\documentclass[11pt]{preprint}

\usepackage[a4paper, left=2.65cm, right=2.65cm, top=2.75cm, bottom=2.75cm]{geometry}
\usepackage[full]{textcomp}
\usepackage[osf]{newtxtext}
\usepackage{cabin}
\usepackage{enumerate}
\usepackage{enumitem}

\usepackage{bm}
\usepackage{dsfont}
\usepackage{stmaryrd}

\usepackage[varqu,varl]{inconsolata}
\usepackage[cal=boondoxo]{mathalfa}

\usepackage{comment}
\excludecomment{versiona}

\usepackage{cancel}
\usepackage{mcomment}
\usepackage{hyperref}
\usepackage{mathrsfs}
\usepackage{booktabs}
\usepackage{caption}

\usepackage{tikz}
\usetikzlibrary{decorations.pathreplacing,decorations.markings}
\usepackage{amsmath} 
\usepackage{mhequ} 
\usepackage{mhsymb} 
\usepackage{mhenvs} 
\usepackage{microtype}
\usepackage{amssymb} 

\usepackage{wasysym}
\usepackage{centernot}

\usepackage{oldgerm}

\renewcommand{\R}{\mathbb{R}}
\renewcommand{\C}{\mathbb{C}}
\renewcommand{\N}{\mathbb{N}}

\newcommand{\T}{\mathbb{T}}
\renewcommand{\Z}{\mathbb{Z}}
\renewcommand{\E}{\mathbb{E}}

\newcommand{\cA}{\mathcal{A}}
\newcommand{\cB}{\mathcal{B}}
\newcommand{\cC}{\mathcal{C}}

\newcommand{\cE}{\mathcal{E}}
\newcommand{\cF}{\mathcal{F}}

\newcommand{\cL}{\mathcal{L}}

\newcommand{\cN}{\mathcal{N}}
\newcommand{\cO}{\mathcal{O}}

\newcommand{\cR}{\mathcal{R}}
\newcommand{\cS}{\mathcal{S}}
\newcommand{\cT}{\mathcal{T}}

\newcommand{\cV}{\mathcal{V}}

\newcommand{\cZ}{\mathcal{Z}}

\newcommand{\fe}{\mathfrak{e}}
\newcommand{\ff}{\mathfrak{f}}

\newcommand{\fh}{\mathfrak{h}}

\newcommand{\fs}{\mathfrak{s}}

\newcommand{\fw}{\mathfrak{w}}

\newcommand{\dd}{\mathrm{d}}

\newcommand{\eff}{\mathrm{eff}}

\newcommand{\fsi}{\mathsf{i}}

\newcommand{\fsP}{\mathsf{P}}
\newcommand{\fsQ}{\mathsf{Q}}
\newcommand{\fsR}{\mathsf{R}}

\def\one{\mathrm{(I)}}
\def\two{\mathrm{(II)}}
\def\three{\mathrm{(III)}}
\def\four{\mathrm{(IV)}}
\def\five{\mathrm{(V)}}
\def\six{\mathrm{(VI)}}
\def\seven{\mathrm{(VII)}}

\newcommand{\vertiii}[1]{{\vert\kern-0.25ex\vert\kern-0.25ex\vert #1 
    \vert\kern-0.25ex\vert\kern-0.25ex\vert}}

\newcommand{\1}{\text{\usefont{U}{dsrom}{m}{n}1}}

\newcommand{\vphi}{\varphi}

\colorlet{darkblue}{blue!90!black}
\colorlet{darkred}{red!90!black}
\colorlet{darkgreen}{green!50!black}
\colorlet{darkyellow}{yellow!90!black}

\newcommand{\fock}[2]{
    \def\order{#2}
    \ifx\order\empty
        \Gamma L^2_{#1}
    \else
        \Gamma H^{#2}_{#1}
    \fi
}

\newcommand{\fockd}[2]{
    \def\order{#2}
    \ifx\order\empty
        \Gamma \ell^2_{#1}
    \else
        \Gamma h^{#2}_{#1}
    \fi
}

\newcommand{\fockdo}[2]{
    \def\order{#2}
    \ifx\order\empty
        \Gamma_0 \ell^2_{#1}
    \else
        \Gamma_0 h^{#2}_{#1}
    \fi
}

\newcommand{\fockb}[2]{
    \def\order{#2}
    \ifx\order\empty
        \bar \Gamma L^2_{#1}
    \else
        \bar \Gamma H^{#2}_{#1}
    \fi
}

\newcommand{\fockdb}[2]{
    \def\order{#2}
    \ifx\order\empty
        \bar \Gamma \ell^2_{#1}
    \else
        \bar \Gamma h^{#2}_{#1}
    \fi
}

\newcommand{\fockdob}[2]{
    \def\order{#2}
    \ifx\order\empty
        \bar \Gamma_0 \ell^2_{#1}
    \else
        \bar \Gamma_0 h^{#2}_{#1}
    \fi
}

\newcommand{\fockcore}[3]{
 \def\order{#2}
    \ifx\order\empty
        \boldsymbol{\Gamma}_{#1} \mathbf{L}^2_{#3}
    \else
        \boldsymbol{\Gamma}_{#1} \mathbf{H}^{#2}_{#3}
    \fi
}

\makeatletter
\def\DeclareSymbol#1#2#3{\expandafter\gdef\csname MH@symb@#1\endcsname{\tikzsetnextfilename{symbol#1}\tikz[baseline=#2,scale=0.15,draw=symbols]{#3}}\expandafter\gdef\csname MH@symb@#1s\endcsname{\scalebox{0.7}{\tikzsetnextfilename{symbol#1s}\tikz[baseline=#2,scale=0.15,draw=symbols]{#3}}}}
\def\<#1>{\csname MH@symb@#1\endcsname}
\makeatother

\definecolor{Red}{rgb}{1,0,0}
\definecolor{Blue}{rgb}{0,0,1}
\definecolor{Olive}{rgb}{0.41,0.55,0.13}
\definecolor{Yarok}{rgb}{0,0.5,0}
\definecolor{Green}{rgb}{0,1,0}
\definecolor{MGreen}{rgb}{0,0.8,0}
\definecolor{DGreen}{rgb}{0,0.65,0}
\definecolor{Yellow}{rgb}{1,1,0}
\definecolor{Cyan}{rgb}{0,1,1}
\definecolor{Magenta}{rgb}{1,0,1}
\definecolor{Orange}{rgb}{1,.5,0}
\definecolor{Violet}{rgb}{.5,0,.5}
\definecolor{Purple}{rgb}{.75,0,.25}
\definecolor{Brown}{rgb}{.75,.5,.25}
\definecolor{Grey}{rgb}{.7,.7,.7}
\definecolor{Black}{rgb}{0,0,0}
\definecolor{dr}{rgb}{0.8,0,0}
\definecolor{db}{rgb}{0,0,0.8}

\newcounter{list_size}
\newcommand{\connect}[4][dr]{
	\edef\myleftx{10000}
	\edef\myrightx{-10000}
	\edef\mycentery{0}
	\setcounter{list_size}{0}
	\foreach \j in {#2}{\stepcounter{list_size}};
	\edef\n{\arabic{list_size}}
	\ifthenelse{\n=1}
	{
		\draw (#3#2) -- (#4#2);
	}
	{
		\foreach \j in {#2} {
			\path (#3\j); \pgfgetlastxy{\XCoord}{\YCoord};
			\pgfmathsetmacro{\lx}{min(\myleftx,\XCoord)};
			\pgfmathsetmacro{\rx}{max(\myrightx,\XCoord)};
			\pgfmathsetmacro{\cy}{\mycentery+\YCoord/(2*\n)};
			\path (#4\j); \pgfgetlastxy{\XCoord}{\YCoord};
			\pgfmathsetmacro{\cy}{\cy+\YCoord/(2*\n)};
			\global\let\myleftx=\lx
			\global\let\myrightx=\rx
			\global\let\mycentery=\cy
		}
		\foreach \j in {#2} {
			\draw (#3\j) -- (#4\j);
		}
		\draw[thick,#1] (\myleftx pt,\mycentery pt) -- (\myrightx pt,\mycentery pt);
		\foreach \j in {#2} {
			\path (#3\j); \pgfgetlastxy{\XCoord}{\YCoord};
			\node at (\XCoord,\mycentery pt) [simple,fill=#1] {};
		}
	}
}

\def\restriction#1#2{\mathchoice
              {\setbox1\hbox{${\displaystyle #1}_{\scriptstyle #2}$}
              \restrictionaux{#1}{#2}}
              {\setbox1\hbox{${\textstyle #1}_{\scriptstyle #2}$}
              \restrictionaux{#1}{#2}}
              {\setbox1\hbox{${\scriptstyle #1}_{\scriptscriptstyle #2}$}
              \restrictionaux{#1}{#2}}
              {\setbox1\hbox{${\scriptscriptstyle #1}_{\scriptscriptstyle #2}$}
              \restrictionaux{#1}{#2}}}
\def\restrictionaux#1#2{{#1\,\smash{\vrule height .8\ht1 depth .85\dp1}}_{\,#2}}

\newcommand{\rap}{{\mathrm{rap}}}

\let\ss=\smallskip
\let\ms=\medskip

\usepackage{tikz-cd}
\usepackage{graphicx}

\def\Log{\mathrm{Log}}

\title{Superdiffusive Invariance Principle for the two-dimensional Asymmetric Simple Exclusion Process}

\begin{document}

\date{}

\maketitle

\vspace{-2cm}
\noindent{\large \bf Quentin Moulard}

\noindent{Technical University of Vienna, Austria\\}
\noindent\email{quentin.moulard@tuwien.ac.at}
\newline

\bigskip\noindent
\begin{abstract}
We consider the Asymmetric Simple Exclusion Process (ASEP) at its critical dimension $d = 2$, started from the Bernoulli product measure with density $\rho \in (0,1)$.
We first derive the sharp $\log(t)^{2/3}$ divergence of the bulk diffusivity, predicted via mode-coupling theory in [H. van Beijeren, R. Kutner, H. Spohn, Phys. Rev. Lett., 1985], including the exact prefactor, thereby removing the subleading diverging multiplicative corrections in the Tauberian bounds of [H.-T. Yau, Ann. of Math., 2004].
More importantly, we establish an invariance principle for the density fluctuation field on the corresponding superdiffusive scales.
On the time scale $\tau$, we rescale space by $\sqrt{\tau} \log(\tau)^{1/3}$ in the direction of the drift and by $\sqrt{\tau}$ in the transverse direction, and prove functional convergence of the rescaled fluctuation process to the solution of a conservative stochastic heat equation with additive noise. This is the same explicit Gaussian Fixed Point obtained in [G. Cannizzaro, Q. Moulard, F. Toninelli, arXiv:2501.00344, 2025] for the two-dimensional Stochastic Burgers Equation (SBE). We also deduce a superdiffusive central limit theorem for the second-class particle under the same anisotropic rescaling. The result holds for any finite-range jump kernel with nonzero drift and uniaxial asymmetry.

This is the first space-time scaling limit result for a nonequilibrium interacting particle system at the critical dimension, where Gaussian fluctuations emerge despite logarithmic corrections to the scaling. The proof uses an approximate fluctuation-dissipation relation, which is obtained by combining a sharp quantitative comparison between the ASEP and SBE generators with refined estimates on the SBE Ansatz. An additional key novelty is a tightness argument that overcomes the limitations of the It\^o trick on superdiffusive scales.
\end{abstract}

\bigskip


\setcounter{tocdepth}{3} 
\tableofcontents

\section{Introduction}

Understanding how macroscopic behaviour emerges from microscopic dynamics is a central question in the theory of interacting particle systems and in nonequilibrium statistical mechanics~\cite{Liggett, KL99, KLO}.
For systems with a conserved particle density, a first step is to establish a hydrodynamic limit describing the evolution of the macroscopic density profile. The natural next step is to investigate the large-scale evolution of the density fluctuations around this profile.
An important question is how these fluctuations depend on the spatial
dimension and on the reversibility of the microscopic dynamics.
In this paper, we establish a superdiffusive invariance principle for the stationary density fluctuations of the two-dimensional Asymmetric Exclusion Process (ASEP), at its critical dimension and with fixed microscopic asymmetry.
\ss

The ASEP is a paradigmatic example of a driven diffusive system.
The particles occupy the sites of $\Z^d$ and evolve in continuous time: a particle at $x$ attempts to jump to $y$ at rate $r(y-x)$, where the jump kernel $r$ has finite range, and the jump is performed only if the target site $y$ is empty. This exclusion rule encodes the interaction between the particles.
For every density $\rho \in (0,1)$, the Bernoulli product measure $\nu_\rho$, under which each site is occupied independently with probability $\rho$, is invariant for the dynamics.
Here, however, we assume that the drift $\fw := \sum_{z \in \Z^d} r(z) \, z$ is nonzero, so that the stationary dynamics is non-reversible and carries the macroscopic particle current $j(\rho) := \rho(1-\rho) \, \fw$.
Despite the simplicity of the invariant measure, the space-time fluctuations are highly non trivial and reflect both the interactions and the irreversible particle current.
\ss

To put the effect of the asymmetry into perspective, recall that for the symmetric simple exclusion process, the stationary density fluctuation field has a diffusive Gaussian scaling limit in every dimension~\cite{KL99, KLO}. The introduction of a non-zero drift leads to a much richer picture, whose nature strongly depends on the spatial dimension $d$.
The propagation of density fluctuations can be probed by a second-class particle, which tracks the single discrepancy between two copies of ASEP evolving under the basic coupling, started from configurations differing at a single site~\cite{Liggett}.
The second-class particle has velocity $v := j'(\rho) = (1 - 2 \rho) \fw$, and its spreading can be measured by the bulk diffusivity matrix $D(t)$~\eqref{e:DiffusivityDef}. More precisely, if we define the longitudinal diffusivity as
\begin{equ} \label{e:Dparallel}
D_{\parallel}(t) := \fe^T D(t) \fe \, , \quad \fe := \frac{\fw}{|\fw|} \, ,
\end{equ}
then $t D_{\parallel}(t)$ gives the mean-square displacement of the second-class particle in the direction of the drift, after centring by $t v$.
Thus, divergence of $D_{\parallel}(t)$ as $t \to \infty$ is a signature of superdiffusive behaviour. On the basis of mode-coupling theory, van Beijeren, Kutner and Spohn~\cite{vBKS} predicted for driven diffusive systems such as ASEP that $D_{\parallel}(t)$ is of order $1$, $t^{1/3}$ and $\log(t)^{2/3}$ in dimensions $d \geq 3$, $d = 1$ and $d = 2$, respectively.
\ms

In dimensions $d \geq 3$, the asymmetry does not change the qualitative diffusive picture.
The mathematical understanding of this regime developed through the study of diffusive hydrodynamic limits for small perturbations of constant density profiles~\cite{EMY94}, the analysis of the diffusion coefficient~\cite{LOY96}, and the construction of a \emph{fluctuation-dissipation} relation~\cite{LY97}. In the reference frame moving with the characteristic velocity $v$, the density fluctuation process converges under diffusive scaling to the solution of a conservative stochastic heat equation with a renormalised diffusion matrix~\cite{CLO01, LOV04}, which also yields a central limit theorem for the second-class particle. These works show how a non-reversible microscopic dynamics can produce Gaussian fluctuations at large scale, with the effect of the asymmetry retained in the effective diffusion coefficient; see also~\cite{KLO}.
\ss

In dimension $d=1$, by contrast, the nonlinear current is relevant and produces KPZ behaviour.
For nearest-neighbour ASEP, coupling estimates for the second-class particle yield the predicted $t^{1/3}$ growth of the diffusivity~\cite{BS10}; the same two-sided bound is also known in a Tauberian sense for ASEP~\cite{QV08}. At the level of the height function fluctuations, convergence to the KPZ fixed point under the $1:2:3$ scaling has been established for TASEP~\cite{MQR21} and, more recently, for nearest-neighbour ASEP~\cite{ACH24b, ACH24}. For TASEP,~\cite{RV25} established a path scaling limit for the second-class particle, which is described by a competition interface in the directed landscape. The identification of these scaling limits relies on the integrable structures of the models. 

A complementary approach concerns weak asymmetry: when the drift is of order $\tau^{-1/4}$ on the time scale $\tau$, the density fluctuations retain diffusive scaling and converge to the one-dimensional Stochastic Burgers Equation, as first proved for WASEP in~\cite{BG97}. Subsequent works developed energy solution methods for broader classes of weakly asymmetric particle systems~\cite{GJ14, GJS15}, with uniqueness of the limiting energy solutions established
in~\cite{GP18}, and obtained KPZ equation limits for further exclusion models~\cite{CT17}. More recently,~\cite{HMW25} derived the KPZ equation from WASEP through the theory of regularity structures, without using the microscopic Cole-Hopf transform.
\ss

At the critical dimension $d=2$, the resolvent approach of~\cite{LQSY04} established superdiffusivity, and the breakthrough work~\cite{Yau04} obtained the predicted logarithmic divergence of $D_{\parallel}(t)$ at density $\rho = 1/2$, up to divergent subleading multiplicative corrections and in a Tauberian formulation: as $\mu \searrow 0$,
\begin{equ} \label{e:Yau}
\mu^2 \int_0^\infty e^{-\mu t} \, t D_\parallel(t) \, \dd t = \log(1/\mu)^{2/3} \exp\big(O(|\log\log\log(1/\mu)|^2) \big).
\end{equ}
These estimates identify the logarithmic correction of the spreading scales, but do not determine the limiting space-time law of the density fluctuations.
The mode-coupling predictions of~\cite{vBKS} do not make an explicit prediction on the scaling limit of the density fluctuation process. However, it is worth mentioning that their computations rely on the assumption that the density fluctuation process is Gaussian: then, they deduce that the density-density correlations after time $t$ must have an anisotropic Gaussian profile, with longitudinal spreading of order $\sqrt{t} \log(t)^{1/3}$ and transversal spreading of order $\sqrt{t}$.
Numerical evidence for this profile and logarithmically enhanced spreading was obtained in~\cite{KNSS18}. 
This raises a stronger question: does the density fluctuation process itself have a Gaussian scaling limit? Establishing such a limit for the $2d$ ASEP is the main purpose of the present work.
\ms

Our principal result is a superdiffusive invariance principle for the stationary density fluctuation field (Theorem~\ref{th:CLT}). On the time scale $\tau \to \infty$, we work in the reference frame moving with the characteristic velocity $v$, and rescale space by $\sqrt{\tau} \log(\tau)^{1/3}$ longitudinally and by $\sqrt{\tau}$ transversally. Under this rescaling, the properly normalised fluctuation process converges in law to the solution of an anisotropic conservative stochastic heat equation with explicit renormalised coefficients. The convergence holds in the space of distribution-valued càdlàg paths and, in particular, includes temporal tightness. It holds at every density $\rho \in (0,1)$ and for the class of jump kernels with uniaxial asymmetry specified below.

As a consequence of the convergence of the fluctuation field, we deduce a superdiffusive central limit theorem for the second-class particle (Theorem~\ref{th:CLT2ndClass}). We also establish the sharp diffusivity asymptotic (Theorem~\ref{th:Diff})
\begin{equ}
D_\parallel(t) \sim d^{\lim} \, \log(t)^{2/3} \, , \mbox{ as } t \to \infty \, , 
\end{equ}
where $d^{\lim} > 0$ is explicit and agrees with the coefficient predicted by mode-coupling theory.
\ss

From a broader perspective, our result concerns the emergence of Gaussian fluctuations at the critical dimension despite logarithmic corrections to the scaling. At equilibrium, this picture has been established for important lattice systems, including the $4d$ Ising and $\phi^4_4$ models~\cite{BBS14,ADC21}. The present work can be seen as the first dynamical, nonequilibrium instance of this phenomenon for a microscopic interacting particle system.

The proof does not rely on integrability. Its main novelties are an asymptotically sharp comparison between the dynamics of the ASEP and of the two-dimensional Stochastic Burgers Equation, and a tightness argument that overcomes the limitations on superdiffusive scales of the Itô trick (the $L^p$ version of the celebrated Kipnis-Varadhan lemma, see~\cite{GJ13}). We discuss these ingredients in more detail in Section~\ref{sec:SBE}. We also give a brief overview of recent progress on nonequilibrium stochastic models at the critical dimension in Section~\ref{sec:Literature}.

\subsection{The model}

We consider the ASEP on $\Z^2$, that is, the Markov process $(\eta_t)_{t \geq 0}$ taking values in $\Omega \eqdef \{0,1\}^{\Z^2}$ whose generator acts on local functions $f \colon \Omega \to \R$ as
\begin{equ} \label{e:GenLocal}
L f(\eta) = \sum_{x, z \in \Z^2} r(z) \, \eta(x) \big(1 - \eta(x+z)\big) \big(f(\eta^{x,x+z}) - f(\eta)\big),
\end{equ}
where $r \colon \Z^2 \to \R_+$ is the jump rate function and is assumed to have finite support, and $\eta^{x,y}$ is the configuration obtained from $\eta$ by swapping $\eta(x)$ and $\eta(y)$. For $\eta \in \Omega$, we denote by $\mathbf P_\eta$ the law of the ASEP started from $\eta_0 = \eta$. We fix $\rho \in (0,1)$, set $\nu_\rho := \bigotimes_{x \in \Z^2} \operatorname{Ber}(\rho)$ and denote by $\mathbf P_{\nu_\rho}$ the law of the stationary ASEP started from $\eta_0 \sim \nu_\rho$.
\ss

The value of $r(0)$ is irrelevant, and we henceforth assume that $r(0)=0$. For $z\in\Z^2$, we define the symmetric and antisymmetric parts of the jump rates by
\begin{equ} \label{e:Sym&AntiSymRate}
s(z) \eqdef \frac{r(z) + r(-z)}{2} \, , \qquad a(z) \eqdef \frac{r(z) - r(-z)}{2} \, .
\end{equ}
Two important associated objects are the covariance matrix and the drift of the jump process:
\begin{equ} \label{e:D0&drift}
D_0 \eqdef \sum_{z \in \Z^2} s(z) \, z z^T \in S_2^+(\R) \, ,
\qquad \fw \eqdef \sum_{z \in \Z^2} a(z) \, z \in \R^2 \, ,
\end{equ}
which are equivalently given by $D_0 = \sum_z r(z) \, z z^T$ and $\fw = \sum_z r(z) \, z$. We assume that $D_0$ is non-degenerate, that is, the jump vectors are not confined to a line, and that $\fw \neq 0$, that is, the jump process has a non-zero drift. For simplicity, we also assume that the jump kernel is irreducible, in the sense that the lattice generated by $\mathrm{Supp}(r)$ is $\Z^2$. This assumption entails no loss of generality, since a reducible dynamics can always be decomposed into finitely many independent irreducible dynamics. More importantly, we assume that the asymmetry is uniaxial in the following sense.

\begin{assumption}[Uniaxial asymmetry] \label{ass:Unidirectional}
The support of the antisymmetric part of the jump kernel is collinear. Equivalently, $\mathrm{Supp}(a) \subset \R \fw$.
\end{assumption}
\ss

We discuss the role of Assumption~\ref{ass:Unidirectional} in Remark~\ref{rmk:Unidirectional} below. Figure~\ref{fig:jump-kernels} illustrates the scope of Assumption~\ref{ass:Unidirectional}: in particular, it does not require the jump kernel to be nearest-neighbour or to have an axis of symmetry. 

\begin{figure}[t]
    \centering
    \resizebox{\textwidth}{!}{%
\begin{tikzpicture}[
    x=0.82cm,
    y=0.82cm,
    >=Latex,
    lattice point/.style={circle,fill=black!35,inner sep=1.15pt},
    origin/.style={circle,fill=black,inner sep=2.0pt},
    rate arrow/.style={-{Latex[length=2.8mm,width=2.0mm]},line cap=round},
    axis line/.style={densely dashed,black!45,line width=0.45pt},
    panel title/.style={font=\small\bfseries,align=center},
    rate label/.style={font=\small,fill=white,inner sep=1pt}
]

\newcommand{\asepdrawlattice}{%
  \foreach \i in {-2,-1,0,1,2}{%
    \foreach \j in {-2,-1,0,1,2}{%
      \node[lattice point] at (\i,\j) {};%
    }%
  }%
  \node[origin] at (0,0) {};%
}

\begin{scope}[shift={(0,0)}]
  \asepdrawlattice
  \draw[axis line] (-2.55,0) -- (2.55,0);
  \node[font=\scriptsize,black!60,above] at (1.65,0.02)
      {$\mathbb R \fw$};

  \draw[rate arrow,line width=1.35pt] (0.14,0) -- (0.94,0)
      node[rate label,midway,above=2pt] {$p$};
  \draw[rate arrow,line width=0.75pt] (-0.14,0) -- (-0.94,0)
      node[rate label,midway,above=2pt] {$q$};
  \draw[rate arrow,line width=0.95pt] (0,0.14) -- (0,0.94)
      node[rate label,midway,right=2pt] {$1$};
  \draw[rate arrow,line width=0.95pt] (0,-0.14) -- (0,-0.94)
      node[rate label,midway,right=2pt] {$1$};

  \node[panel title,text width=5.7cm] at (0,3.05)
      {(a) Nearest-neighbour ASEP\\with horizontal bias $p>q$};
\end{scope}

\begin{scope}[shift={(7.0,0)}]
  \asepdrawlattice
  \draw[axis line] (-2.45,-1.225) -- (2.45,1.225);
  \node[font=\scriptsize,black!60,above,rotate=26.565]
      at (-1.40,-0.70) {$\mathbb R \fw$};

  \draw[rate arrow,line width=0.85pt] (0.14,0) -- (0.94,0)
      node[rate label,midway,below=2pt] {$1$};
  \draw[rate arrow,line width=0.85pt] (-0.14,0) -- (-0.94,0)
      node[rate label,midway,below=2pt] {$1$};
  \draw[rate arrow,line width=0.85pt] (0,0.14) -- (0,0.94)
      node[rate label,midway,left=2pt] {$1$};
  \draw[rate arrow,line width=0.85pt] (0,-0.14) -- (0,-0.94)
      node[rate label,midway,left=2pt] {$1$};
  \draw[rate arrow,line width=1.35pt] (0.13,0.065) -- (1.90,0.95)
      node[rate label,pos=0.72,above=3.5pt,rotate=26.565] {$p$};

  \node[panel title,text width=5.8cm] at (0,3.05)
      {(b) SSEP plus a TASEP jump};
\end{scope}

\begin{scope}[shift={(14.0,0)}]
  \asepdrawlattice
  \draw[axis line] (-1.82,-1.82) -- (1.82,1.82);
  \node[font=\scriptsize,black!60,above,rotate=45] at (1.32,1.32)
      {$\mathbb R \fw$};

  \draw[rate arrow,line width=1.25pt] (0.14,0) -- (0.94,0)
      node[rate label,midway,below=2pt] {$1$};
  \draw[rate arrow,line width=1.25pt] (0,0.14) -- (0,0.94)
      node[rate label,midway,left=2pt] {$1$};

  \node[panel title,text width=5.2cm] at (0,3.05)
      {(c) North-East TASEP};
\end{scope}

\end{tikzpicture}%
    }
    \caption{Examples of jump kernels: $(a)$ and $(b)$ satisfy Assumption~\ref{ass:Unidirectional}, but $(c)$ does not.}
    \label{fig:jump-kernels}
\end{figure}
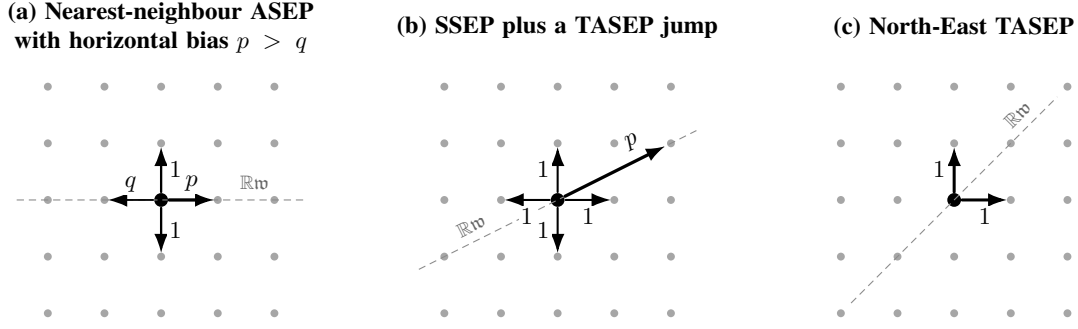

We define $\lambda \eqdef |\fw|$, $\fe \eqdef \fw / |\fw|$, and let $\ff$ be the unique vector such that $(\fe, \ff)$ is a positively oriented orthonormal basis of $\R^2$. For $x, y \in \R^2$, we also write $D_0[x,y] \eqdef x^T D_0 y$.

\subsection{Main results: sharp diffusivity asymptotics and superdiffusive invariance principle}

We use the second-class particle introduced above with the following initialisation: first-class particles occupy the sites of $\Z^2 \setminus \{0\}$ independently with probability $\rho$, while a second-class particle starts at the origin. We denote its position by $X_t$ and the law of the resulting process by $\mathbf P_{\nu_\rho}^{\mathrm{II}}$.
The basic coupling~\cite{Liggett} shows that the second-class particle is distributed as follows, for every $t \geq 0$ and $x \in \Z^2$,
\begin{equ} \label{e:2PointFunction}
\mathbf P_{\nu_\rho}^{\mathrm{II}}(X_t=x) = \frac{1}{\rho \, (1-\rho)} \, \mathbf E_{\nu_\rho}\big[(\eta_0(0)-\rho)(\eta_t(x)-\rho)\big] \, .
\end{equ}
A simple computation shows that its expectation follows the characteristic velocity
\begin{equ} \label{e:TransportTerm2ndClass}
\mathbf E_{\nu_\rho}^{\mathrm{II}}[X_t] = tv, \qquad v\eqdef(1-2\rho) \, \fw \, .
\end{equ}
The spreading of $X_t$ around this characteristic is then quantified by the bulk diffusivity matrix
\begin{equs} \label{e:DiffusivityDef}
D(t) &\eqdef \frac{1}{t} \, \mathbf{Cov}_{\nu_\rho}^{\mathrm{II}}(X_t) = \frac{1}{t \, \rho(1-\rho)}\sum_{x\in\Z^2} \mathbf E_{\nu_\rho}\big[(\eta_0(0)-\rho)(\eta_t(x)-\rho)\big] \ (x-tv) \, (x-tv)^T \, .
\end{equs}

Our first result confirms the conjecture of~\cite{vBKS} and upgrades the estimate of~\cite{Yau04} to a sharp asymptotic of the full diffusivity matrix.

\begin{theorem}[Sharp diffusivity asymptotics] \label{th:Diff}
As $t\to\infty$,
\begin{equ} \label{e:Diff}
D(t) = D_0 + \big(d^{\lim} + o(1)\big) \, \log(t)^{2/3}\, \fe \fe^T \, ,
\end{equ}
where
\begin{equ} \label{e:dlim}
d^{\lim} \eqdef \Big(\frac{3}{2\pi}\Big)^{2/3}\frac{\lambda^{4/3} \, \rho^{2/3} (1 - \rho)^{2/3}}{D_0[\ff,\ff]^{1/3}} \, .
\end{equ}
\end{theorem}

The absolute value of the $o(1)$ term in~\eqref{e:Diff} is bounded by $\log(t)^{-1/6+o(1)}$; see~\eqref{e:DiffStrong}.

\begin{remark}

The explicit form of $d^{\lim}$ has noteworthy features. First, after the adequate identification of parameters (see~\eqref{e:ASEPFormalSBE}), it coincides not only with the prefactor appearing in the mode-coupling computations of~\cite{vBKS}, but also with the diffusivity prefactor obtained for the two-dimensional stochastic Burgers equation in~\cite{CMT25}. More strikingly, the $\lambda^{4/3}$ proportionality is shared by one-dimensional models in the KPZ universality class, even though their diffusivity grows as $t^{1/3}$ rather than logarithmically. In addition, the symmetric jump rates enter the leading prefactor only through the transverse component $D_0[\ff,\ff]$; in particular it is independent of $D_0[\fe,\fe]$. The scaling argument of~\cite[Remark~1.2]{CMT25} draws a connection between this independence and the powers $\lambda^{4/3}$ and $D_0[\ff,\ff]^{-1/3}$.
\end{remark}

Theorem~\ref{th:Diff} identifies the natural spatial scales of the second-class particle: after centring by $tv$, its fluctuations are of order $\sqrt{t} \log(t)^{1/3}$ in the direction of the drift and of order $\sqrt{t}$ in the transverse direction. The next natural step is to determine the shape of the rescaled distribution of the second-class particle: our next result shows that it is asymptotically Gaussian. For $t \gg 1$, define the scaling matrix by
\begin{equ} \label{e:ScalingMatrixAsymp}
\Sigma^t \eqdef \sqrt{t} \, \Big(\log(t)^{1/3} \, \fe \fe^T + \ff\ff^T\Big) \, .
\end{equ}

\begin{theorem}[Central limit theorem for the second class particle] \label{th:CLT2ndClass}
Under $\mathbf P_{\nu_\rho}^{\mathrm{II}}$,
\begin{equ} \label{e:CLT2ndClass}
(\Sigma^t)^{-1}(X_t-tv) \xrightarrow[t\to\infty]{\mathrm{law}} \mathcal N(0,D^{\lim}) \, ,
\end{equ}
where $\mathcal N(0,D^{\lim})$ denotes the centred Gaussian law on $\R^2$ with covariance matrix
\begin{equ} \label{e:Dlim}
D^{\lim} \eqdef d^{\lim} \, \fe \fe^T + D_0[\ff,\ff] \, \ff \ff^T \, .
\end{equ}
\end{theorem}

We stress that the limiting covariance matrix $D^{\lim}$ is consistent with Theorem~\ref{th:Diff}, since it is the limit of $(\Sigma^t)^{-1} \, t D(t) \, (\Sigma^t)^{-1}$ as $t \to \infty$.

\begin{remark}
Theorem~\ref{th:CLT2ndClass} identifies the scaling limit of the second-class particle for one time marginal. It would be interesting to investigate joint convergence for several time marginals, or even a functional scaling limit. We stress however that this is an highly nontrivial question: in particular, we are not aware of any scaling limit results beyond one time marginal for the second-class particle in the ASEP in dimension $d \geq 3$.
\end{remark}
\ms

We now turn to our principal result, which identifies the space-time scaling limit of the density fluctuation field. Letting $\tau \gg 1$ denote the time scale, we use the spatial scaling matrix $\Sigma^\tau$ from~\eqref{e:ScalingMatrix}. For a configuration $\eta \in \Omega$, we define the centred and normalised occupation variables by
\begin{equ} \label{e:Sigma}
\sigma(x) \eqdef \frac{\eta(x)-\rho}{\sqrt{\rho(1 - \rho)}}.
\end{equ}
For $t\geq0$, let $Z_t^\tau[\eta]\in\cS'(\R^2)$ be the tempered distribution given, for every $\vphi\in\cS(\R^2)$, by
\begin{equ} \label{e:Ztau}
Z_t^\tau[\eta](\vphi) \eqdef \det(\Sigma^\tau)^{-1/2} \sum_{x\in\Z^2} \sigma(x) \, \vphi\big((\Sigma^\tau)^{-1}(x - \tau t v)\big),
\end{equ}
and we define the rescaled density fluctuation process by
\begin{equ} \label{e:Ytau}
Y_t^\tau \eqdef Z_t^\tau[\eta_{\tau t}] \, .
\end{equ}
The normalisation in~\eqref{e:Ztau} is chosen so that, when $\eta \sim \nu_\rho$, the random distribution $Z_t^\tau[\eta]$ converges in law to spatial white noise, that is the centred Gaussian distribution $\zeta \in \cS'(\R^2)$ with covariance
\begin{equ} \label{e:whitenoise}
\mathbb E[\zeta(\vphi) \, \zeta(\psi)] = \langle\vphi,\psi\rangle_{L^2(\R^2)} \, .
\end{equ}
We denote by $\mathbb P$ the law of the spatial white noise.

The limiting process turns out to be the solution of the anisotropic conservative stochastic heat equation (sometimes referred to as an Ornstein-Uhlenbeck process)
\begin{equ} \label{e:Ylim}
\partial_tY^{\lim} = \frac{1}{2} \nabla \cdot D^{\lim} \nabla \, Y^{\lim} + \nabla \cdot (D^{\lim})^{1/2} \, \vec{\xi} \, ,
\end{equ}
where $\vec\xi$ is a two-dimensional space-time white noise. For $w \in \cS'(\R^2)$, we denote by $\mathbf P_w^{\lim}$ the law of the solution of~\eqref{e:Ylim} started from $Y_0^{\lim} = w$. A direct computation shows that spatial white noise is invariant for~\eqref{e:Ylim}; hence, let $\mathbf P_{\mathbb P}^{\lim}$ be the stationary law of the process $Y_\cdot^{\lim}$ starting from $Y_0^{\lim} \sim \mathbb P$.
\ss

We can now state the superdiffusive invariance principle for the fluctuation process.

\begin{theorem}[Superdiffusive invariance principle] \label{th:CLT}
As $\tau \to \infty$, the law of $(Y_t^\tau)_{t \geq 0}$ under $\mathbf P_{\nu_\rho}$ converges weakly to the law of $(Y_t^{\lim})_{t \geq 0}$ under $\mathbf P_{\mathbb P}^{\lim}$, in the Skorokhod topology of $D(\R_+, \cS'(\R^2))$. Equivalently, for every bounded continuous functional $F \colon D(\R_+,\cS'(\R^2)) \to \R$,
\begin{equ} \label{e:CLTFunctional}
\mathbf E_{\nu_\rho}[F(Y_\cdot^\tau)] \xrightarrow[\tau\to\infty]{} \mathbf E_{\mathbb P}^{\lim}[F(Y_\cdot^{\lim})] \, .
\end{equ}

Furthermore, the convergence holds conditionally on the microscopic initial configuration $\eta$ with high probability under $\nu_\rho(\dd \eta)$, in the following sense. For every bounded continuous functional $F$ as above,
\begin{equ} \label{e:CLTFunctionalQuenched}
\int_{\Omega} \nu_\rho(\dd \eta) \, \big|\mathbf E_\eta[F(Y_\cdot^\tau)]-\mathbf E_{Z_0^\tau[\eta]}^{\lim}[F(Y_\cdot^{\lim})]\big| \xrightarrow[\tau\to\infty]{} 0 \, .
\end{equ}
\end{theorem}

In words, the statement~\eqref{e:CLTFunctionalQuenched} shows that the superdiffusive invariance principle is not restricted to the stationary process and holds for typical microscopic initial configurations $\eta$ sampled from $\nu_\rho$, up to starting the limiting process $(Y_t^{\lim})_{t \geq 0}$ from the rescaled fluctuation field $Z_0^\tau[\eta]$.

\begin{remark}
Theorem~\ref{th:CLT} exhibits a characteristic feature of systems at their critical dimension. The microscopic fluctuation field must be observed under a logarithmically superdiffusive scaling, whereas the limiting Gaussian process $Y^{\lim}$ is itself invariant under diffusive scaling. Hence, the logarithmic correction disappears after taking the scaling limit. A similar phenomenon occurs for the two-dimensional Stochastic Burgers Equation~\cite{CMT25} and for the diffusion in a critically correlated incompressible random drift~\cite{ABRK24}.
\end{remark}

\begin{remark} \label{rmk:Unidirectional}
We briefly discuss the role of Assumption~\ref{ass:Unidirectional}. In fact, this assumption is only important here to determine the precise prefactor of the diffusivity asymptotics and the precise scaling limit of the fluctuation process. Indeed, with little extra effort, the techniques developed in this paper can be adapted to show that, without Assumption~\ref{ass:Unidirectional}, there still exists $C \in (0, \infty)$ such that for sufficiently large time $t$
\begin{equ}
C^{-1} \log(t)^{2/3} \leq D_{\parallel}(t) \leq C \log(t)^{2/3} \, ,
\end{equ}
and the laws of the rescaled fluctuation process $(Y_t^\tau)_{t \geq 0}$ under $\mathbf P_{\nu_\rho}$ are still tight in $D(\R_+, \cS'(\R^2))$. For more details, see Remark~\ref{rmk:wRemoval}.
\end{remark}

\subsection{Main novelties and connection with the $2d$ Stochastic Burgers Equation}

\label{sec:SBE}

Our analysis builds on a quantitative comparison between the $2d$ ASEP and the $2d$ Stochastic Burgers Equation (SBE).
This connection is suggested by nonlinear fluctuating hydrodynamics~\cite{vBKS,NLFH,CGMO25}. Starting from the Euler-scale hydrodynamic equation $\partial_t\alpha + \nabla\cdot j(\alpha)=0$~\cite{R91}, expanding the current around the reference density $\rho$, and adding the appropriate diffusion term and conservative noise formally gives, in the reference frame moving with characteristic velocity $v$,
\begin{equ} \label{e:ASEPFormalSBE}
\partial_t Y = \frac{1}{2} \, \nabla \cdot D_0 \nabla Y + \sqrt{\rho(1-\rho)} \, \fw\cdot\nabla(Y^2) + \nabla \cdot \sqrt{D_0} \, \vec{\, \xi} \, .
\end{equ}
Here, $Y$ formally represents the density fluctuation field and $\vec{\xi}$ is a
two-dimensional space-time white noise. After an appropriate linear change of variables, this
takes the form of the $2d$ SBE
\begin{equ} \label{e:FormalSBE}
\partial_t u = \frac{1}{2}\,\Delta u
+ \lambda_{\mathrm{SBE}}\,\partial_1(u^2)
+ \nabla\cdot\vec{\xi} \, .
\end{equ}
This equation is only formal, since putative solutions would be too singular to give a meaning to the quadratic nonlinearity. However, one can first regularise the nonlinearity to make sense of the equation, and then investigate the large-scale behaviour of the system, with the microscopic regularisation kept fixed. In particular, for the regularisations used in~\cite{CGT24,DGHS24,CMT25}, the regularised SBE admits the spatial white noise $\mathbb P$ (or its mollified counterpart) as an invariant measure, which is the continuum analogue of the stationarity of $\nu_\rho$ for the ASEP.
\ss

Recent works have substantially advanced the understanding of the large-scale behaviour of the regularised $2d$ SBE.
First,~\cite{DGHS24} obtained Tauberian bounds on the diffusivity analogous to~\eqref{e:Yau}, but with smaller, though still diverging, multiplicative corrections.
Subsequently,~\cite{CMT25} established the sharp large-time asymptotics of the diffusivity, including its explicit prefactor. After matching the coefficients through~\eqref{e:ASEPFormalSBE}, this prefactor agrees with the one obtained here for the ASEP.
Moreover,~\cite{CMT25} proved that the superdiffusively rescaled SBE solution converges, in finite-dimensional distributions, to a Gaussian process of the same form as $Y^{\lim}$ in~\eqref{e:Ylim}.
In fact, this work constructed an accurate \emph{Ansatz} for the resolvent of the generator, which retains much more information across scales on the dynamics than the mere limiting Gaussian description~$Y^{\lim}$.
The extension in~\cite{CKM26} established a similar superdiffusive central limit theorem for regularised SBEs with nonlinearities $\partial_1 F(u)$, for a large class of smooth functions $F$ with nonzero effective coefficient.
Its proof exploits the quadratic SBE Ansatz of~\cite{CMT25}, which still gives an accurate description of the resolvent of the generalised SBEs.

This naturally leads to a broader question: does the SBE Ansatz also capture the large-scale behaviour of discrete microscopic particle systems expected to belong to the same universality class? Transferring the SBE Ansatz description to the ASEP and
identifying the full fluctuation process raises several difficulties.
\begin{itemize}
    \item \emph{The generators act on different spaces.}
    The ASEP and SBE generators act on $L^2(\nu_\rho)$ and $L^2(\mathbb P)$, respectively. Their chaos decompositions lead to a discrete Fock space with a hard-core constraint in the first case, and a continuum Fock space in the second.
    Consequently, the formal similarity between the generators does not entail a direct comparison: both the exclusion constraint and the discrete nature of the ASEP dynamics produce discrepancies that are absent from the SBE analysis.

    \item \emph{A comparison up to multiplicative constants is not sufficient.}
    The classical hardcore removal estimates of~\cite{LQSY04} compare resolvents up to multiplicative factors depending on the chaos degree. Such bounds do not provide the asymptotic precision needed to identify the limiting coefficients or the limiting process.
    Moreover, even after hardcore removal, the discrete and continuum generators genuinely differ at microscopic scales. Their agreement at large scales therefore does not by itself show that the error is negligible when they act on the SBE Ansatz, which retains information across all scales.

    \item \emph{The direct use of the It\^o trick is insufficient at
    the superdiffusive scale.}
    In contrast to the regimes above the critical dimension or at weak coupling, the nonlinear current cannot be uniformly controlled by the Itô trick under the scaling of interest. Relatedly, no uniform graded sector condition holds.
    Thus, estimates based on the symmetric dynamics alone do not provide the control of the asymmetric contribution needed for tightness. This reflects the mismatch between the diffusive relaxation of the symmetric dynamics and the superdiffusive spreading produced by the asymmetry.

    \item \emph{Finite-dimensional convergence does not provide
    temporal tightness.}
    The convergence of the resolvents obtained for the SBE in~\cite{CMT25,CKM26} does not give tightness. In particular, the convergence does not give sufficiently good control of the increments of the fluctuation process over short time intervals. In addition, the higher-moment estimates needed for tightness would also require stronger control of the chaos expansion of the Ansatz.
    A further difficulty specific to the particle dynamics is that the associated Dynkin martingales have jumps, which must also be accounted for in these estimates.
\end{itemize}

We now explain how these difficulties are overcome, first through an asymptotically sharp comparison of the generators, and then through a fluctuation-dissipation relation and a tightness argument adapted to the superdiffusive scale.
\ss

The first main novelty of the present work is an asymptotically sharp comparison of the generators, particularly when they act on the SBE Ansatz.
To make this comparison meaningful, we introduce bridge operators between the discrete and continuum Fock spaces. The comparison proceeds through an intermediate discrete dynamics without the hard-core constraint, separating the removal of the exclusion constraint from the comparison with the continuum operators; see Sections~\ref{sec:Bridges} and~\ref{sec:FromASEPtoSBE}.

The main challenge is then to make the resulting error vanish
under the superdiffusive rescaling. The classical hard-core removal approach of~\cite{LQSY04} controlled the $H^{-1}$ error between the generators through the $H^1$ norm of the input functional. Here, we refine this approach by controlling the $H^{-1}$ error through \emph{higher-order Sobolev} and \emph{anisotropic} $H^1$ norms of the input. This estimate captures the idea that the two generators tend to better agree on spatial scales that are both much larger than the lattice mesh and elongated in the direction of the drift.
However, such refined comparison of the generators is not sufficient by itself to make the comparison error vanish: we further need quantitative information about the spatial scales carrying the Dirichlet energy of the input functional, that is the SBE Ansatz.
To this end, we establish new higher-order Sobolev estimates on the Ansatz showing that its Dirichlet energy is indeed asymptotically concentrated on scales much larger than the lattice mesh.
Combined with the anisotropic $H^1$ estimates of~\cite{CMT25}, these estimates make the discrete-continuum error vanish and eventually show that the SBE Ansatz gives a faithful description of the ASEP dynamics.

We then use the transferred Ansatz to establish an approximate \emph{fluctuation-dissipation}
relation for the ASEP. This decomposes the nonlinear transport into an effective longitudinal
Laplacian and a fluctuation term represented through a \emph{corrector} that is tightly connected to the SBE Ansatz. A fluctuation-dissipation
approach was developed for ASEP in dimension $d \geq 3$ in~\cite{LY97}, and used in~\cite{CGT24} for the SBE in dimensions $d \geq 3$ and $d = 2$ in the weak-coupling regime.
Here, however, both the emergent dissipation and fluctuation terms contribute at the superdiffusive scale, which requires finer estimates on the corrector. From the fluctuation-dissipation relation, we deduce convergence of the finite-dimensional laws of the fluctuation process.

The second main novelty is a tightness argument that further exploits this fluctuation-dissipation relation. This requires two refinements. First, we strengthen the corrector estimates by introducing exponential weights in the chaos degree, in particular for its full Dirichlet energy. This control of the high chaos levels enables higher-moment estimates. Second, we choose the corrector in a way that depends on the length of the time interval under consideration. This allows us to control its boundary terms and the jumps of its Dynkin martingale, which cannot simply be discarded in a tightness argument. Together, these refinements yield the short-time increment bounds needed to apply the Kolmogorov-Chentsov criterion to the time-integrated nonlinear current.

In this way, we obtain temporal tightness and convergence of the full fluctuation process. Our tightness argument could be applied mutatis mutandis to the SBE studied in~\cite{CMT25,CKM26}, extending the finite-dimensional marginal convergence result therein to an invariance principle.
Furthermore, our discrete-continuum comparison argument and tightness argument seem sufficiently robust to be adapted to other stationary interacting particle systems in the same universality class as the the $2d$ SBE, such as the two-dimensional asymmetric zero-range process or speed-change ASEP (see~\cite{GJ14}).

We refer to Section~\ref{sec:Roadmap} for a more detailed outline of the proof.

\subsection{Relation with the literature}

\label{sec:Literature}

Much of the recent progress for nonequilibrium stochastic models at the critical dimension concerns the \emph{weak-coupling} regime, in which the interaction strength is tuned to zero as the observation scale increases, typically as
\begin{equ} \label{e:WeakCouplingIntro}
\lambda_\tau = \frac{\hat\lambda}{\sqrt{\log(\tau)}} \, .
\end{equ}
Gaussian diffusive limits in this regime have been obtained for a range of critical SPDEs and diffusions in random environment~\cite{CGT24,CET23b,KRYZ25,YY26,CG26}. A different weak-coupling picture, involving a non-Gaussian diffusive limit at a critical point $\hat \lambda_c$, arises for the $2d$ multiplicative stochastic heat equation and
directed polymers~\cite{CSZ17,CSZ20,CSZ23a,CSZ23b,Tsa24}; related questions for the elliptic $4d$ Anderson model are studied in~\cite{GR26,DS26}.
\ss

At fixed coupling, the resolvent methods developed for ASEP~\cite{LQSY04,Yau04} have led to logarithmic superdiffusivity bounds for a range of critical stochastic
models~\cite{DGHS24,CET23a,CHST22,N26}.
Quantitative homogenisation techniques have also yielded optimal-order diffusivity bounds for diffusion in the curl of the two-dimensional Gaussian free field~\cite{CMOW25}, as well as a quenched superdiffusive invariance principle for a broader class of critically correlated incompressible random drifts~\cite{ABRK24}.
\ss

Finally, we mention a remarkable but complementary line of work on one-dimensional asymmetric exclusion processes with fixed heavy-tailed jump rates. In the critical regime with stable index $3/2$ and under the $1:2:3$ scaling,~\cite{GJ18} established tightness and identified the limit points as stationary energy solutions of a fractional stochastic Burgers equation; more recently,~\cite{CCGO26} proved convergence to the unique energy solution of such an equation and extended the result to a multi-species model. These equations are scaling-critical, in the sense that they are themselves invariant under the $1:2:3$ scaling. The mechanism is nevertheless substantially different from the one considered here, since the superdiffusive scaling is already that of the symmetric long-jump dynamics. In particular, the microscopic nonlinear current survives as a continuum nonlinear transport term, rather than being replaced by enhanced linear dissipation and an additional effective noise.

\subsection*{Organisation of the paper}

The remainder of the paper is organised as follows. In Section~\ref{sec:Preliminaries}, we collect the main definitions of functional spaces and operators that are used throughout. We eventually outline the overall proof strategy in subsection~\ref{sec:Roadmap} and describe in more detail the main novelties of the paper. The core of our analysis begins in Section~\ref{sec:FromASEPtoSBE} where we develop a sharp comparison between the ASEP generator, its soft-core version, and the SBE generator, which we believe is of independent interest. Section~\ref{sec:AnsatzSBE} is devoted to establishing a superdiffusive \emph{fluctuation-dissipation} relation which captures the large-scale dynamics produced by the asymmetry. The \emph{Ansatz} of~\cite{CMT25} for the SBE resolvent is introduced in Section~\ref{sec:AnsatzDef}, and used to build a fluctuation-dissipation relation for the SBE in Section~\ref{sec:AnsatzA+}. The latter is then turned into an ASEP fluctuation-dissipation relation in Section~\ref{sec:AnsatzA+ASEP} by taking advantage of the comparison of the generators and of estimates on the $H^{1+\gamma}$ and anisotropic $H^1$ norms of the Ansatz. In Section~\ref{sec:Fluctuation}, we deduce from the fluctuation-dissipation relation the invariance principle of $(Y_t^\tau)_{t \geq 0}$. In particular we are able to prove tightness in Section~\ref{sec:Tightness} by exploiting refined estimates on the Ansatz with exponentially fast decay in chaos. We finally derive in Section~\ref{sec:2ndClass} the central limit theorem for the second-class particle and the sharp diffusivity asymptotics, as by-products of our analysis.

\subsection*{Notations}

To introduce the next definitions, let $d \geq 0$. We write $|\cdot|$ for the canonical Euclidean norm on $\R^d$, $B(x,r)$ for the corresponding open balls of centre $x$ and radius $r$, and $x \cdot y := x^T y$ for the inner product. We also write $|\cdot|_\infty$ for the $\ell^\infty$-norm on $\R^d$ (i.e. the maximum of the absolute values of the coordinates), and $B_\infty(x,r)$ for the corresponding open balls. For $p_1, ..., p_n \in \R^2$, we use the shorthand notation $p_{1:n} := (p_1, ..., p_n) \in \R^{2n}$ (note that $|p_{1:n}|^2 = |p_1|^2 + ... + |p_n|^2$). More generally, for any index set $I = \{i_1 < ... < i_k\} \subset \N$, we set $p_I = (p_{i_1}, ..., p_{i_k})$; in particular, $p_{1:n \setminus i,j}$ stands for the vector obtained from $p_{1:n}$ by omitting $p_i$ and $p_j$. If $M$ is a $2 \times 2$ matrix, we write $M p_{1:n} \eqdef (M p_1, ..., M p_n)$.

Let $\mathbb T^d \eqdef \R^d / 2 \pi \Z^d$ be the $d$-dimensional torus of side length $2 \pi$. We equip $\T^d$ with its natural group structure and Lebesgue measure. We systematically identify functions on $\T^d$ with $2 \pi \Z^d$-periodic functions on $\R^d$. For any $x \in \mathbb T^d$, we denote by $\llbracket x \rrbracket$ its unique representative in $(-\pi, \pi]^{d}$.

Let $\cS(\R^d)$ be the Schwartz space and $\cS'(\R^d)$ be its topological dual. For $f \in \cS(\R^d)$, we define its Fourier transform by
\begin{equ}[e:FourierDef]
\fF(f) \, (p) \eqdef \frac{1}{(2 \pi)^{d/2}} \int_{\R^d} f(x) \, e^{- \iota \, p \cdot x} \, \dd x \, , \quad p \in \R^d \, .
\end{equ}
The Fourier transform extends by duality to $\cS'(\R^d)$. Alongside this, we also define for $f \in \ell^1(\Z^d)$ its Fourier transform by
\begin{equ}[e:PerFourierDef]
\cF(f) \, (p) \eqdef \frac{1}{(2 \pi)^{d/2}} \sum_{x \in \Z^d} f(x) \, e^{- \iota \, p \cdot x} \, , \quad p \in \T^d \, .
\end{equ}
This Fourier transform also extends by duality, from functions on $\Z^d$ growing at most polynomially to distributions on $\T^d$. Let us highlight that the operator $\fF$ (resp. $\cF$) is unitary from $L^2(\R^d, \C)$ to $L^2(\R^d, \C)$ (resp. from $\ell^2(\Z^d, \C)$ to $L^2(\T^d, \C)$). We recall the following corollary of the Poisson summation formula: for every $f \in \cS(\R^d)$ and $p \in \T^d$,
\begin{equ} \label{e:PoissonSummation}
\cF (\restriction{f}{\Z^d}) \, (p) = \sum_{k \in p + 2 \pi \Z^d} \fF f(k) \, .
\end{equ}

We fix throughout the paper $l_0 \geq 1$ such that the jump rate function $r$ (and thus also $s$ and $a$) is supported on $B_\infty(0,l_0)$. Let us observe that we have $\mathrm{Supp}(s) = \mathrm{Supp}(r) \cup (- \mathrm{Supp}(r))$ and $\mathrm{Supp}(a) \subset \mathrm{Supp}(s)$.

For any $n \geq 0$, we equip $\Z^{2n}$ with the following graph structure: $x_{1:n}$ and $y_{1:n}$ are neighbours (and we write $x_{1:n} \sim y_{1:n}$) if there exists $i \in \{1, \dots, n\}$ such that $x_i - y_i \in \mathrm{Supp}(s)$ and $x_j = y_j$ for every $j \in \{1, \dots, n\} \setminus \{i\}$. The irreducibility assumption implies the connectedness of this graph.
\ss

To avoid repeatedly writing the lower cut-off $1 \vee [...]$ in estimates, we set for any $y \in \R$: $y_* \eqdef 1 \vee y$. In particular, for $x > 0$, we write $\log_*(x) \eqdef 1 \vee \log(x)$. It is convenient to define the scaling matrix $\Sigma^t$ in~\eqref{e:ScalingMatrixAsymp} for any $t \geq 1$ by
\begin{equ} \label{e:ScalingMatrix}
\Sigma^t \eqdef \sqrt{t} \, (\log_*(t)^{1/3} \, \fe \fe^T + \ff \ff^T) \, .
\end{equ}

Finally, we write $a \lesssim b$ if there exists a constant $C > 0$, independent of any quantity relevant to the result,
such that $a \leq C b$; and $a \asymp b$ if $a \lesssim b$ and $b \lesssim a$.
If we want to highlight the dependence of the constant $C$ on a specific quantity $Q$, we write instead $\lesssim_Q$ or $\asymp_Q$. For simplicity, we will never track the dependence of the constants on the jump rate function $r(\cdot)$ and the density $\rho$.

\subsection*{AI disclosure}

The use of AI in this work was limited to English-language editing, proofreading, and identifying relevant references.
The result was announced in March 2026 at the workshop ``From discrete systems to SPDE'' at TU Wien.

\subsection*{Acknowledgments}

The author is grateful to Giuseppe Cannizzaro and Fabio Toninelli for their invaluable and continued support, and to Damiano De Gaspari for helpful discussions.
This research was funded by the Austrian Science Fund
(FWF) 10.55776/F1002. For open access purposes, the author has applied a CC BY public
copyright license to any author accepted manuscript version arising from this submission.

\section{Preliminaries}

\label{sec:Preliminaries}

First we gather all necessary definitions of functional spaces and operators, including the generator of the ASEP (Section~\ref{sec:ASEPgen}), its soft-core counterpart (Section~\ref{sec:Softcore}), the generator of the SBE (Section~\ref{sec:SBEgen}), bridges to be able to compare them (Section~\ref{sec:Bridges}), and the Green-Kubo formula for the diffusivity (Section~\ref{sec:GK}). Then, we outline the overall proof strategy in Section~\ref{sec:Roadmap}.

\subsection{The ASEP generator and its variants}

Let us first define the semigroup of the ASEP (we refer to~\cite[Chapter 5]{KLO} for more details): if $F \in L^1(\nu_\rho)$ and $t \geq 0$, then $P_t F \in L^1(\nu_\rho)$ is given by
\begin{equ} \label{e:SemiGroup}
P_t F(\eta) \eqdef \mathbf E_{\eta}[F(\eta_t)] \, .
\end{equ}
Because of the $\nu_\rho$-stationarity of the ASEP, this defines a strongly continuous contractive semigroup on $L^p(\nu_\rho)$, for any $p \in [1,\infty)$. Then, the central object of our analysis is the generator $L \colon \mathrm{Dom}(L) \subset L^2(\nu_\rho) \to L^2(\nu_\rho)$ of the semigroup $(P_t)_{t \geq 0}$ acting on $L^2(\nu_\rho)$. The action of $L$ on local functionals is given by~\eqref{e:GenLocal}, and the space of local functionals constitutes a core for $L$. The symmetric part $S$ (resp. antisymmetric part $A$) of $L$ acts on local functionals as in~\eqref{e:GenLocal} after replacing the jump rates $r(\cdot)$ by $s(\cdot)$ (resp. $a(\cdot)$). To have a better understanding of the generator $L$, we will view its action through the Fock space decomposition of $L^2(\nu_\rho)$, as presented in the following subsection.

\subsubsection{The ASEP generator $L$ and the discrete Fock space}

\label{sec:ASEPgen}

It is common (see~\cite[Section 5.4]{KLO}) to write the action of the generator $L$ in the Hilbert basis of $L^2(\nu_\rho)$ consisting of the cylinder functionals
\begin{equ} \label{e:defsigma}
\sigma_A \eqdef \prod_{a \in A} \sigma(a) \, , \ \mbox{ where } \ \sigma(a) \eqdef \frac{\eta(a) - \rho}{\sqrt{\rho(1 - \rho)}} \, ,
\end{equ}
for $A$ running over all finite subsets $A$ of $\Z^2$. However, for our purpose, it is more convenient to parameterise this basis with symmetrised $n$-tuples rather than finite subsets.

To do so, let us define the Fock space $F \ell^2$ as the Hilbert space which is the completion of the orthogonal direct sum $\bigoplus_{n \geq 0} \ell^2(\Z^{2n})$, endowed with the norm
\begin{equ} \label{e:defNormFockd}
\|f\| \eqdef \Big(\sum_{n \geq 0} n! \, \|f_n\|_{\ell^2(\Z^{2n})}^2\Big)^{1/2} \, ,
\end{equ}
and we denote by $\langle \cdot, \cdot \rangle$ the corresponding scalar product. We will mainly work with the symmetric Fock space $\fockd{}{}$, which is the subspace of $F \ell^2$ consisting of functions $f = (f_n)$ such that each $f_n \in \ell^2((\Z^2)^n)$ is symmetric with respect to permutations of its $n$ variables $x_{1:n} \in \Z^{2n}$. Besides, let $\fockdo{}{}$ be the subspace of $\fockd{}{}$ consisting of functions that are supported on the subset $HC$ given by
\begin{equ} \label{e:defHC}
HC \eqdef \bigsqcup_{n \geq 0} \{x_{1:n} \in \Z^{2n} : \, \forall i \ne j \, , x_i \ne x_j \} \subset \bigsqcup_{n \geq 0} \Z^{2n} \eqdef \cZ \, .
\end{equ}
Observe that $HC$ is the set of vectors that satisfy the \emph{hard core} condition. We also sometimes refer to $DI \eqdef \cZ \setminus HC$ as the diagonal (so that $f \in \fockd{}{}$ belongs to $\fockdo{}{}$ if and only if it vanishes on the diagonal).
Finally, for $n \geq 0$ we denote $\fockd{n}{}{} \eqdef \fockd{}{} \cap \ell^2(\Z^{2n})$ and $\fockdo{n}{}{} \eqdef \fockdo{}{} \cap \ell^2(\Z^{2n})$, which we sometimes refer to as the $n$-th chaos.

Since $(\sigma_A)_{A \subset \subset \Z^2}$ is a Hilbert basis of $L^2(\nu_\rho)$, we readily deduce that the operator
\begin{equ} \label{e:defI}
I \colon f \in \fockdo{}{} \mapsto \sum_{n \geq 0} \sum_{x_{1:n} \in \Z^{2n} \cap HC} f(x_{1:n}) \, \sigma_{\{x_1, \dots, x_n\}} \in L^2(\nu_\rho) \, ,
\end{equ}
defines an isometry. From now on, we will identify $L^2(\nu_\rho)$ with $\fockdo{}{}$ via this isometry (in particular, we identify operators $O$ acting on $L^2(\nu_\rho)$ with the operator $I^{-1} O I$ acting on $\fockdo{}{}$, and conversely).
\ms

Now that we have defined the hard-core discrete Fock space $\fockdo{}{}$, we describe the action of the generator $L$ on the latter. First, let us introduce the number operator $N$ and the core $\fockdo{\rap}{}{}$.

\begin{definition} \label{def:N}

Let $\cN$ be the operator which maps $f = (f_n)_{n \geq 0} \in \fockd{}{}$ to $\cN f := (n f_n)_{n \geq 0}$ ; similarly, let $\cN_*$ be the operator which maps $f = (f_n)_{n \geq 0} \in \fockd{}{}$ to $\cN_* f := (n_* f_n)_{n \geq 0}$ where $n_* := 1 \vee n$. We denote by $N$ (and $N_*$) the restriction of $\cN$ (and $\cN_*$) to $\fockdo{}{}$. 

Let $\fockd{\rap}{}$ be the subspace consisting of $f = (f_n)_{n \geq 0} \in \fockd{}{}$ such that $\|\cN_*^k f\|$ is finite for every $k \geq 0$. We denote $\fockdo{\rap}{} = \fockdo{}{} \cap \fockd{\rap}{}$. We equip $\fockd{\rap}{}$ with its natural Fréchet metric.
\end{definition}

Note that the space $\fockdo{\rap}{}$ contains local functions, since they have finite chaos expansions (in particular, $\fockdo{\rap}{}$ will also be a core for $L$). A simple computation (in analogy to~\cite[Section 5.4]{KLO}) shows that the generator $L$ acts on the Fock space as follows.

\begin{lemma} \label{lem:Gen}
The generator $L$ equals $S + A \eqdef S + A_+ + A_- + A_0$ on $\fockdo{\rap}{}{}$ (which is a core for $L$), where the operators $S, A_+, A_-, A_0$ act (continuously) from $\fockdo{\rap}{}{}$ to itself as follows, for $f \in \fockdo{\rap}{}{}$,
\begin{equs}
\label{e:S}
S f (x_{1:n}) &\eqdef \sum_{i=1}^n \sum_{z \in \Z^2} \1_{HC}(x_{1:n}, x_i + z) \, s(z) \, \big(f(x_i + z, x_{1:n \setminus i}) - f(x_{1:n})\big) \, , \\
\label{e:A+}
A_+ f (x_{1:n+1}) &\eqdef \frac{2 \sqrt{\rho (1-\rho)}}{n+1} \sum_{1 \leq i < j \leq n+1} \1_{HC}(x_{1:n+1}) \, a(x_j - x_i) \, \big(f(x_{1:n+1 \setminus i}) - f(x_{1:n+1 \setminus j})\big) \, , \\
\label{e:A-}
A_- f(x_{1:n-1}) &\eqdef - 2 \sqrt{\rho(1-\rho)} \, n \sum_{i=1}^{n-1} \sum_{z \in \Z^2} \1_{HC}(x_{1:n-1}, x_i+z) \, a(z) \, f(x_i + z, x_{1:n-1}) \, , \\
\label{e:A0}
A_0 f(x_{1:n}) &\eqdef (1 - 2 \rho) \sum_{i=1}^n \sum_{z \in \Z^2} \1_{HC}(x_{1:n}, x_i+z) \, a(z) \, \big(f(x_i + z, x_{1:n \setminus i}) - f(x_{1:n})\big) \, .
\end{equs}
On $\fockdo{\rap}{} \subset \fockdo{}{}$, the operator $S$ is essentially self-adjoint while $A$ is skew-symmetric as we have $(A_+)^* = - A_-$ and $(A_0)^* = - A_0$. Observe that, for any $n \geq 0$, the operators $S$, $A_+$, $A_-$, $A_0$ map $\fockdo{n}{}{}$ respectively to $\fockdo{n}{}{}$, $\fockdo{n+1}{}{}$, $\fockdo{n-1}{}{}$, $\fockdo{n}{}{}$.
\end{lemma}

\begin{remark} \label{rmk:IPPS}
The operator $\restriction{S}{\fockdo{n}{}{}}$ can be seen as the $s$-weighted discrete Laplacian on $\Z^{2n} \setminus DI$ with Neumann boundary condition. In particular, the following integration by parts formula holds: if we define for any $z \in \Z^2$ the operator $\nabla_z^0$ acting from $\fockdo{}{}{}$ to $F \ell^2$ as
\begin{equ} \label{e:nablaz0}
\nabla_z^0 f(x_{1:n}) = \1_{HC}(x_1 + z, x_{1:n}) \, (f(x_1 + z, x_{2:n}) - f(x_{1:n})) \, ,
\end{equ}
then, for any $f \in \fockdo{n}{}{}$, we have
\begin{equ} \label{e:NeumannBC}
\|(- S)^{1/2} f\|^2 = \langle (-S) f, f\rangle = \frac{n}{2} \sum_{z \in \Z^2} s(z) \, \|\nabla_z^0 f\|^2 \, .
\end{equ}
\end{remark}

\begin{remark} \label{rmk:A0}
The operator $A_0$ can be heuristically interpreted as a discrete transport operator with corresponding velocity $v = (1 - 2 \rho) \fw$ (in particular, observe that $A_0 = 0$ when $\rho = 1/2$). Indeed, after omitting the hard-core constraint and expanding $f \in \fockdo{n}{}$ at first order around $x_{1:n}$, we formally obtain $A_0 f \approx v \cdot (\nabla_{x_1} + \cdots + \nabla_{x_n}) f$. Therefore, the Galilean change of variables $x \to x + t v$ performed in the definition~\eqref{e:Ztau} of $Y_t^\tau$ is designed to cancel the action of $A_0$.
\end{remark}

\begin{remark} \label{rmk:noGSC}
An important difficulty here is that, contrary to systems at the supercritical dimension or in the weak coupling regime, no graded sector condition holds: the generator is not bounded from $H^1$ to $H^{-1}$, even if it is restricted to a finite number of chaoses. For example, it can be checked that
\begin{equ} \label{e:noGSC}
\sup_{f \in \fockdo{1}{}, \, f \not \equiv 0} \frac{\|(\mu - S)^{-1/2} \, A_+ f\|}{\|(\mu - S)^{1/2} f\|} \asymp \log_*(1/\mu)^{1/2} \, .
\end{equ}
\end{remark}

\subsubsection{The softcore generator $\cL$}

\label{sec:Softcore}

One of the main difficulties in studying the generator $L$ (and more broadly every operator defined in Lemma~\ref{lem:Gen}) is the hard core constraint, which, for instance, makes direct computations via the Fourier transform intractable. To remedy the problem, in analogy to~\cite{LQSY04}, we introduce the ``softcore'' variants of these operators, which act on the discrete Fock space $\fockd{}{}{}$ (instead of $\fockdo{}{}{}$) and are obtained by removing the hard core constraints in the formulas of Lemma~\ref{lem:Gen}.

\begin{definition} \label{def:SoftCore}
Let us define the operators $\cL \eqdef \cS + \cA \eqdef \cS + \cA_+ + \cA_- + \cA_0$ acting (continuously) from $\fockd{\rap}{}{}$ to itself, by the corresponding formulas~\eqref{e:S},~\eqref{e:A+},~\eqref{e:A-},~\eqref{e:A0} with the indicators $\1_{HC}(...)$ removed. In terms of the Fourier transform, they act on $f \in \fockd{\rap}{}$ as follows
\begin{equs}
\label{e:FcS}
\cF (\cS f) \, (k_{1:n}) &= - \frac{1}{2} \, d_0(k_{1:n}) \, \cF f (k_{1:n}) \, , \\
\label{e:FcA+}
\cF (\cA_+ f) \, (k_{1:n+1}) &= \frac{\sqrt{\rho(1-\rho)}}{\pi (n+1)} \sum_{1 \leq i < j \leq n+1} \iota \, (w(k_i) + w(k_j)) \, \ \cF f (k_i + k_j, k_{1:n+1 \setminus i,j}) \, , \\
\label{e:FcA-}
\cF (\cA_- f) \, (k_{1:n-1}) &= \frac{\sqrt{\rho(1-\rho)} \, n}{2 \pi} \sum_{i=1}^{n-1} \int_{\T^2} \dd \ell \, \iota \, (w(\ell) + w(k_i - \ell)) \, \cF f(\ell,k_i-\ell, k_{1:n-1 \setminus i}) \, , \\
\label{e:FcA0}
\cF (\cA_0 f) \, (k_{1:n}) &= (1 - 2 \rho) \ \iota \, \Big(\sum_{i=1}^n w(k_i)\Big) \, \cF f (k_{1:n}) \, ,
\end{equs}
where we have used the shorthand notation, for $k \in \R^2$,
\begin{equ} \label{e:defd0w}
d_0(k) \eqdef 2 \sum_{z \in \Z^2} s(z) \, (1-\cos(z \cdot k)) \, , \quad w(k) \eqdef \sum_{z \in \Z^2} a(z) \, \sin(z \cdot k) \, ,
\end{equ}
and $d_0(k_{1:n}) \eqdef d_0(k_1) + \dots + d_0(k_n)$ for $k_{1:n} \in \R^{2n}$.
Similarly to Lemma~\ref{lem:Gen}, on $\fockd{\rap}{} \subset \fockd{}{}$ the operator $\cS$ is essentially self-adjoint while $(\cA_+)^* = - \cA_-$ and $(\cA_0)^* = - \cA_0$; for any $n \geq 0$, the operators $\cS$, $\cA_+$, $\cA_-$, $\cA_0$ map $\fockd{n}{}{}$ respectively to $\fockd{n}{}{}$, $\fockd{n+1}{}{}$, $\fockd{n-1}{}{}$, $\fockd{n}{}{}$.
\end{definition}

\begin{remark} \label{rmk:IPPcS}
The operator $\restriction{\cS}{\fockd{n}{}}$ can be seen as the $s$-weighted discrete Laplacian on $\Z^{2n}$. It satisfies an integration by parts formula similar to that for $S$ (cf. Remark~\ref{rmk:IPPS}). Indeed, for $z \in \Z^2$, let us define the operator $\nabla_z$ acting from $\fockd{}{}{}$ to $F \ell^2$ as
\begin{equ} \label{e:nablaz}
\nabla_z f(x_{1:n}) = f(x_1 + z, x_{2:n}) - f(x_{1:n}) \, .
\end{equ}
Then the soft-core counterpart of~\eqref{e:NeumannBC} is, for $f \in \fockd{n}{}{}$,
\begin{equ} \label{e:noBC}
\|(- \cS)^{1/2} f\|^2 = \frac{n}{2} \sum_{z \in \Z^2} s(z) \, \|\nabla_z f\|^2 \, .
\end{equ}
\end{remark}

\begin{remark} \label{rmk:NonSymmetricFock}
Because the operator $\nabla_z$ produces non-symmetric functions, we may need to work with the non-symmetric Fock space $F \ell^2$. To this end, we need to generalise a few definitions. First, we extend the counting operator $\cN$ to $F \ell^2$ and define the space $F \ell^2_{\rap}$ in the obvious way (as in Definition~\ref{def:N}). Then, we extend the operator $\cS$ to $F_* \ell^2 \subset F \ell^2$ as follows
\begin{equ} \label{e:NonSymmetricLapl}
\cS f(x_{1:n}) \eqdef \sum_{i=1}^n \sum_{z \in \Z^2} s(z) \, \big(f(x_1, \dots, x_{i-1}, x_i + z, x_{i+1}, \dots, x_n) - f(x_{1:n})\big) \, , 
\end{equ}
which, in terms of the Fourier transform, corresponds to the same formula as~\eqref{e:FcS}. Of course, this extended operator $\cS$ is still essentially self-adjoint. With this extension of $\cS$ in mind, we can for instance state that, for any $\gamma \geq 0$ and $f \in \fockd{n}{}$, it holds
\begin{equ} \label{e:FracLaplNabla}
\|(- \cS)^{(1+\gamma)/2} f\|^2 = \frac{n}{2} \, \sum_{z \in \Z^2} \|(-\cS)^{\gamma/2} \, \nabla_z f\|^2 \, .
\end{equ}
Indeed, this follows from the integration by parts formula~\eqref{e:noBC} applied to $(- \cS)^{\gamma/2} f$, and the observation that $\nabla_z$ and $(- \cS)^{\gamma/2}$ commute.
\end{remark}

\subsubsection{The SBE generator $\fL$ and the continuum Fock space}

\label{sec:SBEgen}

Let us now introduce the continuum analogue of the soft-core generator, which is the generator of the SBE. This generator can be formally obtained by replacing in formulas~\eqref{e:FcS}-\eqref{e:FcA0} the discrete Fourier transform $\cF$ by the continuum one $\fF$ and the trigonometric functions $d_0(\cdot)$ and $w(\cdot)$ by their Taylor expansions at $0$ (up to the first nontrivial order). Here is an example that highlights the gain from working with continuum operators: in~\eqref{e:FcA-}, the factor $w(\ell) + w(k_i - \ell)$ appearing in the integrand depends nontrivially on $\ell$ and thus cannot be factorised outside the integral; however, we always have $\fw \cdot \ell + \fw \cdot (k_i - \ell) = \fw \cdot k_i$ which does not depend on $\ell$.
\ss

To define the SBE generator, we first need to define the continuum Fock space that corresponds to the chaos decomposition of $L^2(\mathbb P)$. Let $\fock{}{}$ be the completion of $\bigoplus_{n \geq 0} \fock{n}{}{}$, where $\fock{n}{}{}$ is\footnote{For $n = 0$, we use the convention $\fock{0}{} = L^2(\R^0) = \R^{\{0\}}$ with the norm $\|f\|_{L^2(\R^0)} \eqdef |f(0)|$.} the subspace of $L^2(\R^{2n}) \simeq L^2((\R^{2})^n)$ made of functions which are symmetric with respect to permutations of their $n$ variables, for the Hilbert norm
\begin{equ} \label{e:FockNorm}
\|f\| = \Big(\sum_{n \geq 0} n! \,  \|f_n\|_{L^2(\R^{2n})}^2\Big)^{1/2} \, ,
\end{equ}
and we denote by $\langle \cdot, \cdot \rangle$ the corresponding scalar product. Then, let us define the so-called Wiener-Itô integration operator $\fI \colon \fock{}{} \to L^2(\mathbb P)$ (we refer to~\cite{peccati2011wiener} for more details) which is the isometry that maps elements of the form $h^{\otimes n} \in \fock{n}{}{}$, for $h \in \cS(\R^2)$, to the random variable $H_n(\zeta(h), \|h\|_{L^2(\R^2)})$, where $\zeta$ is the spatial white noise and $H_n(\cdot , s)$ is the $n$-th Hermite polynomial given by
\begin{equ} \label{e:Hermite}
H_n(X, s) \eqdef \sum_{0 \leq k \leq n/2} c_{k, n} \, (-1)^k \, s^{2k} \, X^{n - 2k} \, , \mbox{ where } c_{k, n} := \frac{n!}{2^k k! (n-2k)!} \, .
\end{equ} 
For later use, we mention that the coefficient $c_{k,n}$ can be interpreted combinatorially as the number of partitions of $\{1, \dots, n\}$ with $k$ blocks of size $2$ and $n - 2 k$ blocks of size $1$. From now on, we will identify $L^2(\mathbb P)$ with $\fock{}{}$ through the isometry $\fI$.
\ss

\begin{definition} \label{def:CoreFock}
Let $\fock{\rap}{} \subset \fock{}{}$ be the subspace consisting of $f = (f_n)_{n \geq 0}$ such that for every $a, b \geq 0$ it holds
\begin{equ} \label{e:defCoreFock}
\sup_{n \geq 0} n! \, (1 + n^a) \int_{\R^{2n}} (1+|k_{1:n}|^2)^{b} \, |\fF f_n(k_{1:n})|^2 \, \dd k_{1:n} < \infty \, .
\end{equ}
We equip $\fock{\rap}{}$ with its natural Fréchet metric. For $n \geq 0$, we denote $\fock{\rap,n}{} \eqdef \fock{\rap}{} \cap \fock{n}{}$.
\end{definition}
Let us mention that the space $\fock{\rap}{}$ contains all smooth cylinder functionals, that is, functionals of the form $\zeta \in \cS'(\R^2) \mapsto f(\zeta(\vphi_1), \dots, \zeta(\vphi_m))$ with $\vphi_1, \dots, \vphi_m \in \cS(\R^2)$ and $f \colon \R^m \to \R$ smooth with at most polynomial growth.

\begin{definition} \label{def:GenSBE}
Let us define the operators $\fL \eqdef \fS + \fA \eqdef \fS + \fA_+ + \fA_- + \fA_0$ which act (continuously) from $\fock{\rap}{}$ to itself as follows, for $f \in \fock{\rap}{}$,
\begin{equs}
\label{e:fS}
\fF (\fS f) \, (k_{1:n}) &\eqdef - \frac{1}{2} D_0(k_{1:n}) \,\fF f(k_{1:n}) \, , \\
\label{e:fA+}
\fF (\fA_+ f) \, (k_{1:n+1}) &\eqdef \frac{\sqrt{\rho(1-\rho)}}{\pi (n+1)} \sum_{1 \leq i < j \leq n+1} \iota \, \fw \cdot (k_i+k_j) \, \Pi_{k_i, k_j} \, \fF f (k_i + k_j, k_{1:n+1 \setminus i,j}) \, , \\
\label{e:fA-}
\fF (\fA_- f) \, (k_{1:n-1}) &\eqdef \frac{\sqrt{\rho(1-\rho)} \, n}{2 \pi} \sum_{i=1}^{n-1} \iota \, \fw \cdot k_i \, \int_{\R^2} \dd \ell \, \Pi_{\ell, k_i-\ell} \, \fF f (\ell, k_i-\ell, k_{1:n-1 \setminus i}) \, , \\
\label{e:fA0}
\fF(\fA_0 f) \, (k_{1:n}) &\eqdef \iota \, \Big(v \cdot \sum_{i=1}^n k_i\Big) \, \fF f(k_{1:n}) \, ,
\end{equs}
where $\fw$ and $D_0$ were defined in~\eqref{e:D0&drift}, $v = (1 - 2 \rho) \, \fw$, and where we have used the shorthand notation, for $k_{1:n} \in \R^{2n}$ and $p, q \in \R^2$,
\begin{equ} \label{e:CutOff}
D_0(k_{1:n}) \eqdef \sum_{i=1}^n k_i^T D_0 \,  k_i \, , \quad  \Pi_{p,q} \eqdef \1_{|p|_\infty < \pi} \, \1_{|q|_\infty < \pi} \, \1_{|p+q|_\infty < \pi} \, .
\end{equ}
Similarly to Lemma~\ref{lem:Gen}, on $\fock{\rap}{} \subset \fock{}{}$, the operator $\fS$ is essentially self-adjoint while $(\fA_+)^* = - \fA_-$ and $(\fA_0)^* = - \fA_0$. Observe that, for any $n \geq 0$, the operators $\fS$, $\fA_+$, $\fA_-$, $\fA_0$ map $\fock{\rap,n}{}$ respectively to $\fock{\rap,n}{}$, $\fock{\rap,n+1}{}$, $\fock{\rap,n-1}{}$, $\fock{\rap,n}{}$.
\end{definition}

\begin{remark} \label{rmk:fLSBE}
Formally, $\fL$ can be interpreted as the generator of the SBE
\begin{equ} \label{e:ASEPTrueSBE}
(\partial_t + v \cdot \nabla) \, u = \frac{1}{2} \, \nabla \cdot D_0 \nabla u + \sqrt{\rho(1-\rho)} \, \fw \cdot \nabla \chi * (\chi * u)^2 + \nabla \cdot \big(\sqrt{D_0} \, \vec{\, \xi}\big) \, ,
\end{equ}
where the mollifier $\chi$ is given by the ultraviolet Fourier cut-off $\fF^{-1}((2 \pi)^{-1} \1_{|\cdot|_\infty < \pi})$. Compare~\eqref{e:ASEPTrueSBE} to the heuristic equation~\eqref{e:ASEPFormalSBE}: here, we have an additional transport term because we have not yet performed the Galilean transformation, and the non-linearity is regularised via the mollifier $\chi$, which is connected to the choice of the cut-off $\Pi_{p,q} = 2 \pi \hat \chi(p) \times 2 \pi \hat \chi(q) \times 2 \pi \hat \chi(p+q)$ in~\eqref{e:CutOff}. Heuristically, $\fS$ corresponds to the generator of the diffusive part $\frac{1}{2} \, \nabla \cdot D_0 \nabla u + \nabla \cdot \sqrt{D_0} \, \vec{\xi}$, while $\fA_+ + \fA_-$ comes from the non-linearity $\sqrt{\rho(1-\rho)} \, \fw \cdot \nabla \chi * (\chi * u^2)$ and $\fA_0$ from the transport term $v \cdot \nabla u$. For a more rigorous statement, we refer to~\cite{CMT25}; we stress however that, although the connection between $\fL$ and the generator of the SBE is enlightening, we do not need this interpretation of $\fL$ in this work (in particular, we do not need the well-posedness of the SPDE~\eqref{e:ASEPTrueSBE}).
\end{remark}

\begin{remark} \label{rmk:fA0}
Note that we can rewrite~\eqref{e:fA0} as $\fA_0 f(x_{1:n}) = (v \cdot \nabla_{x_1} + ... + v \cdot \nabla_{x_n}) f(x_{1:n})$, that is, $\fA_0$ acts as a transport term $v \cdot \nabla$ in each chaos variable, which is consistent with Remark~\ref{rmk:A0}. In particular, let us stress that $\fA_0$ commutes with every translation invariant operator, such as $\fS$, $\fA_+$, $\fA_-$ or any diagonal operator (cf. the definition just below).
\end{remark}

Among the tractable operators that act on $\fock{}{}$ are those that act as a Fourier multiplier on each chaos; we refer to these as \emph{diagonal} operators.

\begin{definition} \label{def:DiagonalOperator}

An operator $\fD$ acting on $\fock{\rap}{}$ is said to be \emph{diagonal} if there exists a family of symmetric kernels $d = (d_n)_{n \geq 0}$ such that, for any $\vphi \in \fock{\rap}{}$, it holds
$\fF(\fD \vphi) (k_{1:n}) = d_n(k_{1:n}) \, \fF \vphi(k_{1:n})$ for every $k_{1:n} \in \R^{2n}$, $n \geq 0$ (to avoid pathologies, we always assume that the kernels are measurable and grow at most as $|d_n(k_{1:n})| \lesssim n^a \, (1 + |k_{1:n}|^b)$ for some $a, b \geq 0$). Given two diagonal operators $\fD^{(1)}$ and $\fD^{(2)}$, we write $\fD^{(1)} \leq \fD^{(2)}$ if their respective kernels $d^{(1)}$ and $d^{(2)}$ satisfy $d_n^{(1)}(k_{1:n}) \leq d_n^{(2)}(k_{1:n})$ for every $k_{1:n} \in \R^{2n}$, $n \geq 0$.
Finally, given a diagonal operator $\fD$ with kernel $d = (d_n)_{n \geq 0}$
and a scalar function $f$, we define (when it makes sense) $f(\fD)$ as the diagonal operator
whose kernel is $f(d) = (f \circ d_n)_{n \geq 0}$.

We define the diagonal operators $\fN$, $\fN_*$, $\fL_0$, $\fL_0^\fe$, $\fL_0^\ff$ whose respective kernels on chaos $n$ are given by: $n$, $1 \vee n$, $- |k_{1:n}|^2/2$, $- |\fe \cdot k_{1:n}|^2/2$, $- |\ff \cdot k_{1:n}|^2/2$.
\end{definition}

Finally, the action of the generator of linear stochastic heat equations, such as~\eqref{e:Ylim}, is particularly simple on the Fock space $\fock{}{}$. More precisely, we can define the semigroup $P_t^{\lim} F$ of the process $Y^{\lim}$, for $t \geq 0$ and $F \in L^1(\mathbb P)$, as
\begin{equ} \label{e:defPlim}
P_t^{\lim} F(\zeta) = \mathbf E_\zeta^{\lim}[F(Y^{\lim}_t)] \, .
\end{equ}
Because of the Markov property and the $\mathbb P$-stationarity of $(Y^{\lim}_t)_{t \geq 0}$, we deduce that $(P_t^{\lim})_{t \geq 0}$ defines a strongly continuous contractive semigroup on $L^p(\mathbb P)$ for any $p \in [1, \infty)$. In particular, we denote by $\fL^{\lim} \colon \mathrm{Dom}(\fL^{\lim}) \subset L^2(\mathbb P) \to L^2(\mathbb P)$ the generator of $(P_t^{\lim})_{t \geq 0}$ acting on $L^2(\mathbb P) \simeq \fock{}{}$, coming from Hille-Yosida theory. A straightforward computation shows the following formula for $\fL^{\lim}$.

\begin{lemma} \label{lem:Llim}
On $\fock{\rap}{} \subset \fock{}{}$, the generator $\fL^{\lim}$ is given by $d^{\lim} \, \fL_0^\fe +  D_0[\ff, \ff] \, \fL_0^{\ff}$.
\end{lemma}

\subsubsection{The bridges between the different Fock spaces}

\label{sec:Bridges}

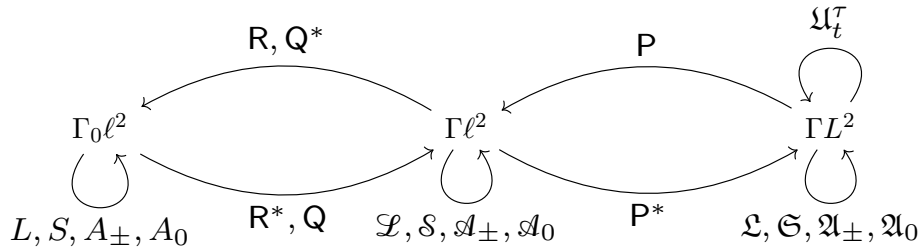
\begin{figure}[ht]
\centering
\begin{tikzcd}[scale=1.5, transform shape, column sep=5em]
{\Gamma_0 \ell^2}
  \arrow["{L, \, S, \, A_{\pm}, \, A_0}"', loop, distance=2em, in=305, out=235]
  \arrow[rr, "{\fsR^*, \, \fsQ}"', bend right]
& &
{\Gamma \ell^2}
  \arrow["{\cL, \, \cS, \, \cA_\pm, \, \cA_0}"', loop, distance=2em, in=305, out=235]
  \arrow[ll, "{\fsR, \, \fsQ^*}"', bend right]
  \arrow[rr, "\mathsf{P}^*"', bend right]
& &
{\Gamma L^2}
  \arrow["{\fL, \, \fS, \, \fA_{\pm}, \, \fA_0}"', loop, distance=2em, in=305, out=235]
  \arrow[ll, "\mathsf{P}"', bend right]
  \arrow["{\fU_t^\tau}"', loop, distance=2em, in=125, out=55]
\end{tikzcd}
\caption{The three Fock spaces and the operators between them}
\label{fig:operator-diagram}
\end{figure}

In this section, we define some additional operators that act between the Fock spaces: in particular, the scaling operator $\fU_t^\tau \colon \fock{}{} \to \fock{}{}$ and the bridges $\fsP \colon \fock{}{} \to \fockd{}{}$ and $\fsR, \fsQ^* \colon \fockd{}{} \to \fockdo{}{}$.

First, we define the scaling operator $\fU_t^\tau$ whose purpose is to produce large scale functionals consistent with the superdiffusive scaling we are interested in: space scales of order $\sqrt{\tau} \log(\tau)^{1/3}$ and $\sqrt{\tau}$ in the directions $\fe$ and $\ff$, together with the translation $\cdot - \tau t v$ corresponding to the Galilean transformation.

\begin{definition} \label{def:Utau}
For any $z \in \R^2$, $\tau \geq 1$, $t \geq 0$, we define the isometries $\fT_{z}, \fU_t^\tau \colon \fock{}{} \to \fock{}{}$ that act on $f \in \fock{}{}$ as follows
\begin{equs}
\fT_z f (x_{1:n}) &\eqdef f(x_1 + z, \dots, x_n + z) \label{e:T} \\
\fU_0^\tau f (x_{1:n}) &\eqdef \det(\Sigma^\tau)^{-n/2} \, f((\Sigma^\tau)^{-1} \, x_{1:n}) \, , \\
\fU^\tau_t f &\eqdef \fT_{- \tau t v} \, \fU^\tau_0 \, f \, , \label{e:Utau}
\end{equs}
where the matrix $\Sigma^\tau$ is given by~\eqref{e:ScalingMatrix} and $v = (1 - 2 \rho) \fw$. In Fourier space, the action of $\fT_{z}$ and $\fU_0^\tau$ can be written as
\begin{equ}
\fF(\fT_z f) \, (k_{1:n}) = \exp\big(\iota \, z \cdot (k_1 + \dots + k_n)\big) \, \fF f(k_{1:n}) \, , \quad \fF(\fU_0^\tau f) \, (k_{1:n}) = \det(\Sigma^\tau)^{n/2} \, \fF f(\Sigma^\tau k_{1:n}) \, . \label{e:UtauF}
\end{equ}
\end{definition}

\begin{remark} \label{rmk:Utau}
The scaling operator $\fU_t^\tau$ is consistent with the scaling defining the fluctuation process $Y_t^\tau$ (cf.~\eqref{e:Ztau} and~\eqref{e:Ytau}): indeed, we can rewrite the random variables $Z_t^\tau(\vphi)$, for any $\vphi \in \cS(\R^2)$, as $Z_t^\tau(\vphi) = I(\restriction{(\fU_t^\tau \vphi)}{\Z^2})$.
\end{remark}

\begin{remark} \label{rmk:GalileanTransform}
    By comparing~\eqref{e:UtauF} and~\eqref{e:fA0}, we note that $\fT_{- \tau t v} = \exp(- \tau t \, \fA_0)$, which in particular gives $(\partial_t + \tau \fA_0) \, \fU^\tau_t = 0$. This identity confirms the idea from Remark~\ref{rmk:A0} that the Galilean change of variables cancels the action of $\fA_0$.
\end{remark}

Now, we introduce the bridge operators which act as embeddings between the different Fock spaces $\fock{}{}$, $\fockd{}{}$ and $\fockdo{}{}$. These will allow us to compare the generators of the ASEP and SBE. First, we define the operators $\fsR \colon \fockd{}{} \to \fockdo{}{}$ and $\fsP \colon \fock{}{} \to \fockd{}{}$ that act on $f \in \fockd{}{}$ and $g \in \fock{}{}$ as
\begin{equ} \label{e:defRP}
\fsR f \, (x_{1:n}) \eqdef \1_{HC}(x_{1:n}) \, f(x_{1:n}) \, , \quad \cF (\fsP g) \, (k_{1:n}) \eqdef \, \fF g(\llbracket k_{1:n} \rrbracket) \, ,
\end{equ}
where we recall that $\llbracket \cdot \rrbracket$ denotes the representative map from $\T^d$ to $(-\pi, \pi]^d$. Intuitively, $\fsP$ projects the continuous Fock space into the discrete one by truncating the Fourier modes outside $(-\pi,\pi)^2$, while $\fsR$ restricts the discrete Fock space to the hardcore one by truncating the diagonal. By duality, these definitions give also rise to operators $\fsR^* \colon \fockdo{}{} \to \fockd{}{}$ and $\fsP^* \colon \fockd{}{} \to \fock{}{}$. Observe that $\fsR^*$ is simply the inclusion of $\fockdo{}{}$ in $\fockd{}{}$, while $\fsP^*$ acts on $f \in \fockd{}{}$ as
\begin{equ} \label{e:defP*}
\fF(\fsP^* f)(k_{1:n}) = \1_{k_{1:n} \in (- \pi, \pi)^{2 n}} \, \cF f(k_{1:n}) \, .
\end{equ}
Furthermore, we define the operator $\fsi_t^\tau \colon \fock{}{} \to \fockdo{}{}$ which simply combines the scaling operator $\fU_t^\tau$ with the previous bridges, as follows
\begin{equ} \label{e:fsittau}
\fsi_t^\tau \eqdef \fsR \, \fsP \, \fU_t^\tau \, .
\end{equ}

We now state elementary bounds on the bridges $\fsP$ and $\fsR$ and $H^{\pm 1}$ norms, for later use.

\begin{lemma} \label{lem:BoundP}
For every $\kappa \in \R$, $f \in \fock{\rap}{}$, $g \in \fockd{\rap}{}$, and $\mu \geq 0$, we have
\begin{equ} \label{e:BoundP}
\|(\mu - \cS)^{\kappa} \, \fsP f\| \lesssim_\kappa \|(\mu - \fL_0)^{\kappa} f\| \, , \quad \|(\mu - \fL_0)^{\kappa} \, \fsP^* g\| \lesssim_\kappa \|(\mu - \cS)^{\kappa} \, g\| \, . 
\end{equ}
\end{lemma}

\begin{proof}[of Lemma~\ref{lem:BoundP}]
The second bound in~\eqref{e:BoundP} can be deduced from the first one (with $- \kappa$ instead of $\kappa$) by using a simple duality argument: thus we focus on the first bound. By using the elementary estimate~\eqref{e:d0D0equivalence} from Lemma~\ref{lem:SoftcoreRemoval} below, we deduce
\begin{equs}
\|(\mu - \cS)^\kappa \, \fsP f\|^2 &= \sum_{n \geq 0} n! \int_{[- \pi, \pi]^{2n}} (\mu + d_0(k_{1:n})/2)^{2 \kappa} \, |\fF f(k_{1:n})|^2 \, \dd k_{1:n} \\
&\lesssim_\kappa \sum_{n \geq 0} n! \int_{[- \pi, \pi]^{2n}} (\mu + |k_{1:n}|^2/2)^{2 \kappa} \, |\fF f(k_{1:n})|^2 \, \dd k_{1:n} \leq \|(\mu - \fL_0)^\kappa f\|^2 \, ,
\end{equs}
which shows the desired bound.
\end{proof}

\begin{lemma} \label{lem:BoundR}
For every $\mu \geq 0$ and $f \in \fockd{}{}$, we have
\begin{equ} \label{e:BoundR}
\|\fsR f\| \leq \|f\| \, , \quad \|(\mu - S)^{1/2} \, \fsR f\| \leq\|(\mu - \cS)^{1/2} f\| \, .
\end{equ}
\end{lemma}

\begin{proof}[of Lemma~\ref{lem:BoundR}]
Recalling the identities~\eqref{e:NeumannBC},~\eqref{e:noBC} and $\|(\mu - S)^{1/2} \cdot\|^2 = \mu \|\cdot\|^2 + \|(- S)^{1/2} \cdot\|^2$ (and similarly for $\cS$), we readily deduce~\eqref{e:BoundR} from the trivial bound $\1_{HC}(\cdot) \leq 1$.
\end{proof}

\begin{remark}
On the contrary, the corresponding bound on $\fsR$ for the $H^{-1}$ norm would be (ignoring possible extra polynomial cost in the number operator $\cN$)
\begin{equ} \label{e:BoundRBad}
\|(\mu - S)^{-1/2} \, \fsR f\| \lesssim \log_{*}(1/\mu)^{1/2} \, \|(\mu - \cS)^{-1/2} f\| \, ,
\end{equ}
where the diverging logarithmic factor $\log_{*}(1/\mu)^{1/2}$ is optimal (in fact, it comes from the logarithmic divergence at the origin of the two-dimensional Green's function).
\end{remark}

Contrary to the bridge $\fsP$ that behaves well in both $H^{1}$ and $H^{-1}$ norms, the bridge $\fsR$ only behaves well in $H^1$ norm. The reason why $\fsR$ is bad in $H^{-1}$ norm can be more easily seen by looking at the dual problem, that is, the fact that $\fsR^*$ is bad in $H^1$ norm. Indeed, given a function $f$ defined on $\Z^{2n} \cap HC$, the bridge $\fsR^*$ simply extends $f$ on the diagonal $DI$ by $0$: this creates a discrepancy between the values of $\fsR^* f$ at and near the diagonal, which prevents $\fsR^* f$ from being ``smooth'' (no matter how ``smooth'' $f$ is). Consequently, we need a bridge $\fsQ \colon \fockdo{}{} \to \fockd{}{}$ that behaves better in $H^1$ norm (or, equivalently, a bridge $\fsQ^* \colon \fockd{}{} \to \fockdo{}{}$ that behaves better in $H^{-1}$ norm).
\ss

We first define a useful partition of the phase space $\cZ = \bigsqcup_{n \geq 0} \Z^{2n}$. For $l \geq 1$, let $\cZ_0(l)$, $\cZ_1(l)$ and $\cZ_{\geq 2}(l)$ be the subsets of $\cZ$ consisting of vectors $x_{1:n} \in \Z^{2n}$ such that
\begin{equ} \label{e:defZi}
\# \{1 \leq i < j \leq n \, , |x_i - x_j|_\infty \leq l\}
\end{equ}
is, respectively, equal to $0$, equal to $1$, and greater than or equal to $2$. When $l = 2 l_0 + 1$ (recall that $l_0$ was introduced at the end of the preamble, as the radius of the support of the jump rates $r$), we simply denote the corresponding sets as $\cZ_0$, $\cZ_1$ and $\cZ_{\geq 2}$. Obviously, we have $\cZ = \cZ_0 \sqcup \cZ_1 \sqcup \cZ_{\geq 2}$, $\cZ_0 \subset HC$ and $DI \subset \cZ_1 \sqcup \cZ_{\geq 2}$. Observe that, if $x_{1:n} \in DI \cap \cZ_1$, then there exist two indices $i \ne j$ such that $x_i = x_j$ while for every $k \not \in \{i,j\}$ we have $|x_k - x_i|_\infty = |x_k - x_j|_\infty > 2 l_0 + 1$; 
in this case, we refer to $x_i = x_j$ as the isolated \emph{double site} of $x_{1:n}$. Let us comment that our partition is similar to the one used in~\cite{LQSY04}, but with the crucial difference that, in the latter reference, elements of $\cZ_1$ are allowed to have multiple isolated double sites, which can substantially deteriorate estimates for large chaos $n$.

Now, we define the operator $\fsQ \colon \fockdo{}{} \to \fockd{}{}$, in a similar vein to~\cite{LQSY04}, that acts on $f \in \fockdo{}{}$ and $x_{1:n} \in \Z^{2n}$ as follows
\begin{equ} \label{e:defQ}
\fsQ f \, (x_{1:n}) \eqdef
\begin{cases}
f(x_{1:n}) &\mbox{if } x_{1:n} \in HC \\
f(x_i, x_j + z_0, x_{1:n \setminus i, j}) &\mbox{if } x_{1:n} \in DI \cap \cZ_1 \mbox{ with the isolated double site } x_i = x_j \\
0 &\mbox{if } x_{1:n} \in DI \cap \cZ_{\geq 2}
\end{cases}
\end{equ}
where $z_0 \in \Z^2 \setminus \{0\}$ is an arbitrary element of $\mathrm{Supp}(s)$ that we fix for the rest of the paper. Note that, by construction, the vector $(x_i,x_j+z_0,x_{1:n \setminus i,j})$ that appears in the definition of $\fsQ f$ always belongs to $HC$. Intuitively, we are given a function $f$ defined only on $HC$ that we want to extend on the diagonal $DI$, but instead of simply extending it by $0$ as $\fsR^*$ does, we want a ``more continuous'' extension, which we obtain by copying the value of the function at a neighbouring point that belongs to $HC$. However, we only run this strategy for $x_{1:n} \in DI \cap \cZ_1$, because finding a ``good'' neighbouring point of $x_{1:n}$ in $HC$ becomes more and more complex as the number of close points in $x_{1:n}$ increases, which could lead to bad (e.g. exponential) dependence on $n$ of the multiplicative constants in various estimates. As for the reason why it is possible to ignore the vectors $x_{1:n} \in \cZ_{\geq 2}$, it will be justified later by Lemmas~\ref{lem:ZiBound} and~\ref{lem:ZiBoundHC}.

By duality, this definition also gives rise to the operator $\fsQ^* \colon \fockd{}{} \to \fockdo{}{}$. It is easily checked that it acts on $g \in \fockd{}{}{}$ and $x_{1:n} \in \Z^{2n} \cap HC$ as follows
\begin{equ} \label{e:defQ*}
\fsQ^* g(x_{1:n}) = g(x_{1:n}) + \frac{1}{2} \sum_{1 \leq i \ne j \leq n} \1_{x_j = x_i + z_0} \, \1_{\cZ_0}(x_{1:n \setminus j}) \, g(x_i, x_i, x_{1:n \setminus i,j}) \, .
\end{equ}

\subsection{The Green-Kubo formula}

\label{sec:GK}

In this section, we recall the Green-Kubo formula (see e.g.~\cite{LOY96}) for the diffusivity matrix $D(t)$ and give a few related definitions. More precisely, we are interested in the following variant of the Green-Kubo formula
\begin{equ} \label{e:GKformulaIntro}
D(t) = D_0 + \frac{\rho (1 - \rho)}{t} \sum_{x \in \Z^2} \mathbf E_{\nu_\rho}\Big[\Big(\int_0^t \sum_{z \in \Z^2} a(z) \, z \, \sigma_{s}(0) \, \sigma_{s}(z) \, \dd s\Big) \Big(\int_0^t \sum_{z \in \Z^2} a(z) \, z \, \sigma_{s}(x) \, \sigma_{s}(x + z) \, \dd s\Big)^T\Big] \, ,
\end{equ}
which is well-defined because, on any finite time horizon, the space-time correlations of the ASEP decay exponentially fast in space. Since we assume $a(\cdot)$ to be supported on the line $\R \fe$, we can further simplify this formula as
\begin{equ} \label{e:GKformulafe}
D(t) = D_0 + d(t) \, \fe \fe^T \, , \quad d(t) \eqdef \frac{1}{t} \sum_{x \in \Z^2} \mathbf E_{\nu_\rho}\Big[\Big(\int_0^t B_+(\1_{\{0\}})[\eta_s] \, \dd s\Big) \Big(\int_0^t B_+(\1_{\{x\}})[\eta_s] \, \dd s\Big)\Big] \, , 
\end{equ}
where the operator $B_+$ is defined as follows.

\begin{definition} \label{def:B+}
We define the operator $B_+ \colon \fockdo{1}{} \to \fockdo{2}{}$ that acts on $f \in \fockdo{1}{}$ as follows
\begin{equ} \label{e:B+}
B_+ f \, (x,y) \eqdef \sqrt{\rho (1 - \rho)} \, a(y-x) \, (\fe \cdot (y-x)) \, \frac{f(x) + f(y)}{2} \, .
\end{equ}
We also define the operator $\fB_+ \colon \fock{\rap,1}{} \to \fock{\rap,2}{}$ that acts on $f \in \fock{\rap,1}{}$ as follows
\begin{equ} \label{e:fB+}
\fF(\fB_+ f) \, (p,q) \eqdef  \frac{\lambda \, \sqrt{\rho(1-\rho)}}{2 \pi} \, \Pi_{p,q} \, \fF f(p+q) \, ,
\end{equ}
where $\Pi_{p,q}$ is the cut-off function given by~\eqref{e:CutOff}.
\end{definition}

\begin{remark}
Compare the formula~\eqref{e:B+} of $B_+$ with the action of $A_+$ on the first chaos
\begin{equ}
A_+ f(x,y) = \sqrt{\rho (1-\rho)} \, a(y-x) \, (f(y) - f(x)) \, .
\end{equ}
Hence, we morally have $A_+ f \approx B_+ \, \partial_\fe f$. In fact, this is an exact identity at the level of the continuum operators: $\fA_+ f = \fB_+ \, \partial_\fe f$.
\end{remark}

Observe that, if there existed an even-symmetric function $\cV^0 \in \cS(\R^2)$ such that $\cV^0 * \cV^0 \equiv 1$, then we would be able to formally rewrite the formula for the diffusivity~\eqref{e:GKformulafe} as
\begin{equ} \label{e:GKneatFormal}
d(t) \overset{?}{=} \frac{1}{t} \, \mathbf E_{\nu_\rho}\Big[\Big(\int_0^t B_+ \cV^0 \, (\eta_s) \, \dd s\Big)^2\Big] \, .
\end{equ}
Although such a function $\cV^0$ certainly does not exist, it is possible to make sense of the formula~\eqref{e:GKneatFormal} by considering an appropriate sequence of functions $\cV^\delta$ such that $\cV^\delta * \cV^\delta \to 1$ as $\delta \to 0$.

\begin{definition} \label{def:V}
Let $\cV \in \cS(\R^2)$ be a Schwartz function that is even-symmetric, with $L^2(\R^2)$-norm equal to $1$, and with Fourier transform supported on $(- \pi, \pi)^2$; we fix such $\cV$ for the rest of the paper. Then, for any $\delta \in (0,1]$, we define $\cV^\delta \in \cS(\R^2)$ as
\begin{equ} \label{e:SqrtDirac}
\cV^\delta(x) = \delta \, \cV(\delta x) \ \mbox{, or equivalently, } \ \fF(\cV^\delta) \, (p) = \delta^{-1} \, \fF \cV(\delta^{-1} p) \, .
\end{equ}
We also denote by $\cV^\delta_{|\Z^2}$ the restriction of $\cV^\delta$ to $\Z^2$. Note that $\cV^\delta$ and $\cV^\delta_{|\Z^2}$ can be seen as elements of, respectively, $\fock{1}{}$ and $\fockdo{1}{}$. 
\end{definition}

\begin{lemma} \label{lem:GKneat}
For every $t > 0$, we have
\begin{equ}
d(t) = \lim_{\delta \to 0} \frac{1}{t} \, \mathbf E_{\nu_\rho}\Big[\Big(\int_0^t B_+ \cV^\delta_{|\Z^2} (\eta_s) \, \dd s\Big)^2\Big] \, . \label{e:GKneat}
\end{equ}
\end{lemma}

\begin{proof}
By writing $B_+ \cV^\delta = \sum_{x \in \Z^2} \cV^\delta(x) \, B_+(\1_{x})$, expanding the square, using the translation invariance of the ASEP, and recalling the even-symmetry of $\cV^\delta$, we obtain
\begin{equs}
d&(t) - \frac{1}{t} \, \mathbf E_{\nu_\rho}\Big[\Big(\int_0^t B_+ \cV^\delta_{|\Z^2} (\eta_s) \, \dd s\Big)^2\Big] \\
&= \frac{1}{t} \, \sum_{x \in \Z^2} \big(1 - \cV^\delta *_{\Z^2} \cV^\delta(x)\big) \, \mathbf E_{\nu_\rho}\Big[\Big(\int_0^t B_+(\1_{\{0\}})[\eta_s] \, \dd s\Big) \Big(\int_0^t B_+(\1_{\{x\}})[\eta_s] \, \dd s\Big)\Big] \, , \label{e:GKneatProof}
\end{equs}
where $*_{\Z^2}$ stands for the $\Z^2$-convolution, defined as: $f *_{\Z^2} g := \sum_{y \in \Z^2} f(\cdot-y) \, g(y)$. Now, we claim that, for every $\delta \in (0,1]$ and $x \in \Z^2$, we have
\begin{equ} \label{e:SquareRootUnity}
|(\cV^\delta *_{\Z^2} \cV^\delta)(x) - 1| \lesssim \delta \, (1 + |x|) \, .
\end{equ}
Then, we deduce the desired convergence~\eqref{e:GKneat} because the space-time correlations of the ASEP decay exponentially fast in space on any finite time window. As for the claim~\eqref{e:SquareRootUnity}, it easily follows by first writing
\begin{equs}
\Big|(\cV^\delta *_{\Z^2} \cV^\delta)(x) - \delta^2 \sum_{y \in \Z^2} \cV(\delta y)^2\Big| &= \Big|\delta^2 \sum_{y \in \Z^2} \cV(\delta y - \delta x) \cV(\delta y) - \delta^2 \sum_{y \in \Z^2} \cV(\delta y)^2\Big| \\
&\leq \|\cV\|_{\mathrm{Lip}} \ \delta |x| \ \delta^2 \sum_{y \in \Z^2} |\cV(\delta y)| \, ,
\end{equs}
and then estimating the appearing Riemann sums, as follows,
\begin{equ}
\delta^2 \sum_{y \in \Z^2} \cV(\delta y)^2 = \|\cV\|_{L^2(\R^2)}^2 + O(\delta) = 1 + O(\delta) \, , \quad \delta^2 \sum_{y \in \Z^2} |\cV(\delta y)| = O(1) \, .
\end{equ}
\end{proof}

\subsection{Roadmap}

\label{sec:Roadmap}

Before turning to the technical details of the proof, we first provide an overview of its underlying strategy and highlight the main ideas that will play a role in the arguments to come.
\ms

{\noindent \it The role of the fluctuation-dissipation relation}

The cornerstone of the proof of Theorem~\ref{th:CLT} is to establish a \emph{fluctuation-dissipation} relation, in a similar fashion to~\cite{LY97} (ASEP in dimension $d \geq 3$) and~\cite{CGT24} (SBE in dimension $d \geq 3$, and in dimension $2$ at weak coupling) but with the crucial difference that here the asymmetry contributes superdiffusively. Schematically, the fluctuation-dissipation relation consists in finding, for any $\vphi \in \cS(\R^2)$, a \emph{corrector} $\mathrm{Cor}_t^\tau(\vphi) \in L^2(\nu_\rho)$ such that
\begin{equ} \label{e:FluctuationDissipation}
\tau A_+ Y_t^\tau(\vphi) \approx \frac{d^{\lim}}{2} \, Y_t^\tau(\partial_{\fe}^2 \vphi) - (\partial_t + \tau L) \, \mathrm{Cor}_t^\tau(\vphi) \quad \mbox{ in $H^{-1}$ norm.}
\end{equ}
The two summands on the right hand side of~\eqref{e:FluctuationDissipation} can be respectively interpreted as the dissipation term and the fluctuation term produced by the asymmetry.
\ss

We briefly explain the interest of the relation~\eqref{e:FluctuationDissipation}. Start from Dynkin's decomposition
\begin{equ} \label{e:DynkinDecomposition}
Y_t^\tau(\vphi) = Y_0^\tau(\vphi) + \int_0^t (\partial_s + \tau L) Y_s^\tau(\vphi) \, \dd s + M_t^\tau(Y_\cdot^\tau) \, ,
\end{equ}
where we use the notation $M_t^\tau(f_\cdot) := f_t(\eta_{\tau t}) - f_0(\eta_{0}) - \int_0^t (\partial_s + \tau L) f_s(\eta_s)$ for the Dynkin martingale. It is then relatively straightforward to justify (see Section~\ref{sec:EasyDynkin}) that, as $\tau \to \infty$,
\begin{equ} \label{e:DynkinEasy}
\int_0^t (\partial_s + \tau (S + A_0 + A_-)) Y_s^\tau(\vphi) \, \dd s \approx \frac{D_0[\ff, \ff]}{2} \int_0^t Y_s^\tau(\partial_\ff^2 \vphi) \, \dd s \, ,
\end{equ}
but capturing the limiting contribution of $A_+$ is more subtle. To do so, we integrate the fluctuation-dissipation relation~\eqref{e:FluctuationDissipation} over time and deduce
\begin{equ} \label{e:DynkinFluctuationDissipation}
\int_0^t \tau A_+ Y_s^\tau(\vphi) \, \dd s \approx \frac{d^{\lim}}{2} \int_0^t Y_t^\tau(\partial_{\fe}^2 \vphi) \, \dd s + M_t^\tau(\mathrm{Cor}_\cdot^\tau(\vphi)) \, ,
\end{equ}
where $M_t^\tau(\mathrm{Cor}^\tau_\cdot)$ is the Dynkin martingale corresponding to the corrector. Note that we have neglected the boundary terms of this martingale in~\eqref{e:DynkinFluctuationDissipation}: this is because the corrector is actually expected to be small in $L^2$ norm. Finally, by combining~\eqref{e:DynkinDecomposition},~\eqref{e:DynkinEasy} and~\eqref{e:DynkinFluctuationDissipation}, we deduce that the fluctuation process $Y_\cdot^\tau$ formally satisfies the martingale problem of $Y_\cdot^{\lim}$ in the limit $\tau \to \infty$, which hints at Theorem~\ref{th:CLT}.
\ms

{\noindent \it The SBE fluctuation-dissipation relation}
\ss 

Before deriving the fluctuation-dissipation relation~\eqref{e:FluctuationDissipation} for the ASEP, it is natural to first work at the level of the SBE since its generator admits nicer expressions. For the SBE, the analogue of~\eqref{e:FluctuationDissipation} becomes
\begin{equ} \label{e:FluctuationDissipationSBE}
\tau \fA_+ \, \fU_t^\tau \vphi \approx \frac{d^{\lim}}{2} \, \fU_t^\tau \partial_\fe^2 \vphi - (\partial_t + \tau \fL) \, \mathfrak{Cor}_t^\tau(\vphi) \, , \quad \mbox{ in $H^{-1}$ norm,}
\end{equ}
where the scaling operator $\fU_t^\tau$ was defined in Definition~\ref{def:Utau}. To build the SBE corrector $\mathfrak{Cor}_t^\tau(\vphi)$ and establish~\eqref{e:FluctuationDissipationSBE}, we shall rely on the \emph{Ansatz} of the SBE resolvent introduced in~\cite{CMT25}.
\ss

Schematically, this Ansatz $\mathfrak{Ans}_\mu \colon \fock{}{} \to \fock{}{}$ takes the form
\begin{equ} \label{e:AnsatzRoadmap}
\mathfrak{Ans}_\mu = \sum_{n = 0} \big((\mu - \fS - \fh_\mu)^{-1} \fA_+\big)^n \, (\mu - \fS - \fh_\mu)^{-1} \, ,
\end{equ}
where $\fh_\mu \colon \fock{}{} \to \fock{}{}$ is an explicit diagonal operator (cf. Definition~\ref{def:DiagonalOperator}) that satisfies
\begin{equ} \label{e:FixedPointRoadmap}
\fh_\mu \approx \fA_- (\mu - \fS - \fh_\mu)^{-1} \, \fA_+ \, .
\end{equ}
This operator $\fh_\mu$ can be heuristically thought of as the generator of the effective dynamics produced by the asymmetry at large scales. We refer to~\cite[Section 2.3]{CMT25} for more details about the importance of the fixed point equation~\eqref{e:FixedPointRoadmap} and the motivation behind the definition~\eqref{e:AnsatzRoadmap}. The key point is that the Ansatz satisfies, for any reasonable functional $\bm \vphi \in \fock{}{}$,
\begin{equ} \label{e:AnsatzHeuristic}
(\mu - \bar \fL) \, \mathfrak{Ans}_\mu(\bm \vphi) \approx \bm \vphi - \fA_- (\mu - \fS - \fh_\mu)^{-1} \bm \vphi \quad \mbox{ in $H^{-1}$ norm,}
\end{equ}
where $\bar \fL := \fL - \fA_0$ is the generator of the centred SBE (i.e.~\eqref{e:ASEPTrueSBE} without the transport term). It turns out that, if $\bm \vphi$ is a macroscopic functional, then the second summand on the right hand side of~\eqref{e:AnsatzHeuristic} is negligible: hence, in this case, $\mathfrak{Ans}_\mu(\bm \vphi)$ can be thought of as an approximation of the resolvent $(\mu - \bar \fL)^{-1} \bm \vphi$ (which is the key to the analysis in~\cite{CMT25}). However, if we instead consider $\bm \vphi = \fA_+ \fU_t^\tau \vphi$ and take $\mu = 1/\tau$, then~\eqref{e:AnsatzHeuristic} leads to
\begin{equ} \label{e:AnsatzHeuristicA+}
(1 - \tau \bar \fL) \, \mathfrak{Ans}_{1/\tau}(\fA_+ \fU_t^\tau \vphi) \approx \tau \fA_+ \fU_t^\tau \vphi - \tau \fh_{1/\tau} \, \fU_t^\tau \vphi \, ,
\end{equ}
because of the fixed point equation~\eqref{e:FixedPointRoadmap}. In fact, it should be possible to replace above $(1 - \tau \bar \fL)$ by $- (\partial_t + \tau \fL)$, and we should have $\tau \,  \fh_\mu \, \fU_t^\tau \vphi \approx \frac{d^{\lim}}{2} \, \fU_t^\tau \partial_\fe^2 \vphi$. Therefore, we end up with the desired fluctuation-dissipation relation~\eqref{e:FluctuationDissipationSBE} by choosing the corrector therein as
\begin{equ} \label{e:SBECorrectorRoadmap}
\mathfrak{Cor}_t^\tau := \mathfrak{Ans}_{1/\tau}(\fA_+ \fU_t^\tau \vphi) \, .
\end{equ}

Proving that the corrector~\eqref{e:SBECorrectorRoadmap} indeed satisfies the desired fluctuation-dissipation relation~\eqref{e:FluctuationDissipationSBE} is one of the purposes of Sections~\ref{sec:AnsatzDef} and~\ref{sec:AnsatzA+} (see e.g.~\eqref{e:AnsatzA+Gen} in Proposition~\ref{p:AnsatzA+}).
\ms

\noindent{\it Transferring the fluctuation-dissipation relation from the SBE to the ASEP}
\ss

Now, we need to deduce a fluctuation-dissipation relation~\eqref{e:FluctuationDissipation} for the ASEP from the previous SBE fluctuation-dissipation relation~\eqref{e:FluctuationDissipationSBE}. It turns out that the hardest step is to show that
\begin{equ} \label{e:ComparisonGenRoadmap}
\fsQ^* \fsP \, \tau \fL \, \mathfrak{Cor}_t^\tau(\vphi) \approx \tau L \, \fsR \fsP \,  \mathfrak{Cor}_t^\tau(\vphi) \mbox{ in $H^{-1}$ norm,}
\end{equ}
where the bridge operators $\fsR$, $\fsQ$ and $\fsP$ were defined in Section~\ref{sec:Bridges}. Indeed, once we have~\eqref{e:ComparisonGenRoadmap}, the idea is to apply the bridge $\fsQ^* \fsP$ to the left of~\eqref{e:FluctuationDissipationSBE} and to morally deduce~\eqref{e:FluctuationDissipation}, where the ASEP corrector therein is $\mathrm{Cor}_t^\tau(\vphi) := \fsR \fsP \, \mathfrak{Cor}_t^\tau(\vphi)$. The comparison of the generators~\eqref{e:ComparisonGenRoadmap} is (together with tightness, which is discussed below) the step that requires the most novelties compared to previous works.

The proof of~\eqref{e:ComparisonGenRoadmap} is naturally divided into two steps: first comparing the ASEP and softcore generators, and then the softcore and SBE generators. For simplicity, we focus our discussion here on the former, that is the \emph{hardcore removal}. A first promising starting point is given by the estimates of~\cite{LQSY04}, which essentially show that, for any $F \in \fockd{n}{}$,
\begin{equ} \label{e:HCremovalNaive}
\|(- S)^{-1/2} \, (L \, \fsR - \fsQ^* \, \cL) F\| \leq C(n) \, \|(- \cS)^{1/2} F\| \, ,
\end{equ}
where $C(n)$ is a factor that grows polynomially fast in the degree $n$ of $F$. However, since the $H^1$ norm of the corrector is of order one, the estimate~\eqref{e:HCremovalNaive} can at best give
\begin{equ} \label{e:HCremovalNaiveCorrector}
\|(- \tau S)^{-1/2} \, (\tau L \, \fsR - \fsQ^* \, \tau \cL) \, \fsP \, \mathrm{Cor}_t^\tau(\vphi)\| \lesssim 1 \, .
\end{equ}
While~\eqref{e:HCremovalNaiveCorrector} is already highly nontrivial (because the generators $L$ and $\cL$ themselves are not bounded from $H^1$ to $H^{-1}$, cf. Remark~\ref{rmk:noGSC}), this is still not sufficient to obtain a vanishing error in the comparison~\eqref{e:ComparisonGenRoadmap}.
\ss

To go beyond~\eqref{e:HCremovalNaive}, the crux of our strategy is to improve the hardcore removal estimate~\eqref{e:HCremovalNaive} by exploiting specific properties of the functional $F$ we are interested in, that is the corrector. In this perspective, our contribution is two-fold: first, we show a refined hardcore removal (see Proposition~\ref{p:HardcoreRemoval}) whose error is bounded by the $H^{1 + \gamma}$ norm of the input (with $\gamma \in [0,1)$)
\begin{equ} \label{e:HCremovalMeso}
\|(- S)^{-1/2} \, (L \fsR - \fsQ^* \cL) F\| \lesssim \|N^{10} \, (- \cS)^{(1+\gamma)/2} F\| \, ,
\end{equ}
and second, we show that the corrector has vanishing $H^{1+\gamma}$ norm when $\gamma > 0$ (see Proposition~\ref{p:AnsatzA+})
\begin{equ} \label{e:CorrectorH1gammaRoadmap}
\tau^{1/2} \, \|N^{10} \, (- \cS)^{(1+\gamma)/2} \, \fsP \, \mathfrak{Cor}_t^\tau(\vphi)\| \lesssim_\gamma \frac{1}{\log_*(\tau)^{1/6 + o(1)}} \xrightarrow[\tau \to \infty]{} 0 \, .
\end{equ}
Intuitively, the extra factor $(-\cS)^\gamma$ in~\eqref{e:HCremovalMeso}  contributes as a factor $\ell^{-2 \gamma}$ on space scales of order~$\ell$: consequently,~\eqref{e:HCremovalMeso} does not improve~\eqref{e:HCremovalNaive} on microscopic scales ($\ell \asymp 1$), but it does on larger scales ($\ell \gg 1$). The property~\eqref{e:CorrectorH1gammaRoadmap} can be then interpreted as follows: the $H^1$ norm of the corrector is concentrated on large space scales (or, equivalently, small Fourier modes). Let us highlight that the latter fact is far from immediate and constitutes a major difference with what happens for the ASEP in dimension $d \geq 3$: in this case, the $H^1$ norm of the corrector would be concentrated on microscopic scales\footnote{Which reflects the fact that the ASEP and SBE in dimensions $d \geq 3$ cannot be sharply compared: their limiting diffusivity generically differs and depends on the microscopic details of the system.}.

Let us mention that the comparison between the softcore and SBE generators follows a similar strategy, but with an additional subtlety: in addition to the $H^{1 + \gamma}$ norm $\|(- \fL_0)^{(1+\gamma)/2} \cdot\|$, we also need to rely on the $\fe$-directed $H^1$ norm $\|(- \fL_0^\fe)^{1/2} \cdot\|$. In fact, the $\fe$-directed $H^1$ norm of the corrector also vanishes, and this crucial property was already noticed and exploited in~\cite{CMT25}. This can be intuitively interpreted as follows: the $H^1$ norm of the corrector is concentrated on anisotropic space scales, that are much more stretched in the direction $\fe$ than in the transverse direction $\ff$.
\ms

\noindent{\it The tightness}
\ss

The previously sketched strategy gives the fluctuation-dissipation relation~\eqref{e:FluctuationDissipation} for the ASEP, from which the convergence of the finite-dimensional laws of the fluctuation process $(Y_t^\tau)_{t \geq 0}$ to $(Y_t^{\lim})_{t \geq 0}$ easily follows. However, we still need to prove the tightness in order to derive the invariance principle of Theorem~\ref{th:CLT}; in particular, the difficulty lies in proving the tightness of the asymmetric term $\int_0^t \tau A_+ \, Y_s^\tau(\vphi) \, \dd s$ in Dynkin's decomposition~\eqref{e:DynkinDecomposition}. As mentioned in the introduction, the difficulty is that, compared to systems at the supercritical dimension or in the weak-coupling regime, the Itô trick is not even sufficient to show that the asymmetric term has bounded second moment as $\tau \to \infty$: indeed, it holds
\begin{equ}
\|(1 - \tau S)^{-1/2} \, \tau A_+ Y_s^\tau(\vphi)\| \asymp_{\vphi} \log_*(\tau)^{1/6} \to \infty \quad \mbox{(provided that $\vphi \not \equiv 0$).}
\end{equ}
The failure of the Itô trick is not just a technical artefact: it reflects that the superdiffusive part of the dynamics cannot be handled by the relaxation properties of the symmetric dynamics (that is diffusive) on which Itô's trick entirely relies.
\ss

To avoid this obstruction, the novelty here is to use once again the fluctuation-dissipation relation. By the above discussion, we morally have
\begin{equ} \label{e:TightnessNaive}
\int_s^t \tau A_+ \, Y_r^\tau(\vphi) \, \dd r \approx \frac{d^{\lim}}{2} \int_s^t Y_r^\tau(\partial_\fe^2 \vphi) \, \dd r + M_t^\tau(\mathrm{Cor}_\cdot^\tau(\vphi)) -  M_s^\tau(\mathrm{Cor}_\cdot^\tau(\vphi)) \, .
\end{equ}
The tightness of the first summand is trivial, so that the difficulty lies in the Dynkin martingale of the corrector. By computing its quadratic variation and using \emph{hypercontractivity} (and ignoring important issues about the jumps of the martingale, which we discuss below), we can bound the increments of this martingale as follows, for $p > 2$,
\begin{equ} \label{e:KolmogorovNaive}
\mathbf E_{\nu_\rho}[|M_t^\tau(\mathrm{Cor}_\cdot^\tau(\vphi)) - M_s^\tau(\mathrm{Cor}_\cdot^\tau(\vphi))|^p] \lesssim_p |t-s|^{p/2} \, \|\alpha(p)^N \, (- \tau S)^{1/2} \, \mathrm{Cor}^\tau(\vphi)\| \, ,
\end{equ}
where $\alpha(p) > 1$ is an explicit constant that satisfies $\alpha(p) \to 1$ as $p \to 2$. Therefore, we could verify a Kolmogorov-Chentsov criterion provided that, uniformly in $\tau$, the $H^1$ norm of the corrector decays \emph{exponentially} fast in the chaos degree. An important contribution of the present paper is to establish the latter property, and more generally to show that the estimates of~\cite{CMT25} on the Ansatz still hold when penalised by an exponential weight $\alpha^N$ (see Proposition~\ref{p:GenericAnsatz} below).
\ss

There are two crucial subtleties that we have omitted in this informal discussion, and that will be taken into account precisely in the actual proof. First, there is actually no chance that~\eqref{e:KolmogorovNaive} holds as written, because the Kolmogorov-Chentsov criterion would then imply that the Dynkin martingale is continuous! A related and in fact much more important issue is that we have neglected in~\eqref{e:TightnessNaive} the boundary terms of the Dynkin martingale, that are given by
\begin{equ} \label{e:BoundaryTermRoadmap}
\mathrm{Cor}_t^\tau(\vphi)(\eta_{\tau t}) - \mathrm{Cor}_s^\tau(\vphi)(\eta_{\tau s}) \, .
\end{equ}
It can be easily shown that the second moment of~\eqref{e:BoundaryTermRoadmap} goes to $0$ as $\tau \to \infty$, but to prove tightness we would need much finer control as $|t-s| \to 0$ (e.g. a Kolmogorov-Chentsov criterion), which we do not have. Our workaround is to tune the fluctuation-dissipation relation used in~\eqref{e:TightnessNaive} in a way that depends on the size of the given time window $|t-s|$: more precisely, we shall choose the parameter of the Ansatz in~\eqref{e:SBECorrectorRoadmap} as $\mu = (|t-s| \, \tau)^{-1}$ instead of $1/\tau$. This modification provides much better control over the boundary terms and the jumps of the corresponding Dynkin martingale, which eventually allows us to obtain a Kolmogorov-Chentsov criterion for the asymmetric term $\int_s^t \tau A_+ Y_t^\tau(\vphi) \, \dd r$.

\section{Comparison between the ASEP and SBE generators}

\label{sec:FromASEPtoSBE}

This section is devoted to the comparison between the ASEP generator $L$ and the SBE generator $\fL$, which will later be the key to transferring the analysis of~\cite{CMT25} to the ASEP setting. More precisely, we establish the following theorem that bounds the $H^{-1}$ error between $L$ and $\fL$ in terms of the $H^{1+\gamma}$ norm and the anisotropic $H^1$ norm of the input.

\begin{theorem} \label{th:ComparisonGen}
Let $\gamma \in [0,1)$. Then, for every $\phi \in \fock{\rap}{}$, we have
\begin{equ}
\|(- S)^{-1/2} \, (L \fsR \fsP - \fsQ^* \fsP \fL) \, \phi\| \lesssim_\gamma \|\fN^{6} (- \fL_0)^{(1+\gamma)/2} \phi\| + \|\fN^{6} (- \fL_0^\fe)^{1/2} \phi\| \, . \label{e:ComparisonGen}
\end{equ}
Further, the bound remains valid if $(L, \fL)$ is replaced by any of $(S, \fS)$, $(A_+, \fA_+)$, $(A_-, \fA_-)$, or $(A_0, \fA_0)$; and the second summand on the right hand side is only needed for $(A_+, \fA_+)$.
\end{theorem}

In particular, we deduce the $H^1 \to H^{-1}$ bound
\begin{equ}
\|(- S)^{-1/2} \, (L \fsR \fsP - \fsQ^* \fsP \fL) \, \phi\| \lesssim \|\fN^{6} \, (- \fL_0)^{1/2} \phi\| \, , \label{e:ComparisonGenEasy}
\end{equ}
which is closer in spirit to the classical hardcore removal estimates from~\cite{LQSY04}. As mentioned above, the comparison~\eqref{e:ComparisonGenEasy} is already highly nontrivial. In fact, it still holds without Assumption~\ref{ass:Unidirectional} (see Remark~\ref{rmk:wRemoval} below) and would suffice to show the results sketched in Remark~\ref{rmk:Unidirectional}. However, the comparison~\eqref{e:ComparisonGen} becomes substantially better than~\eqref{e:ComparisonGenEasy} when the $H^1$ norm of $\phi$ is concentrated on Fourier modes $k$ such that $|k| \ll 1$ and $|\fe \cdot k| \ll |k|$, which intuitively correspond to spatial scales that are, respectively, much larger than the lattice mesh and very stretched in the direction of the drift. We will be able to fully leverage this improvement in the next section, when taking as $\phi$ the Ansatz.
\ss

The rest of this section is organised as follows: we first compare the ASEP and softcore generators in Section~\ref{sec:FromASEPtoSoft}, then compare the softcore and SBE generators in Section~\ref{sec:FromSofttoSBE}, and deduce afterwards Theorem~\ref{th:ComparisonGen} in Section~\ref{sec:FromASEPtoSBEccl}.

\subsection{Comparison between the ASEP generator and the soft core generator}

\label{sec:FromASEPtoSoft}

The main purpose of this section is to compare the generators $L$ and $\cL$ as follows.

\begin{proposition} \label{p:HardcoreRemoval}
Let $\gamma \in [0, 1)$. Then, for every $f \in \fockd{\rap}{}{}$, we have
\begin{equ}
\|(- S)^{-1/2} \, (L \fsR - \fsQ^* \cL) f\| \lesssim_\gamma \, \|\cN^{6} \, (- \cS)^{(1 + \gamma)/2} f\| \, . \label{e:HardcoreRemoval}
\end{equ}
Further, the bound remains valid if $(L, \cL)$ is replaced by any of $(S, \cS)$, $(A_+, \cA_+)$, $(A_-, \cA_-)$, or $(A_0, \cA_0)$.
\end{proposition}

A result very similar to Proposition~\ref{p:HardcoreRemoval} was proved in~\cite{LQSY04} in the case $\rho = 1/2$ and $\gamma = 0$, but, as noted before, the extension to $\gamma > 0$ plays a crucial role in our analysis.
\ss

We also collect a few useful bounds about the bridges $\fsQ$ and $\fsQ^*$.

\begin{lemma} \label{lem:BoundQ}
For every $g \in \fockdo{\rap}{}{}$, $f \in \fockd{\rap}{}{}$, and $\mu \geq 0$, we have
\begin{equs}
\|\fsQ g\| \lesssim \|g\| \, , \quad \|(\mu - &\cS)^{1/2} \, \fsQ g\| \lesssim \|N^3 \, (\mu - S)^{1/2} g\| \, , \label{e:BoundQ} \\
\|\fsQ^* f\| \lesssim \|f\| \, , \quad \|(\mu - &S)^{-1/2} \, \fsQ^* f\| \lesssim \|\cN^3 \, (\mu - \cS)^{-1/2} f\| \, , \label{e:BoundQ*} \,
\end{equs}
and, provided that $\mu > 0$,
\begin{equ}
\|(\mu - S)^{-1/2} \, (\fsQ^* - \fsR) f\| \lesssim \log_{*}(1/\mu) \, \|\cN^2 \, (\mu - \cS)^{1/2} f\| \, . \label{e:RemovalQ*R}
\end{equ}
\end{lemma}
\ss

As in~\cite{LQSY04}, the main technical step to remove the hardcore constraint consists in bounding the $\ell^2$-norm of what lies near the diagonal by the adequate global Sobolev norm. The main novelty here compared to~\cite{LQSY04} is to take advantage of fractional Sobolev norms instead of relying exclusively on the $H^1$ norm. We state a result in this direction first for functions $f \in F \ell^2$ (see Remark~\ref{rmk:NonSymmetricFock} for the extension of $\cS$ on $F \ell^2$) and then for functions $g \in \fockdo{}{}$. We recall that the sets $\cZ_0(l), \cZ_1(l)$ and $\cZ_{\geq 2}(l)$ were defined at the end of Section~\ref{sec:Bridges}; we also denote $\cZ_{\geq 1}(l) := \cZ_1(l) \sqcup \cZ_{\geq 2}(l)$.

\begin{lemma} \label{lem:ZiBound}
Let $\gamma \in [0,1)$. Then, for every $f \in F \ell_*^2$, $l \geq 1$, and $\mu > 0$, we have
\begin{equs}
\|\1_{\cZ_{\geq 1}(l)} \, f\| &\lesssim_\gamma \min\big\{l \, \|\cN \, (-\cS)^{\gamma/2} f\| \, , \ l \, \sqrt{\log_{*}(1/\mu)} \, \|\cN \, (\mu-\cS)^{1/2} f\|\big\} \, , \label{e:Z1Bound} \\
\|\1_{\cZ_{\geq 2}(l)} \, f\| &\lesssim_\gamma l^2 \, \|\cN^2 \, (-\cS)^{(1+\gamma)/2} f\| \, , \label{e:Z2Bound}
\end{equs}
\end{lemma}

\begin{lemma} \label{lem:ZiBoundHC}
For every $g \in \fockdo{}{}$, $\ell \geq 1$, and $\mu > 0$, we have
\begin{equs}
\|\1_{\cZ_{\geq 1}(l)} \, g\| &\lesssim l \, \sqrt{\log_{*}(1/\mu)} \, \|\cN \, (\mu-S)^{1/2} g\| \, , \label{e:Z1BoundHC} \\
\|\1_{\cZ_{\geq 2}(l)} \, g\| &\lesssim l^2 \, \|\cN^2 \, (-S)^{1/2} g\| \, . \label{e:Z2BoundHC}
\end{equs}
\end{lemma}

Informally, Lemma~\ref{lem:ZiBound} shows that restricting to $\cZ_{\geq 1}(l)$ and $\cZ_{\geq 2}(l)$ allows one to gain a factor $(- \cS)^{\gamma/2}$ and $(- \cS)^{(1+\gamma)/2}$, respectively, for any $\gamma \in [0,1)$. In fact, the same estimate holds for $\gamma = 1$ at the cost of a logarithmically diverging factor $\sqrt{\log_{*}(1/\mu)}$, which is however too large to be useful in most applications. Furthermore, Lemma~\ref{lem:ZiBoundHC} extends these estimates to the hardcore Fock space when the upper bound is given in terms of the $H^1$-norm: indeed, in the hardcore setting, it is much more delicate to work with fractional Sobolev norms as in Lemma~\ref{lem:ZiBound}.

\begin{proof}[of Lemma~\ref{lem:ZiBound}]

First, we claim that for any $d \geq 1$, $\nu \in [0,d/2]$, $\mu > 0$, $F \in \ell^2(\Z^{2d})$, and $y_{1:d} \in \Z^{2d}$, we have
\begin{equ} \label{e:ZiBoundClaim}
|F(y_{1:d})|^2 \lesssim_{d, \nu} \log_{*}(1/\mu)^{\1_{\nu = d/2}} \, \|(\mu - \cS)^\nu F\|_{\ell^2(\Z^{2d})}^2 \, .
\end{equ}
In particular, for any $l \geq 1$, we deduce
\begin{equ} \label{e:ZiBoundClaimCorollary}
\|1_{|x_{1:d}|_\infty \leq l} \, F\|^2 \lesssim_{d, \nu} l^{2d} \, \log_{*}(1/\mu)^{\1_{\nu = d/2}} \, \|(\mu - \cS)^\nu F\|_{\ell^2(\Z^{2d})}^2 \, .
\end{equ}
To prove the claim~\eqref{e:ZiBoundClaim}, we write the Fourier inversion formula 
\begin{equ}
F(y_{1:d}) = (2 \pi)^{-d} \int_{[- \pi, \pi]^{2d}} \cF F (k_{1:d}) \, \exp(\iota \, y_{1:d} \cdot k_{1:d}) \, \dd k_{1:d} \, ,
\end{equ}
and we apply the Cauchy-Schwarz inequality to deduce that $|F(y_{1:d})|^2$ is bounded by 
\begin{equs}
(2 &\pi)^{-2d} \, \Big(\int_{[- \pi, \pi]^{2d}} (\mu + d_0(k_{1:d})/2)^{- 2 \nu} \, \dd k_{1:d}\Big) \Big(\int_{[- \pi, \pi]^{2d}} (\mu + d_0(k_{1:d})/2)^{2 \nu} \, |\cF F(k_{1:d})|^2 \, \dd k_{1:d}\Big) \\
&\lesssim_{d, \nu} \log_{*}(1/\mu)^{\1_{\nu = d/2}} \, \|(\mu - \cS)^\nu F\|_{\ell^2(\Z^{2d})}^2 \, ,
\end{equs}
where, in order to bound the first integral by $C(d, \nu) \log_{*}(1/\mu)^{\1_{\nu = d/2}}$, we used the estimate $d_0(k_{1:d}) \asymp |k_{1:d}|^2$ for $k_{1:d} \in [- \pi, \pi]^{2d}$ (see~\eqref{e:d0D0equivalence} from Lemma~\ref{lem:SoftcoreRemoval} below).
\ss

Now, we turn to the proof of~\eqref{e:Z1Bound} and~\eqref{e:Z2Bound}. By orthogonality of chaoses, it is sufficient to prove these bounds for any $f \in F \ell^2_n$ and $n \geq 1$. Then notice that
\begin{equs}
\|\1_{\cZ_{\geq 1}(l)} f\|_{\ell^2(\Z^{2n})}^2 &\leq \sum_{x_{1:n} \in \Z^{2n}} \sum_{i \in \{1, \dots, n\} \, ; \, j \in \{1, \dots, n\} \setminus \{i\}} \1_{|x_i - x_j|_\infty \leq l} \, f(x_{1:n})^2 \, , \\
\|\1_{\cZ_{\geq 2}(l)} f\|_{\ell^2(\Z^{2n})}^2 &\leq \sum_{x_{1:n} \in \Z^{2n}} \sum_{i_1 \ne i_2 \in \{1, \dots, n\} \, ; \, j_1, j_2 \in \{1, \dots, n\} \setminus \{i_1, i_2\}} \1_{|x_{i_1} - x_{j_1}|_\infty \vee |x_{i_2} - x_{j_2}|_\infty \leq l} \, f(x_{1:n})^2 \, .
\end{equs}
Hence, to obtain~\eqref{e:Z1Bound} and~\eqref{e:Z2Bound}, it is sufficient to show that for any $d \in \{1,2\}$, $\nu \in [0,d/2]$, $\mu > 0$, $j_1, \dots, j_d \in \{d+1, \dots, n\}$ and $f \in F \ell^2_n$ we have
\begin{equ} \label{e:ZiBoundReductionBis}
\|\1_{|x_1 - x_{j_1}|_\infty \vee \dots \vee |x_d - x_{j_d}|_\infty \leq l} \, f\|_{\ell^2(\Z^{2n})}^2 \lesssim_\nu l^{2d} \, \log_{*}(1/\mu)^{\1_{\nu = d/2}} \, \|(\mu - \cS)^{\nu} f\|_{\ell^2(\Z^{2n})}^2 \, .
\end{equ}
(In particular, if $\nu < d/2$, then we can let $\mu \to 0$ and extend the statement to $\mu = 0$.)
To show~\eqref{e:ZiBoundReductionBis}, we first write
\begin{equ}
\|\1_{|x_1 - x_{j_d}|_\infty \vee \dots \vee |x_d - x_{j_d}|_\infty \leq l} \, f\|_{\ell^2(\Z^{2n})}^2 = \sum_{x_{d+1:n} \in \Z^{2(n-d)}} \sum_{z_{1:d}\in \Z^{2d}} \1_{|z_{1:d}|_\infty \leq l} \, f(z_1 + x_{j_1}, \dots, z_d + x_{j_d}, x_{d+1:n})^2 \, ,
\end{equ}
and then notice that the second sum corresponds to the left hand side of~\eqref{e:ZiBoundClaimCorollary} with the function $F$ that maps $z_{1:d} \in \Z^{2d}$ to $f(z_1 + x_{j_1}, \dots, z_d + x_{j_d}, x_{d+1:n})$. As a consequence, this second sum is bounded, up to a multiplicative constant $C(d, \nu) \, l^{2 d} \, \log_{*}(1/\mu)^{\1_{\nu = d / 2}}$, by
\begin{equs}
\int_{[-\pi, \pi]^{2 d}} &\dd k_{1:d} \, (\mu + d_0(k_{1:d})/2)^{2 \nu} \, |e^{\iota (k_1 \cdot x_{j_1} + \dots + k_d \cdot x_{j_d})} \, \cF_{1 : d} f(k_{1:d}, x_{d+1 : n})|^2 \\
&= \int_{[-\pi, \pi]^{2d}} \dd k_{1:d} \, (\mu + d_0(k_{1:d})/2)^{2 \nu} \, |\cF_{1 : d} f(k_{1:d}, x_{d+1 : n})|^2 \, ,
\end{equs}
where $\cF_{1 : d}$ stands for the Fourier transform applied to the first $d$ coordinates. Then we take the sum over $x_{d+1:n} \in \Z^{2(n-d)}$ and deduce that $\|\1_{|x_1 - x_{j_d}|_\infty \vee \dots \vee |x_d - x_{j_d}|_\infty \leq l} \, f\|_{\ell^2(\Z^{2n})}^2$ is bounded, up to a multiplicative constant $C(d,\nu) \, l^{2d} \, \log_{*}(1/\mu)^{\1_{\nu = d / 2}}$, by
\begin{equs}
\int_{[-\pi, \pi]^{2 n}} &\dd k_{1:n} \, (\mu + d_0(k_{1:d})/2)^{2 \nu} \, |\cF f(k_{1:n})|^2 \leq \|(\mu - \cS)^{\nu} f\|_{\ell^2(\Z^{2n})}^2 \, ,
\end{equs}
which shows~\eqref{e:ZiBoundReductionBis} and concludes the proof.
\end{proof}

\begin{proof}[of Lemma~\ref{lem:ZiBoundHC}]

The statement can be deduced from Lemma~\ref{lem:ZiBound} (more precisely, the second bound in~\eqref{e:Z1Bound} and the bound~\eqref{e:Z2Bound} with $\gamma = 0$) by adapting, mutatis mutandis, the proof of~\cite[Lemmas 4.1 and 4.2]{LQSY04}. For the reader's convenience, and because our sets $\cZ_{\geq 1}(\ell)$ and $\cZ_{\geq 2}(\ell)$ are slightly more general than their counterparts therein, we provide a sketch of the proof in Appendix~\ref{app:HCSC}.
\end{proof}

We are now ready to prove Lemma~\ref{lem:BoundQ} and  Proposition~\ref{p:HardcoreRemoval}.

\begin{proof}[of Lemma~\ref{lem:BoundQ}]
First, note that it is sufficient to show each bound for every $f \in \fockd{n}{}{}$ and $g \in \fockdo{n}{}{}$, for any $n \geq 1$. We start with the first bound of~\eqref{e:BoundQ}. From the definition~\eqref{e:defQ} of $\fsQ$, we immediately deduce
\begin{equs}
\|\fsQ g\|_{\ell^2(\Z^{2n})}^2 &= \|g\|_{\ell^2(\Z^{2n})}^2 + \sum_{y_{1:n} \in \Z^{2n}} g(y_{1:n})^2 \, \#\{1 \leq i \ne j \leq n \, : \, y_j - y_i = z_0 \, \mbox{ and } y_{1:n \setminus j} \in \cZ_0\} \\
&\leq 2 \, \|g\|_{\ell^2(\Z^{2n})}^2 \, , \label{e:BoundL2QProof}
\end{equs}
where in the last bound we noticed that, for any $y_{1:n} \in \Z^{2n}$, there can exist at most one pair $(i,j)$ satisfying the desired condition: indeed, such a pair must satisfy $|y_i - y_j|_\infty < l_0$ while $|y_k - y_\ell|_\infty \geq (2 l_0 + 1) - l_0 > l_0$ for every $\{k, \ell\} \ne \{i,j\}$, which determines $\{i,j\}$ (and then $(i,j)$ is determined by the relation $y_j = y_i + z_0$).

Now we turn to the second bound of~\eqref{e:BoundQ}. Since $\|(\mu - \cS)^{1/2} \cdot\|^2 = \mu \, \|\cdot\|^2 + \|(-\cS)^{1/2} \cdot\|^2$ (and similarly for $S$), it only remains to show
\begin{equ} \label{e:BoundQH1dot}
\|(- \cS)^{1/2} \, \fsQ g\|^2_{\ell^2(\Z^{2n})} \lesssim n^6 \, \|(- S)^{1/2} \, g\|_{\ell^2(\Z^{2n})}^2 \, .
\end{equ}
The latter can be proved by adapting, mutatis mutandis, the proof of~\cite[Lemma 4.3]{LQSY04} (and using, in particular,~\eqref{e:Z2BoundHC}). For the reader's convenience, and because our operator $\fsQ$ differs slightly from its counterpart therein (in particular, here $\fsQ g(x_{1:n})$ vanishes as soon as $x_{1:n}$ contains multiple isolated double sites), we give a sketch of the proof in Appendix~\ref{app:HCSC}.
\ss

As for the bounds in~\eqref{e:BoundQ*}, observe that they readily follow from those in~\eqref{e:BoundQ} by using a simple duality argument.
\ss

Finally, we tackle~\eqref{e:RemovalQ*R}. Let $h \in \fockdo{n}{}{}$. By the definition~\eqref{e:defQ} of $\fsQ$, we have
\begin{equ}
\langle (\fsQ^* - \fsR) f, h \rangle = \langle f, (\fsQ - \fsR^*) h \rangle = \langle f, \1_{DI \cap \cZ_1} \fsQ h \rangle = \langle \1_{DI \cap \cZ_1} f, \fsQ \, \1_{\cZ_1(3 l_0 + 1)} \, h \rangle \, ,
\end{equ}
where in the last equality we were able to add the indicator $\1_{\cZ_1(3 l_0 + 1)}$ in front of $h$ by using~\eqref{e:defQ} and observing that, if $x_{1:n} \in DI \cap \cZ_1 = DI \cap \cZ_1(2 l_0 + 1)$, then $(x_i, x_j + z_0, x_{1:n \setminus i,j}) \in \cZ_1(3 l_0 + 1)$.
Now, by first using the Cauchy Schwarz inequality and~\eqref{e:BoundQ}, and then~\eqref{e:Z1Bound} applied with $f$ and~\eqref{e:Z1BoundHC} applied with $h$, we deduce
\begin{equ}
|\langle \fsQ^* f - \fsR f, h \rangle| \lesssim \|\1_{\cZ_1} f\| \, \|\1_{\cZ_1(3 l_0 + 1)} h\| \lesssim \log_{*}(1/\mu) \, n^2 \, \|(\mu - \cS)^{1/2} f\| \, \|(\mu - S)^{1/2} h\| \, .
\end{equ}
By taking the supremum over $h$ such that $\|(\mu - S)^{1/2} h\| = 1$, we deduce the desired estimate~\eqref{e:RemovalQ*R} and conclude the proof.
\end{proof}

\begin{proof}[of Proposition~\ref{p:HardcoreRemoval}]

We show that, for any $(O, \cO) \in \{(S, \cS), (A_+, \cA_+), (A_-, \cA_-), (A_0, \cA_0)\}$ and for every $f \in \fockd{n}{}$, $n \geq 0$, $g \in \fockd{\rap}{}$, it holds
\begin{equ} \label{e:HCremovalOclaim}
\big|\langle O \, \fsR f, \fsR g \rangle - \langle \cO f, g\rangle\big| \lesssim_{\gamma} n^3 \, \|(- \cS)^{(1+\gamma)/2} f\| \, \|(- \cS)^{1/2} g\| \, , 
\end{equ}
Then, by noticing that $\fsQ^* \fsR^* = \mathrm{Id}$ and using~\eqref{e:BoundQ*}, we deduce
\begin{equs}
\|(- S)^{-1/2} \, (O \fsR - \fsQ^* \cO) f\| &= \|(- S)^{-1/2} \, \fsQ^* (\fsR^* O \fsR - \cO) f\| \\
&\lesssim n^3 \, \|(- \cS)^{-1/2} \, (\fsR^* O \fsR - \cO) f\| \\
&\lesssim_\gamma n^6 \, \|(- \cS)^{(1+\gamma)/2} f\| \, ,
\end{equs}
where in the last bound we used~\eqref{e:HCremovalOclaim}. Finally, by using orthogonality of chaoses, we can extend this bound to any $f \in \fockd{}{}$ up to replacing the factor $n^6$ by $\cN^6$, which results in the desired estimate~\eqref{e:HardcoreRemoval} and would conclude the proof.
\ss

Now, we prove~\eqref{e:HCremovalOclaim}. Obviously, we can assume that $g$ belongs respectively to $\fockdo{n}{}{}$, $\fockdo{n+1}{}{}$, $\fockdo{n-1}{}{}$, $\fockdo{n}{}{}$ when we prove the statement for $O = S, A_+, A_-, A_0$.
\ss

We tackle the case $(O, \cO) = (S, \cS)$. We write
\begin{equs}
\langle (-\cS) f, g\rangle - \langle (- S) \fsR f, \fsR g\rangle = n! \times n \sum_{z \in \Z^2} &s(z) \sum_{x_{1:n} \in \Z^{2n}} \big(1 -  \1_{HC}(x_{1:n}) \1_{HC}(x_1 + z, x_{2:n})\big) \\
&\times (f(x_1 + z, x_{2:n}) - f(x_{1:n})) \, (g(x_1 + z, x_{2:n}) - g(x_{1:n})) \, .
\end{equs}
Now, observe that if either $x_{1:n}$ or $(x_1 + z, x_{2:n})$ belongs to the diagonal $DI$ (with $z \in \mathrm{Supp}(s)$ so that $|z|_\infty \leq l_0$), then they both belong to $\cZ_{\geq 1}$. Thus we deduce, by the Cauchy-Schwarz inequality,
\begin{equs}
|\langle (-\cS) f, g\rangle - \langle (- S) \fsR f, \fsR g\rangle| &\lesssim n \sum_{z \in \Z^2} s(z) \, \|\1_{\cZ_{\geq 1}} \nabla_z f\| \, \|\1_{\cZ_{\geq 1}} \nabla_z g\|\ \\
&\lesssim n \, \Big(\sum_{z \in \Z^2} s(z) \, \|\1_{\cZ_{\geq 1}} \nabla_z f\|^2\Big)^{1/2} \Big(\sum_{z \in \Z^2} s(z) \, \|\1_{\cZ_{\geq 1}} \nabla_z g\|^2\Big)^{1/2} \, .
\end{equs}
By using~\eqref{e:Z1Bound} and~\eqref{e:FracLaplNabla}, we deduce
\begin{equ} \label{e:HCremovalS}
|\langle (-\cS) f, g\rangle - \langle (- S) \fsR f, \fsR g\rangle|
\lesssim_{\gamma} n \, \|(- \cS)^{(1+\gamma)/2} f\| \|(- \cS)^{(1+\gamma)/2} g\| \, ,
\end{equ}
which proves~\eqref{e:HCremovalOclaim} for $(O, \cO) = (S, \cS)$ (note that we can use the crude bound $-\cS \lesssim \cN$ to drop a factor $(-\cS)^\gamma$ in the second norm).
\ss

We tackle the case $(O, \cO) = (A_+, \cA_+)$. We write
\begin{equs}
\langle &\cA_+ f, g\rangle - \langle A_+ \fsR f, \fsR g\rangle \\
&= (n+1)! \times (-1) \, \sqrt{\rho(1-\rho)} \, n \sum_{z \in \Z^2} a(z) \sum_{x_{1:n} \in \Z^{2n}} \big(1 - \1_{HC}(x_{1:n}) \1_{HC}(x_1 + z, x_{2:n})\big) \\
&\qquad \qquad \qquad \qquad \qquad \qquad \qquad \qquad \times (f(x_1 + z, x_{2:n}) - f(x_{1:n})) \, g(x_1 + z, x_{1:n}) \, .
\end{equs}
Now, observe that if either $x_{1:n}$ or $(x_1 + z, x_{2:n})$ belongs to the diagonal $DI$ (with $z \in \mathrm{Supp}(a)$), then not only do they belong to $\cZ_{\geq 1}$ (as noticed previously), but also $(x_1+z, x_{1:n})$ belongs to $\cZ_{\geq 2}$. Thus we deduce, by the Cauchy-Schwarz inequality and the bound $|a(z)| \leq s(z)$,
\begin{equs}
|\langle \cA_+ f, g\rangle - \langle A_+ \fsR f, \fsR g\rangle| &\lesssim n^{1/2} \times n \sum_{z \in \Z^2} s(z) \, \|\1_{\cZ_{\geq 1}} \nabla_z f\| \|\1_{\cZ_{\geq 2}} g\| \\
&\lesssim n^{3/2} \, \Big(\sum_{z \in \Z^2} s(z) \, \|\1_{\cZ_{\geq 1}} \nabla_z f\|^2\Big)^{1/2} \Big(\sum_{z \in \Z^2} s(z) \, \|\1_{\cZ_{\geq 2}} g\|^2\Big)^{1/2} \, .
\end{equs}
By using~\eqref{e:Z1Bound},~\eqref{e:FracLaplNabla}, and~\eqref{e:Z2Bound}, we deduce 
\begin{equ} \label{e:HCremovalA+}
|\langle \cA_+ f, g\rangle - \langle A_+ \fsR f, \fsR g\rangle|
\lesssim_{\gamma} n^{5/2} \, \|(- \cS)^{(1+\gamma)/2} f\| \|(-\cS)^{(1+\gamma)/2} g\| \, ,
\end{equ}
which proves~\eqref{e:HCremovalOclaim} for $(O, \cO) = (A_+, \cA_+)$ (once again, we can use the crude bound $-\cS \lesssim \cN$ to drop a factor $(-\cS)^\gamma$ in the second norm).
\ss

We tackle the case $(O, \cO) = (A_-, \cA_-)$. By starting from~\eqref{e:HCremovalA+} applied with $(g, f)$ instead of $(f,g)$ (hence, applied with $n-1$ instead of $n$) and using that $(A_+)^* = - A_-$ and $(\cA_+)^* = - \cA_-$, we obtain
\begin{equ}
|\langle \cA_- f, g\rangle - \langle A_- \fsR f, \fsR g\rangle|
\lesssim_{\gamma} n^{5/2} \, \|(- \cS)^{(1+\gamma)/2} f\| \|(-\cS)^{(1+\gamma)/2} g\| \, ,
\end{equ}
which proves~\eqref{e:HCremovalOclaim} for $(O, \cO) = (A_-, \cA_-)$ (once again, we can use the crude bound $-\cS \lesssim \cN$ to drop a factor $(-\cS)^\gamma$ in the second norm).
\ss

Finally, we tackle the case $(O, \cO) = (A_0, \cA_0)$. We write
\begin{equs}
\langle \cA_0 f, g\rangle - \langle A_0 \fsR f, \fsR g \rangle = n! \times (1 - 2 \rho) \, n \sum_{z \in \Z^2} a(z) &\sum_{x_{1:n} \in \Z^{2n}} \big(1 - \1_{HC}(x_{1:n}) \1_{HC}(x_1 + z, x_{2:n})\big) \\
&\quad \times \big(f(x_1 + z, x_{2:n}) - f(x_{1:n})\big) \, g(x_{1:n}) \, .
\end{equs}
Observe that, as previously noticed, if either $x_{1:n}$ or $(x_1+z,x_{2:n})$ belongs to the diagonal $DI$ (with $z \in \mathrm{Supp}(a)$), then they both belong to $\cZ_{\geq 1}$. Thus we deduce, by the Cauchy-Schwarz inequality and the bound $|a(z)| \leq s(z)$,
\begin{equs}
|\langle \cA_0 f, g\rangle - \langle A_0 \fsR f, \fsR g \rangle| &\lesssim n \sum_{z \in \Z^2} s(z) \, \|\1_{\cZ_{\geq 1}} \nabla_z f\| \, \|\1_{\cZ_{\geq 1}} g\| \, \\
&\lesssim n \Big(\sum_{z \in \Z^2} s(z) \, \|\1_{\cZ_{\geq 1}} \nabla_z f\|^2\Big)^{1/2} \Big(\sum_{z \in \Z^2} s(z) \, \|\1_{\cZ_{\geq 1}} g\|^2\Big)^{1/2} \, .
\end{equs}
By using~\eqref{e:Z1Bound} and~\eqref{e:FracLaplNabla}, we deduce
\begin{equ}
|\langle \cA_0 f, g\rangle - \langle A_0 \fsR f, \fsR g \rangle| \lesssim_{\gamma} n \, \|(-\cS)^{(1+\gamma)/2} f\| \,  \|(- \cS)^{\gamma/2} g\|\, .
\end{equ}
Now, by exchanging $f$ and $g$ and using that $(A_0)^* = - A_0$ and $(\cA_0)^* = - \cA_0$, we also obtain
\begin{equ}
|\langle \cA_0 f, g\rangle - \langle A_0 \fsR f, \fsR g \rangle| \lesssim_{\gamma} n \, \|(-\cS)^{\gamma/2} f\| \,  \|(- \cS)^{(1+\gamma)/2} g\|\, .
\end{equ}
Hence, we can apply Lemma~\ref{lem:Interpolation} (by taking $\mathscr S$ therein as $\restriction{- \cS}{\fockd{n}{}}$, which is bounded as $\lesssim n$) to deduce that, for any $\eps > 0$,
\begin{equ}
|\langle \cA_0 f, g\rangle - \langle A_0 \fsR f, \fsR g \rangle| \lesssim_{\gamma, \eps} n^{1+\eps} \, \|(-\cS)^{\gamma - \eps} f\| \,  \|(- \cS)^{1/2} g\|\, .
\end{equ}
Since for any $\gamma' \in [0,1)$ we can obviously find $\gamma, \eps \in [0,1)$ such that $(1 +\gamma')/2 = \gamma - \eps$, we deduce~\eqref{e:HCremovalOclaim} for $(O, \cO) = (A_0, \cA_0)$ and conclude the proof.
\end{proof}

\subsection{Comparison between the soft core generator and the SBE generator}

\label{sec:FromSofttoSBE}

The main purpose of this section is to compare the generators $\cL$ and $\fL$ as follows.

\begin{proposition} \label{p:SoftcoreRemoval}
Let $\gamma \in [0,1]$. For every $f \in \fock{\rap}{}$, we have
\begin{equ} \label{e:SoftcoreRemoval}
\|(- \cS)^{-1/2} \, (\cL \fsP - \fsP \fL) f\| \lesssim \|\fN \, (- \fL_0)^{(1+\gamma)/2} f\| + \|\fN \, (-\fL_0^\fe)^{1/2} f\| \, .
\end{equ}
Further, the bound remains valid if $(\cL, \fL)$ is replaced by any of $(\cS, \fS)$, $(\cA_+, \fA_+)$, $(\cA_-, \fA_-)$, or $(\cA_0, \fA_0)$; the second summand on the right hand side is only needed for $(\cA_+, \fA_+)$.
\end{proposition}

The main technical ingredient to derive these estimates is the following lemma, which essentially follows from the Taylor expansion of the sine and cosine around $0$.

\begin{lemma} \label{lem:SoftcoreRemoval}
For every $k_{1:n} \in [- \pi, \pi]^{2n}$, $n \geq 0$, we have
\begin{equ} \label{e:d0D0equivalence}
d_0(k_{1:n}) \asymp |k_{1:n}|^2 \, .
\end{equ}
For every $k_{1:n} \in \R^{2n}$, $n \geq 0$, and $p, q \in \R^2$ we have
\begin{equs}
|d_0(k_{1:n}) - D_0(k_{1:n})| &\lesssim |k_{1:n}|^4 \, , \label{e:d0Removal}  \\
\Big|\sum_{i=1}^n w(k_i) - \fw \cdot \sum_{i=1}^n k_i\Big| &\lesssim |k_{1:n}|^3 \, , \label{e:wRemoval} \\ 
|w(p) + w(q) - \fw \cdot (p+q)| &\lesssim |\fe \cdot (p+q)| \, (|p|^2 + |q|^2) \, . \label{e:wpqRemoval}
\end{equs}
\end{lemma}

\begin{remark} \label{rmk:wRemoval}
The bound~\eqref{e:wpqRemoval} is the only place in the paper where we directly use Assumption~\ref{ass:Unidirectional}. Without this assumption, we could not say much more than
\begin{equ} \label{e:wRemovalWeak}
|w(p) + w(q) - \fw \cdot (p+q)| \lesssim |p+q| \, (|p|^2 + |q|^2) \, .
\end{equ}
It is easily checked from the proof of Proposition~\ref{p:SoftcoreRemoval} below that, if we use~\eqref{e:wRemovalWeak} instead of~\eqref{e:wpqRemoval}, then the error term $\|\fN \, (- \fL_0^\fe)^{1/2} f\|$ appearing in~\eqref{e:SoftcoreRemoval} when comparing $\cA_+$ and $\fA_+$ would now become $\|\fN \, (- \fL_0)^{1/2} f\|$ (in particular, it would make the choice $\gamma > 0$ useless). Consequently, we would eventually end up with the weaker estimate~\eqref{e:ComparisonGenEasy} instead of~\eqref{e:ComparisonGen} in Theorem~\ref{th:ComparisonGen}.
\end{remark}

\begin{proof}[of Lemma~\ref{lem:SoftcoreRemoval}]
First,~\eqref{e:d0D0equivalence} is a consequence of the irreducibility assumption on the jump rates. Indeed, on the one hand $d_0(k) \sim D_0(k) \asymp |k|^2$ as $k \to 0$ by the Taylor expansion of $\cos$ at $0$, and on the other hand the function $d_0$ on $\R^2$ is continuous, non-negative and only vanishes on $2 \pi \Z^2$ by the irreducibility assumption. Thus $d_0(k) \asymp |k|^2$ on $[- \pi, \pi]^2$, which gives~\eqref{e:d0D0equivalence}.
\ss

Now we prove~\eqref{e:d0Removal} and~\eqref{e:wRemoval}. We use the Taylor estimates $|\cos(\theta) - (1 - \theta^2/2)| \lesssim |\theta|^4$ and $|\sin(\theta) - \theta| \lesssim |\theta|^3$, and deduce for every $k \in \R^2$
\begin{equs}
|d_0(k) - D_0(k)| &\leq 2 \sum_{z \in \Z^2} s(z) \, \Big|1 - \cos(z \cdot k) - \frac{(z \cdot k)^2}{2}\Big| \lesssim |k|^4 \, , \\
|w(k) - \fw \cdot k| &\leq \sum_{z \in \Z^2} |a(z)| \, |\sin(z \cdot k) - z \cdot k| \lesssim |k|^3 \, ,
\end{equs}
which respectively give~\eqref{e:d0Removal} and~\eqref{e:wRemoval} (by noticing that $\sum_{i=1}^n |k_i|^p \leq |k_{1:n}|^p$ for $p \geq 2$).
\ss

Finally, we turn to the proof of~\eqref{e:wpqRemoval}. Using the identity $\sin(x) + \sin(y) = 2 \sin(\frac{x + y}{2}) \cos(\frac{x - y}{2}) $ together with the Taylor estimates $|\sin(\theta) - \theta| \lesssim |\theta|^3$ and $|\cos(\theta) - 1| \lesssim |\theta|^2$ , we obtain for every $x, y \in \R$
\begin{equ}
|\sin(x) + \sin(y) - (x + y)| \lesssim |x+y|^3 + |x + y| \, |x-y|^2 \lesssim |x+y| \, (|x|^2 + |y|^2) \, .
\end{equ}
As a consequence, for every $p,q \in \R^2$ we have
\begin{equ}
|w(p) + w(q) - \fw \cdot (p+q)| \lesssim \sum_{z \in \Z^2} |a(z)| \, |\sin(z \cdot p) + \sin(z \cdot q) - z \cdot (p+q)| \lesssim |\fe \cdot (p+q)| \, (|p|^2 + |q|^2) \, ,
\end{equ}
where in the last bound we used that $\mathrm{Supp}(a) \subset \R \fe$. This shows~\eqref{e:wpqRemoval} and concludes the proof.
\end{proof}

We are now ready to prove Proposition~\ref{p:SoftcoreRemoval}.

\begin{proof}[of Proposition~\ref{p:SoftcoreRemoval}]

We show that, for any $(\cO, \fO) \in \{(\cS, \fS), (\cA_+, \fA_+), (\cA_-, \fA_-), (\cA_0, \fA_0)\}$ and for every $f \in \fock{\rap, n}{}$, $n \geq 0$, $g \in \fock{\rap}{}$, it holds
\begin{equ} \label{e:SoftcoreRemovalO}
\big|\langle \cO \, \fsP f, \fsP g \rangle - \langle \fO f, g\rangle\big| \lesssim_{\gamma} n \, \big(1_{\fO \in \{\fS, \, \fA_-, \, \fA_0\}} \|(- \fL_0)^{(1+\gamma)/2} f\| + 1_{\fO = \fA_+} \|(- \fL_0^\fe)^{1/2} f\|\big) \, \|(- \fL_0)^{1/2} g\| \, , 
\end{equ}
Then, by noticing that $\fsP \fsP^* = \mathrm{Id}$ and using~\eqref{e:BoundP}, we deduce
\begin{equs}
\|(- \cS)^{-1/2} \, (\cO \fsP - \fsP \fO) f\| &= \|(- \cS)^{-1/2} \, \fsP \, (\fsP^* \cO \fsP - \fO) f\| \\
&\lesssim \|(- \fL_0)^{-1/2} \, (\fsP^* \cO \fsP - \fO) f\| \\
&\lesssim n \, \big(1_{\fO \in \{\fS, \, \fA_-, \, \fA_0\}} \|(- \fL_0)^{(1+\gamma)/2} f\| + 1_{\fO = \fA_+} \|(- \fL_0^\fe)^{1/2} f\|\big) \, ,
\end{equs}
where in the last bound we used~\eqref{e:SoftcoreRemovalO}. Finally, by using orthogonality of chaoses, we can extend this bound to any $f \in \fock{\rap}{}$ up to replacing the factor $n$ by $\fN$, which results in the desired estimate~\eqref{e:SoftcoreRemoval} and would conclude the proof.
\ss

Now, we prove~\eqref{e:SoftcoreRemovalO}. Obviously, we can further assume that $g$ belongs respectively to $\fock{n}{}$, $\fock{n+1}{}$, $\fock{n-1}{}$, $\fock{n}{}$ when we prove the statement for $\cO = \cS, \, \cA_+, \, \cA_-, \, \cA_0$.
\ss

We start with the case $(\cO, \fO) = (\cS, \fS)$. We have
\begin{equ}
\langle \cS \fsP f, \fsP g \rangle - \langle \fS f, g \rangle = n! \ \frac{1}{2} \int_{\R^{2n}} \dd k_{1:n} \, \fF f(k_{1:n}) \, \overline{\fF g(k_{1:n})} \, (\1_{|k_{1:n}|_\infty < \pi} \, d_0(k_{1:n}) - D_0(k_{1:n})) \, .
\end{equ}
Now, we use~\eqref{e:d0Removal} and observe that
\begin{equ}
|\1_{|k_{1:n}|_\infty < \pi} \, d_0(k_{1:n}) - D_0(k_{1:n})| \lesssim \1_{|k_{1:n}|_\infty < \pi} \, |k_{1:n}|^4 + \1_{|k_{1:n}|_\infty \geq \pi} \, |k_{1:n}|^2 \lesssim n \, (|k_{1:n}|^4 \wedge |k_{1:n}|^2) \, ,
\end{equ}
where for the second bound we used that $|k_{1:n}|^2 \lesssim n$ if $|k_{1:n}|_\infty < \pi$ and that $1 \lesssim |k_{1:n}|^2$ if $|k_{1:n}|_\infty \geq \pi$. By the Cauchy-Schwarz inequality, we then deduce
\begin{equs}
|\langle &\cS \fsP f, \fsP g \rangle - \langle \fS f, g \rangle| \\
&\lesssim n! \, n \, \Big(\int_{\R^{2n}} \dd k_{1:n} \, (|k_{1:n}|^4 \wedge |k_{1:n}|^2) \, |\fF f(k_{1:n})|^2\Big)^{1/2} \Big(\int_{\R^{2n}} \dd k_{1:n} \, (|k_{1:n}|^4 \wedge |k_{1:n}|^2) \, |\fF g(k_{1:n})|^2\Big)^{1/2} \\
&\lesssim n \|((- \fL_0) \wedge (- \fL_0)^{1/2}) \, f\| \, \|((- \fL_0) \wedge (- \fL_0)^{1/2}) \, g\| \label{e:SCremovalSNeat} \\
&\leq n \, \|(- \fL_0)^{(1+\gamma)/2} f\| \, \|(- \fL_0)^{1/2} \, g\| \, , \label{e:SCremovalSFinal}
\end{equs}
which shows~\eqref{e:SoftcoreRemovalO} for $(\cO, \fO) = (\cS, \fS)$.
\ss

We turn to the case $(\cO, \fO) = (\cA_+, \fA_+)$. We have
\begin{equ} \label{e:SoftcoreRemovalA+Proof}
\langle \cA_+ \fsP f, \fsP g \rangle - \langle \fA_+ f, g \rangle = (n+1)! \, n \, \frac{\iota \, \sqrt{\rho(1-\rho)}}{2 \pi}  \int_{\R^{2(n+1)}} \dd k_{1:n+1} \, \overline{\fF g(k_{1:n+1})} \, K_+^f(k_{1:n+1}) \, , \end{equ}
where $K_+^f(k_{1:n+1})$ is a shorthand notation for
\begin{equ}
\1_{|k_{1:n+1}|_\infty < \pi} \, (w(k_1) + w(k_2)) \, \fF f (\llbracket k_1 + k_2 \rrbracket, k_{3:n+1}) - \Pi_{k_1, k_2} \, \fw \cdot (k_1 + k_2) \, \fF f (k_1 + k_2, k_{3:n+1}) \, . \label{e:SoftCoreRemovalA+Kernel}
\end{equ}
Note that $K_+^f(k_{1:n+1})$ vanishes if $|k_{1:2}|_\infty \geq \pi$, so we can assume that $|k_{1:2}|_\infty < \pi$. Now, we estimate $K_+^f(k_{1:n+1})$ by distinguishing the following three regions
\begin{equ}
R_1 := \{|k_{1:n+1}|_\infty < \pi, \, \Pi_{k_1, k_2} = 1\}, \ R_2 := \{|k_{1:n+1}|_\infty < \pi, \,\Pi_{k_1, k_2} = 0\}, \ R_3 := \{|k_{1:n+1}|_\infty \geq \pi\} \, . \label{e:SoftCoreRemovalA+Regions}
\end{equ}

We start with the region $R_1$. We first observe that $\llbracket k_1 + k_2 \rrbracket = k_1 + k_2$ because $k_1 + k_2 \in (- \pi, \pi)^2$ since $\Pi_{k_1, k_2} = 1$ on $R_1$. Then we use~\eqref{e:wpqRemoval} and obtain
\begin{equs}
|K_+^f(k_{1:n+1})| &\lesssim |\fe \cdot (k_1 + k_2)| \, (|k_1|^2 + |k_2|^2) \, |\fF f (k_1 + k_2, k_{3:n+1})| \\
&\lesssim |\fe \cdot (k_1 + k_2)| \, (|k_{1:n+1}|^2 \wedge 1) \, |\fF f (k_1 + k_2, k_{3:n+1})| \, , \label{e:SoftcoreRemovalA+Technical1}
\end{equs}
where in the second bound we used that $|k_{1:2}|_\infty < \pi$.

We continue with the region $R_2$ (note that only the first summand of~\eqref{e:SoftCoreRemovalA+Kernel} survives here). We first observe that
\begin{equ}
|w(k_1) + w(k_2)| = |w(k_1) + w(\llbracket k_1 + k_2 \rrbracket - k_1)| \lesssim |\fe \cdot \llbracket k_1 + k_2 \rrbracket| \, (1 + |k_1|^2 + |\llbracket k_1 + k_2 \rrbracket - k_1|^2) \lesssim |\fe \cdot \llbracket k_1 + k_2 \rrbracket| \, ,
\end{equ}
where we first used the $2 \pi \Z^2$ periodicity of $w$, then~\eqref{e:wpqRemoval} and finally $|k_{1:2}|_\infty < \pi$. Therefore, we deduce
\begin{equs}
|K_+^f(k_{1:n+1})| &\lesssim |\fe \cdot \llbracket k_1 + k_2 \rrbracket| \, |\fF f (\llbracket k_1 + k_2 \rrbracket, k_{3:n+1})| \\
&\lesssim |\fe \cdot \llbracket k_1 + k_2 \rrbracket| \, (|k_{1:n+1}|^2 \wedge 1) \, |\fF f (\llbracket k_1 + k_2 \rrbracket, k_{3:n+1})| \, , \label{e:SoftcoreRemovalA+Technical3}
\end{equs}
where we used in the second bound that $|k_{1:2}|_\infty \geq \pi/2$ since $\Pi_{k_1, k_2} = 0$ on $R_2$.

We finish with the region $R_3$ (note that only the second summand of~\eqref{e:SoftCoreRemovalA+Kernel} survives here), as follows
\begin{equs}
|K_+^f(k_{1:n+1})| &\lesssim |\fe \cdot (k_1 + k_2)| \, |\fF f (k_1 + k_2, k_{3:n+1})| \\
&\lesssim |\fe \cdot (k_1 + k_2)| \, (|k_{1:n+1}|^2 \wedge 1) \, |\fF f (k_1 + k_2, k_{3:n+1})| \, , \label{e:SoftcoreRemovalA+Technical2}
\end{equs}
where in the second bound we used that $|k_{1:n+1}|_\infty \geq \pi$ on $R_3$.

Coming back to~\eqref{e:SoftcoreRemovalA+Proof}, recalling that we can restrict the integral to $k_{1:2} \in [- \pi, \pi]^4$, gathering~\eqref{e:SoftcoreRemovalA+Technical1},~\eqref{e:SoftcoreRemovalA+Technical2} and~\eqref{e:SoftcoreRemovalA+Technical3}, and using the Cauchy-Schwarz inequality, we deduce
\begin{equs}
|\langle &\cA_+ \fsP f, \fsP g \rangle - \langle \fA_+ f, g \rangle| \lesssim n! \, n^2 \, \Big(\int_{\R^{2(n+1)}} \dd k_{1:n+1} \, (|k_{1:n+1}|^4 \wedge 1) \, |\fF g(k_{1:n+1})|^2\Big)^{1/2} \\
&\qquad \qquad \qquad \times \Big(\int_{[- \pi, \pi]^{4} \times \R^{2(n-1)}} \dd k_{1:n+1} \, |\fe \cdot (k_1 + k_2)|^2 \, |\fF f (k_1 + k_2, k_{3:n+1})|^2 \\
&\qquad \qquad \qquad \qquad \qquad+ \int_{[- \pi, \pi]^{4} \times \R^{2(n-1)}} \dd k_{1:n+1} \, |\fe \cdot \llbracket k_1 + k_2 \rrbracket|^2 \, |\fF f (\llbracket k_1 + k_2 \rrbracket, k_{3:n+1})|^2\Big)^{1/2} \\
&\lesssim n^2 \, n! \, \Big(\int_{\R^{2(n+1)}} \dd k_{1:n+1} \, (|k_{1:n+1}|^4 \wedge 1) \, |\fF g(k_{1:n+1})|^2\Big)^{1/2} \Big(\int_{\R^{2n}} \dd \ell_{1:n} \, |\fe \cdot \ell_1|^2 \, |\fF f (\ell_{1:n})|^2\Big)^{1/2} \\
&\lesssim n \, \|((- \fL_0) \wedge 1) \, g\| \, \|(- \fL_0^\fe)^{1/2} f\| \label{e:SoftcoreRemovalA+Neat} \\
&\leq n \, \|(- \fL_0)^{1/2} g\| \, \|(- \fL_0^\fe)^{1/2} f\| \, , \label{e:SoftcoreRemovalA+Final} 
\end{equs}
where for the second bound we made the changes of variables $\ell_{1:n} = (k_1 + k_2, k_{3:n+1})$ and $\ell_{1:n} = (\llbracket k_1 + k_2 \rrbracket, k_{3:n+1})$ in the two last integrals, respectively.
This concludes the proof of~\eqref{e:SoftcoreRemovalO} for $(\cO, \fO) = (\cA_+, \fA_+)$.
\ss

We turn to the case $(\cO, \fO) = (\cA_-, \fA_-)$. We start from~\eqref{e:SoftcoreRemovalA+Neat}, but applied with $(g,f)$ instead of $(f,g)$ (and hence $n-1$ instead of $n$), and then we use $(\cA_+)^* = - \cA_-$ and $(\fA_+)^* = - \fA_-$, which leads to
\begin{equ}
|\langle \cA_- \fsP f, \fsP g \rangle - \langle \fA_- f, g \rangle| \lesssim n \, \|((- \fL_0) \wedge 1) f\| \, \|(- \fL_0^\fe)^{1/2} g\| \leq n \, \|(- \fL_0)^{(1+\gamma)/2} \, f\| \, \|(- \fL_0)^{1/2} g\| \, .
\end{equ}
This shows~\eqref{e:SoftcoreRemovalO} for $(\cO, \fO) = (\cA_-, \fA_-)$.
\ss

Finally, we tackle the case $(\cO, \fO) = (\cA_0, \fA_0)$. We have
\begin{equ}
\langle \cA_0 \fsP f, \fsP g \rangle - \langle \fA_0 f, g \rangle = n! \, (1 - 2 \rho) \int_{\R^{2n}} \dd k_{1:n} \, \fF f(k_{1:n}) \, \overline{\fF g(k_{1:n})} \, \Big(\1_{|k_{1:n}|_\infty < \pi} \sum_{i=1}^n w(k_i) - \sum_{i=1}^n \fw \cdot k_i\Big) \, .
\end{equ}
Now, we use~\eqref{e:wRemoval} and observe that
\begin{equ}
\Big|\1_{|k_{1:n}|_\infty < \pi} \sum_{i=1}^n w(k_i) - \sum_{i=1}^n \fw \cdot k_i\Big| \lesssim \1_{|k_{1:n}|_\infty < \pi} \, |k_{1:n}|^3 + \1_{|k_{1:n}|_\infty \geq \pi} \, n^{1/2} \, |k_{1:n}| \lesssim n \, (|k_{1:n}|^3 \wedge |k_{1:n}|) \, ,
\end{equ}
where in the first bound we used that $\sum_{i=1}^n |k_i| \leq n^{1/2} \, |k_{1:n}|$, and in the second bound we used that $|k_{1:n}|^2 \lesssim n$ if $|k_{1:n}|_\infty < \pi$ and that $1 \lesssim |k_{1:n}|^2$ if $|k_{1:n}|_\infty \geq \pi$. By the Cauchy-Schwarz inequality, we then deduce
\begin{equs}
|\langle &\cA_0 \fsP f, \fsP g \rangle - \langle \fA_0 f, g \rangle | \\
&\lesssim n! \, n\, \Big(\int_{\R^{2n}} \dd k_{1:n} \, (|k_{1:n}|^4 \wedge 1) \, |\fF f(k_{1:n})|^2\Big)^{1/2} \Big(\int_{\R^{2n}} \dd k_{1:n} \, |k_{1:n}|^2 \, |\fF g(k_{1:n})|^2\Big)^{1/2} \\
&\lesssim n \, \|((- \fL_0) \wedge 1) \, f\| \, \|(- \fL_0)^{1/2} \, g\| \\
&\leq n \, \|(- \fL_0)^{(1+\gamma)/2} f\| \, \|(- \fL_0)^{1/2} g\| ,
\end{equs}
which gives~\eqref{e:SoftcoreRemovalO} for $(\cO, \fO) = (\cA_0, \fA_0)$, and concludes the proof.
\end{proof}

\subsection{Proof of Theorem~\ref{th:ComparisonGen}}

\label{sec:FromASEPtoSBEccl}

We can finally prove Theorem~\ref{th:ComparisonGen} by gathering the results of the two previous sections.

\begin{proof}[of Theorem~\ref{th:ComparisonGen}]

First, we write
\begin{equ}
\|(- S)^{-1/2} \, (L \fsR \fsP - \fsQ^* \fsP \fL) \phi\| \leq \|(- S)^{-1/2} \, (L \fsR - \fsQ^* \cL) \fsP \phi\| + \|(- S)^{-1/2} \, \fsQ^* (\cL \fsP - \fsP \fL) \phi\| \, .
\end{equ}
To bound the first summand we use Proposition~\ref{p:HardcoreRemoval} and~\eqref{e:BoundP} (the latter trivially holds with an extra factor $\cN^{6}$ and $\fN^{6}$ on both sides, by orthogonality of chaoses), which gives
\begin{equ}
\|(- S)^{-1/2} \, [L \fsR - \fsQ^* \cL] \fsP \phi\| \lesssim_\gamma \|\cN^{6} \, (- \cS)^{(1+\gamma)/2} \, \fsP \phi\| \lesssim \|\fN^{6} \, (- \fL_0)^{(1+\gamma)/2} \phi\| \, .
\end{equ}
To bound the second summand we use~\eqref{e:BoundQ*} and Proposition~\ref{p:SoftcoreRemoval} (the latter trivially holds with an extra factor $\cN^3$ and $\fN^3$ on both sides, by orthogonality of chaoses), which gives
\begin{equs}
\|(- S)^{-1/2} \, \fsQ^* (\cL \fsP - \fsP \fL) \phi\| &\lesssim \|\cN^3 \, (- \cS)^{-1/2} \, (\cL \fsP - \fsP \fL) \phi\| \\
&\lesssim \|\fN^{4} (- \fL_0)^{(1+\gamma)/2} \phi\| + \|\fN^{4} (- \fL_0^\fe)^{1/2} \phi\| \, ,
\end{equs}
By gathering the previous bounds, we deduce the desired estimate~\eqref{e:ComparisonGen}. Obviously, we can proceed similarly with any of $(S, \cS, \fS)$, $(A_+, \cA_+, \fA_+)$, $(A_-, \cA_-, \fA_-)$ or $(A_0, \cA_0, \fA_0)$ instead of $(L, \cL, \fL)$, since Propositions~\ref{p:HardcoreRemoval} and~\ref{p:SoftcoreRemoval} apply to each of them (besides, the last summand $\|\fN^4 (- \fL_0^\fe)^{1/2} \phi\|$ is only required to deal with $(A_+, \cA_+, \fA_+)$). This concludes the proof.
\end{proof}

From the results of the two previous sections, we also deduce the following estimate that compares the two $H^1$ norms related to $S$ and $\fS$, which will be useful later to estimate the diffusivity.

\begin{lemma} \label{lem:ComparisonH1}
For every $\phi \in \fock{\rap}{}$ and $\gamma \in [0,1)$, we have
\begin{equ} \label{e:ComparisonH1}
\Big|\|(- S)^{1/2} \, \fsR \fsP \phi\|^2 - \|(- \fS)^{1/2} \phi\|^2\Big| \lesssim_\gamma \|\fN \, (- \fL_0)^{(1 + \gamma)/2} \phi\|^2 \, .
\end{equ}
\end{lemma}

\begin{proof}
By orthogonality of chaoses, it is sufficient to prove~\eqref{e:ComparisonH1} for every $\phi \in \fock{\rap,n}{}$ and $n \geq 1$. By using~\eqref{e:HCremovalS} (together with~\eqref{e:BoundP}) and~\eqref{e:SCremovalSNeat}, we obtain
\begin{equs}
\big|\langle (- S) \fsR \fsP \phi, \fsR \fsP \phi \rangle - \langle (- \cS) \fsP \phi, \fsP \phi \rangle\big| &\lesssim_\gamma n \, \|(- \cS)^{(1+\gamma)/2} \, \fsP \phi\|^2 \lesssim_\gamma n \, \|(- \fL_0)^{(1+\gamma)/2} \phi\|^2 \, , \\
\big|\langle (- \cS) \fsP \phi, \fsP \phi \rangle - \langle (-\fS) \phi, \phi \rangle\big| &\lesssim n \, \|((- \fL_0) \wedge (- \fL_0)^{1/2}) \, \phi\|^2 \leq n \, \|(- \fL_0)^{(1+\gamma)/2} \phi\|^2 \, ,
\end{equs}
which gives, after summing the two bounds, the desired estimate~\eqref{e:ComparisonH1} and concludes the proof.
\end{proof}

\section{From the SBE Ansatz to the fluctuation-dissipation relation}

\label{sec:AnsatzSBE}

In this section we establish a fluctuation-dissipation relation for the ASEP, of the form
\begin{equ} \label{e:FormalFluctuationDissipationASEP}
    \tau A_+ \fsi_t^\tau \vphi \approx d^{\lim} \frac{\log_{*}(1/\mu)^{2/3}}{\log_{*}(\tau)^{2/3}} \, \fsi_t^\tau \fL_0^\fe \vphi - (\partial_t + \tau L) \, \mathrm{Cor}^{\tau, \mu}_{t}(\vphi) \quad \mbox{ in $H^{-1}$ norm,}
\end{equ}
for any $\vphi \in \fock{\rap}{}$, where $\fsi_t^\tau$ is the rescaling operator given by~\eqref{e:fsittau} and $\mathrm{Cor}^{\tau, \mu}_{t}(\vphi)$ is a corrector that we construct below. We refer to Proposition~\ref{p:AnsatzA+ASEP} for a more precise statement. We start by recalling the definition of the SBE Ansatz $\mathfrak{Ans}_\mu(\bm \vphi)$ from~\cite{CMT25} and giving related estimates in Section~\ref{sec:AnsatzDef}. We then use the Ansatz in Section~\ref{sec:AnsatzA+} to construct a fluctuation-dissipation relation for the SBE. Finally, by using the comparison of generators from Theorem~\ref{th:ComparisonGen}, we deduce the fluctuation-dissipation relation~\eqref{e:FormalFluctuationDissipationASEP} in Section~\ref{sec:AnsatzA+ASEP}. Incidentally, we construct a specific corrector that captures the asymptotic of the diffusivity in Section~\ref{sec:AnsatzB+}.

\subsection{The generic SBE Ansatz: definition and properties}

\label{sec:AnsatzDef}

Let $\mu > 0$ be a parameter which, informally, represents the inverse of the time scale of interest (in particular, we have in mind $\mu \ll 1$). We also introduce an auxiliary parameter $\nu_0 \geq 0$, which will be chosen sufficiently large and fixed in Appendix~\ref{app:AnsatzSBE} (see Definition~\ref{def:nu}).
\ss

We first define modulations of the diagonal operator $\fL_0^\fe$ (recall Definition~\ref{def:DiagonalOperator}). Given a scalar function $f \colon \R_+ \to \R_+$, let $\fL_0^\fe[f]_\mu$ be the diagonal operator given by
\begin{equ} \label{e:Modulation}
\fL_0^\fe[f]_\mu \eqdef f(\Log(\mu - \fL_0)) \, \fL_0^\fe \, ,
\end{equ}
where $\Log(x) \eqdef \log(1/x) \, \1_{0 < x \leq 1}$. The modulation we are primarily interested in is the solution $\Upsilon \colon \R_+ \to \R_+$ of the ODE
\begin{equ} \label{e:ODE}
\dot{\Upsilon} = \frac{C_\Upsilon}{\sqrt{\det(D_0) + D_0[\ff, \ff] \, \Upsilon}} \, , \quad \Upsilon(0) = \nu_0 \, , \ \ \mbox{ where } C_\Upsilon \eqdef \frac{\lambda^2 \, \rho (1 - \rho)}{\pi} \, ,
\end{equ}
which is explicitly given by
\begin{equ} \label{e:ODEexplicit}
\Upsilon(\ell) \eqdef \Big(\frac{3}{2} \frac{C_\Upsilon}{D_0[\ff, \ff]^{1/2}} \, \ell + \Big(\frac{\det(D_0)}{D_0[\ff, \ff]} + \nu_0\Big)^{3/2}\Big)^{2/3} - \frac{\det(D_0)}{D_0[\ff, \ff]} \, .
\end{equ}
Heuristically, the operator $\Log(\mu - \fL_0)$ appearing in~\eqref{e:Modulation} parametrises spatial scales logarithmically, while $\Upsilon(\ell)$ encodes the effective diffusivity at logarithmic scale $\ell$. In this perspective, the ODE~\eqref{e:ODE} can be thought of as a flow equation that describes the growth of the diffusivity with logarithmic scale. Observe that $\Upsilon(\ell) \sim d^{\lim} \, \ell^{2/3}$ as $\ell \to \infty$, which echoes the diffusivity asymptotics of Theorem~\ref{th:Diff}. For later use, we mention that we have the finer estimate
\begin{equ} \label{e:AsymptoticUps}
|\Upsilon(\ell) - d^{\lim} \ell^{2/3}| \lesssim 1 \, .
\end{equ}

The operator $\fL_0^\fe[\Upsilon]_\mu$ plays the role of $\fh_\mu$ in the schematic definition of the Ansatz~\eqref{e:AnsatzRoadmap}. In particular, the modulation $\Upsilon$ is chosen to yield the approximate fixed-point equation~\eqref{e:FixedPointRoadmap} (cf.~\cite[Lemma 3.6]{CMT25}). We now define the diagonal operator
\begin{equ} \label{e:Leff}
\fL_{\mu}^{\eff} \eqdef \fS + \fL_0^\fe[\Upsilon]_\mu \, ,
\end{equ}
which intuitively captures the effective generator of the system across scales, up to the time scale $1/\mu$. Motivated by~\eqref{e:AnsatzRoadmap}, it would be natural to consider an Ansatz of the form
\begin{equ} \label{e:NaiveAnsatz}
\sum_{n \geq 0} \big((\mu - \fL_{\mu}^{\eff})^{-1} \, \fA_+\big)^{n} \, (\mu - \fL_{\mu}^{\eff})^{-1} \, \bm \vphi \, .
\end{equ}
However, this Ansatz may be too singular, because the operator $\fA_+$ produces so-called off-diagonal terms that become difficult to control when iterating successively $(\mu - \fL_{\mu}^{\eff})^{-1} \, \fA_+$. Following~\cite[Section 3.2]{CMT25}, we therefore cut the problematic component of $\fA_+$ in the definition of the Ansatz. More precisely, we decompose $\fA_+$ as $\fA_+^\sharp + \fA_+^\flat$ according to the following rule: a transformation of Fourier modes of the form $k_{1:n} \to (p, k_i - p, k_{1:n \setminus i})$ is assigned to $\fA_+^\sharp$ if the created modes $p$ and $k_i - p$ have norms at least three times that of the original modes $k_{1:n}$; otherwise, it is assigned to $\fA_+^\flat$.

\begin{definition} \label{def:Sharp}
We define the operators $\bar \fA^\sharp := \fA_+^\sharp + \fA_-^\sharp$ and $\bar \fA^\flat := \fA_+^\flat + \fA_-^\flat$ acting (continuously) from $\fock{\rap}{}$ to itself as follows. For $f \in \fock{\rap}{}$, we have
\begin{equs}
\fF (\fA_+^{\sharp / \flat} f) \, (k_{1:n+1}) &\eqdef \frac{\sqrt{\rho(1-\rho)}}{\pi (n+1)} \sum_{1 \leq i < j \leq n+1} \chi^{\sharp / \flat}(k_i, k_j; k_i + k_j, k_{1:n+1 \setminus i,j}) \label{e:fA+Sharp} \\
&\qquad \qquad \qquad \qquad \qquad \qquad \times \iota (\fw \cdot (k_i+k_j)) \, \Pi_{k_i, k_j} \, \fF f (k_i + k_j, k_{1:n+1 \setminus i,j}) \\
\fF (\fA_-^{\sharp / \flat} f) \, (k_{1:n-1}) &\eqdef \frac{\sqrt{\rho(1-\rho)} \, n}{2 \pi} \sum_{i=1}^{n-1} \iota (\fw \cdot k_i) \, \int_{\R^2} \dd \ell \ \chi^{\sharp / \flat}(\ell, k_i - \ell; k_{1:n-1}) \label{e:fA-Sharp} \\
&\qquad \qquad \qquad \qquad \qquad \qquad \qquad \qquad \times \Pi_{\ell, k_i-\ell} \, \fF f (\ell, k_i-\ell, k_{1:n-1 \setminus i}) \, ,
\end{equs}
where $\Pi_{\cdot, \cdot}$ was defined in~\eqref{e:CutOff}, and for $a, b \in \R^2$ and $c_{1:n} \in \R^{2n}$ we set
\begin{equ}
\chi^\sharp(a, b ; c_{1:n}) \eqdef \1_{|a| \wedge |b| \geq 3 \, |c_{1:n}|} \, , \quad \chi^\flat \eqdef 1 - \chi^\sharp \, .
\end{equ}
By construction, $\fA_{\pm} = \fA_{\pm}^\sharp + \fA_{\pm}^\flat$ (compare~\eqref{e:fA+Sharp} and~\eqref{e:fA-Sharp} with~\eqref{e:fA+} and~\eqref{e:fA-}). Moreover, $\bar \fA^\sharp$ and $\bar \fA^\flat$ are skew-symmetric on $\fock{\rap}{} \subset \fock{}{}$; more precisely, $(\fA_+^\sharp)^* = - \fA_-^\sharp$ and $(\fA_+^\flat)^* = - \fA_-^\flat$ (see~\cite[Lemma 3.2]{CMT25}).
\end{definition}

We now define the Ansatz $\mathfrak{Ans}_{\mu} \colon \fock{\rap}{} \to \fock{\rap}{}$ by
\begin{equ} \label{e:TrueAnsatz}
\mathfrak{Ans}_{\mu}(\bm \vphi) \eqdef \sum_{n \geq 0} \big((\mu - \fL_{\mu}^{\eff})^{-1} \, \fA_+^\sharp\big)^n \, (\mu - \fL_{\mu}^{\eff})^{-1} \bm \vphi \, .
\end{equ}
For example, if $\bm \vphi$ belongs to a single chaos $n_0$, then the summand of index $n \geq 0$ in~\eqref{e:TrueAnsatz} belongs to chaos $n_0 + n$: hence, the series decomposes the Ansatz chaos by chaos. The following lemma, analogous to~\cite[Lemma 4.2]{CMT25}, shows that~\eqref{e:TrueAnsatz} is well defined. Its proof is given in Step 1 of Appendix~\ref{app:GenericAnsatz}.

\begin{lemma} \label{lem:SummabilityAnsatz}
For every $\mu > 0$, the series in~\eqref{e:TrueAnsatz} converges normally in $\fock{\rap}{}$ in the following sense: for all $a, b \geq 0$ and $\bm \vphi \in \fock{\rap}{}$, we have
\begin{equ} \label{e:SummabilityAnsatz}
\sum_{n = 0}^\infty \|\fN_*^a \, (1 - \fL_0)^b \ \big((\mu - \fL_{\mu}^{\eff})^{-1} \, \fA_+^\sharp\big)^n \, (\mu - \fL_{\mu}^{\eff})^{-1} \bm \vphi\| \lesssim_{\mu, a, b} \|\fN_*^a \, (1 - \fL_0)^b \, \bm \vphi\| \, .
\end{equ}
Consequently,~\eqref{e:TrueAnsatz} defines a continuous operator $\mathfrak{Ans}_{\mu} \colon \fock{\rap}{} \to \fock{\rap}{}$, and the sum may be interchanged with any continuous operator $\fO \colon \fock{\rap}{} \to \fock{\rap}{}$.
\end{lemma}

\begin{remark} \label{rmk:AnsatzIdentity}
The recursive structure of the Ansatz can be expressed through the identity
\begin{equ} \label{e:AnsatzIdentity}
\mathfrak{Ans}_{\mu}(\bm \vphi) = (\mu - \fL_\mu^{\eff})^{-1} \, \fA_+^\sharp \, \mathfrak{Ans}_{\mu}(\bm \vphi) + (\mu - \fL_\mu^{\eff})^{-1} \bm \vphi \, .
\end{equ}
Indeed,~\eqref{e:AnsatzIdentity} follows directly from~\eqref{e:TrueAnsatz}, since Lemma~\ref{lem:SummabilityAnsatz} allows us to interchange the sum with the continuous operator $(\mu - \fL_\mu^{\eff})^{-1} \, \fA_+^\sharp$.
\end{remark}

\begin{remark} \label{rmk:ComparisonCMTansatz}
The Ansatz $\mathfrak{Ans}_\mu(\bm\vphi)$ differs from the construction in~\cite{CMT25} in a few technical respects, which we briefly describe.
First, we work with the ``microscopic'' SBE generator $\fL$, whereas that paper uses the superdiffusively rescaled generator $\fL^\tau = \tau \, (\fU_0^\tau)^{-1} \, \bar \fL \, \fU_0^\tau$. Thus, transferring its results to our setting requires an appropriate conjugation by $\fU^\tau$ and an adjustment of the powers of $\tau$ in the corresponding estimates.
Second, the cut-off $\Pi_{p,q}$ is defined there as $(2 \pi \fF \rho(p))^2 (2 \pi \fF \rho(q))^2 (2 \pi \fF \rho(p+q))^2$ for a smooth compactly supported test function $\rho$, rather than by the Fourier truncation~\eqref{e:CutOff}. Relatedly, the norm on $\fock{n}{}$ used there contains the additional Fourier weight $\prod_{i=1}^n (2 \pi \fF \rho(k_i))^2$.
Third, that paper treats only the case $D_0 = \mathrm{Id}$. Although one can always reduce to this case by appropriately rescaling space, this would be inconvenient here: for instance, it would replace $\Log(\mu - \fL_0)$ in~\eqref{e:Modulation} by $\Log(\mu - \fS)$.
Finally, the auxiliary parameter $\nu_0$ (denoted by $M$ in that paper) enters the construction differently. Instead of $\Upsilon$, the earlier construction uses $\nu_0 \vee \Upsilon_0$, where $\Upsilon_0$ solves~\eqref{e:ODE} with initial condition $0$. Our choice simplifies some subsequent calculations.
\end{remark}

We end this section with the main estimates for the Ansatz. They are related to~\cite[Propositions 4.3, 4.6 and 4.8]{CMT25}; see Remark~\ref{rmk:ComparisonCMTestimates} for a more detailed comparison. We first introduce, for $\theta \in [0,1]$, the function $\tilde \Upsilon^{\theta} \colon \R_+ \to \R_+$ defined by
\begin{equ} \label{e:Upstheta}
\tilde \Upsilon^\theta \eqdef (\det(D_0) + D_0[\ff, \ff] \, \Upsilon)^\theta \, .
\end{equ}

\begin{proposition} \label{p:GenericAnsatz}
Let $\mu > 0$, $\theta \in (1/2,1]$, $\alpha \in [1, \sqrt{2 \theta})$, $\kappa \geq 0$ and $\gamma \in (0,1]$. For every $\bm \vphi \in \fock{\rap}{}$ we have
\begin{equs}
\, \sqrt{\mu} \, \|\fN_{*}^\kappa \alpha^\fN \, \mathfrak{Ans}_{\mu}(\bm \vphi)\| &\lesssim_{\alpha, \kappa} \sqrt{\mu} \, K_{\alpha, \kappa}^{0}(\bm \vphi) \, , \label{e:AnsatzL2} \\
\|\fN_{*}^\kappa \alpha^\fN \, (- \fL_0)^{1/2} \, \mathfrak{Ans}_{\mu}(\bm \vphi)\| &\lesssim_{\alpha, \kappa} K_{\alpha, \kappa}^{1/2}(\bm \vphi) + K_{\alpha, \kappa}^{1/2, \fe, 1}(\bm \vphi) \, , \label{e:AnsatzH1} \\
\|\fN_{*}^\kappa \alpha^\fN \, (- \fL_0^\fe)^{1/2} \,  \mathfrak{Ans}_{\mu}(\bm \vphi)\| &\lesssim_{\alpha, \kappa, \theta} K_{\alpha, \kappa}^{1/2, \fe, \theta}(\bm \vphi)  \, , \label{e:AnsatzH1fe} \\
\|\fN_{*}^\kappa \alpha^\fN \, (- \fL_0)^{(1+\gamma)/2} \, \mathfrak{Ans}_{\mu}(\bm \vphi)\| &\lesssim_{\alpha, \kappa, \theta, \gamma} K_{\alpha, \kappa}^{(1+\gamma)/2}(\bm \vphi) + K_{\alpha, \kappa}^{1/2, \fe, \theta}(\bm \vphi) \, , \label{e:AnsatzH1+gamma} \\
\|\fN_{*}^\kappa \alpha^\fN \, (- \fL_0)^{-1/2} \, \big((\mu - \bar \fL) \, \mathfrak{Ans}_{\mu}(\bm \vphi) - \bm \vphi &+ \fA_-^\sharp (\mu - \fL_{\mu}^{\eff})^{-1} \bm \vphi\big)\| \lesssim_{\alpha, \kappa, \theta} K_{\alpha, \kappa + 1}^{1/2, \fe, \theta}(\bm \vphi) \, , \quad \label{e:AnsatzGen} \\
\Big|\|(\mu - \fS)^{1/2} \, \mathfrak{Ans}_{\mu}(\bm \vphi)\|^2 - \|(\mu &- \fL_{\mu}^{\eff})^{-1/2} \, \bm \vphi\|^2\Big| \lesssim_{\theta}
K_{1, 0}^{1/2, \fe, \theta}(\bm \vphi)^2 \, , \label{e:AnsatzH1precise}
\end{equs}
where we additionally assume for~\eqref{e:AnsatzH1precise} that $\bm \vphi$ belongs to a homogeneous chaos (that is, $\bm \vphi \in \fock{n}{}$ for some $n \geq 0$), and where we have written $\bar \fL \eqdef \fL - \fA_0$ and set
\begin{equs}[e:AnsatzConstantK]
K_{\alpha, \kappa}^s(\bm \vphi) &:= \|\fN_{*}^\kappa \alpha^\fN \, (- \fL_0)^{s} \, (\mu - \fL_{\mu}^{\eff})^{-1} \bm \vphi\| \, , \\
K_{\alpha, \kappa}^{1/2, \fe, \theta}(\bm \vphi) &:= \|\fN_{*}^\kappa \alpha^\fN \, (- \fL_0^\fe\big[\tilde \Upsilon^\theta\big]_\mu)^{1/2} \, (\mu - \fL_{\mu}^{\eff})^{-1} \bm \vphi\| \, .
\end{equs}
\end{proposition}

The first four estimates~\eqref{e:AnsatzL2},~\eqref{e:AnsatzH1},~\eqref{e:AnsatzH1fe}, and~\eqref{e:AnsatzH1+gamma} provide bounds on the $L^2$, $H^1$, $\fe$-directed $H^1$, and $H^{1+\gamma}$ norms of the Ansatz, respectively. The estimate~\eqref{e:AnsatzGen} establishes the key property~\eqref{e:AnsatzHeuristic} of the Ansatz. Finally,~\eqref{e:AnsatzH1precise} gives a precise evaluation of its $H^1$-norm, which will be used to estimate the diffusivity.

\begin{remark} \label{rmk:ComparisonCMTestimates}
Compared with the corresponding results in~\cite{CMT25}, Proposition~\ref{p:GenericAnsatz} provides three key refinements:
\begin{itemize}
    \item The $H^{1+\gamma}$ estimate~\eqref{e:AnsatzH1+gamma} is new. As indicated by Theorem~\ref{th:ComparisonGen} and Section~\ref{sec:Roadmap} (cf.~\eqref{e:HCremovalMeso} and~\eqref{e:CorrectorH1gammaRoadmap}), the control of the $H^{1+\gamma}$-norm of the Ansatz is essential for a sharp comparison between the ASEP and SBE generators. As shown below, this estimate can be deduced from the $\fe$-directed $H^1$ bound~\eqref{e:AnsatzH1fe} by exploiting the recursive structure~\eqref{e:AnsatzIdentity}.

    \item Our estimates allow both polynomial and exponential weights in the number operator $\fN_*^\kappa \, \alpha^\fN$. These weights are essential to compensate for two sources of losses at high chaos levels: the polynomial factor $\fN^6$ arising in the comparison of generators in Theorem~\ref{th:ComparisonGen}, and the exponential factor $\alpha(p)^N$ arising from hypercontractivity (see Proposition~\ref{p:TightnessBlackBox}), which enters the tightness argument. An exponentially weighted version of the $\fe$-directed $H^1$ estimate~\eqref{e:AnsatzH1fe} was already established in~\cite[Proposition 4.8 and Remark 4.10]{CMT25}. In fact, similar weighted estimates can be obtained for~\eqref{e:AnsatzL2} and~\eqref{e:AnsatzGen} with little additional work. However, obtaining the exponentially weighted estimate for the full $H^1$ norm~\eqref{e:AnsatzH1} requires a new argument, based on a different use of the \emph{recursive replacement lemma}~\cite[Lemma 3.9]{CMT25}.

    \item The estimate corresponding to~\eqref{e:AnsatzGen} in~\cite[Proposition 4.6]{CMT25} controls the difference between the Ansatz and the resolvent $(\mu - \bar \fL)^{-1} \bm \vphi$ in $H^1$ norm, rather than directly controlling the action of the generator in $H^{-1}$ norm as in~\eqref{e:AnsatzGen}. In fact, starting from~\eqref{e:AnsatzGen}, one recovers such a bound by moving the term $\fA_-^\sharp (\mu - \fL_\mu^{\eff})^{-1} \bm \vphi$ to the right-hand side using the triangle inequality, and then applying the standard resolvent bound from $H^{-1}$ to $H^1$. However, either of these two losses would be too crude for establishing a fluctuation-dissipation relation. A small adaptation of the proof of~\cite{CMT25} avoids both losses and gives an estimate of the form~\eqref{e:AnsatzGen}.
\end{itemize}
\end{remark}

\begin{proof}[of Proposition~\ref{p:GenericAnsatz}]
For the reader's convenience, we give a complete proof of Proposition~\ref{p:GenericAnsatz}. Most of the details are deferred to Appendices~\ref{app:ReplLemma} and~\ref{app:GenericAnsatz}; here we present the main new argument, which yields the first two improvements discussed in Remark~\ref{rmk:ComparisonCMTestimates}. We begin by outlining the four steps of the proof. During the first two steps, the auxiliary parameter $\nu_0$ (the initial condition of $\Upsilon$, cf.~\eqref{e:ODE}) will not yet be fixed. We therefore work with a free parameter $\nu \geq 0$ in its place, and add a subscript $\nu$ to write the corresponding objects (e.g., $\Upsilon_\nu$ or $\fL^{\eff}_{\mu, \nu}$).
\ss

{\it Step 0. The replacement lemmas:} The main technical tools in the analysis of the SBE Ansatz are the replacement lemmas. They estimate expressions of the form $\langle \fD \fA_+^\sharp \phi, \fA_+^\sharp \psi\rangle$ for specific diagonal operators $\fD$ built from modulations of $\fL_0^\fe$. More precisely, we shall use the following two estimates, corresponding to the replacement lemma and the recursive replacement lemma of~\cite[Lemmas 3.6 and 3.9]{CMT25}.

\begin{lemma} \label{lem:ReplLemma}
1) For every $\nu \geq 0$, $\mu > 0$ and $\phi, \psi \in \fock{\rap}{}$, we have
\begin{equ} \label{e:ReplFixPt}
\big|\langle \big(- \fA_-^\sharp (\mu - \fL_{\mu, \nu}^{\eff})^{-1} \, \fA_+^\sharp + \fL_0^\fe[\Upsilon_\nu]_\mu\big) \, \phi, \psi \rangle\big| \lesssim_\nu  \|(- \fL_0^\fe)^{1/2} \, \phi\| \, \|(- \fL_0^\fe)^{1/2} \, \psi\| \, ,
\end{equ}

2) There exists $\nu_1 \geq 0$, depending only on $D_0$ and $\lambda$, such that for every $\nu \geq \nu_1$, $\theta \in (0, 1]$, $\mu > 0$ and $\phi \in \fock{\rap}{}$, we have
\begin{equ} \label{e:ReplContraction}
\|(- \fL_0^\fe\big[\tilde \Upsilon_\nu^\theta\big]_\mu)^{1/2} \, (\mu - \fL_{\mu, \nu}^{\eff})^{-1} \fA_+^\sharp \phi\|^2 \leq \frac{1}{2 \theta} \, \|(- \fL_0^\fe \big[\tilde \Upsilon_\nu^\theta\big]_\mu)^{1/2} \, \phi\|^2 \, .
\end{equ}
\end{lemma}

The first estimate shows that $- \fL_0^\fe[\Upsilon_\nu]_\mu$ approximately solves the fixed point equation~\eqref{e:FixedPointRoadmap}. The second shows that, when $\theta \in (1/2, 1]$, the operator $(\mu - \fL_{\mu, \nu}^{\eff})^{-1} \, \fA_+^\sharp$ is a strict contraction in the norm $\|(- \fL_0^\fe\big[\tilde \Upsilon_\nu^\theta\big]_\mu)^{1/2} \cdot\|$. Both statements are special cases of Lemma~\ref{lem:ReplLemmaGeneral}, proved in Appendix~\ref{app:ReplLemma}. Here, applying the recursive replacement lemma directly with the modulations $\tilde \Upsilon_\nu^\theta$, as in~\eqref{e:ReplContraction}, slightly simplifies the approach of~\cite{CMT25} and, importantly, will yield in addition the exponentially weighted $H^1$ estimate~\eqref{e:AnsatzH1}.
\ss

{\it Step 1. The $L^2$ estimates:} In Appendix~\ref{app:GenericAnsatz}, we prove~\eqref{e:AnsatzL2} and Lemma~\ref{lem:SummabilityAnsatz} using an $L^2$ contraction estimate for $(\mu - \fL_{\mu, \nu}^{\eff})^{-1} \fA_+^\sharp$, following the strategy of~\cite[Proposition 4.3]{CMT25}. At this stage we also fix $\nu = \nu_0$ as in Definition~\ref{def:nu}.
\ss

{\it Step 2. The $H^1$-type estimates:} We then establish~\eqref{e:AnsatzH1},~\eqref{e:AnsatzH1fe} and~\eqref{e:AnsatzH1+gamma} by combining the recursive replacement lemma~\eqref{e:ReplContraction} with the recursive identity~\eqref{e:AnsatzIdentity}. Since this argument is the main new ingredient, we give it in full below.
\ss

{\it Step 3. The generator estimates:} Finally, we prove in Appendix~\ref{app:GenericAnsatz} the estimates~\eqref{e:AnsatzGen} and~\eqref{e:AnsatzH1precise} by combining the replacement lemma~\eqref{e:ReplFixPt} with the $\fe$-directed $H^1$ bound~\eqref{e:AnsatzH1fe}.
\ms

We now carry out Step 2. The starting point is the following estimate for the Ansatz in the modulated $\fe$-directed $H^1$ norm $\|(- \fL_0^\fe\big[\tilde \Upsilon^\theta\big]_\mu)^{1/2} \cdot\|$, which follows directly from the recursive replacement lemma~\eqref{e:ReplContraction}.

\begin{lemma} \label{lem:H1feEstimate}
Let $\theta \in (1/2, 1]$, $\alpha \in [1, \sqrt{2 \theta})$ and $\kappa \geq 0$. For every $\mu > 0$ and $\bm{\vphi} \in \fock{\rap}{}$, we have
\begin{equ}  \label{e:H1feEstimate}
\|\fN_*^\kappa \alpha^{\fN} \, (- \fL_0^\fe\big[\tilde \Upsilon^\theta\big]_\mu)^{1/2} \, \mathfrak{Ans}_{\mu}(\bm \vphi)\| \lesssim_{\alpha, \kappa, \theta} \|\fN_*^\kappa \alpha^{\fN} \, (-\fL_0^\fe\big[\tilde \Upsilon^\theta\big]_\mu)^{1/2} \, (\mu - \fL_{\mu}^\eff)^{-1} \bm \vphi\| = K_{\alpha, \kappa}^{1/2, \fe, \theta}(\bm \vphi) \, .
\end{equ}
\end{lemma}

\begin{proof}[of Lemma~\ref{lem:H1feEstimate}]
Iterating~\eqref{e:ReplContraction}, we find that, for any $\bm \vphi \in \fock{\rap}{}$ and $n \geq 0$,
\begin{equ}
\|(- \fL_0^\fe\big[\tilde \Upsilon^\theta\big]_\mu)^{1/2} \, \big((\mu - \fL_{\mu}^\eff)^{-1} \fA_+^\sharp\big)^n (\mu - \fL_{\mu}^\eff)^{-1} \bm \vphi\| \leq \frac{1}{(2 \theta)^{n/2}} \, \|(- \fL_0^\fe\big[\tilde \Upsilon^\theta\big]_\mu)^{1/2} \, (\mu - \fL_{\mu}^\eff)^{-1} \bm \vphi\| \, .
\end{equ}
Since the operator $\big((\mu - \fL_{\mu}^{\eff})^{-1} \fA_+^\sharp\big)^n \, (\mu - \fL_{\mu}^{\eff})^{-1}$ maps $\fock{\rap,n_0}{}$ into $\fock{\rap,n_0 + n}{}$ for any $n_0 \geq 0$, the orthogonality of chaoses and the bound $(n_0 + n)_*^\kappa \, \alpha^{n_0 + n} \lesssim_\kappa n_*^\kappa \alpha^n \, (n_0)_*^\kappa \alpha^{n_0}$ further give, for any $\bm \vphi \in \fock{\rap}{}$ and $n \geq 0$,
\begin{equs}
\|\fN_{*}^\kappa \alpha^\fN (- \fL_0^\fe\big[\tilde \Upsilon^\theta\big]_\mu)^{1/2} \big((\mu - &\fL_{\mu}^{\eff})^{-1} \fA_+^\sharp\big)^n \, (\mu - \fL_{\mu}^{\eff})^{-1}  \bm \vphi\| \\
&\lesssim_\kappa \frac{n_*^\kappa \alpha^n}{(2 \theta)^{n/2}} \, \|\fN_{*}^\kappa \alpha^\fN \, (- \fL_0^\fe\big[\tilde \Upsilon^\theta\big]_\mu)^{1/2} (\mu - \fL_{\mu}^{\eff})^{-1} \bm \vphi\| \, .
\end{equs}
By summing this bound over $n \geq 0$ (recall that $\alpha < \sqrt{2 \theta}$) and recognising~\eqref{e:TrueAnsatz}, we deduce~\eqref{e:H1feEstimate}.
\end{proof}

A direct application of Lemma~\ref{lem:H1feEstimate} gives the $\fe$-directed $H^1$ estimate~\eqref{e:AnsatzH1fe}. Indeed, by using the trivial lower bound $\tilde \Upsilon^\theta \geq \det(D_0)^{\theta} \gtrsim 1$, we obtain
\begin{equ}
\|\fN_{*}^\kappa \alpha^{\fN} \, (- \fL_0^\fe)^{1/2} \, \mathfrak{Ans}_{\mu}(\bm \vphi)\| \lesssim \|\fN_{*}^\kappa \alpha^{\fN} \, (- \fL_0^\fe\big[\tilde \Upsilon^\theta\big]_\mu)^{1/2} \, \mathfrak{Ans}_{\mu}(\bm \vphi)\| \, ,
\end{equ}
which shows~\eqref{e:AnsatzH1fe} after using~\eqref{e:H1feEstimate}.
\ss

To prove~\eqref{e:AnsatzH1} and~\eqref{e:AnsatzH1+gamma}, the idea is to exploit the recursive identity~\eqref{e:AnsatzIdentity} which, thanks to the next lemma, makes it possible to bound the $H^1$ and $H^{1+\gamma}$ norms of the Ansatz by appropriate modulated $\fe$-directed $H^1$ norms.

\begin{lemma} \label{lem:ConvertionH1GA+}
For every $\gamma \in (0,1]$, $\mu > 0$ and $\phi \in \fock{\rap}{}$, we have
\begin{equs}
\|(\mu - \fL_0)^{1/2} \, (\mu - \fL_{\mu}^{\eff})^{-1} \fA_+^\sharp \phi\| &\lesssim \|(- \fL_0^\fe\big[\tilde \Upsilon^1\big]_\mu)^{1/2} \, \phi\| \, , \label{e:ConvertionH1feH1GA+} \\
\|(\mu - \fL_0)^{(1+\gamma)/2} \, (\mu - \fL_{\mu}^{\eff})^{-1} \fA_+^\sharp \phi\| &\lesssim_\gamma \|(- \fL_0^\fe)^{1/2} \, \phi\| \, . \label{e:ConvertionH1feH1gammaGA+}
\end{equs}
\end{lemma}

\begin{proof}[of Lemma~\ref{lem:ConvertionH1GA+}]

We start with~\eqref{e:ConvertionH1feH1GA+}. The trivial bound $- \fL_0 \leq - \|D_0^{-1}\| \, \fL_{\mu}^{\eff}$ gives
\begin{equs}
\|(\mu - \fL_0)^{1/2} \, (\mu - \fL_{\mu}^{\eff})^{-1} \fA_+^\sharp \, \phi\|^2 &\lesssim \langle (\mu - \fL_{\mu}^{\eff})^{-1} \, \fA_+^\sharp \phi, \fA_+^\sharp \phi\rangle \\
&\lesssim \langle - \fL_0^\fe[\Upsilon]_\mu \, \phi, \phi\rangle + \langle - \fL_0^\fe \, \phi, \phi\rangle \\
&\lesssim \|(- \fL_0^\fe\big[\tilde \Upsilon^1\big]_\mu)^{1/2} \, \phi\|^2 \, ,
\end{equs}
where the second inequality follows from the first replacement lemma~\eqref{e:ReplFixPt}. This proves~\eqref{e:ConvertionH1feH1GA+}.

We turn to~\eqref{e:ConvertionH1feH1gammaGA+}. By orthogonality of chaoses, it is sufficient to prove the bound for $\phi \in \fock{\rap,n}{}$, with $n \geq 1$. We apply~\eqref{e:SharpDiag} from Lemma~\ref{lem:Diag} and use the notation~\eqref{e:XiNotation}, which leads to
\begin{equs}
\|&(\mu - \fL_0)^{(1+\gamma)/2} (\mu - \fL_{\mu}^\eff)^{-1} \, \fA_+^\sharp \phi\|^2 = \langle (\mu - \fL_0)^{1 + \gamma} (\mu - \fL_{\mu}^\eff)^{-2} \, \fA_+^\sharp \phi, \fA_+^\sharp \phi \rangle_{\mathrm{diag}} \\
&\lesssim n! \int_{\R^{2n}} \dd k_{1:n} \, |\fF \phi(k_{1:n})|^2 \sum_{i=1}^n |\fe \cdot k_i|^2 \int_{\R^2} \frac{(\mu + \Xi_0)^{1 + \gamma} \, \1_{|p|_\infty \leq \pi} \, \dd p}{(\mu + \Xi + \Upsilon(\Log(\mu + \Xi_0)) \, \Xi_0^\fe)^2} \, \\
&\lesssim n! \int_{\R^{2n}} \dd k_{1:n} \, |\fF \phi(k_{1:n})|^2 \sum_{i=1}^n |\fe \cdot k_i|^2 \int_{\R^2} \frac{\1_{|p|_\infty \leq \pi} \, \dd p}{|p|^{2 - 2 \gamma}} \, , \\
&\lesssim_\gamma \|(- \fL_0^\fe)^{1/2} \phi\|^2 \, ,
\end{equs}
where in the third line we used $(\mu + \Xi_0)^{1 + \gamma} \lesssim (\mu + \Xi)^{1 + \gamma}$ and $\Xi \gtrsim |p|^2$ (cf.~\eqref{e:Triangular}), and in the fourth line we noticed that the last integral is finite since $\gamma > 0$. This proves~\eqref{e:ConvertionH1feH1gammaGA+}.
\end{proof}

We are now ready to show~\eqref{e:AnsatzH1} and~\eqref{e:AnsatzH1+gamma}. By applying~\eqref{e:ConvertionH1feH1GA+} and~\eqref{e:ConvertionH1feH1gammaGA+} in each chaos and accounting for the weights in the number operator, we obtain for any $\phi \in \fock{\rap}{}$
\begin{equs}
\|\fN_{*}^\kappa \alpha^\fN \, (- \fL_0)^{1/2} \, (\mu - \fL_{\mu}^{\eff})^{-1} \fA_+^\sharp \phi\| &\lesssim_{\kappa, \alpha} \|\fN_{*}^\kappa \alpha^\fN \, (- \fL_0^\fe\big[\tilde \Upsilon^1\big]_\mu)^{1/2} \, \phi\| \, , \label{e:ConvertionH1feH1GA+bis} \\
\|\fN_{*}^\kappa \alpha^\fN \, (- \fL_0)^{(1+\gamma)/2} \, (\mu - \fL_{\mu}^{\eff})^{-1} \fA_+^\sharp \phi\| &\lesssim_{\kappa, \alpha, \gamma} \|\fN_{*}^\kappa \alpha^\fN \, (- \fL_0^\fe)^{1/2} \, \phi\| \, . \label{e:ConvertionH1feH1gammaGA+bis}
\end{equs}
By taking $\phi = \mathfrak{Ans}_{\mu}(\bm \vphi)$ and using the recursive identity~\eqref{e:AnsatzIdentity}, we deduce
\begin{equs}
\|\fN_{*}^\kappa \alpha^\fN \, &(- \fL_0)^{1/2} \, \mathfrak{Ans}_{\mu}(\bm \vphi)\| \\
&\lesssim_{\kappa, \alpha} \|\fN_{*}^\kappa \alpha^\fN \, (- \fL_0^\fe\big[\tilde \Upsilon^1\big]_\mu)^{1/2} \, \mathfrak{Ans}_{\mu}(\bm \vphi)\| + \|\fN_{*}^\kappa \alpha^\fN \, (- \fL_0)^{1/2} \, (\mu - \fL_{\mu}^{\eff})^{-1} \bm \vphi\| \, , \\
\|\fN_{*}^\kappa \alpha^\fN \, &(- \fL_0)^{(1+\gamma)/2} \, \mathfrak{Ans}_{\mu}(\bm \vphi)\| \\
&\lesssim_{\kappa, \alpha, \gamma} \|\fN_{*}^\kappa \alpha^\fN \, (- \fL_0^\fe)^{1/2} \, \mathfrak{Ans}_{\mu}(\bm \vphi)\| + \|\fN_{*}^\kappa \alpha^\fN \, (- \fL_0)^{(1+\gamma)/2} \, (\mu - \fL_{\mu}^{\eff})^{-1} \bm \vphi\| \, .
\end{equs}
The first terms on the two right-hand sides can be bounded by using~\eqref{e:H1feEstimate} with $\theta = 1$ and~\eqref{e:AnsatzH1fe}, respectively. This results in the desired estimates~\eqref{e:AnsatzH1} and~\eqref{e:AnsatzH1+gamma}, and completes Step 2 of the proof of Proposition~\ref{p:GenericAnsatz}.
\end{proof}

\subsection{The fluctuation-dissipation relation for the SBE}

\label{sec:AnsatzA+}

In this section, we obtain a fluctuation-dissipation relation for the SBE by using the previous Ansatz: for $\vphi \in \fock{\rap}{}$, we show that
\begin{equ} \label{e:FormalFluctuationDissipationSBE}
\tau \fA_+ \fU_t^\tau \vphi \approx d^{\lim} \, \frac{\log_{*}(1/\mu)^{2/3}}{\log_{*}(\tau)^{2/3}} \, \fU_t^\tau \fL_0^\fe \vphi - (\partial_t + \tau \fL) \, \mathfrak{Ans}_\mu(\fA_+^\sharp \fU_t^\tau \vphi) \quad \mbox{ in $H^{-1}$ norm.}
\end{equ}
Here, $\fU_t^\tau$ is the scaling operator from Definition~\ref{def:Utau}. The most straightforward choice for the above parameter $\mu$ is the inverse of the time scale $\tau$. In fact, as sketched in Section~\ref{sec:Roadmap} (recall the discussion after~\eqref{e:BoundaryTermRoadmap}), when proving tightness we will instead need to choose $\mu$ as $(\tau |t-s|)^{-1}$ to have better control of the increments $Y_t^\tau - Y_s^\tau$ when $|t-s| \ll 1$. We stress that, in the two cases of interest, we have in particular $\mu \tau \geq 1$, which we shall always assume from now on.
\ss

The next proposition details all the estimates we need on the corrector $\mathfrak{Ans}_\mu(\fA_+^\sharp \fU_t^\tau \vphi)$. For later use, we define the following constants: for any $\vphi \in \fock{\rap}{}$, $s \geq 0$, $\alpha \geq 1$ and $\kappa \geq 0$,
\begin{equ} \label{e:Cvphi}
C_{\alpha, \kappa}^s(\vphi) \eqdef \|\fN_{*}^\kappa \, \alpha^\fN \, (1 - \fL_0)^s \, \vphi\| \, .
\end{equ}

\begin{proposition} \label{p:AnsatzA+}

Let $\eps \in (0,1/6)$, $\alpha \in [1, \sqrt{1 + 6 \eps})$, $\kappa \geq 0$ and $\gamma \in (0,1]$. For every $\tau \geq 1$, $\mu > 0$ such that $\mu \tau \geq 1$, and $\vphi \in \fock{\rap}{}$, it holds
\begin{equs}
\sqrt{\mu \tau} \, \|\fN_{*}^\kappa \alpha^\fN \, \mathfrak{Ans}_\mu(\fA_+^\sharp \fU_t^\tau \vphi)\| &\lesssim_{\alpha, \kappa} \log_{*}(\tau)^{-1/3} \, C_{\alpha, \kappa}^{1/2}(\vphi) \, , \label{e:AnsatzA+L2} \\
\|\fN_{*}^\kappa \alpha^\fN \, (- \tau \fL_0)^{1/2} \, \mathfrak{Ans}_\mu(\fA_+^\sharp \fU_t^\tau \vphi)\| &\lesssim_{\alpha, \kappa} C_{\alpha, \kappa}^{1/2}(\vphi) \, , \label{e:AnsatzA+H1} \\
\|\fN_{*}^\kappa \alpha^\fN \, (- \tau \fL_0^\fe)^{1/2} \,  \mathfrak{Ans}_\mu(\fA_+^\sharp \fU_t^\tau \vphi)\| &\lesssim_{\alpha, \kappa, \eps} \log_{*}(\tau)^{-1/6 + \eps} \, C_{\alpha, \kappa}^{1/2}(\vphi) \, , \label{e:AnsatzA+H1fe} \\
\|\fN_{*}^\kappa \alpha^\fN \, \tau^{- \gamma/2} (- \tau \fL_0)^{(1+\gamma)/2} \, \mathfrak{Ans}_\mu(\fA_+^\sharp \fU_t^\tau \vphi)\| &\lesssim_{\alpha, \kappa, \gamma, \eps} \log_{*}(\tau)^{-1/6 + \eps} \, C_{\alpha, \kappa}^{1}(\vphi) \, . \label{e:AnsatzA+H1+gamma} \\
\Big\|\fN_{*}^\kappa \alpha^\fN \, (\mu \tau - \tau \fL_0)^{-1/2} \, \Big(\tau \fA_+ \fU_t^\tau \vphi + (\partial_t + \tau \fL) \, &\mathfrak{Ans}_\mu(\fA_+^\sharp \fU_t^\tau \vphi) - d^{\lim} \frac{\log_{*}(1/\mu)^{2/3}}{\log_{*}(\tau)^{2/3}} \, \fU_t^\tau \fL_0^\fe \vphi\Big)\Big\| \\
&\lesssim_{\alpha, \kappa, \eps} \log_{*}(\tau)^{-1/6 + \eps} \, C^1_{\alpha, \kappa+1}(\vphi) \, , \label{e:AnsatzA+Gen}
\end{equs}
\end{proposition}

Estimates~\eqref{e:AnsatzA+L2} and~\eqref{e:AnsatzA+H1} tell us that the corrector has small $L^2$ norm and bounded $H^1$ norm, which reflects the fact that the corresponding Dynkin martingale (e.g. the one that appears in~\eqref{e:DynkinFluctuationDissipation}) should have small boundary terms and bounded second moment. As for the smallness of the $\fe$-directed $H^1$ norm~\eqref{e:AnsatzA+H1fe} and of the $H^{1+\gamma}$ norm~\eqref{e:AnsatzA+H1+gamma} of the corrector, it will be crucial in the next section to exploit the comparison between the SBE and ASEP generators from Theorem~\ref{th:ComparisonGen}.
\ss

The rest of this section is devoted to the proof of Proposition~\ref{p:AnsatzA+}. We first state an intermediate result which is the analogue of Proposition~\ref{p:GenericAnsatz} but where the input of the Ansatz is of the form $\fA_+^\sharp \phi$.

\begin{lemma} \label{lem:AnsatzA+bis}
In the notation of Proposition~\ref{p:GenericAnsatz}, for every $\phi \in \fock{\rap}{}$ we have
\begin{equs}
\, \sqrt{\mu} \, \|\fN_{*}^\kappa \alpha^\fN \, \mathfrak{Ans}_{\mu}(\fA_+^\sharp \phi)\| &\lesssim_{\alpha, \kappa} \|\fN_{*}^\kappa \alpha^\fN \, (- \fL_0^\fe)^{1/2} \phi\| \, , \label{e:AnsatzA+L2bis} \\
\|\fN_{*}^\kappa \alpha^\fN \, (- \fL_0)^{1/2} \, \mathfrak{Ans}_{\mu}(\fA_+^\sharp \phi)\| &\lesssim_{\alpha, \kappa} \|\fN_{*}^\kappa \alpha^\fN (- \fL_0^\fe\big[\tilde \Upsilon^1\big]_\mu)^{1/2} \phi\| \, , \label{e:AnsatzA+H1bis} \\
\|\fN_{*}^\kappa \alpha^\fN \, (- \fL_0^\fe)^{1/2} \,  \mathfrak{Ans}_{\mu}(\fA_+^\sharp \phi)\| &\lesssim_{\alpha, \kappa, \theta} \|\fN_{*}^\kappa \alpha^\fN (- \fL_0^\fe\big[\tilde \Upsilon^\theta\big]_\mu)^{1/2} \phi\| \, , \label{e:AnsatzA+H1febis} \\
\|\fN_{*}^\kappa \alpha^\fN \, (- \fL_0)^{(1+\gamma)/2} \, \mathfrak{Ans}_{\mu}(\fA_+^\sharp \phi)\| &\lesssim_{\alpha, \kappa, \theta, \gamma} \|\fN_{*}^\kappa \alpha^\fN (- \fL_0^\fe\big[\tilde \Upsilon^\theta\big]_\mu)^{1/2} \phi\| \, , \label{e:AnsatzA+H1gammabis} \\
\|\fN_{*}^\kappa \alpha^\fN (\mu - \fL_0)^{-1/2} \big(\bar \fL \, \mathfrak{Ans}_{\mu}(\fA_+^\sharp \phi) + \fA_+ \phi &- \fL_0^\fe[\Upsilon]_\mu \phi\big)\| \lesssim_{\alpha, \kappa, \theta} \|\fN_{*}^{\kappa+1} \alpha^\fN (- \fL_0^\fe\big[\tilde \Upsilon^\theta\big]_\mu)^{1/2} \phi\| \, . \qquad \ \ \label{e:GenAnsatzA+bis}
\end{equs}
\end{lemma}

\begin{proof}
Recall the constants $K$ defined in~\eqref{e:AnsatzConstantK} in Proposition~\ref{p:GenericAnsatz}. We first claim that
\begin{equs}
K_{\alpha, \kappa}^{1/2}(\fA_+^\sharp \phi) &\lesssim_{\alpha, \kappa} \|\fN_{*}^\kappa \alpha^\fN \, (- \fL_0^{\fe}\big[\tilde \Upsilon^1\big]_\mu)^{1/2} \phi\| \, , \label{e:BoundK1A+} \\
K_{\alpha, \kappa}^{(1+\gamma)/2}(\fA_+^\sharp \phi) &\lesssim_{\alpha, \kappa}  \|\fN_{*}^\kappa \alpha^\fN \, (- \fL_0^\fe)^{1/2} \phi\| \, , \label{e:BoundK1gammaA+} \\
K_{\alpha, \kappa}^{1/2, \fe, \theta}(\fA_+^\sharp \phi) &\lesssim_{\alpha, \kappa}  \|\fN_{*}^\kappa \alpha^\fN \, (- \fL_0^\fe\big[\tilde \Upsilon^\theta\big]_\mu)^{1/2} \phi\| \, , \label{e:BoundK1feA+} \\
K_{\alpha, \kappa}^0(\fA_+^\sharp \phi) &\lesssim_{\alpha, \kappa} \mu^{-1/2} \, \|\fN_{*}^\kappa \alpha^\fN \, (- \fL_0^{\fe})^{1/2} \phi\| \, . \label{e:BoundK0A+}
\end{equs}
By orthogonality of chaoses, to prove the above claim we can reduce to the case $\phi \in \fock{n}{}$ with $n \geq 1$, and $\alpha = 1$, $\kappa = 0$ (up to dividing both sides by $n_*^\kappa \alpha^n$). Note that~\eqref{e:BoundK1A+} and~\eqref{e:BoundK1gammaA+} immediately follow from~\eqref{e:ConvertionH1feH1GA+} and~\eqref{e:ConvertionH1feH1gammaGA+}. As for~\eqref{e:BoundK1feA+}, it readily follows from the recursive replacement lemma~\eqref{e:ReplContraction}. We turn to the proof of~\eqref{e:BoundK0A+}: by using the trivial bound $- \fL^\eff_\mu \geq \|D_0\|^{-1} \, (- \fL_0)$, the identity~\eqref{e:SharpDiag} from Lemma~\ref{lem:Diag}, and the integral estimate $\int_{\R^2} \frac{\dd p}{(\mu + |p|^2)^2} \lesssim \mu^{-1}$, we deduce
\begin{equs}
K_{1, 0}^0(\fA_+^\sharp \phi)^2 &\lesssim \langle (\mu - \fL_0)^{-2} \, \fA_+^\sharp \phi, \fA_+^\sharp \phi \rangle_{\mathrm{diag}} \\
&\lesssim n! \int_{\R^{2n}} \dd k_{1:n} \, |\fF \phi(k_{1:n})|^2 \sum_{i=1}^n |\fe \cdot k_i|^2 \int_{\R^2} \frac{\dd p}{(\mu + |p|^2)^2} \\
&\lesssim \mu^{-1} \, \|(- \fL_0^\fe)^{1/2} \phi\|^2 \, .
\end{equs}

Now that we have proved the claim, the estimates~\eqref{e:AnsatzA+L2bis},~\eqref{e:AnsatzA+H1bis},~\eqref{e:AnsatzA+H1febis},~\eqref{e:AnsatzA+H1gammabis} readily follow from their counterparts~\eqref{e:AnsatzL2},~\eqref{e:AnsatzH1},~\eqref{e:AnsatzH1fe},~\eqref{e:AnsatzH1+gamma} in Proposition~\ref{p:GenericAnsatz} (using in particular the trivial bound $1 \lesssim \tilde \Upsilon^\theta \lesssim \tilde \Upsilon^1$). We turn to the proof of~\eqref{e:GenAnsatzA+bis}: we first rearrange its left hand side by the triangle inequality, as follows
\begin{equs}
\|\fN_{*}^\kappa \alpha^\fN \, &(\mu - \fL_0)^{-1/2} \, (\bar \fL \, \mathfrak{Ans}_\mu(\fA_+^\sharp \phi) + \fA_+ \phi - \fL_0^\fe\big[\Upsilon\big]_\mu)\| \\
&\leq \|\fN_{*}^\kappa \alpha^\fN \, (- \fL_0)^{-1/2} \, \big((\mu - \bar \fL) \, \mathfrak{Ans}_\mu(\fA_+ \phi) - \fA_+^\sharp \phi + \fA_-^\sharp (\mu - \fL_{\mu}^{\eff})^{-1} \fA_+^\sharp \phi\big)\| \\
&\quad + \mu^{-1/2} \, \|\fN_{*}^\kappa \alpha^\fN \, \mathfrak{Ans}_\mu(\fA_+^\sharp \phi)\| + \|\fN_{*}^\kappa \alpha^\fN \, (- \fL_0)^{-1/2} \, \fA_+^\flat \phi\| \\
&\quad + \|\fN_{*}^\kappa \alpha^\fN \, (- \fL_0)^{-1/2} \, (\fA_-^\sharp (\mu - \fL_{\mu}^{\eff})^{-1} \fA_+^\sharp - \fL_0^\fe[\Upsilon]_\mu) \, \phi\| \, .
\end{equs}
Now we can bound the first term by using~\eqref{e:AnsatzGen} (with $\bm \vphi = \fA_+^\sharp \vphi$) together with~\eqref{e:BoundK1feA+}, the second term by using~\eqref{e:AnsatzA+L2bis}, the third term by using~\eqref{e:Aflat} (which clearly holds with extra $\fN_*^\kappa \alpha^\fN$ on both sides up to paying a multiplicative cost $\lesssim_{\alpha, \kappa}$, because $\fA_+^\flat$ increases the number operator $\fN$ by $1$), and the fourth term by the first replacement lemma~\eqref{e:ReplFixPt} (and by noticing that $\fA_-^\sharp (\mu - \fL_{\mu}^{\eff})^{-1} \fA_+^\sharp - \fL_0^\fe[\Upsilon]_\mu$ preserves the number operator $\fN$). Thus, we deduce
\begin{equs}
\|&\fN_{*}^\kappa \alpha^\fN \, (\mu - \fL_0)^{-1/2} \, (\bar \fL \, \mathfrak{Ans}_\mu(\fA_+^\sharp \phi) + \fA_+ \phi - \fL_0^\fe\big[\Upsilon\big]_\mu)\| \\
&\lesssim \|\fN_{*}^{\kappa} \alpha^\fN (- \fL_0^\fe\big[\tilde \Upsilon^\theta\big]_\mu)^{1/2} \phi\| + \|\fN_{*}^{\kappa} \alpha^\fN (- \fL_0^\fe)^{1/2} \phi\| + \|\fN_{*}^{\kappa+1} \alpha^\fN (- \fL_0^\fe)^{1/2} \phi\| + \|\fN_{*}^{\kappa} \alpha^\fN (- \fL_0^\fe)^{1/2} \phi\| \, \\
&\lesssim \|\fN_{*}^{\kappa+1} \alpha^\fN (- \fL_0^\fe\big[\tilde \Upsilon^\theta\big]_\mu)^{1/2} \phi\| \, ,
\end{equs}
which proves~\eqref{e:GenAnsatzA+bis} and concludes the proof.
\end{proof}

\begin{proof}[of Proposition~\ref{p:AnsatzA+}]

Let us first state a simple estimate that we shall use repeatedly in the paper: for any $a, b \geq 0$, $\tau \geq 1$ and $\phi \in \fock{\rap}{}$,
\begin{equ}
\|(- \fL_0)^{a} (- \fL_0^\fe)^{b} \, \fU_t^\tau \phi\| \leq \tau^{- a - b} \, \log_{*}(\tau)^{-b/3} \, \|(- \fL_0)^{a+b} \phi\| \, , \label{e:HsUtau}
\end{equ}
and also, for any $\theta \in [0,1]$ and $\mu > 0$ such that $\mu \tau \geq 1$,
\begin{equ}
\|(- \fL_0^\fe\big[\tilde \Upsilon^\theta\big]_\mu)^{1/2} \, \fU_t^\tau \vphi\| \lesssim \tau^{-1/2} \log_{*}(\tau)^{-(1-\theta)/3} \, \|(- \fL_0^\fe)^{1/2} \, \vphi\| \, . \label{e:H1fethetaUtau}
\end{equ}
Obviously,~\eqref{e:HsUtau} and~\eqref{e:H1fethetaUtau} still hold if we add an extra factor $\fN_*^\kappa \alpha^\fN$ on both sides. These two bounds easily follow from Definition~\ref{def:Utau} and the trivial estimates
\begin{equs}
|\fe \cdot (\Sigma^\tau)^{-1} \, k_{1:n}|^2 &= \tau^{-1} \log_{*}(\tau)^{-2/3} \, |\fe \cdot k_{1:n}|^2 \, , \\
|(\Sigma^\tau)^{-1} k_{1:n}|^2 &\leq \tau^{-1} \, |k_{1:n}|^2 \, , \\
\tilde \Upsilon^\theta(\Log(\mu + |p_{1:n}|^2/2)) &\lesssim \log_*(\tau)^{2 \theta/3} \, ,
\end{equs}
where, for the last bound, we used that $\mu + |p_{1:n}|^2/2 \geq 1/\tau$ and the trivial bound $\tilde \Upsilon^\theta(\Log(\cdot)) \lesssim \log_*(1/\cdot)^{2 \theta/3}$ (recall~\eqref{e:Upstheta},~\eqref{e:ODEexplicit}, and $\Log(\cdot) = \log(1/\cdot) \, \1_{0 < \cdot \leq 1}$).
\ss

By using~\eqref{e:HsUtau} and~\eqref{e:H1fethetaUtau}, the estimates~\eqref{e:AnsatzA+L2},~\eqref{e:AnsatzA+H1},~\eqref{e:AnsatzA+H1fe},~\eqref{e:AnsatzA+H1+gamma} readily follows from their counterpart~\eqref{e:AnsatzA+L2bis},~\eqref{e:AnsatzA+H1bis},~\eqref{e:AnsatzA+H1febis},~\eqref{e:AnsatzA+H1gammabis} in Lemma~\ref{lem:AnsatzA+bis} (with $\theta$ therein chosen as $1/2 + 3 \eps \in (1/2, 1)$).

We turn to the proof of~\eqref{e:AnsatzA+Gen}. We first observe that $\fA_0$ commutes with $\fA_+^\sharp$ and $\mathfrak{Ans}_{\mu}$ (which is clear from Remark~\ref{rmk:fA0} and formulas~\eqref{e:fA+Sharp},~\eqref{e:TrueAnsatz}). By recalling that $(\partial_t + \tau \fA_0) \, \fU_t^\tau = 0$ (cf. Remark~\ref{rmk:GalileanTransform}), we deduce
\begin{equ} \label{e:GalileanCancellation}
(\partial_t + \tau \fA_0) \, \mathfrak{Ans}_\mu(\fA_+^\sharp \fU_t^\tau \vphi) = \mathfrak{Ans}_\mu\big(\fA_+^\sharp \, (\partial_t + \tau \fA_0) \, \fU_t^\tau \vphi\big) = 0 \, ,
\end{equ}
and, equivalently: $(\partial_t + \tau \fL) \, \mathfrak{Ans}_\mu(\fA_+^\sharp \fU_t^\tau \vphi) = \tau \bar \fL \, \mathfrak{Ans}_\mu(\fA_+^\sharp \fU_t^\tau \vphi)$. We also observe that Definition~\ref{def:Utau} readily implies: $\fU_t^\tau \fL_0^\fe = \tau \log_*(\tau)^{2/3} \, \fL_0^\fe \, \fU_t^\tau$. Consequently, we deduce
\begin{equs}
\Big\|&\fN_{*}^\kappa \alpha^\fN \, (\mu \tau - \tau \fL_0)^{-1/2} \, \Big((\partial_t + \tau \fL) \, \mathfrak{Ans}_\mu(\fA_+^\sharp \fU_t^\tau \vphi) + \tau \fA_+ \fU_t^\tau \vphi - d^{\lim} \frac{\log_{*}(1/\mu)^{2/3}}{\log_{*}(\tau)^{2/3}} \, \fU_t^\tau \fL_0^\fe \vphi\Big)\Big\| \\
&\leq \tau^{1/2} \, \|\fN_{*}^\kappa \alpha^\fN \, (\mu - \fL_0)^{-1/2} \, (\bar \fL \, \mathfrak{Ans}_\mu(\fA_+^\sharp \fU_t^\tau \vphi) + \fA_+ \fU_t^\tau \vphi - \fL_0^\fe[\Upsilon]_\mu \, \fU_t^\tau\vphi)\| \\
&\quad + \tau^{1/2} \, \|\fN_{*}^\kappa \alpha^\fN \, (- \fL_0)^{-1/2} \, (\fL_0^\fe[\Upsilon]_\mu - d^{\lim} \log_{*}(1/\mu)^{2/3} \fL_0^\fe) \, \fU_t^\tau \vphi\| \, \\
&\lesssim \log(\tau)^{-1/6 + \eps} \, C_{\alpha, \kappa + 1}^{1/2}(\vphi) + \|\fN_{*}^\kappa \alpha^\fN \, (- \tau \fL_0)^{-1/2} \, (\fL_0^\fe[\Upsilon]_\mu - d^{\lim} \log_{*}(1/\mu)^{2/3} \fL_0^\fe) \, \fU_t^\tau \vphi\| , \label{e:AnsatzA+GenProof}
\end{equs}
where in the last bound we used~\eqref{e:GenAnsatzA+bis} (with $\theta = 1/2 + 3 \eps$) together with~\eqref{e:HsUtau}. It only remains to estimate the last term. First, we observe that for any $k_{1:n} \in \R^{2n}$ we have
\begin{equ}
|\Log(\mu + |k_{1:n}|^2/2) - \log_{*}(1/\mu)| \lesssim 1 + \log(1 + \mu^{-1} \, |k_{1:n}|^2/2) \lesssim 1 + \mu^{-1/2} \, |k_{1:n}| \, .
\end{equ}
Then, we use~\eqref{e:AsymptoticUps} (and the trivial bound $|a^{2/3} - b^{2/3}| \lesssim 1 + |a - b|$) and deduce
\begin{equ} \label{e:gnewglim}
|\Upsilon(\Log(\mu + |k_{1:n}|^2/2)) - d^{\lim} \log_{*}(1/\mu)^{2/3}| \lesssim 1 + \mu^{-1/2} \, |k_{1:n}| \leq 1 + \tau^{1/2} \, |k_{1:n}| \, ,
\end{equ}
where in the last bound we used that $\mu \tau \geq 1$. By combining~\eqref{e:gnewglim} and~\eqref{e:HsUtau}, we obtain
\begin{equs}
\tau^{1/2} \, \|\fN_{*}^\kappa \alpha^\fN \, &(- \fL_0)^{-1/2} \, (\fL_0^\fe[\Upsilon]_\mu - d^{\lim} \log_{*}(1/\mu)^{2/3} \fL_0^\fe) \, \fU_t^\tau \vphi\| \\
&\lesssim \tau^{1/2} \, \|\fN_{*}^\kappa \alpha^\fN \, (- \fL_0^\fe)^{1/2} \, \fU_t^\tau \vphi\| + \tau \,  \|\fN_{*}^\kappa \alpha^\fN \, (- \fL_0^\fe)^{1/2} (- \fL_0)^{1/2} \, \fU_t^\tau \vphi\| \\
&\lesssim \log(\tau)^{-1/6} \, C_{\alpha, \kappa}^1(\vphi) \, . \label{e:LnewLlim}
\end{equs}
Plugging this bound in~\eqref{e:AnsatzA+GenProof} gives the desired estimate~\eqref{e:AnsatzA+Gen} and concludes the proof.
\end{proof}

\subsection{The fluctuation-dissipation relation for the ASEP}

\label{sec:AnsatzA+ASEP}

We now derive the fluctuation-dissipation relation~\eqref{e:FormalFluctuationDissipationASEP} by transferring the previous SBE fluctuation-dissipation relation to the ASEP setting. More precisely, the ASEP corrector is simply obtained by passing the SBE corrector $\mathfrak{Ans}_\mu(\fA_+^\sharp \fU_t^\tau \vphi)$ from $\fock{}{}$ to $\fockdo{}{}$ via the bridge $\fsR \fsP$.

\begin{definition} \label{def:AnsatzA+ASEP}
For any $\tau \geq 1$, $\mu > 0$ with $\mu \tau \geq 1$, and $\vphi \in \fock{\rap}{}$ we set
\begin{equ} \label{e:AnsatzA+ASEP}
\mathrm{Cor}^{\tau, \mu}_{t}(\vphi) \eqdef \fsR \fsP \, \mathfrak{Ans}_\mu(\fA_+^\sharp \fU_t^\tau \vphi) \, .
\end{equ}
\end{definition}

By combining Theorem~\ref{th:ComparisonGen} and Proposition~\ref{p:AnsatzA+}, we deduce the following proposition which details all the estimates we need on the ASEP corrector.

\begin{proposition} \label{p:AnsatzA+ASEP}
Let $\tau \geq 1$ and $\mu > 0$ such that $\mu \tau \geq 1$, and let $\eps \in (0,1/6)$, $\alpha \in [1, \sqrt{1 + 6 \eps})$, $\kappa \geq 0$. Then it holds
\begin{equs}
\sqrt{\mu \tau} \, \|N_{*}^\kappa \alpha^N \, \mathrm{Cor}^{\tau, \mu}_{t}(\vphi)\| &\lesssim_{\alpha, \kappa} \log_{*}(\tau)^{-1/3} \, C_{\alpha, \kappa}^{1/2}(\vphi) \, , \label{e:AnsatzA+L2ASEP} \\
\|N_{*}^\kappa \alpha^N \, (- \tau S)^{1/2} \, \mathrm{Cor}^{\tau, \mu}_{t}(\vphi)\| &\lesssim_{\alpha, \kappa} C_{\alpha, \kappa}^{1/2}(\vphi) \, , \label{e:AnsatzA+H1ASEP} \\
\Big\|N_{*}^\kappa \alpha^N (\mu \tau - \tau S)^{-1/2} \Big(\tau A_+ \fsi_t^\tau
\vphi + (\partial_t + \tau L) \, &\mathrm{Cor}^{\tau, \mu}_{t}(\vphi) - d^{\lim} \frac{\log_{*}(1/\mu)^{2/3}}{\log_{*}(\tau)^{2/3}} \, \fsi_t^\tau \fL_0^\fe \vphi\Big)\Big\| \\
&\lesssim_{\alpha, \kappa, \eps} \log_{*}(\tau)^{-1/6 + \eps} \, C_{\alpha, \kappa+6}^{3/2}(\vphi) \, , \label{e:AnsatzA+GenASEP}
\end{equs}
where $\fsi_t^\tau$ is defined in~\eqref{e:fsittau}, and the constants $C^s_{\alpha, \kappa}(\vphi)$ are given by~\eqref{e:Cvphi}.
\end{proposition}

\begin{proof}
We first comment that, by orthogonality of chaoses, the estimates of Lemma~\ref{lem:BoundR}, Lemma~\ref{lem:BoundQ}, Lemma~\ref{lem:BoundP}, and~\eqref{e:HsUtau} remain valid if we add an extra factor $N_{*}^\kappa \alpha^N$, $\cN_{*}^\kappa \alpha^\cN$ or $\fN_{*}^\kappa \alpha^\fN$ (depending on the context) on both sides of the inequalities. Furthermore, the same holds for Theorem~\ref{th:ComparisonGen}, up to paying an extra multiplicative cost only depending on $\alpha$ and $\kappa$ (because $(n \pm 1)_*^\kappa \alpha^{n \pm 1} \lesssim_{\kappa, \alpha} n_{*}^\kappa \alpha^n)$. For the upcoming estimates, it is useful to choose any $\gamma \in (0,1)$ that we fix for the rest of the proof.
\ss

We start with the proof of~\eqref{e:AnsatzA+L2ASEP} and~\eqref{e:AnsatzA+H1ASEP}. By first using~\eqref{e:BoundR} and~\eqref{e:BoundP}, and then applying~\eqref{e:AnsatzA+L2} and~\eqref{e:AnsatzA+H1}, we readily deduce
\begin{equs}
\sqrt{\mu \tau} \, \|N_{*}^\kappa \alpha^N \, \mathrm{Cor}^{\tau, \mu}_t(\vphi)\| &\lesssim \sqrt{\mu \tau} \, \|\fN_{*}^\kappa \alpha^\fN \, \mathfrak{Ans}_\mu(\fA_+^\sharp \fU_t^\tau \vphi)\| \lesssim \log_{*}(\tau)^{-1/3} \, C_{\alpha, \kappa}^{1/2}(\vphi) \, , \\
\|N_{*}^\kappa \alpha^N \, (- \tau S)^{1/2} \, \mathrm{Cor}^{\mu, \tau}_t(\vphi)\| &\lesssim \|\fN_{*}^\kappa \alpha^\fN \, (- \tau \fL_0)^{1/2} \, \mathfrak{Ans}_\mu(\fA_+^\sharp \fU_t^\tau \vphi)\| \lesssim C_{\alpha, \kappa}^{1/2}(\vphi) \, ,
\end{equs}
which proves the two desired estimates.
\ss

We turn to the proof of~\eqref{e:AnsatzA+GenASEP}. We can rearrange its left hand side by the triangle inequality, as follows
\begin{equs}
\Big\|&N_{*}^\kappa \alpha^N (\mu \tau - \tau S)^{-1/2} \Big(\tau A_+ \fsi_t^\tau
\vphi + (\partial_t + \tau L) \, \mathrm{Cor}^{\tau, \mu}_{t}(\vphi) - d^{\lim} \frac{\log_{*}(1/\mu)^{2/3}}{\log_{*}(\tau)^{2/3}} \, \fsi_t^\tau \fL_0^\fe \vphi\Big)\Big\| \\
&\leq \Big\|N_{*}^\kappa \alpha^N (\mu \tau - \tau S)^{-1/2} \, \fsQ^* \fsP \Big(\tau \fA_+ \fU_t^\tau \vphi + (\partial_t + \tau \fL) \, \mathfrak{Ans}_\mu(\fA_+^\sharp \fU_t^\tau \vphi) - d^{\lim} \frac{\log_{*}(1/\mu)^{2/3}}{\log_{*}(\tau)^{2/3}} \, \fU_t^\tau \fL_0^\fe \vphi\Big)\Big\| \\
&\qquad + \|N_{*}^\kappa \alpha^N (\mu \tau - \tau S)^{-1/2} \, (\tau A_+ \fsR \fsP - \fsQ^* \fsP \, \tau \fA_+) \, \fU_t^\tau \vphi\| \\
&\qquad + \|N_{*}^\kappa \alpha^N (\mu \tau - \tau S)^{-1/2} \, (\fsR - \fsQ^*) \, \fsP \, \partial_t \,  \mathfrak{Ans}_\mu(\fA_+^\sharp \fU_t^\tau \vphi)\| \\
&\qquad + \|N_{*}^\kappa \alpha^N (\mu \tau - \tau S)^{-1/2} \, (\tau L \fsR \fsP - \fsQ^* \fsP \, \tau \fL) \, \mathfrak{Ans}_\mu(\fA_+^\sharp \fU_t^\tau \vphi)\| \\
&\qquad + \frac{d^{\lim} \log_{*}(1/\mu)^{2/3}}{\log_{*}(\tau)^{2/3}} \, \|N_{*}^\kappa \alpha^N (\mu \tau - \tau S)^{-1/2} \, (\fsR - \fsQ^*) \fsP \, \fU_t^\tau \fL_0^\fe \vphi\| \\
&=: \one + \two + \three + \four + \five \, . \label{e:AnsatzA+ASEP0}
\end{equs}

For the first summand, we use~\eqref{e:BoundQ*},~\eqref{e:BoundP} and~\eqref{e:AnsatzA+Gen}, which results in
\begin{equ}
\one \lesssim_{\alpha, \kappa, \eps} \|\fN_{*}^{\kappa + 3} \alpha^\fN \, (\mu \tau - \tau \fL_0)^{-1/2} \, (...)\| \lesssim \log_{*}(\tau)^{-1/6+\eps} \, C_{\alpha, \kappa + 4}^{1}(\vphi) \, , \label{e:AnsatzA+ASEP1}
\end{equ}
where $(...)$ stands for the expression inside the large parenthesis in the definition of $\one$.

For the second summand, we use Theorem~\ref{th:ComparisonGen} and~\eqref{e:HsUtau}, which gives
\begin{equs}
\two &\lesssim_{\gamma} \tau^{1/2} \, \|\fN_{*}^{\kappa + 6} \alpha^\fN \, (- \fL_0)^{(1+\gamma)/2} \, \fU_t^\tau \vphi\| + \tau^{1/2} \, \|\fN_{*}^{\kappa + 6} \alpha^\fN \, (- \fL_0^\fe)^{1/2} \, \fU_t^\tau \vphi\| \\
&\lesssim \tau^{- \gamma/2} \, C_{\alpha, \kappa + 6}^{(1+\gamma)/2}(\vphi) + \log_{*}(\tau)^{-1/3} \, C_{\alpha, \kappa + 6}^{1/2}(\vphi) \, . \label{e:AnsatzA+ASEP2}
\end{equs}

For the third summand, we first observe that $\partial_t \, \fU_t^\tau = - \tau \fA_0 \, \fU_t^\tau = - \tau^{1/2} \log_{*}(\tau)^{-1/3} \, \fU_t^\tau \, \fA_0$ (recall Definition~\ref{def:Utau} and~\eqref{e:fA0}). We then deduce
\begin{equs}
\three &= \tau^{1/2} \log_{*}(\tau)^{-1/3} \, \|N_{*}^\kappa \alpha^N (\mu \tau - \tau S)^{-1/2} \, (\fsR - \fsQ^*) \, \fsP \, \mathfrak{Ans}_\mu(\fA_+^\sharp \fU_t^\tau \fA_0 \vphi)\| \\
&\lesssim \tau^{1/2} \log_{*}(\tau)^{-1/3} \times \log_{*}(1/\mu) \, \tau^{-1} \|\fN_{*}^{\kappa+2} \alpha^\fN (\mu \tau - \tau \fL_0)^{1/2} \, \mathfrak{Ans}_\mu(\fA_+^\sharp \fU_t^\tau \fA_0 \vphi)\| \\
&\lesssim_{\alpha, \kappa} \tau^{-1/2} \log_{*}(\tau)^{2/3} \, C_{\alpha, \kappa + 2}^1(\fA_0 \vphi) \\
&\lesssim \tau^{-1/2} \log_{*}(\tau)^{2/3} \, C_{\alpha, \kappa + 5/2}^{3/2}(\vphi) \, , \label{e:AnsatzA+ASEP3}
\end{equs}
where in the first bound we used~\eqref{e:RemovalQ*R} and~\eqref{e:BoundP}, in the second bound we used~\eqref{e:AnsatzA+L2} and~\eqref{e:AnsatzA+H1} (and $\mu \tau \geq 1$), and in the third bound we used the trivial bound $|v \cdot (k_1 + \dots + k_n)| \lesssim \sqrt{n} \, |k_{1:n}|$ due to the Cauchy-Schwarz inequality. 

For the fourth summand, we first apply Theorem~\ref{th:ComparisonGen} and then use~\eqref{e:AnsatzA+H1+gamma} and~\eqref{e:AnsatzA+H1fe}, which leads to
\begin{equs}
\four &\lesssim_{\gamma} \tau^{1/2} \, \|\fN_{*}^{\kappa + 6} \alpha^\fN \, (- \fL_0)^{(1+\gamma)/2} \, \mathfrak{Ans}_\mu(\fA_+^\sharp \fU_t^\tau \vphi)\| + \tau^{1/2} \, \|\fN_{*}^{\kappa + 6} \alpha^\fN \, (- \fL_0^\fe)^{1/2} \, \mathfrak{Ans}_\mu(\fA_+^\sharp \fU_t^\tau \vphi)\| \\
&\lesssim_{\alpha, \kappa, \eps} \tau^{- \gamma/2} \, C_{\alpha, \kappa + 6}^{(1+\gamma)/2}(\vphi) + \log_{*}(\tau)^{-1/6 + \eps} \, C_{\alpha, \kappa + 6}^{1}(\vphi) \, . \label{e:AnsatzA+ASEP4}
\end{equs}

For the fifth summand, we use $\mu \tau \geq1$,~\eqref{e:RemovalQ*R},~\eqref{e:BoundP} and~\eqref{e:HsUtau}, which gives
\begin{equs}
\five &\lesssim \|N_{*}^\kappa \alpha^N (1 - \tau S)^{-1/2} \, (\fsR - \fsQ^*) \fsP \, \fU_t^\tau \fL_0^\fe \vphi\| \\
&\lesssim \tau^{-1} \log_{*}(\tau) \, \|\fN_{*}^{\kappa+2} \alpha^\fN \, (1 - \tau \fL_0)^{1/2} \fU_t^\tau \fL_0^\fe \vphi\| \\
&\lesssim \tau^{-1} \log_{*}(\tau) \,  C_{\alpha, \kappa + 2}^{1/2}(\fL_0^\fe \vphi) \\
&\leq \tau^{-1} \log_{*}(\tau) \, C_{\alpha, \kappa + 2}^{3/2}(\vphi) \, . \label{e:AnsatzA+ASEP5}
\end{equs}

We can finally gather~\eqref{e:AnsatzA+ASEP0},~\eqref{e:AnsatzA+ASEP1},~\eqref{e:AnsatzA+ASEP2},~\eqref{e:AnsatzA+ASEP3},~\eqref{e:AnsatzA+ASEP4},~\eqref{e:AnsatzA+ASEP5} and deduce the estimate~\eqref{e:AnsatzA+GenASEP}, which concludes the proof.
\end{proof}

\subsection{The corrector for the diffusivity}

\label{sec:AnsatzDiffusivity}

\label{sec:AnsatzB+}

In this section, we introduce a corrector $\mathrm{Cor}_{\mu, \delta}^{B+}$ which formally satisfies
\begin{equ}
B_+ \cV_{|\Z^2}^\delta \approx - L \, \mathrm{Cor}_{\mu, \delta}^{B+} \quad \mbox{ in $H^{-1}$ norm,}
\end{equ}
where the operator $B_+$ and the function $\cV_{|\Z^2}^\delta$ are given by Definitions~\ref{def:B+} and~\ref{def:V}. The role of this corrector is to capture the sharp asymptotic of the diffusivity, because its Dynkin martingale should approximate the time integral of $B_+ \cV^\delta$ appearing in Lemma~\ref{lem:GKneat}.
\ss

At the level of the SBE, this corrector is given by $\mathfrak{Ans}_\mu(\fB_+ \cV^\delta)$ and the next proposition details all the related estimates we need. The proof can be found in Appendix~\ref{app:AnsatzProof}, and mainly follows from Proposition~\ref{p:GenericAnsatz} after evaluating the constants appearing therein with $\bm \vphi = \fB_+ \cV^\delta$.

\begin{proposition} \label{p:AnsatzDiff}
Let $\mu > 0$, $\delta \in (0,1]$, $\eps \in (0,1/6)$, $\alpha \in [1, \sqrt{1 + 6 \eps})$, and $\gamma \in (0,1]$. Then it holds
\begin{equs}
\sqrt{\mu} \, \|\alpha^\fN \, \mathfrak{Ans}_\mu(\fB_+ \cV^\delta)\| &\lesssim_\alpha 1 \, , \label{e:AnsatzDiffL2} \\
\|\alpha^\fN \, (- \fL_0^\fe)^{1/2} \, \mathfrak{Ans}_\mu(\fB_+ \cV^\delta)\| &\lesssim_{\alpha, \eps} \log_{*}(1/\mu)^{1/6 + \eps} \, , \label{e:AnsatzDiffH1fe} \\
\|\alpha^\fN \, (- \fL_0)^{(1+\gamma)/2} \, \mathfrak{Ans}_\mu(\fB_+ \cV^\delta)\| &\lesssim_{\alpha, \eps, \gamma} \log_{*}(1/\mu)^{1/6 + \eps} \, ,
\label{e:AnsatzDiffH1+gamma} \\
\limsup_{\delta \to 0} \big\|\alpha^\fN \, (\mu - \fL_0)^{-1/2} \, \big(\fL \, \mathfrak{Ans}_\mu(\fB_+ \cV^\delta) + \fB_+ &\cV^\delta\big)\big\| \lesssim_{\alpha, \eps} \log_{*}(1/\mu)^{1/6 + \eps} \, , \label{e:AnsatzDiffGen} \\
\limsup_{\delta \to 0} \big|\, 2 \, \|(- \fS)^{1/2} \, \mathfrak{Ans}_\mu(\fB_+ \cV^\delta)\|^2 - d^{\lim} \log_{*}(1/&\mu)^{2/3}\big| \lesssim_{\eps} \log_{*}(1/\mu)^{1/3 + 2 \eps} \, .\label{e:AnsatzDiffH1}
\end{equs}
\end{proposition}

We can now transfer the corrector $\mathfrak{Ans}_\mu(\fB_+ \cV^\delta)$ to the ASEP setting in a similar vein to Definition~\ref{def:AnsatzA+ASEP}.

\begin{definition}
\label{def:AnsatzDiffASEP}
For every $\mu > 0$ and $\delta \in (0,1]$, we define
\begin{equ}
\mathrm{Cor}_{\mu, \delta}^{B+} \eqdef \fsR \fsP \, \mathfrak{Ans}_\mu(\fB_+ \cV^\delta) \, .
\end{equ}
\end{definition}

By combining Proposition~\ref{p:AnsatzDiff} and Theorem~\ref{th:ComparisonGen}, we deduce the following key estimates for the diffusivity corrector.

\begin{proposition} \label{p:AnsatzDiffASEP}
In the notation of Definition~\ref{def:AnsatzDiffASEP}, let $\eps \in (0,1/6)$ and $\alpha \in [1, \sqrt{1 + 6 \eps})$. Then it holds
\begin{equs}
\sqrt{\mu} \, \|\alpha^N \, \mathrm{Cor}_{\mu, \delta}^{B+}\| &\lesssim_\alpha 1 \, , \label{e:AnsatzDiffL2ASEP} \\
\limsup_{\delta \to 0} \big\|\alpha^N (\mu - S)^{-1/2} \big(L \, \mathrm{Cor}_{\mu, \delta}^{B+} + B_+ \cV^\delta_{|\Z^2}\big)\big\| &\lesssim_{\alpha, \eps} \log_{*}(1/\mu)^{1/6 + \eps} \, , \label{e:AnsatzDiffGenASEP} \\
\limsup_{\delta \to 0} \big|\, 2 \, \|(- S)^{1/2} \, \mathrm{Cor}_{\mu, \delta}^{B+}\|^2 - d^{\lim} \log_{*}(1/\mu)^{2/3}\big| &\lesssim_{\eps} \log_{*}(1/\mu)^{1/3 + 2 \eps} \, . \label{e:AnsatzDiffH1ASEP}
\end{equs}
\end{proposition}

\begin{proof}
For the upcoming estimates, it is useful to choose any $\gamma \in (0,1)$ and $\beta \in (\alpha, \sqrt{1 + 6 \eps})$ that we fix for the proof.
\ss

First, since the bridges $\fsR$ and $\fsP$ trivially contract the $L^2$-norm $\|\cdot\|$ (and preserve the number operator), we readily deduce~\eqref{e:AnsatzDiffL2ASEP} from~\eqref{e:AnsatzDiffL2}.

Next, we turn to the estimate~\eqref{e:AnsatzDiffH1ASEP}. By using Lemma~\ref{lem:ComparisonH1} and the bound~\eqref{e:AnsatzDiffH1+gamma} (with any $\alpha \in (1, \sqrt{1 + 6 \eps})$, so that we can crudely bound $n \lesssim_\alpha \alpha^n$), we deduce
\begin{equs}
\big|\|(- S)^{1/2} \, \mathrm{Cor}^{B+}_{\mu, \delta}\|^2 - \|(-\fS)^{1/2} \,  \mathfrak{Ans}_\mu(\fB_+ \cV^\delta)\|^2\big| &\lesssim_\gamma \|\fN \, (- \fL_0)^{(1+\gamma)/2} \mathfrak{Ans}_\mu(\fB_+ \cV^\delta)\|^2 \\
&\lesssim_{\alpha, \eps, \gamma} \log_{*}(1/\mu)^{1/3 + 2\eps} \, .
\end{equs}
By combining this estimate with~\eqref{e:AnsatzDiffH1}, we readily deduce~\eqref{e:AnsatzDiffH1ASEP}.
\ss

Finally, we tackle~\eqref{e:AnsatzDiffGenASEP}. By the triangle inequality, we write
\begin{equs}
\big\|\alpha^N (\mu - S)^{-1/2} \big(L \, \mathrm{Cor}_{\mu, \delta}^{B+} + B_+ \cV^\delta_{|\Z^2}\big)\big\| &\leq \|\alpha^N (\mu - S)^{-1/2} \, \fsQ^* \fsP \big(\fL \, \mathfrak{Ans}_\mu(\fB_+ \cV^\delta) + \fB_+ \cV^\delta\big)\| \\
&\quad + \|\alpha^N (\mu - S)^{-1/2} \big(L \fsR \fsP - \fsQ^* \fsP \fL) \, \mathfrak{Ans}_\mu(\fB_+ \cV^\delta)\| \\
&\quad + \|\alpha^N (\mu - S)^{-1/2} (B_+ \cV^\delta_{|\Z^2} - \fsQ^* \fsP \fB_+ \cV^\delta)\| \\
&=: \one + \two + \three \, . \label{e:AnsatzDiffGenASEP0}
\end{equs}

For the first summand, we use~\eqref{e:BoundQ*},~\eqref{e:BoundP}, and~\eqref{e:AnsatzDiffGen} (with $\beta$ instead of $\alpha$, so that we can crudely bound $n^3 \alpha^n \lesssim_{\alpha, \beta} \beta^n$), which gives
\begin{equ}
\limsup_{\delta \to 0} \one \lesssim \limsup_{\delta \to 0} \|\fN^3 \alpha^\fN \, (\mu - \fL_0)^{-1/2} \big(\fL \, \mathfrak{Ans}_\mu(\fB_+ \cV^\delta) + \fB_+ \cV^\delta\big)\| \lesssim_{\alpha, \beta, \eps} \log_{*}(1/\mu)^{1/6 + \eps} \, .  \label{e:AnsatzDiffGenASEP1}
\end{equ}

For the second summand, we use Theorem~\ref{th:ComparisonGen},~\eqref{e:AnsatzDiffH1+gamma} and~\eqref{e:AnsatzDiffH1fe} (with $\beta$ instead of $\alpha$, so that we can bound $n^6 \alpha^n \lesssim_{\alpha, \beta} \beta^n$), which results in
\begin{equs}
\two &\lesssim \|\fN^6 \alpha^\fN \, (- \fL_0)^{(1+\gamma)/2} \, \mathfrak{Ans}_\mu(\fB_+ \cV^\delta)\| + \|\fN^6 \alpha^\fN \, (- \fL_0^\fe)^{1/2} \, \mathfrak{Ans}_\mu(\fB_+ \cV^\delta)\| \\
&\lesssim_{\alpha, \beta, \eps} \log_{*}(1/\mu)^{1/6 + \eps} \, . \label{e:AnsatzDiffGenASEP2}
\end{equs}

For the third summand, we first notice that $B_+ \cV_{|\Z^2}^\delta = \fsQ^* B_+ \cV_{|\Z^2}^\delta$ because $B_+ \cV_{|\Z^2}^\delta \in \fockdo{}{} \subset \fockd{}{}$. Then we use~\eqref{e:BoundQ} (and note that the number operator simply contributes here as $2$) to deduce
\begin{equ}
\three = \alpha^2 \, \|(\mu - S)^{-1/2} \, \fsQ^* (B_+ \cV_{|\Z^2}^\delta -\fsP \fB_+ \cV^\delta)\| \lesssim \|(\mu - \cS)^{-1/2} \, (B_+ \cV_{|\Z^2}^\delta - \fsP  \fB_+ \cV^\delta)\| \, . \label{e:AnsatzDiffGenASEP3}
\end{equ}
Now, we claim that the last term can be bounded as
\begin{equ} \label{e:B+Replacement}
\|(- \cS)^{-1/2} \, (B_+ \cV_{|\Z^2}^\delta - \fsP  \fB_+ \cV^\delta)\| \lesssim 1 \, .
\end{equ}
Then, by gathering~\eqref{e:AnsatzDiffGenASEP0},~\eqref{e:AnsatzDiffGenASEP1},~\eqref{e:AnsatzDiffGenASEP2},~\eqref{e:AnsatzDiffGenASEP3}, and the claim~\eqref{e:B+Replacement}, we deduce the desired estimate~\eqref{e:AnsatzDiffGenASEP}.
\ss

Consequently, it only remains to show the claim~\eqref{e:B+Replacement}. Let us first compute the Fourier transforms of the functionals $\fsP \fB_+ \cV^\delta$ and $B_+ \cV_{|\Z^2}^\delta$. For the first functional, we deduce from the definition~\eqref{e:fB+} of $\fB_+$ that, for any $p, q \in (- \pi, \pi)^2$, we have
\begin{equ} \label{e:FourierfB+}
\cF(\fsP \fB_+ \cV^\delta) \, (p,q) = \frac{\sqrt{\rho(1 - \rho)}}{2 \pi} \, \lambda \, \Pi_{p, q} \ \fF \cV^\delta (\llbracket p+q \rrbracket) \, .
\end{equ}
As for the second functional, an easy computation starting from the definition~\eqref{e:B+} of $B_+$ shows that
\begin{equ} \label{e:FourierB+}
\cF(B_+ \cV_{|\Z^2}^\delta) \, (p,q) = \frac{\sqrt{\rho(1 - \rho)}}{2 \pi} \, \frac{\lambda(p) + \lambda(q)}{2} \, \cF (\cV_{|\Z^2}^\delta) (p+q) \, ,
\end{equ}
where we have defined, for any $k \in \T^2$,
\begin{equ} \label{e:LambdaTrigo}
\lambda(k) \eqdef \sum_{z \in \Z^2} a(z) \, (\fe \cdot z) \, \cos(z \cdot k) \, .
\end{equ}
Furthermore, since $\fF \cV^\delta$ is supported on $(- \pi, \pi)^2$ (by the assumption on $\cV$, cf. Definition~\ref{def:V}), the Poisson summation formula~\eqref{e:PoissonSummation} shows that, for any $k \in \T^2$, we have $\cF(\cV_{|\Z^2}^\delta) \, (k) = \fF \cV^\delta \, (\llbracket k \rrbracket)$. Hence, it holds
\begin{equ} \label{e:FourierB+Bis}
\cF(B_+ \cV_{|\Z^2}^\delta) \, (p,q) = \frac{\sqrt{\rho(1 - \rho)}}{2 \pi} \, \frac{\lambda(p) + \lambda(q)}{2} \, \fF \cV^\delta (\llbracket p+q \rrbracket) \, .
\end{equ}
Now we compare~\eqref{e:FourierfB+} and~\eqref{e:FourierB+}. On the one hand, from the Taylor estimate $|\cos(\theta) - 1| \lesssim \theta^2$ we deduce that, for any $k \in \R^2$,
\begin{equ}
|\lambda(k) - \lambda| \leq \sum_{z \in \Z^2} |a(z)| \, |\fe \cdot z| \, |\cos(z \cdot k) - 1| \lesssim |k|^2 \, .
\end{equ}
On the other hand, we can trivially bound $|\Pi_{p,q} - 1| \lesssim |p|^2 + |q|^2$ for any $p, q \in \R^2$. Consequently, we deduce for every $p, q \in (- \pi, \pi)^2$
\begin{equ}
|\cF(B_+ \cV_{|\Z^2}^\delta - \fsP \fB_+ \cV^\delta) \, (p,q)| \lesssim (|p|^2 + |q|^2) \ |\fF \cV^\delta \, (\llbracket p+q \rrbracket)| \, .
\end{equ}
Then, by squaring, dividing by $d_0(p,q) \asymp |p|^2 + |q|^2$ (recall~\eqref{e:d0D0equivalence}), integrating over $p, q \in (- \pi, \pi)^2$ and recalling that $\|\cV^\delta\|_{L^2(\R^2)}^2 = 1$, we deduce the desired claim~\eqref{e:B+Replacement} and conclude the proof.
\end{proof}

\section{The superdiffusive invariance principle}

\label{sec:Fluctuation}

This section is devoted to the proof of the superdiffusive invariance principle for the fluctuation process, as stated in Theorem~\ref{th:CLT}. The plan is, on the one hand, to show the convergence of the finite-dimensional laws (Section~\ref{sec:FinDim}), by proving first the convergence of the semigroups through the rescaling operator $\fsi^\tau_t$ (Section~\ref{sec:CVsemigroup}), and, on the other hand, to show tightness (Section~\ref{sec:Tightness}). Both directions crucially rely on the fluctuation-dissipation relation~\eqref{e:FormalFluctuationDissipationASEP}, which we derived in the previous section. In addition, we need some general estimates on Dynkin's martingales and additive functionals, which we expose in the next section.

\subsection{General estimates about Dynkin's martingales and additive functionals}

\label{sec:TightBlackBox}

In this section, we give useful $L^p$-estimates for increments of Dynkin's martingales $M_\cdot^\tau(f)$ and additive functionals $W_\cdot^\tau(g)$, which are defined just below. Let $(\cF_t)_{t \geq 0}$ be the augmented filtration generated by the ASEP $(\eta_t)_{t \geq 0}$.

\begin{definition} \label{def:MI}
Let $\tau \geq 1$ and $T > 0$. For any $f = (f_t)_{t \in [0,T]} \in C^1([0, T], \fockdo{\rap}{})$ and $g = (g_t)_{t \in [0,T]} \in C([0, T], \fockdo{\rap}{})$, we define the processes $(M_t^\tau(f))_{t \in [0,T]}$ and $(W_t^\tau(g))_{t \in [0,T]}$ on $(\bm{\Omega}, \mathbf P_{\nu_\rho})$ as follows
\begin{equ} \label{e:DefMI}
M_t^\tau(f) \eqdef f_t(\eta_{\tau t}) - f_0(\eta_0) - \int_0^t (\partial_r + \tau L) f_{r} \, (\eta_{\tau r}) \ \dd r \, , \quad W_t^\tau(g) \eqdef \int_0^t g_r(\eta_{\tau r}) \, \dd r \, .
\end{equ}
It is classical that $(M_\cdot^\tau(f))_{t \geq 0}$ is a càdlàg $(\cF_{\tau t})_{t \geq 0}$-martingale and that $(W_t^\tau(g))_{t \geq 0}$ is a continuous $(\cF_{\tau t})_{t \geq 0}$-adapted finite-variation process.
For $0 \leq s \leq t \leq T$, we set $M_{s, t}^\tau(f) := M_t^\tau(f) - M_s^\tau(f)$ and $W_{s,t}^\tau(g) := W_t^\tau(g) - W_s^\tau(g)$.
When $\tau = 1$ we omit the superscript $\tau$ in the previous notations; let us notice that $M_t^\tau(f_\cdot) = M_{\tau t}(f_{\cdot/\tau})$ and $W_t^\tau(g_\cdot) = \tau^{-1} \, W_{\tau t}(g_{\cdot/\tau})$.
\end{definition}
\ss

We first recall two classical $L^2$-bounds for $M_\cdot^\tau(f)$ and $W_\cdot^\tau(g)$.

\begin{lemma} \label{lem:MIL2}
In the notation of Definition~\ref{def:MI}, for every $0 \leq s < t \leq T$, we have
\begin{equs}
\mathbf E_{\nu_\rho}\Big[\sup_{s \leq a \leq b \leq t} |M_{a,b}^\tau(f)|^2\Big] &\lesssim |t-s| \sup_{r \in [0,T]} \|(- \tau S)^{1/2} f_r\|^2 \label{e:ML2} \\
\mathbf E_{\nu_\rho}\Big[\sup_{s \leq a \leq b \leq t} |W^\tau_{a,b}(f)|^2\Big] &\lesssim |t-s| \sup_{r \in [0,T]} \|(|t-s|^{-1} - \tau S)^{-1/2} f_r\|^2 \, . \label{e:IL2}
\end{equs}
\end{lemma}

\begin{proof}
The first estimate~\eqref{e:ML2} follows readily from the BDG inequality and the computation of the predictable quadratic variation $\langle M_\cdot(f) \rangle$ (see~\cite[Section 2.4]{KLO}). The second estimate~\eqref{e:IL2} then follows from the first by the classical time-reversal (Itô) trick (see~\cite[Section 2.5]{KLO}).
\end{proof}

While the $L^2$-estimates of Lemma~\ref{lem:MIL2} will be sufficient to establish the convergence of the semigroup and of the finite-dimensional laws, they are not sufficient for tightness. Indeed, the problem comes from the fact that the right hand sides of~\eqref{e:ML2} and~\eqref{e:IL2} are only of order $O(|t-s|)$, which prevents the use of the Kolmogorov-Chentsov criterion. Thus, we need to consider higher moments of the increments and improve the power of $|t-s|$. The next proposition shows that this is indeed possible, at the price of an exponential cost in the number operator $N$ (which is tightly connected to \emph{hypercontractivity}, see Section~\ref{sec:Hypercontractivity} for more details).

\begin{proposition} \label{p:TightnessBlackBox}
Let $(\alpha(p))_{p > 2}$ be the family of constants from Lemma~\ref{lem:UltimateHypercontractivity} (which, in particular, satisfy $\alpha(p) \to 1$ as $p \to 2$). We use the notations of Definition~\ref{def:MI}, and let $p > 2$. For every $0 \leq s < t \leq T$, we have
\begin{equs}
\label{e:KolmogorovM}
\mathbf E_{\nu_\rho}\Big[\sup_{s \leq a \leq b \leq t} |M_{a,b}^\tau(f)|^p\Big] &\lesssim_p |t-s|^{p/2} \sup_{r \in [0,T]} \|\alpha(p)^N \, (|t-s|^{-1} - \tau S)^{1/2} f_r\|^p \, , \\
\label{e:KolmogorovI}
\mathbf E_{\nu_\rho}\Big[\sup_{s \leq a \leq b \leq t} |W_{a,b}^\tau(g)|^p\Big] &\lesssim_p |t-s|^{p/2} \sup_{r \in [0,T]} \|\alpha(p)^N \, (|t-s|^{-1} - \tau S)^{-1/2} g_r\|^p \, .
\end{equs}
Moreover, for any $\delta > 0$ and $(\cF_{\tau t})_{t \geq 0}$-stopping times $\rm T_1, \rm T_2$ taking values in $[0, T]$ such that $\rm T_1 \leq \rm T_2 \leq \rm T_1 + \delta$, we have
\begin{equ}
\label{e:AldousM}
\mathbf E_{\nu_\rho}\Big[|M_{\rm T_1, \rm T_2}^\tau(f)|^2\Big] \lesssim_p \delta^{1 - \frac{2}{p}} \, T^{\frac{2}{p}} \sup_{r \in [0, T]} \|\alpha(p)^N \, (- \tau S)^{1/2} f_r\|^2 \, .
\end{equ}
\end{proposition}

\begin{remark}
Compared to other similar estimates that can be found in the literature, the bounds~\eqref{e:KolmogorovM} and~\eqref{e:KolmogorovI} have two key advantages: on the one hand, we can consider functionals $f$ and $g$ with infinite chaos decompositions, provided that we control their exponential decay in $N$ (which we do, for example, in the estimates of Proposition~\ref{p:AnsatzA+ASEP}); on the other hand, we do not need any control on quantities such as $\|\nabla f\|_\infty := \sup_{\eta \in \Omega, \, x \sim y \in \Z^2} |f(\eta^{x,y}) - f(\eta)|$ to control the size of the jumps of $M_\cdot^\tau(f)$ (which we do not have, for example, for the corrector $\mathrm{Cor}_t^{\tau, \mu}(\vphi)$).
\end{remark}

\begin{proof}[of Proposition~\ref{p:TightnessBlackBox}]

First, it is easily checked that the statement for general $\tau \geq 1$ can be deduced from the one with $\tau = 1$ by using the scaling relations $M_{s,t}^\tau(f) = M_{\tau s, \tau t}(f_{\cdot/\tau})$ and $W^\tau_{s,t}(g) = \tau^{-1} \, W_{\tau s, \tau t}(g_{\cdot/\tau})$. Consequently, we assume $\tau = 1$.

Second, an easy approximation argument shows that it is sufficient to prove the statement for smooth local paths $f$ and $g$ (that is to say functionals of the form $[0,T] \times \Omega \ni (t, \eta) \mapsto h(t, \restriction{\eta}{\Lambda})$ for some finite subset $\Lambda \subset \Z^2$ and smooth function $h \colon [0,T] \times \{0, 1\}^{\Lambda} \to \R$). This reduction allows us to avoid unnecessary technical difficulties.
\ss

Before delving into the proof, we first state a useful fact about $M_t(f)$. For $q \geq 2$, let us define the $q$-variation process as
\begin{equ} \label{e:pVariationMartingale}
V_t^{(q)}(f) := \sum_{0 < r \leq t} |M_r(f) - M_{r-}(f)|^q = \sum_{0 < r \leq t} |f_r(\eta_r) - f_r(\eta_{r-})|^q \, ,
\end{equ}
which generalises the quadratic variation $[M_t(f)] = V_t^{(2)}(f)$. We denote by $C^{(q)}_t(f)$ the compensator of $V_t^{(q)}(f)$, which generalises the predictable quadratic variation $\langle M(f) \rangle_t = C_t^{(2)}(f)$. Then, it is clear that
\begin{equ} \label{e:PredictablepVariationMartingale}
C_t^{(q)}(f) = \int_0^t \Gamma_q f_r(\eta_r) \, \dd r \, ,
\end{equ}
where the ``puissance $q$ du champ'' $\Gamma_q \psi \colon \Omega \to \R_+$, for any local functional $\psi \colon \Omega \to \R$, is given by
\begin{equ} \label{e:PuissanceDuChamp}
\Gamma_q \psi \, (\eta) \eqdef \sum_{x, y \in \Z^2} r(y-x) \, \eta(x) (1 - \eta(y)) \, \big|\psi(\eta^{x,y}) - \psi(\eta)\big|^q \, .
\end{equ}
In particular, we recover the classical formula of $\langle M(f) \rangle_t$ in terms of the ``carré du champ''.
\ss

We start the proof of~\eqref{e:KolmogorovM}. By the BDG inequality, we know that
\begin{equ} \label{e:KolmogorovBDG}
\mathbf E_{\nu_\rho}\Big[\sup_{s \leq a \leq b \leq t} |M_{a,b}(f)|^p\Big] \lesssim_p \mathbf E_{\nu_\rho}\big[\big([M(f)]_t - [M(f)]_s\big)^{p/2}\big] \, .
\end{equ}
We emphasise that, because $p > 2$ and $M_\cdot(f)$ is not continuous, it is not possible to replace the quadratic variation $[M(f)]$ by its compensator $\langle M(f) \rangle$. However, we can use~\cite[Theorem 2.1]{GeneralisedBDG} and~\eqref{e:PredictablepVariationMartingale} to deduce
\begin{equs}
\mathbf E_{\nu_\rho}\big[\big([M(f)]_t - [M(f)]_s\big)^{p/2}\big] &\lesssim_p \mathbf E_{\nu_\rho}\big[\big(\langle M(f) \rangle_t - \langle M(f) \rangle_s\big)^{p/2}\big] + \mathbf E_{\nu_\rho}[C_t^{(p)} - C_s^{(p)}] \, , \label{e:GeneralisedBDG} \\
&= \mathbf E_{\nu_\rho}\Big[\Big(\int_s^t \Gamma_2 f_r(\eta_r) \, \dd r\Big)^{p/2}\Big] + \mathbf E_{\nu_\rho}\Big[\int_s^t \Gamma_p f_r(\eta_r) \, \dd r\Big] \, ,
\end{equs}
Now, we first use Jensen's inequality and the $\nu_\rho$-stationarity of the ASEP, and then we apply the hypercontractivity estimates~\eqref{e:UltHyperGamma2} and~\eqref{e:UltHyperGammap}, which lead to
\begin{equs}
\, &\mathbf E_{\nu_\rho}\Big[\sup_{s \leq a \leq b \leq t} |M_{a,b}(f)|^p\Big] \lesssim_p |t-s|^{p/2} \sup_{r \in [0,T]} \nu_\rho\big(|\Gamma_2 f_r|^{p/2}\big) + |t-s| \sup_{r \in [0,T]} \nu_\rho\big(\Gamma_p f_r\big) \\
&\lesssim_p |t-s|^{p/2} \sup_{r \in [0,T]} \|\alpha(p)^N \, (-S)^{1/2} f_r\|^p + |t-s| \sup_{r \in [0,T]} \|\alpha(p)^N f_r\|^{p - 2} \|\alpha(p)^N \, (-S)^{1/2} f_r\|^2 \\
&\lesssim |t-s|^{p/2} \, \sup_{r \in [0,T]} \|\alpha(p)^N \, (|t-s|^{-1}-S)^{1/2} f_r\|^p \, ,  \label{e:KolmogorovMproof}
\end{equs}
where in the last line we used the trivial bounds $\|(- S)^{1/2} \cdot \| \leq \|(|t-s|^{-1} - S)^{1/2} \cdot\|$ and $\|\cdot\| \leq |t-s|^{1/2} \, \| (|t-s|^{-1} - S)^{1/2} \cdot\|$. This shows the desired estimate~\eqref{e:KolmogorovM}.
\ss

We now prove~\eqref{e:KolmogorovI}, which follows from~\eqref{e:KolmogorovM} and the Itô trick. More precisely, we first define $h = (|t-s|^{-1} - S)^{-1} \, g$ and we split
\begin{equ} \label{e:KolmogorovIdecomposition}
W_{a,b}(g) = |t-s|^{-1} W_{a,b}(h) + W_{a,b}(- S h) \, .
\end{equ}
To deal with the first summand of~\eqref{e:KolmogorovIdecomposition}, we crudely bound $|W_{a,b}(h)| \leq \int_s^t |h(\eta_r)| \, \dd r$, use Jensen's inequality, the $\nu_\rho$-stationarity of the process, and the hypercontractivity estimate~\eqref{e:UltHyper}, which gives
\begin{equ} \label{e:KolmogorovItrivial}
\mathbf E_{\nu_\rho}\Big[\sup_{s \leq a \leq b \leq t} \big||t-s|^{-1} W_{a,b}(h)\big|^p\Big] \leq \sup_{r \in [0,T]} \nu_\rho(|h_r|^p) \leq \sup_{r \in [0,T]}  \|\alpha(p)^{N} h_r\|^p \, ,
\end{equ}
To deal with the second summand of~\eqref{e:KolmogorovIdecomposition}, we use the Itô trick by writing
\begin{equs}
W_{a,b}(- S h) &= \frac{1}{2} \, \big(M_{a,b}(h) + \tilde M_{a+,b+}(\tilde h)\big) \, , \label{e:ItoTrick} \\
\mbox{where } \tilde M_{a,b}(\tilde h) &\eqdef \tilde h_{T-a}(\tilde \eta_{T-a}) - \tilde h_{T-b}(\tilde \eta_{T-b}) - \int_{T-b}^{T-a} (\partial_r + \tilde L) \, \tilde h_r \, (\tilde \eta_r) \, \dd r \, ,
\end{equs}
and where we introduced the time-reversed path $\tilde h_t := h_{T-t}$, the time-reversed ASEP\footnote{The time-reversal we consider is $\eta_{(T-t)-}$ rather than $\eta_{T-t}$ because we want to work with càdlàg processes.} $\tilde \eta_t := \eta_{(T-t)-}$ and the time-reversed generator $\tilde L := S - A$. Now, we recall the classical fact that the time-reversed ASEP $(\tilde \eta_r)_{r \in [0, T]}$ under $\mathbb P_{\nu_\rho}(\dd \eta_\cdot)$ has the same law as the $\nu_\rho$-stationary skewed ASEP, that is, with jump rates $s(\cdot) - a(\cdot)$ and generator $\tilde L$. Thus, the process $(\tilde M_{a,b}(\tilde h))_{0 < a \leq b < T}$ has the same law as the process $(M_{T-b,T-a}(\tilde h))_{0 < a \leq b < T}$ related to the skewed ASEP. Obviously, the estimate~\eqref{e:KolmogorovM} also holds for the skewed ASEP (indeed, the right hand side of~\eqref{e:KolmogorovM} depends only on the symmetric part of the jump rates, which is unchanged). Thus, we deduce from~\eqref{e:ItoTrick} and~\eqref{e:KolmogorovM} that it holds
\begin{equs}
\mathbf E_{\nu_\rho}\Big[\sup_{s \leq a \leq b \leq t} |W_{a,b}(- S h)|^p\Big] &= \mathbf E_{\nu_\rho}\Big[\sup_{s < a \leq b < t} |W_{a,b}(- S h)|^p\Big] \\
&\lesssim_p \mathbf E_{\nu_\rho}\Big[\sup_{s < a \leq b < t} |M_{a,b}(h)|^p\Big] + \mathbf E_{\nu_\rho}\Big[\sup_{s < a \leq b < t} |\tilde M_{a,b}(\tilde h)|^p\Big] \\
&\lesssim_p |t-s|^{p/2} \sup_{r \in [0,T]} \|\alpha(p)^N \, (|t-s|^{-1}-S)^{1/2} h_r\|^p \, . \label{e:KolmogorovItoTrick}
\end{equs}
Gathering~\eqref{e:KolmogorovIdecomposition},~\eqref{e:KolmogorovItrivial},~\eqref{e:KolmogorovItoTrick}, and recalling that $h_r = (|t-s|^{-1} - S)^{-1} g_r$, we deduce
\begin{equ}
\mathbf E_{\nu_\rho}\Big[\sup_{s \leq a \leq b \leq t} |W_{a,b}(g)|^p\Big] \lesssim_p |t-s|^{p/2} \sup_{r \in [s,t]} \|\alpha(p)^N \, (|t-s|^{-1} - S)^{-1/2} g_r\|^p \, ,
\end{equ}
which shows~\eqref{e:KolmogorovI}.
\ss

We finish with the proof of~\eqref{e:AldousM}. By using classical properties of martingales and stopping times, we deduce
\begin{equ}
\mathbf E_{\nu_\rho}[M_{\rm T_1, \rm T_2}(f)^2] = \mathbf E_{\nu_\rho}\big[\langle M(f) \rangle_{\rm T_2} - \langle M(f) \rangle_{\rm T_1}\big] = \mathbf E_{\nu_\rho}\Big[\int_{\rm T_1}^{\rm T_2} \Gamma_2 f_r (\eta_r) \, \dd r\Big] \, .
\end{equ}
By Hölder's inequality, we can deterministically bound the last integral (recall that $|\rm T_2 - \rm T_1| \leq \delta$) as
\begin{equ}
\int_{\rm T_1}^{\rm T_2} \Gamma_2 f_r (\eta_r) \, \dd r \leq \delta^{1 - \frac{2}{p}} \, \Big(\int_0^T |\Gamma_2 f_r (\eta_r)|^{p/2} \, \dd r\Big)^{2/p} \, .
\end{equ}
Then, by using Jensen's inequality, the $\nu_\rho$-stationarity of the ASEP, and finally using the hypercontractivity estimate~\eqref{e:UltHyperGamma2}, we deduce
\begin{equ}
\mathbf E_{\nu_\rho}[M_{\rm T_1, \rm T_2}(f)^2] \leq \delta^{1 - \frac{2}{p}} \, T^{\frac{2}{p}} \sup_{r \in [0,T]} \nu_\rho(|\Gamma_2 f_r|^{p/2})^{2/p} \, \lesssim_p \delta^{1 - \frac{2}{p}} \, T^{\frac{2}{p}} \sup_{r \in [0, T]} \|\alpha^N \, (- \tau S)^{1/2} f_r\|^2 \, ,
\end{equ}
which shows~\eqref{e:AldousM} and concludes the proof.
\end{proof}

\subsection{The easy part of the Dynkin martingale}

\label{sec:EasyDynkin}

The starting point to investigate the large scale behaviour of the fluctuation process $Y_t^\tau$ is to understand the limiting behaviour as $\tau \to \infty$ of the following Dynkin martingale, for any $\vphi \in \fock{\rap}{}$,
\begin{equ}
M_t^\tau(\fsi_\cdot^\tau \vphi) = \fsi_t^\tau \vphi \, (\eta_{\tau t}) - \fsi_0^\tau \vphi \, (\eta_0) - \int_0^t (\partial_s + \tau L) \, \fsi_s^\tau \vphi \, (\eta_{\tau s}) \, \dd s \, ,
\end{equ}
where $\fsi_t^\tau$ is the rescaling operator defined in~\eqref{e:fsittau}. As sketched in the beginning of Section~\ref{sec:Roadmap}, the most delicate contribution here comes from the action of $A_+$, which can be captured by using the previous fluctuation-dissipation relation~\eqref{e:FormalFluctuationDissipationASEP}. In this section, we collect useful estimates about the remaining part of this Dynkin martingale.
\ss

More precisely, the next lemma shows that in particular $\tau S$ produces the Laplacian term $\frac{1}{2} D_0[\ff, \ff] \, \partial_\ff^2$, while $\tau A_-$ and $\partial_t + \tau A_0$ are negligible (in $N$-weighted $H^{-1}$ norm).

\begin{lemma} \label{lem:SimpleGen}
Let $\tau \geq 1$, $t \geq 0$, $\alpha \geq 1$, $\kappa \geq 0$, $\gamma \in (0,1)$, and $\vphi \in \fock{\rap}{}$. Then it holds
\begin{equs}
\|N_{*}^\kappa \alpha^N \, (1 - \tau S)^{1/2} \, \fsi_t^\tau \vphi\| &\lesssim_{\alpha, \kappa} C_{\alpha, \kappa}^{2}(\vphi) \, , \label{e:SimpleH1} \\
\|N_{*}^\kappa \alpha^N \, (1 - \tau S)^{-1/2} \, (\tau S \, \fsi_t^\tau \vphi - D_0[\ff, \ff] \, \fsi_t^\tau \fL_0^\ff \vphi)\| &\lesssim_{\alpha, \kappa} \log_{*}(\tau)^{-1/3} \, C_{\alpha, \kappa+6}^{2}(\vphi) \, , \label{e:SimpleS} \\
\|N_{*}^\kappa \alpha^N \, (1 - \tau S)^{-1/2} \ \tau A_- \, \fsi_t^\tau \vphi\| &\lesssim_{\alpha, \kappa} \log_{*}(\tau)^{-1/2} \, C_{\alpha, \kappa+6}^{2}(\vphi) \, , \label{e:SimpleA-} \\
\|N_{*}^\kappa \alpha^N \, (1 - \tau S)^{-1/2} \, (\partial_t + \tau A_0) \, \fsi_t^\tau \vphi\| &\lesssim_{\alpha, \kappa, \gamma} \tau^{- \gamma/2} \, C_{\alpha, \kappa+6}^{2}(\vphi) \, , \label{e:SimpleA0}
\end{equs}
where the constant $C_{\alpha, \kappa}^s(\vphi)$ is defined in~\eqref{e:Cvphi}.
\end{lemma}

\begin{proof}
By using the orthogonality of chaoses, it is sufficient to show the statement for every $\vphi \in \fock{\rap,n}{}$ and $n \geq 1$. In addition, we can assume $\alpha = 1$ and $\kappa = 0$ up to dividing both sides by $n_{*}^\kappa \alpha^n$. Next, we recall the useful estimate~\eqref{e:HsUtau} and give another similar statement: for any $a, b \geq 0$ and $\phi \in \fock{\rap}{}$, we have
\begin{equ}
\|(- \fL_0)^{a} (- \fL_0^\fe)^{b} \, \fA_0 \, \fU_t^\tau \phi\| \lesssim \tau^{- a - b - 1/2} \, \log_{*}(\tau)^{-b/3} \, \|\fN \, (- \fL_0)^{a+b+1/2} \phi\| \, \label{e:SimpleHsA0}
\end{equ}
Indeed, it follows from the same proof as~\eqref{e:HsUtau} by noticing in addition the trivial bound
\begin{equ} \label{e:SimpleScalingA0}
|v \cdot (\Sigma^\tau)^{-1} (k_1 + ... + k_n)| = \tau^{-1/2} \log_{*}(\tau)^{-1/3} \, |v \cdot (k_1 + \dots + k_n)| \lesssim \tau^{-1/2} \log_{*}(\tau)^{-1/3} \, n \, |k_{1:n}| \, .
\end{equ}

We start with the proof of~\eqref{e:SimpleH1}. This is a straightforward consequence of~\eqref{e:BoundR},~\eqref{e:BoundP} and~\eqref{e:HsUtau}, as follows
\begin{equ}
\|(1 - \tau S)^{1/2} \, \fsR \fsP \fU_t^\tau \phi\| \lesssim \|(1 - \tau \fL_0)^{1/2} \, \fU^\tau \vphi\| \lesssim \|(- \fL_0)^{1/2} \vphi\| \, .
\end{equ}

We turn to~\eqref{e:SimpleS}. We decompose its left hand side as follows
\begin{equs}
\|(1 - \tau S)^{-1/2} \, (\tau S \, \fsi_t^\tau \vphi -  \fsi_t^\tau D_0[\ff, \ff] \fL_0^\ff \vphi)\| &\leq \|(- \tau S)^{-1/2} \, (\tau S \, \fsR \fsP - \fsQ^* \fsP \, \tau \fS) \fU_t^\tau \vphi\| \\
&\quad + \|(- \tau S)^{-1/2} \, \fsQ^* \fsP (\tau \fS \, \fU_t^\tau \vphi - \fU_t^\tau D_0[\ff, \ff] \fL_0^\ff \vphi)\| \\
&\quad + \|(1 - \tau S)^{-1/2} \, (\fsQ^* - \fsR) \fsP \, \fU_t^\tau D_0[\ff, \ff] \fL_0^\ff \vphi\| \\
&=: \one + \two + \three \, ,
\end{equs}
and we bound individually each summand. For the first one, we use Theorem~\ref{th:ComparisonGen} and~\eqref{e:HsUtau}, which gives
\begin{equ} \label{e:SimpleS1}
\one \lesssim_\gamma n^6 \, \tau^{1/2} \, \|(- \fL_0)^{3/4} \, \fU^\tau \vphi\| \lesssim n^6 \, \tau^{-1/4} \, \|(- \fL_0)^{3/4} \vphi\| \, .
\end{equ}
For the second summand, we use~\eqref{e:BoundQ*},~\eqref{e:BoundP}, and deduce
\begin{equ} \label{e:SimpleS2}
\two \lesssim n^3 \, \|(- \fL_0)^{-1/2}(\tau \fS \, \fU_t^\tau \vphi - \fU_t^\tau D_0[\ff, \ff] \fL_0^\ff \vphi)\| \lesssim n^3 \log_{*}(\tau)^{-1/3} \, \|(-\fL_0)^{1/2} \vphi\| \, ,
\end{equ}
where in the last bound we used the definition of $\fU^\tau$ and the simple bound
\begin{equ}
\big|\tau D_0((\Sigma^\tau)^{-1} k_{1:n}) - D_0[\ff, \ff] \, |\ff \cdot k_{1:n}|^2\big| \lesssim \log_{*}(\tau)^{-1/3} \, |k_{1:n}|^2 \, .
\end{equ}
For the third summand, we use~\eqref{e:RemovalQ*R},~\eqref{e:BoundP} and~\eqref{e:HsUtau}, which results in
\begin{equ} \label{e:SimpleS3}
\three \lesssim n^2 \, \log_{*}(\tau) \, \tau^{-1} \, \|(1 - \tau \fL_0)^{1/2} \, \fU^\tau \fL_0^\ff \vphi\| \lesssim n^2 \log_{*}(\tau) \, \tau^{-1} \, \|(- \fL_0)^{3/2} \vphi\| \, .
\end{equ}
By combining~\eqref{e:SimpleS1},~\eqref{e:SimpleS2} and~\eqref{e:SimpleS3}, we deduce the desired estimate~\eqref{e:SimpleS}.
\ss

We tackle the proof of~\eqref{e:SimpleA-}. We write
\begin{equs}
\|(1 - \tau S)^{-1/2} \, \tau A_- \fsi_t^\tau \vphi\| &\leq \|(- \tau S)^{-1/2} \, (\tau A_- \fsR \fsP - \fsQ^* \fsP \, \tau \fA_-) \, \fU_t^\tau\vphi\| + \|\fsQ^* \fsP \, \tau \fA_- \, \fU_t^\tau\vphi\| \\
&\lesssim_\gamma n^6 \, \tau^{-1/4} \, \|(- \fL_0)^{3/4} \vphi\| + \tau \, \|\fA_- \fU_t^\tau\vphi\| \, , \label{e:SimpleA-1}
\end{equs}
where in the last bound we used Theorem~\ref{th:ComparisonGen} and~\eqref{e:HsUtau} to deal with the first summand (similarly to~\eqref{e:SimpleS1}), and we used~\eqref{e:BoundQ*} and~\eqref{e:BoundP} to deal with the second summand. Now, we estimate further the last term and write
\begin{equs}
\|&\fA_- \fU^\tau_t \vphi\|^2 \lesssim (n-1)! \int_{\R^{2(n-1)}} \dd k_{1:n-1} \Big|n \, \sum_{i=1}^{n-1} (\fe \cdot k_i) \int_{\R^{2}} \dd \ell \, \Pi_{\ell, k_i - \ell} \, \fF(\fU^\tau_t \vphi) \, (\ell, k_i - \ell, k_{1:n-1 \setminus i})\Big|^2 \\
&\lesssim n! \, n^3 \int_{\R^{2(n-1)}} \dd k_{1:n-1} \, |\fe \cdot k_1|^2 \Big|\int_{\R^2} \dd \ell \, \Pi_{\ell, k_1 - \ell} \, \fF(\fU^\tau_t \vphi) \, (\ell, k_1 - \ell, k_{2:n-1})\Big|^2 \\
&\lesssim n! \, n^3 \int_{\R^{2(n-1)}} \dd k_{1:n-1} \dd \ell \, |\fe \cdot k_1|^2 \, (1 + \tau \log_{*}(\tau)^{2/3} |\fe \cdot \ell|^2) \, (1 + \tau |\ff \cdot \ell|^2) \, |\fF(\fU^\tau_t \vphi) \, (\ell, k_1 - \ell, k_{2:n-1})|^2 \\
&\qquad \times \int_{\R^2} \dd \ell \, \frac{1}{(1 + \tau \log_{*}(\tau)^{2/3} |\fe \cdot \ell|^2) \, (1 + \tau |\ff \cdot \ell|^2)} \\
&\lesssim_\nu n! \, n^3 \int \dd p_{1:n} \, |\fe \cdot p_{1:n}|^2 \, (1 + \tau \log_{*}(\tau)^{2/3} |\fe \cdot p_{1:n}|^2) \, (1 + \tau |\ff \cdot p_{1:n}|^2) \, |\fF(\fU^\tau_t \vphi)(p_{1:n})|^2 \\
&\qquad \times \tau^{-1/2} \log_{*}(\tau)^{-1/3} \times \tau^{-1/2} \\
&\lesssim_\nu n^3 \, \tau^{-1} \log_{*}(\tau)^{-2/3} \|(- \fL_0^\fe)^{1/2} (1 - \fL_0^\fe)^{1/2} (1 - \fL_0^\ff)^{1/2} \vphi\|^2 \times \tau^{-1} \log_{*}(\tau)^{-1/3} \\
&\lesssim n^3 \, \tau^{-2} \log_{*}(\tau)^{-1} \, \|(1 - \fL_0)^{3/2} \vphi\|^2 \, , \label{e:SimpleA-2}
\end{equs}
where we proceeded as follows: in the first bound we used the definition of $\fA_-$ as in~\eqref{e:fA-}; in the second bound we used the Cauchy-Schwarz inequality for the sum over $i$, and the symmetry of $f$ with respect to permutations of variables to reduce to $i = 1$; in the third bound, we used the Cauchy-Schwarz inequality for the integral over $\ell$, after biasing by $(1 + \tau \log_{*}(\tau)^{2/3} |\fe \cdot \ell|^2)^{\pm 1} \, (1 + \tau |\ff \cdot \ell|^2)^{\pm 1}$; in the fourth bound, we made the change of variable $p_{1:n} = (\ell, k_1 - \ell, k_{2:n})$ for the first integral, and we bounded the second integral by using the general estimate $\int_{\R} \frac{1}{1 + A x^2} \lesssim A^{-1/2}$ for $A > 0$;
in the fifth bound, we made the change of variable $q_{1:n} = (\Sigma^\tau)^{-1} p_{1:n}$ and used the definition of $\fU_t^\tau$ (cf. Definition~\ref{def:Utau}). Therefore, by plugging~\eqref{e:SimpleA-2} into~\eqref{e:SimpleA-1}, we deduce
\begin{equ}
\|(1 - \tau S)^{-1/2} \, \tau A_- \fsi_t^\tau \vphi\| \lesssim n^6 \log_{*}(\tau)^{-1/2} \, \|(1 - \fL_0)^{3/2} \vphi\| \, ,
\end{equ}
which shows the desired bound~\eqref{e:SimpleA-}.
\ss

Finally, we prove~\eqref{e:SimpleA0}. We first recall that $(\partial_t + \tau \fA_0) \, \fU_t^\tau = 0$ (cf. Remark~\ref{rmk:GalileanTransform}), which gives $\partial_t \, \fsi_t^\tau \vphi = - \fsR \fsP \, \tau \fA_0 \, \fU_t^\tau \vphi$. Hence, we obtain
\begin{equs}
\|(1 - \tau S)^{-1/2} \, (\partial_t + \tau A_0) \, &\fsi_t^\tau \vphi\| \leq \|(- \tau S)^{-1/2} \, (\tau A_0 \fsR \fsP - \fsQ^* \fsP \, \tau \fA_0) \, \fU_t^\tau\vphi\| \\
&\qquad \quad + \|(1 - \tau S)^{-1/2} \, (\fsQ^* - \fsR) \, \fsP \, \tau \fA_0 \, \fU_t^\tau\vphi\| \\
&\lesssim_\gamma n^6 \, \tau^{1/2} \, \|(- \fL_0)^{3/4} \, \fU^\tau \vphi\| + n^2 \log_{*}(\tau) \, \|(1 - \tau \fL_0)^{1/2} \, \fA_0 \fU_t^\tau \vphi\| \, , \\
&\lesssim n^6 \, \tau^{-1/4} \, \|(- \fL_0)^{3/4} \vphi\| + n^2 \log_{*}(\tau) \, \tau^{-1/2} \, \|(1 - \fL_0)^{1/2} (- \fL_0)^{1/2} \vphi\| \, ,
\end{equs}
where in the second bound we used Theorem~\ref{th:ComparisonGen} and~\eqref{e:RemovalQ*R}, and in the last bound we used~\eqref{e:HsUtau} and~\eqref{e:SimpleHsA0}. This shows the desired bound~\eqref{e:SimpleA0} and concludes the proof.
\end{proof}

\subsection{Convergence of the semigroup, at the level of the Fock space}

\label{sec:CVsemigroup}

In this section, we show that, after rescaling and embedding the Fock space $\fock{}{}$ into $\fockdo{}{}$ by using the map $\fsi_\cdot^\tau$ (recall Definition~\ref{def:Utau}), the $\tau$-accelerated ASEP semigroup $P_{\tau t}$ converges to the semigroup $P_t^{\lim}$ of the limiting stochastic heat equation.

\begin{proposition} \label{p:CVsemigroupFock}
Let $\eps \in (0,1/6)$. Then, for every $\tau \geq 1$, $t \geq 0$ and $\vphi \in \fock{\rap}{}$, it holds
\begin{equ} \label{e:QuantitativeCVsemigroup}
\|P_{\tau t} \, \fsi_t^\tau \vphi - \fsi_0^\tau P_t^{\lim} \vphi\| \lesssim_{\eps} \frac{(1 \vee t) \, C_{1, 6}^{2}(\vphi)}{\log_{*}(\tau)^{1/6 - \eps}} \, ,
\end{equ}
where the constant $C_{\alpha, \kappa}^s(\vphi)$ is defined in~\eqref{e:Cvphi}. As a consequence, for every $\vphi \in \fock{}{}$, it holds
\begin{equ} \label{e:CVsemigroupFock}
\sup_{t_0, t \geq 0} \|P_{\tau t} \, \fsi_{t_0 + t}^\tau \vphi - \fsi_{t_0}^\tau P_t^{\lim} \vphi\| \xrightarrow[\tau \to \infty]{} 0 \, .
\end{equ}
\end{proposition}

Intuitively,~\eqref{e:CVsemigroupFock} investigates the evolution of the rescaled fluctuation process $Y_\cdot^\tau$ between times $t_0$ and $t_0 + t$. The reason why it is useful to take into account the starting time $t_0$ is that, although the process $\eta_\cdot$ is stationary under $\mathbf P_{\nu_\rho}$, the fluctuation process $Y_\cdot^\tau$ is not (as soon as $\rho \ne 1/2$) because we are looking at the process on the moving lattice $\Z^2 - v t$.
\ss 

We first give a brief sketch of the proof of~\eqref{e:QuantitativeCVsemigroup}. We start from Dynkin's decomposition
\begin{equs}
\fsi_t^\tau \vphi(\eta_{\tau t}) &= \fsi_0^\tau \vphi (\eta_0) + \int_0^t (\tau S + \partial_r + \tau A_0 + \tau A_-) \, \fsi_r^\tau \vphi \ (\eta_{\tau r}) + \int_0^t \tau A_+ \fsi_r^\tau \vphi \ (\eta_{\tau r}) \, \dd r + M_t^\tau(\fsi_\cdot^\tau \vphi) \\
&\approx \fsi_0^\tau \vphi (\eta_0) + \int_0^t \fsi_r^\tau \, (D_0[\ff, \ff] \, \fL_0^\ff + d^{\lim} \, \fL_0^\fe) \, \vphi \, (\eta_{\tau r}) \, \dd r + M_t^\tau(\mathrm{Cor}_\cdot^{\tau, 1/\tau}(\vphi)) + M_t^\tau(\fsi_\cdot^\tau \vphi) \, ,
\end{equs}
where the latter approximation follows by using Lemma~\ref{lem:SimpleGen} and the fluctuation-dissipation relation~\eqref{e:FormalFluctuationDissipationASEP} (with $\mu = 1/\tau$). Now, by recalling that $\fL^{\lim} = D_0[\ff, \ff] \, \fL_0^\ff + d^{\lim} \, \fL_0^\fe$ (cf. Lemma~\ref{lem:Llim}) and taking the conditional expectation with respect to $\cF_0$ (which in particular eliminates the two martingales), we formally deduce
\begin{equ} \label{e:SketchCVsemigroup}
P_{\tau t} \, \fsi_t^\tau \vphi \approx \fsi_0^\tau \vphi + \int_0^t P_{\tau r} \, \fsi_r^\tau \, \fL^{\lim} \vphi \, \dd r \, .
\end{equ}
This estimate suggests that $P_{\tau s} \, \fsi_s^\tau \approx \fsi_0^\tau P_s^{\lim}$ for every $s \geq 0$, but it is not really convenient to work with directly. In fact, we shall rather start from the Dynkin decomposition of $\fsi_t^\tau P_{T - t}^{\lim} \, \vphi$ for some fixed $T > 0$. Indeed, this choice leads to computations similar to before, but with an additional time integral of $- \fL^{\lim}$ that arises from the action of the time derivative in Dynkin's decomposition on the time-reversed limiting semigroup $P_{T - \, \cdot}^{\lim}$. This additional summand eventually cancels with the time integral appearing in~\eqref{e:SketchCVsemigroup}, and we end up with
\begin{equ}
P_{\tau t} \, P_{T-t}^{\lim} \, \fsi_t^\tau \vphi \approx \fsi_0^\tau \, P_{T}^{\lim} \vphi \, .
\end{equ}
By taking $t = T$, we then deduce the desired convergence of the semigroups $P_T \, \fsi_T^\tau \vphi \approx \fsi_0^\tau P_{T}^{\lim} \vphi$.

\begin{proof}[of Proposition~\ref{p:CVsemigroupFock}]

{\noindent \it Part I. Proof of~\eqref{e:QuantitativeCVsemigroup}}

In this proof, we take $\mu := 1/\tau$ and write $\mathrm{Cor}_t^\tau := \mathrm{Cor}_t^{\tau, 1/\tau}$ for conciseness (recall Definition~\ref{def:AnsatzA+ASEP}). Let $T > 0$ and set $\vphi_t := P_{T - t}^{\lim} \, \vphi$ for any $t \in [0,T]$. We write the following two Dynkin decompositions (recall Definition~\ref{def:MI})
\begin{equs}
\fsi_t^{\tau} \vphi_t \, (\eta_{\tau t}) &= \fsi_0^{\tau} \vphi_0 \, (\eta_0) + W_t^\tau\big([(\partial_\cdot + \tau L) \, \fsi_\cdot^\tau] \, \vphi_\cdot - \fsi_\cdot^\tau \fL^{\lim} \vphi_\cdot\big) + M_t^\tau(\fsi_\cdot^{\tau} \vphi_\cdot) \, , \label{e:CVDynkin1} \\ \mathrm{Cor}_t^\tau \vphi_t \, (\eta_{\tau t}) &= \mathrm{Cor}_0^\tau \vphi_0 \, (\eta_0) + W_t^\tau\big([(\partial_\cdot + \tau L) \, \mathrm{Cor}_\cdot^{\tau}] \, \vphi_\cdot - \mathrm{Cor}_\cdot^{\tau} \fL^{\lim} \vphi_\cdot\big) + M_t^\tau(\mathrm{Cor}_\cdot^\tau \vphi_\cdot) \, . \label{e:CVDynkin2}
\end{equs}
By summing the two formulas and rearranging the terms (in particular, in~\eqref{e:CVDynkin1}, we expand $L$ as $S + A_- + A_0 + A_+$ and $\fL^{\lim}$ as $d^{\lim} \fL_0^\fe + D_0[\ff, \ff] \fL_0^\ff$), we deduce
\begin{equ} \label{e:MildProcessDecomposition}
\fsi_t^{\tau} \vphi_t \, (\eta_{\tau t}) = \fsi_0^{\tau} \vphi_0 \, (\eta_0) + M_t^\tau(\fsi_\cdot^{\tau} \vphi_\cdot) + M_t^\tau(\mathrm{Cor}_\cdot^\tau \vphi_\cdot) + \cR_t^\tau(\vphi) \, ,
\end{equ}
where the remainder term $\cR_t^\tau(\vphi)$ is given by
\begin{equs}
\cR_t^\tau(\vphi) &:= W_t^\tau(\tau S \, \fsi_\cdot^\tau \vphi_\cdot - D_0[\ff, \ff] \, \fsi_\cdot^\tau \fL_0^\ff \vphi_\cdot)
+ W_t^\tau(\tau A_- \fsi_\cdot^{\tau} \vphi_\cdot)
+ W_t^\tau\big([(\partial_{\cdot} + \tau A_0) \, \fsi_\cdot^\tau] \vphi_\cdot\big) \\
&\quad + W_t^\tau(\tau A_+ \fsi_\cdot^\tau \vphi_\cdot + [(\partial_{\cdot} + \tau L) \, \mathrm{Cor}_\cdot^\tau] \, \vphi_\cdot - d^{\lim} \, \fsi_\cdot^\tau \fL_0^\fe \vphi_\cdot\big) \\
&\quad - W_t^\tau(\mathrm{Cor}_\cdot^\tau \fL^{\lim} \vphi_\cdot) - \mathrm{Cor}_t^{\tau} \vphi_t \, (\eta_{\tau t}) + \mathrm{Cor}_0^{\tau} \vphi_0 \, (\eta_{0}) \, . \label{e:CVsemigroupSplitting}
\end{equs}
Now, start from~\eqref{e:MildProcessDecomposition}, take the conditional expectation with respect to $\cF_0$ (which eliminates the two martingales) and choose $t = T$ (note that $\vphi_0 = P_T^{\lim} \vphi$ and $\vphi_T = \vphi$), which leads to
\begin{equ} \label{e:MildSemigroup}
P_{\tau T} \, \fsi_t^\tau \vphi - \fsi_0^\tau P_T^{\lim} \vphi = \mathbf E_{\nu_\rho}[\cR_T^\tau(\vphi)  | \cF_0] \, . 
\end{equ}
Consequently, to prove the desired estimate~\eqref{e:QuantitativeCVsemigroup}, it only remains to show that
\begin{equ} \label{e:CVsemigroupClaim}
\mathbf E_{\nu_\rho}\Big[\big|\cR_T^\tau(\vphi)\big|^2\Big]^{1/2} \lesssim_\eps T_* \, C_{1,10}^2(\vphi) \, \log_{*}(\tau)^{-1/6 + \eps} \, .
\end{equ}

To prove~\eqref{e:CVsemigroupClaim}, we treat separately each summand appearing in~\eqref{e:CVsemigroupSplitting} (that we refer to as $\one$, $\two$, $\three$, $\four$, $\five$, $\six$, $\seven$ respectively). For later use, we give some trivial bounds: from the definition~\eqref{e:Cvphi} of $C_{\alpha, \kappa}^s$ we readily deduce that, for any $t \in [0,T]$,
\begin{equ} \label{e:CvphiSemigroup}
C^s_{\alpha, \kappa}(\vphi_t) \leq C^s_{\alpha, \kappa}(\vphi) \, , \quad C^s_{\alpha, \kappa}( \fL^{\lim} \, \vphi_t) \lesssim C^{s+1}_{\alpha, \kappa}(\vphi) \, ,
\end{equ}
and next, for any $\psi \in \fock{}{}$, we have
\begin{equ} \label{e:TrivialBoundmuH-1}
\|(T^{-1} - \tau S)^{-1/2} \, \psi\| \leq T_*^{1/2} \, \|(1 - \tau S)^{-1/2} \, \psi\| \leq T_*^{1/2} \, \|\psi\| \, .
\end{equ}
To deal with $\one$, we use~\eqref{e:IL2},~\eqref{e:TrivialBoundmuH-1},~\eqref{e:SimpleS},~\eqref{e:CvphiSemigroup} and deduce
\begin{equs}
\mathbf E_{\nu_\rho}\Big[\big|W_T^\tau(\tau S \, \fsi_\cdot^\tau \vphi_\cdot - D_0[\ff, \ff] \, \fsi_\cdot^\tau \fL_0^\ff \vphi_\cdot)\big|^2\Big]^{1/2} &\lesssim T_* \sup_{r \in [0,T]} \|(1 - \tau S)^{-1/2} \, (\tau S \, \fsi_r^\tau - D_0[\ff, \ff] \, \fsi_r^\tau \fL_0^\ff) \, \vphi_r\| \\
&\lesssim T_* \, C_{1, 6}^2(\vphi) \log_{*}(\tau)^{-1/2} \, . \label{e:QuantitativeCVsemigroup1}
\end{equs}
To deal with $\two$, we use~\eqref{e:IL2},~\eqref{e:TrivialBoundmuH-1},~\eqref{e:SimpleA-},~\eqref{e:CvphiSemigroup} and obtain
\begin{equ}
\mathbf E_{\nu_\rho}\Big[\big|W_T^\tau(\tau A_- \fsi_\cdot^{\tau} \vphi_\cdot)\big|^2\Big]^{1/2} \lesssim T_* \sup_{r \in [0,T]} \|(1 - \tau S)^{-1/2} \, \tau A_- \fsi_r^\tau \vphi_r\| \lesssim T_* \, C_{1, 6}^2(\vphi) \log_{*}(\tau)^{-1/2} \, . \label{e:QuantitativeCVsemigroup2}
\end{equ}
To deal with $\three$, we use~\eqref{e:IL2},~\eqref{e:TrivialBoundmuH-1},~\eqref{e:SimpleA0} (with any fixed $\gamma \in (0,1)$) and~\eqref{e:CvphiSemigroup}, which yields
\begin{equ}
\mathbf E_{\nu_\rho}\Big[\big|W_T^\tau\big([(\partial_{\cdot} + \tau A_0) \, \fsi_\cdot^\tau] \vphi_\cdot\big)\big|^2\Big]^{1/2} \lesssim T_* \sup_{r \in [0,T]} \|(1 - \tau S)^{-1/2} \, [(\partial_r + \tau A_0) \, \fsi_r^\tau] \vphi_r\| \lesssim_\gamma T_* \, C_{1, 6}^2(\vphi) \, \tau^{- \gamma/2} \, . \label{e:QuantitativeCVsemigroup3}
\end{equ}
To deal with $\four$, we use~\eqref{e:IL2},~\eqref{e:TrivialBoundmuH-1},~\eqref{e:AnsatzA+GenASEP} (recall that $\mu = 1/\tau$) and~\eqref{e:CvphiSemigroup}, which shows
\begin{equs}
\mathbf E_{\nu_\rho}\Big[\big|W_T^\tau&(\tau A_+ \fsi_\cdot^\tau \vphi_\cdot + [(\partial_{\cdot} + \tau L) \, \mathrm{Cor}_\cdot^\tau] \, \vphi_\cdot - d^{\lim} \, \fsi_\cdot^\tau \fL_0^\fe \vphi_\cdot\big)\big|^2\Big]^{1/2} \\
&\lesssim T_* \sup_{r \in [0,T]} \|(1 - \tau S)^{-1/2} \, (\tau A_+ \fsi_r^\tau \vphi_r + [(\partial_r + \tau L) \, \mathrm{Cor}_r^\tau] \, \vphi_r - d^{\lim} \, \fsi_r^\tau \fL_0^\fe \vphi_r)\| \\
&\lesssim T_* \, C_{1,6}^{3/2}(\vphi) \, \log_{*}(\tau)^{-1/6 + \eps} \, . \label{e:QuantitativeCVsemigroup4}
\end{equs}
To deal with $\five$, we use~\eqref{e:IL2},~\eqref{e:TrivialBoundmuH-1},~\eqref{e:AnsatzA+L2ASEP} (recall that $\mu = 1/\tau$),~\eqref{e:CvphiSemigroup}, and deduce
\begin{equ}
\mathbf E_{\nu_\rho} \Big[\big|W_t^\tau(\mathrm{Cor}_\cdot^\tau \fL^{\lim} \vphi_\cdot)\big|^2\Big]^{1/2} \lesssim T_* \sup_{r_1, r_2 \in [0,T]} \|\mathrm{Cor}_r^{\tau} \, \fL^{\lim} \vphi_r\| \lesssim (1 \vee T) \, C_{1,6}^{3/2}(\vphi) \, \log_{*}(\tau)^{-1/3} \, . \label{e:QuantitativeCVsemigroup5}
\end{equ}
Finally, we finish with the boundary terms $\six$ and $\seven$. We can easily bound them by using the $\nu_\rho$-stationarity,~\eqref{e:AnsatzA+L2ASEP} (recall that $\mu = 1/\tau$) and~\eqref{e:CvphiSemigroup}, which leads to
\begin{equ}
\sup_{t \in [0,T]} \mathbf E_{\nu_\rho}\Big[\big|\mathrm{Cor}_t^{\tau} \vphi_t \, (\eta_{\tau t})\big|^2\Big]^{1/2} \lesssim \sup_{t \in [0,T]} \|\mathrm{Cor}_t^{\tau} \vphi_r\| \lesssim C_{1,0}^{1/2}(\vphi) \, \log_{*}(\tau)^{-1/3} \, . \label{e:QuantitativeCVsemigroup6}
\end{equ}
By gathering~\eqref{e:QuantitativeCVsemigroup1},~\eqref{e:QuantitativeCVsemigroup2},~\eqref{e:QuantitativeCVsemigroup3},~\eqref{e:QuantitativeCVsemigroup4},~\eqref{e:QuantitativeCVsemigroup5},~\eqref{e:QuantitativeCVsemigroup6}, we deduce~\eqref{e:CVsemigroupClaim} and conclude the proof of~\eqref{e:QuantitativeCVsemigroup}.
\ms 

{\noindent \it Part II. Proof of~\eqref{e:CVsemigroupFock}}

First, by recalling Definition~\ref{def:Utau} and noticing that the translation operator $\fT_\cdot$ commutes with $P_t^{\lim}$, we have
\begin{equ}
\fsi^\tau_{t_0 + t} = \fsi^\tau_t \, \fT_{\tau t_0 v} \, , \quad \fsi^\tau_{t_0} P_t^{\lim} = \fsi^\tau_{0} \, P_t^{\lim} \, \fT_{\tau t_0 v} \, .
\end{equ}
Hence, by applying~\eqref{e:QuantitativeCVsemigroup} to $\fT_{\tau t_0 v} \, \vphi$ (and noticing that $C_{\alpha, \kappa}^s(\fT_{\tau t_0 v} \vphi) = C_{\alpha, \kappa}^s(\vphi)$) we deduce
\begin{equ}
\|P_{\tau t} \fsi_{t_0 + t}^\tau \vphi - \fsi_{t_0}^\tau P_t^{\lim} \vphi\| \lesssim_\eps \frac{(1 \vee t) \, C_{1,6}(\vphi)}{\log_{*}(\tau)^{1/6-\eps}} \, .
\end{equ}
In particular, for any $\vphi \in \fock{\rap}{}$ and $T > 0$, we have
\begin{equ} \label{e:CVsemigroupLocal}
\sup_{t_0 \geq 0, \, t \in [0,T]} \|P_{\tau t} \fsi_{t_0 + t}^\tau \vphi - \fsi_{t_0}^\tau P_t^{\lim} \vphi\| \xrightarrow[\tau \to \infty]{} 0 \, .
\end{equ}
Then, a straightforward density argument allows us to extend~\eqref{e:CVsemigroupLocal} to any $\vphi \in \fock{}{}$, because the operators $P_{\tau t} \fsi_{t_0 + t}^\tau$ and $\fsi_{t_0}^\tau P_t^{\lim}$ are uniformly bounded by $1$ from $\fock{}{}$ to $\fockdo{}{}$.
\ss

Hence, it only remains to extend~\eqref{e:CVsemigroupLocal} to $T = \infty$. First, note that we can assume $\E[\vphi] = 0$ up to replacing $\vphi$ by $\vphi - \E[\vphi]$. Next, by using the contractivity of $P_\cdot$, $P_\cdot^{\lim}$ and $\fsi^\tau_\cdot$, we deduce, for any $t_0 \geq 0$ and $t \geq T \geq 0$ (let us denote $t_1 := t_0 + t - T \geq 0$),
\begin{equs}
\|P_{\tau t} \fsi_{t_0 + t}^\tau \vphi - \fsi_{t_0}^\tau P_t^{\lim} \vphi\| &\leq \|P_{\tau t} \fsi_{t_0 + t}^\tau \vphi\| + \|\fsi_{t_0}^\tau P_t^{\lim} \vphi\| \\
&\leq \|P_{\tau T} \fsi_{t_0 + t}^\tau \vphi\| + \|P_T^{\lim} \vphi\| \\
&\leq \|P_{\tau T} \fsi_{t_1 + T}^\tau \vphi - \fsi_{t_1}^\tau P_T^{\lim} \vphi\| + 2 \|P_T^{\lim} \vphi\| \, ,
\end{equs}
which gives, for any $T \geq 0$,
\begin{equ} \label{e:ExtensionTinfty}
\sup_{t_0 \geq 0, \, t \geq 0} \|P_{\tau t} \fsi_{t_0 + t}^\tau \vphi - \fsi_{t_0}^\tau P_t^{\lim} \vphi\| \leq 2 \, \|P_T^{\lim} \vphi\| + \sup_{t_0 \geq 0, \, t \in [0,T]} \|P_{\tau t} \fsi_{t_0 + t}^\tau \vphi - \fsi_{t_0}^\tau P_t^{\lim} \vphi\| \, .
\end{equ}
Now, we notice that $\|P_T^{\lim} \vphi\| \to 0$ as $T \to \infty$ because $\E[\vphi] = 0$ (which can be easily checked by a direct computation using the chaos decomposition). Consequently, by taking in~\eqref{e:ExtensionTinfty} the limsup first as $\tau \to \infty$ (which eliminates the second summand of the right hand side, by~\eqref{e:CVsemigroupLocal}) and then as $T \to \infty$ (to eliminate the first summand of the right hand side, by the previous observation), we deduce~\eqref{e:CVsemigroupFock} and conclude the proof.
\end{proof}

\subsection{Convergence of the finite dimensional laws}

\label{sec:FinDim}

In this section we show the convergence of the finite dimensional laws of the fluctuation process, as follows.

\begin{proposition} \label{p:CVfindim}
For any $n \geq 1$, $0 \leq t_1 \leq ... \leq t_n$, $\vphi_1, ..., \vphi_n \in \cS(\R^2)$, $f \colon \R^n \to \R$ continuous with at most polynomial growth, $p \in [1,\infty)$, it holds
\begin{equs}
\mathbf E_{\nu_\rho}&[f(Y^\tau_{t_1}(\vphi_1), \dots, Y^\tau_{t_n}(\vphi_n))] \xrightarrow[\tau \to \infty]{} \mathbf E_{\mathbb P}^{\lim}[f(Y^{\lim}_{t_1}(\vphi_1), \dots, Y^{\lim}_{t_n}(\vphi_n))] \, , \label{e:CVfindimAnnealed} \\
\int_{\Omega} \nu_\rho(\dd \eta) \, \big|\mathbf E_{\eta}&[f(Y^\tau_{t_1}(\vphi_1), \dots, Y^\tau_{t_n}(\vphi_n))] - \mathbf E_{Z_0^\tau(\eta)}^{\lim}[f(Y^{\lim}_{t_1}(\vphi_1), \dots, Y^{\lim}_{t_n}(\vphi_n))]\big|^p \xrightarrow[\tau \to \infty]{} 0 \, . \quad\label{e:CVfindimQuenched}
\end{equs}
\end{proposition}

\begin{remark}

It would be possible to strengthen Proposition~\ref{p:CVfindim} in two ways. First, the convergences~\eqref{e:CVfindimAnnealed} and~\eqref{e:CVfindimQuenched} hold uniformly over the sampling times $t_1, \dots, t_n \geq 0$. Second, if we assume $f$ to be smooth with at most polynomial growth, then the errors with respect to the convergences~\eqref{e:CVfindimAnnealed} and~\eqref{e:CVfindimQuenched} can be bounded by $O(\log_{*}(\tau)^{-1/6 + o(1)})$. Indeed, these improvements could be obtained, respectively, from the uniformity in time in the convergence~\eqref{e:CVsemigroupFock} and from the quantitative estimate~\eqref{e:QuantitativeCVsemigroup}. However, since implementing these refinements would require additional work, we prefer to stick to the simpler statement of Proposition~\ref{p:CVfindim}.
\end{remark}

Before diving further into the proof of Proposition~\ref{p:CVfindim}, we give some useful definitions and basic properties. First, we mention the following elementary concentration bounds (see Lemma~\ref{lem:Subgaussian} for a stronger statement): for any $\vphi \in \cS(\R^2)$ and $p \geq 1$,
\begin{equ} \label{e:Concentration}
\sup_{\tau \geq 1, \ t_0, t \geq 0} \mathbb E[|\zeta(\vphi)|^p] \vee \nu_\rho[|Z_{t}^\tau(\vphi)|^p] \vee \nu_\rho[\mathbf E_{Z_{t_0}^\tau}^{\lim}[|Y_t^{\lim}(\vphi)|^p]] < \infty \, .
\end{equ}
Next, in this section we work mainly with (spatial) cylinder functionals that we now define.

\begin{definition} \label{def:CylinderFunctional}
A \emph{cylinder functional} is a function $F \colon \cS'(\R^2) \to \R$ such that there exist an integer $m \geq 1$, a continuous function $f \in \cC(\R^m, \R)$ that grows at most polynomially fast, and test functions $\vphi_1, \dots, \vphi_m \in \cS(\R^2)$, such that for every $w \in \cS'(\R^2)$ it holds
\begin{equ} \label{e:CylinderFunctional}
F(w) = f(w(\vphi_1), \dots, w(\vphi_m)) \, .
\end{equ}
We denote by $\mathrm{Cyl}$ the space of cylinder functionals. We can define the semigroup $(P_t^{\lim})_{t \geq 0}$ on $\mathrm{Cyl}$, simply by using the formula $(P_t^{\lim} F)(w) \eqdef \mathbf E^{\lim}_w[F(Y_t^{\lim})]$. Indeed, the fact that $P_t^{\lim} F \in \mathrm{Cyl}$ can easily be checked from the following explicit representation of the process $(Y_t^{\lim})_{t \geq 0}$
\begin{equ} \label{e:YlimExplicit}
Y_t^{\lim}(\vphi) = Y_0^{\lim}(K^{\lim}_t \vphi) + \int_0^t \vec{\xi}(\dd s) \, \big(K^{\lim}_{t-s} \,  \sqrt{D^{\lim}} \, \nabla \vphi\big) \, ,
\end{equ}
where we denoted by $K^{\lim}_r \eqdef \exp(r \, \nabla \cdot D^{\lim} \nabla / 2)$ the tweaked heat semigroup.
\end{definition}

To prove Proposition~\ref{p:CVfindim}, we need to convert the convergence of semigroups at the level of the Fock space, as stated in Proposition~\ref{p:CVsemigroupFock}, to a convergence of semigroups at the level of the fluctuation process. In this regard, the key step is to replace the bridge operator $\fsi_t^\tau \colon L^2(\mathbb P) \to L^2(\nu_\rho)$ by the evaluation on $Y_t^\tau$, in the following sense: for any cylinder functional $F$, we should have
\begin{equ} \label{e:CVChaosHeuristic}
F(Z_t^\tau) \approx \fsi_t^\tau (F(\zeta)) \, .
\end{equ}
The next lemma, whose proof is postponed to Appendix~\ref{app:Cyl}, justifies this approximation.

\begin{lemma} \label{lem:CVchaos}
For every $F \in \mathrm{Cyl}$ and $t \geq 0$, it holds
\begin{equ} \label{e:CVchaos}
\|F(Z_t^\tau) - \fsi_t^\tau (F(\zeta))\|_{L^2(\nu_\rho)} \xrightarrow[\tau \to \infty]{} 0 \, .
\end{equ}
\end{lemma}

\begin{remark}
Intuitively, this lemma tells us that the Bernoulli and Gaussian chaos decompositions are consistent at large space scales. To better see why, we focus on the following example: let $F(w) = H_n(w(h))$ (where $H_n(X) := H_n(X, 1)$ is the $n$-th Hermite polynomial as in~\eqref{e:Hermite}) for some $n \geq 0$ and $h \in \cS(\R^2)$ such that $\|h\|_{L^2(\R^2)} = 1$. Note that, by definition of the isometry $\fI \colon \fock{}{} \to L^2(\mathbb P)$, we have $F(\zeta) = \fI(h^{\otimes n})$. Writing $h_t^\tau := \fU_t^\tau h$ and granting the elementary approximation $\fsP h_t^\tau \approx \restriction{h_t^\tau}{\Z^2}$, we then deduce
\begin{equ} \label{e:ConsistencyExample}
\fsi_t^\tau (F(\zeta)) = I \, \fsR \fsP \, (h_t^\tau)^{\otimes n} \approx I (\restriction{(h_t^\tau)^{\otimes n}}{\Z^{2n} \cap HC}) = \sum_{x_{1:n} \in \Z^{2n} \cap HC} h_t^\tau(x_1) \dots h_t^\tau(x_n) \, \sigma(x_1) \dots \sigma(x_n) \, .
\end{equ}
If we omit the hardcore condition in the last sum of~\eqref{e:ConsistencyExample}, then we obtain $Z_t^\tau(h)^n$. But the missing $x_{1:n} \in \Z^{2n} \cap DI$ cannot be neglected: more precisely, they result in terms of smaller degree (depending on the ``collision diagram'' of $x_{1:n}$), which eventually leads to a linear combination of smaller powers of $Z_t^\tau(h)$. What Lemma~\ref{lem:CVchaos} really tells us is that the combinatorics behind the ``collision diagrams'' matches, in the limit $\tau \to \infty$, the coefficients $c_{k,n}$ of the Hermite polynomial (recall~\eqref{e:Hermite}), so that~\eqref{e:ConsistencyExample} eventually approximates $H_n(Z_t^\tau(h)) = F(Z_t^\tau(h))$.
\end{remark}
\ss

Now, from Proposition~\ref{p:CVsemigroupFock} and Lemma~\ref{lem:CVchaos}, we can easily deduce the following convergence of semigroups, at the level of the fluctuation process.

\begin{lemma} \label{lem:CVsemigroupCyl}
For every $F \in \mathrm{Cyl}$, $t_0, t \geq 0$ and $p \in [1,\infty)$, it holds
\begin{equ} \label{e:CVsemigroupCyl}
\|P_{\tau t} (F(Z^{\tau}_{t_0 + t})) - (P_t^{\lim} F) (Z_{t_0}^\tau)\|_{L^p(\nu_\rho)} \xrightarrow[\tau \to \infty]{} 0 \, .
\end{equ}
\end{lemma}

\begin{proof}[of Lemma~\ref{lem:CVsemigroupCyl}]

By using the triangle inequality, we obtain
\begin{equs}
\|P_{\tau t}(F(Z^{\tau}_{t_0 + t})) - (P_t^{\lim} F)(Z_{t_0}^\tau)\|_{L^2(\nu_\rho)} &\leq \|P_{\tau t} \fsi_{t_0 + t}^\tau (F(\zeta)) - \fsi_{t_0}^\tau P_t^{\lim} (F(\zeta))\|_{L^2(\nu_\rho)} \\
&\quad + \|P_{\tau t} [F(Z^{\tau}_{t_0 + t}) - \fsi_{t_0 + t}^\tau (F(\zeta))]\|_{L^2(\nu_\rho)} \\
&\quad + \|(P_t^{\lim} F) (Z^\tau_{t_0}) - \fsi_{t_0}^\tau \, ((P_t^{\lim} F)(\zeta))\|_{L^2(\nu_\rho)} \\
&\xrightarrow[\tau \to \infty]{} 0 \, ,
\end{equs}
where we used Proposition~\ref{p:CVsemigroupFock} to deal with the first summand and Lemma~\ref{lem:CVchaos} (together with the contractivity of $P_{\tau t}$) to deal with the two last. This shows the desired statement~\eqref{e:CVsemigroupCyl} for $p = 2$. But note that, for any $q \geq 1$,
\begin{equ}
\|P_{\tau t} (F(Z^{\tau}_{t_0 + t})) - (P_t^{\lim} F) (Z_{t_0}^\tau)\|_{L^q(\nu_\rho)} \leq \nu_\rho[|F(Z^{\tau}_{t_0 + t})|^q]^{1/q} + \mathbf E_{\nu_\rho}[\mathbf E^{\lim}_{Y_{t_0}^\tau}[|F(Y_{t}^{\lim})|^q]]^{1/q} \, ,
\end{equ}
which is uniformly bounded in $\tau$, because of~\eqref{e:Concentration} and the at most polynomial growth of cylinder functionals. Therefore, by using Hölder's inequality, we can extend the statement~\eqref{e:CVsemigroupCyl} from $p = 2$ to any $p \geq 1$, which concludes the proof.
\end{proof}

Finally, by iterating Lemma~\ref{lem:CVsemigroupCyl} and the Markov property, we can easily deduce the convergence of finite-dimensional laws as stated in Proposition~\ref{p:CVfindim}.

\begin{proof}[of Proposition~\ref{p:CVfindim}]

First, from the concentration bounds~\eqref{e:Concentration} (and the $\nu_\rho$-stationarity of the ASEP, the $\mathbb P$-stationarity of $Y_\cdot^{\lim}$), we immediately deduce that, for any $\vphi \in \cS(\R^2)$, $t \geq 0$, and $q \geq 1$,
\begin{equ}
\sup_{\tau \geq 1} \mathbf E_{\nu_\rho}[|Y_t^\tau(\vphi)|^q] \vee \mathbf E_{\nu_\rho}[\mathbf E_{Y_0^\tau}^{\lim}[|Y_t^{\lim}(\vphi)|^q]] \vee \mathbf E_{\mathbb P}^{\lim}[|Y_t^{\lim}(\vphi)|^q] < \infty \, .
\end{equ}
Consequently, an easy density/approximation argument shows that it is enough to prove Proposition~\ref{p:CVfindim} for a subset of functions $f$ whose span is dense in $\cC_0(\R^n, \R)$. Hence, we can assume without loss of generality that $f$ is of the form $f(y_1, \dots, y_n) = f_1(y_1) \dots f_n(y_n)$ with functions $f_i \in \cC_0(\R, \R)$.
\ss

We denote $F_i = f_i(\cdot(\vphi_i)) \in \mathrm{Cyl}$, and let $t_0 := 0$. Let us define iteratively, for $k \in \{0, \dots, n\}$, the functionals $G_k^\tau \colon \Omega \to \R$ and $G_k^{\lim} \in \mathrm{Cyl}$ as follows: $G_n^\tau = G_n^{\lim} = 1$ and for $k \in \{0, \dots, n-1\}$
\begin{equ} \label{e:MarkovIteration}
G_k^\tau := P_{\tau (t_{k+1} - t_k)} \big(F_{k+1}(Z_{t_{k+1}}^\tau) \times G_{k+1}^\tau\big) \, , \quad G_k^{\lim} := P^{\lim}_{t_{k+1} - t_k} \big(F_{k+1} \times G_{k+1}^{\lim}\big) \, .
\end{equ}
Note that these functionals are uniformly bounded by $C := (1 \vee \|f_1\|_\infty) \dots (1 \vee \|f_n\|_\infty)$. The interest of this definition is that, by the Markov property, we have
\begin{equ} \label{e:MarkovCascade}
G_0^\tau(\eta) = \mathbf E_{\eta}[F_1(Y^\tau_{t_1}) \dots F_n(Y^{\tau}_{t_n})] \, , \quad G_0^{\lim}(Z_0^\tau) = E_{Z_0^\tau}^{\lim}[F_1(Y^{\lim}_{t_1}) \dots F_n(Y^{\lim}_{t_n})] \, .
\end{equ}
Let us show, by (backward) induction on the index $k$, that $G_k^\tau - G_k^{\lim}(Z_{t_k}^\tau) \to 0$ in $L^p(\nu_\rho)$ for every $p \geq 1$. For the index $n$, this is trivial since $G_n^\tau \equiv 1 \equiv G_n^{\lim}$. Now, assume that the property holds for index $k+1$ and let us prove it for index $k$. Observe that
\begin{equs}
G_k^\tau - G_k^{\lim}(Z_k^\tau) &= P_{\tau (t_{k+1} - t_k)} \, \big((F_{k+1} \times G_{k+1}^{\lim})(Z_{t_{k+1}}^\tau)\big) - \big(P^{\lim}_{t_{k+1} - t_k} (F_{k+1} \times G_{k+1}^{\lim})\big) (Z_{t_k}^\tau) \\
&\quad + P_{\tau (t_{k+1} - t_k)} \, \big(F_{k+1}(Z_{t_{k+1}}^\tau) \times (G_{k+1}^\tau - G_{k+1}^{\lim}(Z_{t_{k+1}}^\tau))\big) \, .
\end{equs}
Notice that the two lines on the right hand side converge to $0$ in $L^p(\nu_\rho)$ as $\tau \to \infty$: for the first line this follows from Lemma~\ref{lem:CVsemigroupCyl} (applied to the cylinder functional $F_{k+1} \times G_{k+1}^{\lim}$), and for the second line this follows from the $L^p(\nu_\rho)$-contractivity of the ASEP semigroup, the bound $\|F_{k+1}\|_\infty \leq C$ and the induction hypothesis $\|G_{k+1}^\tau - G_{k+1}^{\lim}(Z_{t_{k+1}}^\tau)\|_{L^p(\nu_\rho)} \to 0$. This concludes the proof of the induction: we deduce that $\|G_0^\tau - G_0^{\lim}(Z_0^\tau)\|_{L^p(\nu_\rho)} \to 0$, which, in view of~\eqref{e:MarkovCascade}, is exactly the desired statement~\eqref{e:CVfindimQuenched}.
\ss

Finally, we prove~\eqref{e:CVfindimAnnealed}. By using the trivial bound
\begin{equs}
    \Big|\mathbf E_{\nu_\rho}&[f(Y^\tau_{t_1}(\vphi_1), \dots, Y^\tau_{t_n}(\vphi_n))] - \int_{\Omega} \nu_\rho(\dd \eta) \, \mathbf E_{Z_0^\tau(\eta)}^{\lim}[f(Y^{\lim}_{t_1}(\vphi_1), \dots, Y^{\lim}_{t_n}(\vphi_n))]\Big| \\
    &\leq \int_{\Omega} \nu_\rho(\dd \eta) \, \Big|\mathbf E_{\eta}[f(Y^\tau_{t_1}(\vphi_1), \dots, Y^\tau_{t_n}(\vphi_n))] - \mathbf E_{Z_0^\tau(\eta)}^{\lim}[f(Y^{\lim}_{t_1}(\vphi_1), \dots, Y^{\lim}_{t_n}(\vphi_n))]\Big| \, ,
\end{equs}
and since we have already proved~\eqref{e:CVfindimQuenched}, we see that it only remains to show that, as $\tau \to \infty$, the process $Y_\cdot^{\lim}$ started from $Z_0^\tau$ (with $\eta \sim \nu_\rho$) converges in law to the stationary process $\mathbf P_{\mathbb P}^{\lim}$ in $C(\R_+, \cS'(\R^2))$. But this is a straightforward consequence of the convergence in law of $Z_0^\tau$ to $\mathbb P$, and the explicit representation~\eqref{e:YlimExplicit} for the process $Y^{\lim}_\cdot$. This concludes the proof.
\end{proof}

\subsection{Tightness}

\label{sec:Tightness}

In this section, we show the tightness of the fluctuation process.

\begin{proposition} \label{p:Tightness}
The laws of $\{Y_\cdot^\tau\}_{\tau \geq 1}$ under $\mathbf P_{\nu_\rho}$ are tight in $D(\R_+, \cS'(\R^2))$.
\end{proposition}

\begin{proof}
By Mitoma's criterion~\cite{Mitoma}, it is sufficient to show that, for every $\vphi \in \cS(\R^2)$, the laws of $\{Y_\cdot^\tau(\vphi)\}_{\tau \geq 1}$ under $\mathbf P_{\nu_\rho}$ are tight in $D(\R_+, \R)$. Recall that we have $Y_t^\tau(\vphi) = \restriction{(\fU_t^\tau \vphi)}{\Z^2} (\eta_{\tau t})$. Then, we write the Dynkin decomposition (recall Definition~\ref{def:MI})
\begin{equ}
Y_t^\tau(\vphi) = Y_0^\tau(\vphi) + W_t^\tau\big((\partial_\cdot + \tau L) \restriction{(\fU_\cdot^\tau \vphi)}{\Z^2}\big) + M_t^\tau\big(\restriction{(\fU_\cdot^\tau \vphi)}{\Z^2}\big) \, .
\end{equ}
The first summand is trivially tight (it is constant in time and converges in law to $\cN(0, \|\vphi\|_{L^2(\R^2)}^2)$). As for the two last summands, we show that they satisfy, respectively, the Kolmogorov-Chentsov\footnote{We use the Kolmogorov-Chentsov criterion rather than Aldous' criterion primarily because it provides much better control over the delicate boundary terms produced by the fluctuation-dissipation relation.} and Aldous' tightness criteria (see~\cite[Th. 16.10]{Bil99} for the latter). More precisely, we prove that there exists $p > 2$ such that for any $\tau \geq 1$, $0 \leq s \leq t \leq T$, $\delta > 0$, ${\rm T}_1, {\rm T}_2$ two $(\cF_{\tau t})_{t \geq 0}$-stopping times taking values in $[0,T]$ such that ${\rm T}_1 \leq {\rm T}_2 \leq {\rm T}_1 + \delta$, it holds
\begin{equs}
\mathbf E_{\nu_\rho} \Big[\big|W_{s,t}^\tau\big((\partial_\cdot + \tau L) \restriction{(\fU_\cdot^\tau \vphi)}{\Z^2}\big)\big|^p\Big] &\lesssim_{T, \vphi} |t-s|^{p/2} \, , \label{e:TightI} \\
\mathbf E_{\nu_\rho} \Big[\big|M_{{\rm T}_1, {\rm T}_2}^\tau \big(\restriction{(\fU_\cdot^\tau \vphi)}{\Z^2}\big)\big|^2\Big] &\lesssim_{T, \vphi} \delta^{1 - 2/p} \, . \label{e:TightM}
\end{equs}
Consequently, we deduce that the laws of $\{W_\cdot^\tau\big((\partial_\cdot + \tau L) \restriction{(\fU_\cdot^\tau \vphi)}{\Z^2}\big)\}_{\tau \geq 1}$ and $\{M_t^\tau\big(\restriction{(\fU_\cdot^\tau \vphi)}{\Z^2}\big)\}_{\tau \geq 1}$ are tight in $C(\R_+, \R)$ and $D(\R_+, \R)$ respectively, which implies the desired statement.
\ss

We divide the proof of~\eqref{e:TightI} and~\eqref{e:TightM} into three steps. First, we show that we can replace $\restriction{(\fU_\cdot^\tau \vphi)}{\Z^2}$ by $\fsi_\cdot^\tau \vphi$. Then, we deal with the diffusive contributions $W^\tau((\partial_\cdot + \tau L - \tau A_+) \fsi_\cdot^\tau \vphi)$ and $M^\tau(\fsi_\cdot^\tau \vphi)$. Finally, we tackle the superdiffusive contribution $W_\cdot^\tau(\tau A_+ \fsi_\cdot^\tau \vphi)$. For the upcoming estimates, it is convenient to fix any $\eps \in (0,1/6)$, $\alpha \in (1, \sqrt{1 + 6 \eps})$ and $p > 2$ such that the constant $\alpha(p)$ coming from Lemma~\ref{lem:UltimateHypercontractivity} is smaller than $\alpha$. We do not track below the dependence of the hidden constants in $\lesssim$ on $\eps$, $\alpha$ and $p$.
\ss

{\noindent \it Step 1. The replacement with the bridge $\fsi_\cdot^\tau$}
\ss

Let us write $\mathrm{Err}_r^\tau(\psi) := \restriction{(\fU_r^\tau \psi)}{\Z^2} - \fsi_r^\tau \psi$ for any $\psi \in \cS(\R^2)$. We show that, for any $m \geq 0$, we have
\begin{equs}
\mathbf E_{\nu_\rho} \Big[\big|W_{s,t}^\tau\big((\partial_\cdot + \tau L) \mathrm{Err}_\cdot^\tau(\vphi)\big)\big|^p\Big] &\lesssim_{T, \vphi, m} \tau^{-m} \, |t-s|^{p/2} \, , \label{e:TightIRepl} \\
\mathbf E_{\nu_\rho} \Big[\big|M_{{\rm T}_1, {\rm T}_2}^\tau \big(\mathrm{Err}_r^\tau(\vphi)\big)\big|^2\Big] &\lesssim_{T, \vphi, m} \tau^{-m} \, \delta^{1 - 2/p} \, , \label{e:TightMRepl}
\end{equs}
which is much stronger than needed. In fact, these estimates are easy consequences of Proposition~\ref{p:TightnessBlackBox} and the following claim: for any $\psi \in \cS(\R^2)$ and $m \geq 2$, we have
\begin{equ} \label{e:ReplacementRP}
\|\mathrm{Err}_r^\tau(\psi)\| \, \lesssim \tau^{-m} \, \|(- \fL_0)^{m} \, \psi\| \, .
\end{equ}

We start by showing~\eqref{e:ReplacementRP}. By using the Poisson summation formula~\eqref{e:PoissonSummation} and the definition of $\fsi_r^\tau$ (and noticing that the hardcore and softcore first chaoses coincide: $\fockd{1}{} = \fockdo{1}{}$), we obtain for any $k \in \square:=(- \pi, \pi)^2$
\begin{equ}
\cF(\mathrm{Err}_r^\tau(\psi)) \, (k) = \sum_{\ell \in 2 \pi \Z^2 \setminus \{0\}} \fF(\fU_r^\tau \psi) \, (k + \ell) \, .
\end{equ}
In particular, we deduce
\begin{equ}
\|\mathrm{Err}_r^\tau(\psi)\| \leq \sum_{\ell \in 2 \pi \Z^2 \setminus \{0\}} \|\fF(\fU_r^\tau \psi)\|_{L^2(\ell + \square)} \lesssim_m \sum_{\ell \in 2 \pi \Z^2 \setminus \{0\}} |\ell|^{-2m} \, \|(- \fL_0)^m \, \fU_r^\tau \psi\| \lesssim_m \|(- \fL_0)^m \, \fU_r^\tau \psi\| \, .
\end{equ}
Then, by using the definition of $\fU_r^\tau$ and the bound $|(\Sigma^\tau)^{-1} \, k|^2 \leq \tau^{-1} \, |k|^2$, we obtain $\|(- \fL_0)^m \, \fU_r^\tau \psi\| \leq \tau^{-m} \|(- \fL_0)^m \, \psi\|$ and deduce the claim~\eqref{e:ReplacementRP}.
\ss

Now, note that we deduce from~\eqref{e:ReplacementRP} the crude bounds
\begin{equ} \label{e:ReplacementRPCrude}
    \|(-\tau S)^{1/2} \, \mathrm{Err}_r^\tau(\psi)\| \lesssim_{m, \psi} \tau^{1/2 - m} \, , \quad \|\tau L \, \mathrm{Err}_r^\tau(\psi)\| \lesssim_{m, \psi} \tau^{1 - m} \, , \quad \|\partial_r \, \mathrm{Err}_r^\tau(\psi)\| \lesssim_{m, \psi} \tau^{1/2-m} \, ,
\end{equ}
where we used that $S$, $L$ are bounded from $\fockdo{1}{}$ to $\fockdo{}{}$, and we noticed that Definition~\ref{def:Utau} gives $\partial_r \mathrm{Err}_r^\tau(\psi) = - \tau^{1/2} \log_*(\tau)^{-1/3} \, \mathrm{Err}_r^\tau(v \cdot \nabla \psi)$.
\ss

Finally, we prove~\eqref{e:TightIRepl} and~\eqref{e:TightMRepl}. On the one hand, by using~\eqref{e:KolmogorovI} (with the trivial bound $\|(|t-s|^{-1} - \tau S)^{-1/2} \cdot\| \leq |t-s|^{1/2} \|\cdot\|$) and~\eqref{e:ReplacementRPCrude}, we obtain
\begin{equ} \label{e:TightIReplbis}
\mathbf E_{\nu_\rho} \Big[\big|W_{s,t}^\tau\big((\partial_\cdot + \tau L) \mathrm{Err}_\cdot^\tau(\vphi)\big)\big|^p\Big] \lesssim_{\vphi} |t-s|^{p} \, \tau^{(1-m)p} \, .
\end{equ}
On the other hand, by using~\eqref{e:AldousM} and~\eqref{e:ReplacementRPCrude}, we deduce
\begin{equ}  \label{e:TightMReplbis}
\mathbf E_{\nu_\rho} \Big[\big|M_{{\rm T}_1, {\rm T}_2}^\tau\big(\mathrm{Err}_\cdot^\tau(\vphi)\big)\big|^2\Big] \lesssim_{\vphi, T} \delta^{1-2/p} \, \tau^{1-2m} \, .
\end{equ}
Since $m$ can be chosen arbitrarily large, we deduce the desired estimates~\eqref{e:TightIRepl} and~\eqref{e:TightMRepl}.
\ms

{\noindent \it Step 2. The diffusive part of the dynamics}
\ss

Now, we turn to the diffusive contributions and show that
\begin{equs}
\mathbf E_{\nu_\rho} \Big[\big|W_{s,t}^\tau\big((\partial_\cdot + \tau L - \tau A_+) \, \fsi_\cdot^\tau \vphi\big)\big|^p\Big] &\lesssim_{T, \vphi} |t-s|^{p/2} \, , \label{e:TightIdiff} \\
\mathbf E_{\nu_\rho} \Big[\big|M_{{\rm T}_1, {\rm T}_2}^\tau (\fsi_\cdot^\tau \vphi)\big|^2\Big] &\lesssim_{T, \vphi} \delta^{1 - 2/p} \, . \label{e:TightMdiff}
\end{equs}
In fact, these bounds are easy consequences of Proposition~\ref{p:TightnessBlackBox} and Lemma~\ref{lem:SimpleGen}.
\ss

First, observe that~\eqref{e:TightMdiff} readily follows from~\eqref{e:AldousM} and~\eqref{e:SimpleH1}.

We turn to the proof of~\eqref{e:TightIdiff}. First, note that $(\partial_r + \tau L - \tau A_+) \, \fsi_r^\tau \vphi = (\tau S + \partial_r + \tau A_0) \, \fsi_r^\tau \vphi$; indeed, $A_- \fsi_r^\tau \vphi = 0$ because $\vphi \in \fockdo{1}{}$. By using~\eqref{e:KolmogorovI} (with the trivial bound $\|(|t-s|^{-1} - \tau S)^{-1/2} \cdot\| \leq \|(- \tau S)^{-1/2} \cdot\|$) and~\eqref{e:SimpleH1}, we obtain
\begin{equ} \label{e:TightIdiffS}
\mathbf E_{\nu_\rho} \Big[\big|W_{s,t}^\tau(\tau S \, \fsi_\cdot^\tau \vphi)\big|^p\Big] \lesssim_{T, \vphi} |t-s|^{p/2} \, .
\end{equ}
By using~\eqref{e:KolmogorovI} (with the trivial bound $\|(|t-s|^{-1} - \tau S)^{-1/2} \cdot\| \leq (1 \vee T)^{1/2} \, \|(1- \tau S)^{-1/2} \cdot\|$) and~\eqref{e:SimpleA0}, we deduce that, for any $\gamma \in (0,1)$,
\begin{equ} \label{e:TightIdiffA0}
\mathbf E_{\nu_\rho} \Big[\big|W_{s,t}^\tau\big((\partial_\cdot + \tau A_0) \, \fsi_\cdot^\tau \vphi\big)\big|^p\Big] \lesssim_{T, \vphi, \gamma} |t-s|^{p/2} \, \tau^{- \gamma p/2} \, .
\end{equ}
By combining~\eqref{e:TightIdiffS} and~\eqref{e:TightIdiffA0}, we deduce the desired estimate~\eqref{e:TightIdiff}.
\ms

{\noindent \it Step 3. The superdiffusive part of the dynamics}
\ss

Finally, we tackle the superdiffusive contribution and show that
\begin{equ}
    \mathbf E_{\nu_\rho} \Big[\big|W_{s,t}^\tau\big(\tau A_+ \, \fsi_\cdot^\tau \vphi\big)\big|^p\Big] \lesssim_{T, \vphi} |t-s|^{p/2} \, . \label{e:TightIA+}
\end{equ}
This is the most interesting term because it requires the use of the fluctuation dissipation relation~\eqref{e:FormalFluctuationDissipationASEP}. As mentioned before, we choose here the parameter $\mu$ as $\mu := \tau^{-1} \, |t-s|^{-1}$. Note that, without loss of generality, it is sufficient to show~\eqref{e:TightIA+} for times $s,t$ that satisfy $|t-s| \leq 1$: then, we have $\mu \tau \geq 1$ as required in Proposition~\ref{p:AnsatzA+}. Now, we write
\begin{equs}
W_{s,t}^\tau\big(\tau A_+ \, \fsi_\cdot^\tau \vphi\big) &= M_{s,t}^\tau\big(\mathrm{Cor}^{\tau, \mu}_\cdot(\vphi)\big) + d^{\lim} \frac{\log_{*}(1/\mu)^{2/3}}{\log_{*}(\tau)^{2/3}} W_{s,t}(\fsi_\cdot^\tau \fL_0^\fe \vphi) + \mathrm{Cor}^{\tau, \mu}_s(\vphi)(\eta_{\tau s}) - \mathrm{Cor}^{\tau, \mu}_t(\vphi)(\eta_{\tau t}) \\
&\quad + W_{s,t}^\tau\Big(\tau A_+ \, \fsi_\cdot^\tau \vphi + (\partial_\cdot + \tau L) \mathrm{Cor}^{\tau, \mu}_\cdot(\vphi) - d^{\lim} \frac{\log_{*}(1/\mu)^{2/3}}{\log_{*}(\tau)^{2/3}} \fsi_\cdot^\tau \fL_0^\fe \vphi\Big) \, . \label{e:TightFluctuationDissipation}
\end{equs}

To estimate the fluctuation term, we use~\eqref{e:KolmogorovM},~\eqref{e:AnsatzA+L2ASEP} and~\eqref{e:AnsatzA+H1ASEP}, which gives
\begin{equ} \label{e:TightIA+M}
\mathbf E_{\nu_\rho}\Big[\big|M_{s,t}^\tau\big(\mathrm{Cor}^{\tau, \mu}_\cdot(\vphi)\big)\big|^p\Big] \lesssim_{T, \vphi} |t-s|^{p/2} \, (\log_{*}(\tau)^{-(1/6 - \eps) p} \, |t-s|^{p/2} + 1) \lesssim_{T, \vphi} |t-s|^{p/2} \, .
\end{equ}

To estimate the dissipation term, we recall that $1/\mu \leq \tau$, use~\eqref{e:KolmogorovI} (with the trivial bound $\|(|t-s|^{-1} - \tau S)^{-1/2} \cdot\| \leq |t-s|^{1/2}$) and the $L^2$ contractivity of $\fsi_\cdot^\tau$, which leads to
\begin{equ} \label{e:TightIA+Lnew}
\mathbf E_{\nu_\rho}\Big[\Big|d^{\lim} \frac{\log_{*}(1/\mu)^{2/3}}{\log_{*}(\tau)^{2/3}} \, W_{s,t}^\tau(\fsi_\cdot^\tau \fL_0^\fe \vphi)\Big|^p\Big] \lesssim_{T, \vphi} |t-s|^{p} \, .
\end{equ}

To estimate the boundary terms, we use the $\nu_\rho$-stationarity, the hypercontractivity~\eqref{e:UltHyper} and~\eqref{e:AnsatzA+L2ASEP} (recall that $\mu \tau = |t-s|^{-1}$), which results in (for $r \in \{s,t\}$)
\begin{equ} \label{e:TightIA+bc}
\mathbf E_{\nu_\rho}\Big[\big|\mathrm{Cor}^{\tau, \mu}_r(\vphi)\big|^p\Big] = \|\mathrm{Cor}^{\tau, \mu}_\cdot(\vphi)\|_{L^p(\nu_\rho)}^p \lesssim \|\alpha^N \mathrm{Cor}^{\tau, \mu}_\cdot(\vphi)\|^p \lesssim_{\vphi} |t - s|^{p/2} \, \log_{*}(\tau)^{-p/3} \, .
\end{equ}

To estimate the error term, we use~\eqref{e:KolmogorovI} and~\eqref{e:AnsatzA+GenASEP}, which gives
\begin{equs} \label{e:TightIA+Error}
\mathbf E_{\nu_\rho}\Big[\Big|W_{s,t}^\tau\big(\tau A_+ \, \fsi_\cdot^\tau \vphi + (\partial_\cdot + \tau L) \mathrm{Cor}^{\tau, \mu}_\cdot(\vphi) - &d^{\lim} \frac{\log_{*}(1/\mu)^{2/3}}{\log_{*}(\tau)^{2/3}} \fsi_\cdot^\tau \fL_0^\fe \vphi\big)\Big|^p\Big] \\
&\lesssim_{T, \vphi} |t-s|^{p/2} \, \log_{*}(\tau)^{-(1/6 - \eps) p} \, .
\end{equs}

By combining~\eqref{e:TightIA+M},~\eqref{e:TightIA+Lnew},~\eqref{e:TightIA+bc}, and~\eqref{e:TightIA+Error}, we deduce~\eqref{e:TightIA+} and conclude the proof.
\end{proof}

\subsection{The invariance principle (proof of Theorem~\ref{th:CLT})}

\label{sec:CLTccl}

We conclude with the proof of Theorem~\ref{th:CLT}.

\begin{proof}[of Theorem~\ref{th:CLT}]

The convergence in law of the process $(Y_t^\tau)_{t \geq 0}$ under $\mathbf P_{\nu_\rho}$ to the process $(Y_t^{\lim})_{t \geq 0}$ under $\mathbf P_{\mathbb P}^{\lim}$ is a standard consequence of the convergence of the finite-dimensional laws (\eqref{e:CVfindimAnnealed} in Proposition~\ref{p:CVfindim}) and tightness (Proposition~\ref{p:Tightness}). However, deducing the ``functional convergence in law conditioned on the initial condition'' as stated in~\eqref{e:CLTFunctionalQuenched} from its finite-dimensional marginal analogue (\eqref{e:CVfindimQuenched} in Proposition~\ref{p:CVfindim}) and tightness is much less standard; hence, we give more details.
\ss

The key ingredient is the following elementary result which states that bounded continuous functionals on $D(\R_+, \cS'(\R^2))$ can always be uniformly approximated on compact subsets by space-time cylinder functionals.
\begin{lemma} \label{lem:ApproxCyl}
Let $F \colon D(\R_+, \cS'(\R^2)) \to \R$ be a bounded continuous function, $K \subset D(\R_+, \cS'(\R^2))$ be a compact subset, and $\eps > 0$. Then, there exist $n \geq 1$, $0 \leq t_1 \leq ... \leq t_n$, $\vphi_1, ..., \vphi_n \in \cS(\R^2)$, and a continuous function $g \colon \R^n \to \R$ bounded by $\|F\|_\infty$ such that, for every $y_\cdot \in K$,
\begin{equ} \label{e:ApproxCyl}
|F(y_\cdot) - g(y_{t_1}(\vphi_1), ..., y_{t_n}(\vphi_n))| < \eps \, .
\end{equ}
\end{lemma}
The subtlety here is that the space-time cylinder functionals themselves are \emph{not continuous} on $D(\R_+, \cS'(\R^2))$. In particular, it is not possible to deduce Lemma~\ref{lem:ApproxCyl} from the Stone-Weierstrass theorem (contrary to the analogous statement with $C(\R_+, \cS'(\R^2))$ instead of $D(\R_+, \cS'(\R^2))$). We postpone the proof of Lemma~\ref{lem:ApproxCyl} to the end of this section.
\ss

Now, let $\eps > 0$. First, we recall that Proposition~\ref{p:Tightness} shows the tightness of the laws of $Y_\cdot^\tau$ under $\mathbf P_{\nu_\rho}$ in $D(\R_+, \cS'(\R^2))$. Second, as discussed at the end of the proof of Proposition~\ref{p:CVfindim}, we know that the process $Y_\cdot^{\lim}$ started from $Z_0^\tau(\eta)$ with $\eta \sim \nu_\rho$ converges in law to $\mathbf P^{\lim}_{\mathbb P}$ as $\tau \to \infty$ in $C(\R_+, \cS'(\R^2))$. In particular, the laws of the latter processes are also tight in $D(\R_+, \cS'(\R^2))$. Consequently, there exists a compact subset $K \subset D(\R_+, \cS'(\R^2))$ such that, for every $\tau \geq 1$,
\begin{equ} \label{e:CompactTightness}
\mathbf P_{\nu_\rho}[Y_\cdot^{\tau} \not \in K] < \eps \, , \quad \nu_\rho[\mathbf P_{Z_0^\tau(\eta)}^{\lim}[Y_\cdot^{\lim} \not \in K]] < \eps \, .
\end{equ}
Then, we apply Lemma~\ref{lem:ApproxCyl} (let us write $G(y_\cdot) := g(y_{t_1}(\vphi_1), ..., y_{t_n}(\vphi_n))$ for the corresponding space-time cylinder functional) and deduce, for every $y_\cdot \in K$,
\begin{equ} \label{e:SpaceTimeCylinderApprox}
|F(y_\cdot) - G(y_{\cdot})| < \eps \, .
\end{equ}
Then, by gathering~\eqref{e:CompactTightness} and~\eqref{e:SpaceTimeCylinderApprox} (and recalling that $\|G\|_\infty \leq \|F\|_\infty$), we immediately deduce
\begin{equs}
\int_\Omega \dd \nu_\rho(\eta) \, \big|\mathbf E_{\eta}[F(Y_\cdot^\tau)] &- \mathbf E_{Z_0^\tau(\eta)}[F(Y_\cdot^{\lim})]\big| \\
&\leq 4 \|F\|_\infty \, \eps + 2 \eps + \int_\Omega \dd \nu_\rho(\eta) \, \big|\mathbf E_{\eta}[G(Y_\cdot^\tau)] - \mathbf E_{Z_0^\tau(\eta)}^{\lim}[G(Y_\cdot^{\lim})]\big| \\
&\leq 4 \|F\|_\infty \, \eps + 3 \eps \, ,
\end{equs}
where the last line holds for sufficiently large $\tau$ because of the convergence~\eqref{e:CVfindimQuenched}. Since $\eps$ can be taken arbitrarily small, we deduce the desired statement~\eqref{e:CLTFunctionalQuenched} and conclude the proof.
\end{proof}

We finish this section by proving Lemma~\ref{lem:ApproxCyl}.

\begin{proof}[of Lemma~\ref{lem:ApproxCyl}]
For any $n \geq 1$ and any path $x = (x_t)_{t \geq 0}$ (taking values in either $\R$ or $\cS'(\R^2)$), we define the dyadic piecewise constant path $x^{[n]} = (x_t^{[n]})_{t \geq 0}$ as follows
\begin{equ}
x_t^{[n]} \eqdef \sum_{k=0}^{4^n-1} x_{k \, 2^{-n}} \, \1_{[k \, 2^{-n}, \, (k+1) \, 2^{-n})}(t) \, .
\end{equ}
We leave to the reader the proof of the following elementary fact about the Skorokhod topology. Let $d$ be a metric inducing the Skorokhod topology on $D(\R_+, \R)$. For any $f \in D(\R_+, \R)$ and $\delta > 0$, there exist $n_0 \geq 1$ and $\delta' > 0$ such that, for every $n \geq n_0$ and $g \in D(\R_+, \R)$ such that $d(g, f) < \delta'$, we have $d(g^{[n]}, f) < \delta$.

Let $\eps > 0$. Now, by combining the continuity of $F$ and the previous fact, we deduce the following property: for every $x_\cdot \in D(\R_+, \cS'(\R^2))$ there exist $n_0 \geq 1$, $\delta > 0$ and a finite subset $\cT \subset \cS(\R^2)$ such that, for every $n \geq n_0$ and $y_\cdot \in U(x_\cdot, \cT, \delta) := \{y_\cdot \in D(\R_+, \cS'(\R^2)) \mbox{ such that } \forall \vphi \in \cT, \, d(y_\cdot(\vphi), x_\cdot(\vphi)) < \delta\}$, we have $|F(y_\cdot) - F(x_\cdot)| < \eps/2$ and $|F(y_\cdot^{[n]}) - F(x_\cdot)| < \eps/2$. By compactness of $K$, we can consider a finite covering by such open subsets: let $K \subset \bigcup_{i=1}^m U(x_\cdot^{(i)}, \cT^{(i)}, \delta^{(i)}) =: \bigcup_{i=1}^m U^{(i)}$. Thus, there exists $n \geq 1$ such that, for every $i \in \{1, ..., m\}$ and $y_\cdot \in U^{(i)}$, we have $|F(y_\cdot) - F(x_\cdot^{(i)})| < \eps/2$ and $|F(y_\cdot^{[n]}) - F(x_\cdot^{(i)})| < \eps/2$.

Finally, let $V \subset \cS(\R^2)$ be the linear span of the elements of $\cT^{(1)}, \dots, \cT^{(m)}$, and let $\pi_V \colon \cS'(\R^2) \to V (\subset \cS'(\R^2))$ be the orthogonal projection onto $V$, that is $\pi_V(T) := \sum_{i=1}^M T(\vphi_i) \, \langle \vphi_i, \cdot \rangle_{L^2(\R^2)}$ where $(\vphi_1, \dots, \vphi_M)$ is an $L^2(\R^2)$-orthonormal basis of $V$. Recalling the definition of $U^{(i)} = U(x_\cdot^{(i)}, \cT^{(i)}, \delta^{(i)})$, we deduce the following property: if $y_\cdot \in U^{(i)}$, then we also have $\pi_V(y_\cdot) \in U^{(i)}$, which in particular gives
\begin{equ}
|F(y_\cdot) - F(\pi_V(y_\cdot)^{[n]})| \leq |F(y_\cdot) - F(x_\cdot^{(i)})| + |F(\pi_V(y_\cdot)^{[n]}) - F(x_\cdot^{(i)})| < \eps \, .
\end{equ}
In particular, we deduce that for any $y \in K$ it holds $|F(y) - F(\pi_V(y_\cdot)^{[n]})| < \eps$. This concludes the proof after noticing that, for any $y_\cdot \in D(\R_+, \cS'(\R^2))$, we have $F(\pi_V(y_\cdot)^{[n]}) = g((y_{k 2^{-n}}(\vphi_j))_{0 \leq k \leq 4^n - 1, \, 1 \leq j \leq M})$ where we have defined $g \colon \R^{4^n \times M} \to \R$ as follows
\begin{equ}
g((\alpha_{k,j})_{k, j}) := F\Big(\sum_{k,j} \alpha_{k, j} \, \vphi_j \, \1_{[k \, 2^{-n}, \, (k+1) \, 2^{-n})}(\cdot)\Big) \, ,
\end{equ}
the latter function being indeed continuous and bounded by $\|F\|_\infty$.
\end{proof}

\section{Large time behaviour of the second class particle}

\label{sec:2ndClass}

We prove the central limit theorem for the second class particle (Theorem~\ref{th:CLT2ndClass}) in Section~\ref{sec:CLT2ndClass}, and the sharp asymptotics of the diffusivity (Theorem~\ref{th:Diff}) in Section~\ref{sec:Diff}.

\subsection{The central limit theorem (proof of Theorem~\ref{th:CLT2ndClass})}

\label{sec:CLT2ndClass}

In this section, we show the central limit theorem for the second class particle, as stated in Theorem~\ref{th:CLT2ndClass}. In fact, this is an easy consequence of the convergence of the density-density correlation function: more precisely,
as a corollary of Proposition~\ref{p:CVfindim}, we deduce that for any $\vphi_1, \vphi_2 \in \cS(\R^2)$ it holds, as $\tau \to \infty$,
\begin{equ} \label{e:2ptCV}
\mathbf{E}_{\nu_\rho}[Y_0^\tau(\vphi_1) Y_1^\tau(\vphi_2)] \to \mathbf {E}_{\mathbb P}^{\lim}[Y_0^{\lim}(\vphi_1) Y_1^{\lim}(\vphi_2)] = \langle \vphi_1(- \, \cdot) * \vphi_2, g^{\lim} \rangle_{L^2(\R^2)} \, ,
\end{equ}
where $g^{\lim}$ stands for the density of the Gaussian vector $\cN(0, D^{\lim})$.

\begin{proof}[of Theorem~\ref{th:CLT2ndClass}]

Let $\vphi_1, \vphi_2 \in \cS(\R^2)$ such that $\vphi_1$ is even-symmetric and observe that, by translation invariance of the law $\mathbf{P}_{\nu_\rho}$ and the formula~\eqref{e:2PointFunction}, it holds
\begin{equs}
\mathbf{E}_{\nu_\rho}[Y_0^\tau(\vphi_1) &Y_1^\tau(\vphi_2)] = \det(\Sigma^\tau)^{-1} \sum_{x, y \in \Z^2} \vphi_1((\Sigma^\tau)^{-1} x) \, \vphi_2((\Sigma^\tau)^{-1}(y - \tau v)) \, \mathbf{E}_{\nu_\rho}[\sigma_0(x) \, \sigma_\tau(y)] \\
&= \det(\Sigma^\tau)^{-1} \sum_{x, z \in \Z^2} \vphi_1((\Sigma^\tau)^{-1} x) \, \vphi_2((\Sigma^\tau)^{-1}(z - \tau v) + (\Sigma^\tau)^{-1} x) \, \mathbf{E}_{\nu_\rho}[\sigma_0(0) \, \sigma_\tau(z)] \\
&= \mathbf{E}_{\nu_\rho}^{\mathrm{II}}\big[(\vphi_1 *_\tau \vphi_2) \, ((\Sigma^\tau)^{-1}(X_\tau - \tau v))\big] \, , \label{e:Correspondence2ndClassFluctuation}
\end{equs}
where we have defined the discrete convolution $\vphi_1 *_\tau \vphi_2 \in \cS(\R^2)$ as follows
\begin{equ}
\vphi_1 *_\tau \vphi_2 \, (\mathrm{z}) \eqdef \det(\Sigma^\tau)^{-1} \sum_{x \in \Z^2} \vphi_1((\Sigma^\tau)^{-1} x) \, \vphi_2(\mathrm{z} - (\Sigma^\tau)^{-1} x) \, .
\end{equ}
Recognizing a Riemann sum, it is clear that $\vphi_1 *_\tau \vphi_2$ converges uniformly to $\vphi_1 * \vphi_2$ as $\tau \to \infty$. Consequently, as $\tau \to \infty$, it holds
\begin{equ}
\mathbf{E}_{\nu_\rho}[Y_0^\tau(\vphi_1) Y_1^\tau(\vphi_2)] = \mathbf{E}_{\nu_\rho}^{\mathrm{II}}\big[(\vphi_1 * \vphi_2) \, ((\Sigma^\tau)^{-1}(X_\tau - \tau v))\big] + o(1) \, .
\end{equ}

Therefore, we deduce from~\eqref{e:2ptCV} that, as $\tau \to \infty$,
\begin{equ} \label{e:preCVlaw2ndClass}
\mathbf{E}_{\nu_\rho}^{\mathrm{II}}\big[(\vphi_1 * \vphi_2) \, ((\Sigma^\tau)^{-1}(X_\tau - \tau v))\big] \to \langle \vphi_1 * \vphi_2, g^{\lim} \rangle_{L^2(\R^2)} \, ,
\end{equ}
for any $\vphi_1, \vphi_2 \in \cS(\R^2)$ such that $\vphi_1$ is even-symmetric. Then, since the span of the functions of the form $\vphi_1 * \vphi_2$ with such $\vphi_1$ and $\vphi_2$ is clearly dense in $\cC_0(\R^2)$ (observe for instance that it contains every linear combination of Gaussian densities), we deduce that
\begin{equ} \label{e:CVlaw2ndClass}
\mathbf{E}_{\nu_\rho}^{\mathrm{II}}\big[\vphi ((\Sigma^\tau)^{-1}(X_\tau - \tau v))\big] \to \langle \vphi, g^{\lim} \rangle_{L^2(\R^2)} \, ,
\end{equ}
for any $\vphi \in \cC_0(\R^2)$. Finally, since $g^{\lim}$ is a probability distribution, we can extend the convergence~\eqref{e:CVlaw2ndClass} to any $\vphi \in \cC_b(\R^2)$ and conclude the proof of Theorem~\ref{th:CLT2ndClass}.
\end{proof}

\subsection{The sharp asymptotic of the diffusivity (proof of Theorem~\ref{th:Diff})}

\label{sec:Diff}

In this section, we show the sharp asymptotic of the diffusivity matrix $D(t)$, as stated in Theorem~\ref{th:Diff}. In fact, we show the following stronger asymptotic, as $t \to \infty$, (recall that $D(t) = D_0 + d(t) \, \fe \fe^T$~\eqref{e:GKformulafe})
\begin{equ} \label{e:DiffStrong}
d(t) = d^{\lim} \log(t)^{2/3} + O(\log(t)^{1/2 + o(1)}) \, .
\end{equ}
Recalling the formula of $d(t)$ given by Lemma~\ref{lem:GKneat}, we need to capture the asymptotic of
\begin{equ}
\frac{1}{t} \, \mathbf E_{\nu_\rho}\Big[\Big(\int_0^t B_+ \cV^\delta(\eta_s) \, \dd s\Big)^2\Big] \, ,
\end{equ}
as first $\delta \to 0$ and then $t \to \infty$. The key idea is to exploit the corrector $\mathrm{Cor}_{\mu, \delta}^{B+}$ (cf. Definition~\ref{def:AnsatzDiffASEP}) and, in particular, to show that $\int_0^{t} B_+ \cV^\delta(\eta_s) \, \dd s$ is close to the Dynkin martingale $M_t(\mathrm{Cor}_{1/t, \delta}^{B+})$.

\begin{proof}[of Theorem~\ref{th:Diff}]

Let $t \geq 1$, set $\mu := 1/t$ and fix $\eps \in (0,1/6)$. Recall Definition~\ref{def:MI}. We write
\begin{equ} \label{e:DiffProof}
\int_0^t B_+ \cV^\delta(\eta_s) \, \dd s = M_t(\mathrm{Cor}_{\mu, \delta}^{B+}) - \mathrm{Cor}_{\mu, \delta}^{B+}(\eta_{t}) + \mathrm{Cor}_{\mu, \delta}^{B+}(\eta_0) + W_t(L \, \mathrm{Cor}_{\mu, \delta}^{B+} + B_+ \cV^\delta) \, .
\end{equ}
The boundary terms can be easily bounded by using the $\nu_\rho$-stationarity of the ASEP and~\eqref{e:AnsatzDiffL2ASEP},
\begin{equ} \label{e:DiffProof1}
\mathbf E_{\nu_\rho}[\mathrm{Cor}_{\mu, \delta}^{B+}(\eta_t)^2] \vee \mathbf E_{\nu_\rho}\big[\mathrm{Cor}_{\mu, \delta}^{B+}(\eta_0)^2\big] = \|\mathrm{Cor}_{\mu, \delta}^{B+}\|^2 \lesssim t \, .
\end{equ}
The last term in~\eqref{e:DiffProof} can be bounded by using the Itô trick bound~\eqref{e:IL2} together with~\eqref{e:AnsatzDiffGenASEP},
\begin{equ} \label{e:DiffProof2}
\limsup_{\delta \to 0} \mathbf E_{\nu_\rho}[W_t(L \, \mathrm{Cor}_{\mu, \delta}^{B+} + B_+ \cV^\delta)^2] \lesssim_\eps t \log_*(t)^{1/3 + 2 \eps} \, .
\end{equ}
Finally, by using the classical identity $\mathbf E_{\nu_\rho}[M_t(f)^2] = 2 \, t \, \|(- S)^{1/2} f\|^2$ (see~\cite[Section 2.4]{KLO}) together with~\eqref{e:AnsatzDiffH1ASEP}, we obtain
\begin{equ} \label{e:DiffProof3}
\limsup_{\delta \to 0} \big|\mathbf E_{\nu_\rho}[M_t(\mathrm{Cor}_{\mu, \delta}^{B+})^2] - d^{\lim} \, t \log_*(t)^{2/3}| \lesssim_\eps t \log_*(t)^{1/3 + 2 \eps} \, .
\end{equ}
By gathering~\eqref{e:DiffProof},~\eqref{e:DiffProof1},~\eqref{e:DiffProof2},~\eqref{e:DiffProof3}, we deduce
\begin{equ}
\limsup_{\delta \to 0} \Big|\frac{1}{t} \, \mathbf E_{\nu_\rho}\Big[\Big(\int_0^t B_+ \cV^\delta(\eta_s) \, \dd s\Big)^2\Big] - d^{\lim} \log_*(t)^{2/3}\Big| \lesssim_\eps \log_*(t)^{1/2 + \eps} \, ,
\end{equ}
which, in view of Lemma~\ref{lem:GKneat}, establishes~\eqref{e:DiffStrong} and concludes the proof.
\end{proof}

\begin{appendix}

\section{Estimates on the SBE Ansatz (in the spirit of~\cite{CMT25})}

\label{app:AnsatzSBE}

The purpose of this section is to first show Proposition~\ref{p:GenericAnsatz}  by following the procedure of~\cite{CMT25}, and then to deduce Propositions~\ref{p:AnsatzA+} and~\ref{p:AnsatzDiff}. But before diving into the proof, we state the following elementary fact that we will use repeatedly.
\ss

Most future computations involve quantities of the form $\langle \fD \fA_+ \phi, \fA_+ \phi\rangle$ where $\fD$ is a diagonal operator (recall Definition~\ref{def:DiagonalOperator}). A key observation is that these brackets can be decomposed into a \emph{diagonal} and an \emph{off-diagonal} contribution: more precisely, writing both occurrences of $\fA_+ \phi$ as a sum over the indices $i < j$ and $i' < j'$ respectively, as suggested by the formula~\eqref{e:fA+}, then the diagonal contribution consists in the sum over the indices such that $(i,j) = (i',j')$. For our purpose, we only need two simple facts: first, we can always bound the full bracket $\langle \fD \fA_+ \phi, \fA_+ \phi\rangle$ by its diagonal contribution, up to a factor that is polynomial in the number operator; second, the off-diagonal contribution always vanishes if we replace $\fA_+$ by $\fA_+^\sharp$ (this property is actually the main motivation behind the definition of $\fA_+^\sharp$).

\begin{lemma} \label{lem:Diag}
We define, for any positive diagonal operator $\fD$ on $\fock{}{}$ with kernel $(d_n)_{n \geq 0}$ and any $\phi, \psi \in \fock{\rap}{}$,
\begin{equs} \label{e:Diag}
\langle \fD \fA_+ \phi, \fA_+ \psi\rangle_{\mathrm{diag}} \eqdef \sum_{n \geq 0} n! \, \int_{\R^{2n}} &\dd k_{1:n} \, \fF \phi(k_{1:n}) \,  \overline{\fF \psi(k_{1:n})} \\
&\times \frac{1}{2} \sum_{i=1}^n (\fe \cdot k_i)^2 \ \frac{\lambda^2 \rho(1-\rho)}{\pi^2} \int_{\R^2} \dd p \, \Pi_{p,k_i-p} \, d_{n+1}(p, k_i - p, k_{1:n \setminus i}) \, .
\end{equs}
We also define $\langle \fD \fA_+^{\bullet} \phi, \fA_+^{\bullet} \psi\rangle_{\mathrm{diag}}$ for $\bullet \in \{\sharp, \flat\}$ by adding a factor $\chi^{\bullet}(p, k_i - p ; k_{1:n})$ in the last integrand. Then, for any $\bullet \in \{\emptyset,\sharp, \flat\}$, it holds
\begin{equ}
\langle \fD \fA_+^{\bullet} \phi, \fA_+^{\bullet} \phi \rangle \lesssim \langle \fN^2 \, \fD \, \fA_+^{\bullet} \phi, \fA_+^{\bullet} \phi\rangle_{\mathrm{diag}} \, . \label{e:CSDiag}
\end{equ}
Moreover, when $\bullet = \sharp$, we have
\begin{equ}
\langle \fD \fA_+^{\sharp} \phi, \fA_+^{\sharp} \psi \rangle = \langle \fD \fA_+^{\sharp} \phi, \fA_+^{\sharp} \psi\rangle_{\mathrm{diag}} \, .\label{e:SharpDiag}
\end{equ}
\end{lemma}

\begin{proof}
It is clearly sufficient to prove the claims for $\phi = \psi \in \fock{\rap, n}{}$ and any $n \geq 1$. Let $\bullet \in \{\emptyset, \sharp, \flat\}$ and write $\fA_+^\bullet \phi$ in the form $\sum_{1 \leq i < j \leq n+1} \fA_+^{\bullet,i,j} \, \phi$ as suggested by~\eqref{e:fA+} or~\eqref{e:fA+Sharp}. The key observation is the following computation (we use the convention $\chi^\emptyset \equiv 1$)
\begin{equs}
\ &\sum_{1 \leq i < j \leq n+1} \langle \fD \fA_+^{\bullet, i, j} \phi, \fA_+^{\bullet, i, j} \phi \rangle \quad \\
&= \sum_{1 \leq i < j \leq n+1} (n+1)! \, \frac{\rho(1 - \rho)}{\pi^2 (n+1)^2} \int_{\R^{2(n+1)}} \dd k_{1:n+1} \, d_{n+1}(k_{1:n+1}) \, \chi^\bullet(k_i, k_j ; k_i + k_j, k_{1:n+1 \setminus i, j}) \, \Pi_{k_i, k_j} \, \\
&\qquad \qquad \qquad \qquad \qquad \qquad \qquad  \qquad \qquad \qquad \qquad \times |\fw \cdot (k_i + k_j)|^2 \, |\fF \phi(k_i + k_j, k_{1:n+1 \setminus i, j})|^2 \\
&= \frac{n(n+1)}{2} \ (n+1)! \, \frac{\lambda^2 \rho(1 - \rho)}{\pi^2 (n+1)^2} \int_{\R^{2n}} \dd q_{1:n} \int_{\R^2} \dd p \, d_{n+1}(p, q_1 - p, q_{2:n}) \, \chi^\bullet(p, q_1 - p ; q_{1:n}) \, \Pi_{p, q_1 - p} \\
&\qquad \qquad \qquad \qquad \qquad \qquad \qquad  \qquad \qquad \qquad \qquad \times \, |\fe \cdot q_1|^2 \, |\fF \phi(q_{1:n})|^2 \\
&= \langle \fD \fA_+^\bullet \phi, \fA_+^\bullet \phi \rangle_{\mathrm{diag}} \, , 
\end{equs}
where in the second equality we made the change of variables $q_{1:n} = (k_i + k_j, k_{1:n \setminus i, j})$ and $p = k_i$, and in the last equality we noticed that the sum over $i$ in~\eqref{e:Diag} can be replaced by $n$ times the summand $i=1$ because $\phi$ is symmetric.
\ss

In particular, using the Cauchy-Schwarz inequality, we deduce
\begin{equs} \label{e:BracketDiag}
\langle \fD \fA_+^\bullet \phi, \fA_+^\bullet \phi \rangle = \Big\|\fD^{1/2} \sum_{1 \leq i < j \leq n+1} \fA_+^{\bullet, i, j} \phi\Big\|^2 &\leq n(n+1) \sum_{1 \leq i < j \leq n+1} \|\fD^{1/2} \, \fA_+^{\bullet, i, j} \phi\|^2 \\
&= n(n+1) \, \langle \fD \fA_+^\bullet \phi, \fA_+^\bullet \phi\rangle_{\mathrm{diag}} \, ,
\end{equs}
which shows~\eqref{e:CSDiag}.

As for the proof of~\eqref{e:SharpDiag}, we first write
\begin{equ}
\langle \fD \fA_+^\sharp \phi, \fA_+^\sharp \psi \rangle =  \sum_{i < j} \langle \fD \fA_+^{\sharp, i, j} \phi, \fA_+^{\sharp, i, j} \psi \rangle +  \sum_{i < j, \, i' < j', \, (i,j) \ne (i', j')} \langle \fD \fA_+^{\sharp, i, j} \phi, \fA_+^{\sharp, i', j'} \psi \rangle \, .
\end{equ}
By the previous computation, the first sum equals $\langle \fD \fA_+^\sharp \phi, \fA_+^\sharp \psi\rangle_{\mathrm{diag}}$. As for the second sum, observe that each term can be written as an integral over $k_{1:n+1} \in \R^{2(n+1)}$ whose integrand is proportional to
\begin{equ} \label{e:OffDiagCancellation}
\chi^\sharp(k_i, k_j ; k_{1:n+1 \setminus i,j}, k_i + k_j) \, \chi^\sharp(k_{i'}, k_{j'} ; k_{1:n+1 \setminus i',j'}, k_{i'} + k_{j'}) \, .
\end{equ}
But the first (resp. second) indicator imposes that $k_i$ and $k_j$ (resp. $k_i'$ and $k_j'$) have the two largest norms among $k_1, \dots k_{n+1}$. Therefore~\eqref{e:OffDiagCancellation} always vanishes if $\{i,j\} \ne \{i',j'\}$, which shows~\eqref{e:SharpDiag}.
\end{proof}

Finally, as suggested by~\eqref{e:BracketDiag}, a lot of formulas will display vectors of the form $(p, k_i-p, k_{1:n \setminus i})$. Hence, to shorten notations, we will often write
\begin{equs} \label{e:XiNotation}
\Xi &:= \frac{1}{2} \, \big(|D_0^{1/2} \, k_{1:n \setminus i}|^2 + |D_0^{1/2} \, p|^2 + |D_0^{1/2} \, (k_i - p)|^2\big) \, , \\
\Xi_0 &:= \frac{1}{2} \, \big(|k_{1:n \setminus i}|^2 + |p|^2 + |k_i - p|^2\big) \, , \quad \Xi_0^\fe := \frac{1}{2} \big(|\fe \cdot k_{1:n \setminus i}|^2 + |\fe \cdot p|^2 + |\fe \cdot (k_i - p)|^2\big) \, , \\
\chi^\bullet&(...) := \chi^\bullet(p, k_i - p ; k_{1:n}) \quad \mbox{ with the convention $\chi^\emptyset \equiv 1$}
\end{equs}
where $p$, $k_{1:n}$ and $i$ will always be suggested by the context. Note that, because of the triangle inequality, we always have
\begin{equ} \label{e:Triangular}
\Xi \asymp \Xi_0 \asymp |k_{1:n}|^2 + |p|^2 \, , \quad \Xi_0^\fe \asymp |\fe \cdot k_{1:n}|^2 + |\fe \cdot p|^2 \, .
\end{equ}

\subsection{The Replacement Lemmas}

\label{app:ReplLemma}

In this section, we prove Lemma~\ref{lem:ReplLemma}. In fact, the latter follows from the following more general replacement lemma, which is the analogue of~\cite[Lemma C.1]{CMT25}.

\begin{lemma} \label{lem:ReplLemmaGeneral}
Let $f \colon \R_+ \to \R_+$ and $g \colon \R_+ \to \R_+$ be $C^1$ functions such that for some constants $C_f, C_g \geq 1$ and for every $x > 0$ it holds
\begin{equ} \label{e:AssReplLemma}
\frac{|f'(x)|}{1 + f(x)} \leq \frac{C_f}{x} \, , \quad \frac{g(x)}{1 + f(x)} \leq C_g \, , \quad \frac{|g'(x)|}{1 + f(x)} \leq \frac{C_g}{x} \, .
\end{equ}
Then, with the notations~\eqref{e:XiNotation}, for every $\bullet \in \{\emptyset, \sharp\}$, $k_{1:n} \in \R^{2n}$ ($n \geq 1$), $1 \leq i \leq n$, $\mu > 0$, we have
\begin{equs}
\label{e:FirstReplLemmaInt}
\Big|&\frac{\lambda^2 \rho(1 - \rho)}{\pi^2} \int_{\R^2} \dd p \, \frac{\chi^\bullet(...) \, \Pi_{p, k_i - p}}{\mu + \Xi +  f(\Log(\mu + \Xi_0)) \, \Xi_0^\fe} - J(f) \, (\Log(\mu + |k_{1:n}|^2/2))\Big| \lesssim C_f  \, , \\
\label{e:SecondReplLemmaInt}
\Big|&\frac{\lambda^2 \rho(1 - \rho)}{\pi^2} \int_{\R^2} \dd p \, \frac{\chi^\bullet(...) \, \Pi_{p, k_i - p} \ g(\Log(\mu + \Xi_0)) \, \Xi_0^\fe}{(\mu + \Xi + f(\Log(\mu + \Xi_0)) \, \Xi_0^\fe)^2} - J'(f;g) \, (\Log(\mu + |k_{1:n}|^2/2))\Big| \lesssim C_f \, C_g \, , \qquad \ \
\end{equs}
where the functions $J(f), \, J'(f; g) \colon \R_+ \to \R_+$ are given by (recall that $C_\Upsilon = \frac{\lambda^2 \rho(1 - \rho)}{\pi}$)
\begin{equ} \label{e:Jint}
J(f) \, (x) \eqdef \int_0^x \frac{C_\Upsilon \, \dd \ell}{(\det(D_0) + D_0[\ff, \ff] \, f(\ell))^{1/2}} \, , \quad J'(f;g) \, (x) \eqdef \frac{1}{2} \int_0^x \frac{C_\Upsilon \, D_0[\ff, \ff] \, g(\ell) \, \dd \ell}{(\det(D_0) + D_0[\ff,\ff] \, f(\ell))^{3/2}} \, .
\end{equ}
In particular, we deduce, for any $\mu > 0$ and $\phi, \psi \in \fock{\rap}{}$,
\begin{equs}
\label{e:FirstReplLemma}
\big|&\langle (\mu - \fS - \fL_0^\fe[f]_\mu)^{-1} \, \fA_+^\sharp \phi, \fA_+^\sharp \psi \rangle - \langle - \fL_0^\fe[J(f)]_\mu \, \phi, \psi \rangle\big| \lesssim C_f \, \|(- \fL_0^\fe)^{1/2} \, \phi\| \, \|(- \fL_0^\fe)^{1/2} \, \psi\| \, , \qquad \ \\
\label{e:SecondReplLemma}
\big|&\langle - \fL_0^\fe[g]_\mu \, (\mu - \fS - \fL_0^\fe[f]_\mu)^{-2} \, \fA_+^\sharp \phi, \fA_+^\sharp \psi \rangle - \langle - \fL_0^\fe [J'(f;g)]_\mu \, \phi, \psi \rangle\big| \\
&\qquad \qquad \qquad \qquad \qquad \qquad \qquad \qquad \qquad \qquad \lesssim C_f \, C_g \, \|(- \fL_0^\fe)^{1/2} \, \phi\| \,  \|((- \fL_0^\fe)^{1/2} \, \psi\| \, .
\end{equs}
Finally, the assumptions~\eqref{e:AssReplLemma} are satisfied by the functions $(f, g) = (\Upsilon_\nu, \tilde \Upsilon_\nu^\theta)$, for any $\nu \geq 0$ and $\theta \in (0,1]$, with constants $C_f, C_g$ that only depend on $D_0$ and $\lambda$; and it holds
\begin{equ} \label{e:JJ'Ups}
J(\Upsilon_\nu) = \Upsilon_\nu - \nu \, , \quad J'(\Upsilon_\nu ; \tilde \Upsilon_\nu^\theta) = \frac{1}{2 \theta} \, \big(\tilde \Upsilon_\nu^\theta - (\det(D_0) + D_0[\ff, \ff] \nu)^\theta\big) \, .
\end{equ}
\end{lemma}

\begin{remark}
We emphasize that the two integrals appearing in~\eqref{e:FirstReplLemmaInt} and~\eqref{e:SecondReplLemmaInt} are tightly connected: more precisely, if we denote them by $I_1(f)$ and $I_2(f ; g)$, then it holds $I_2(f ; g) = - \restriction{\frac{\partial}{\partial \delta}}{\delta = 0} \, I_1(f + \delta g)$. Similarly, note that we have $J'(f;g) = - \restriction{\frac{\partial}{\partial \delta}}{\delta = 0} \, J(f + \delta g)$. In particular, the factor $1/2$ in the definition of $J'(f;h)$ comes from the differentiation of the square root in the formula of $J(f)$. In fact, this factor $1/2$ is crucial: indeed, it reappears in~\eqref{e:ReplContraction} and gives a strict contraction estimate (provided $\theta > 1/2$).
\end{remark}

\begin{proof}[of Lemma~\ref{lem:ReplLemma}, assuming Lemma~\ref{lem:ReplLemmaGeneral}]
\ss

First, by applying~\eqref{e:FirstReplLemma} with $f = \Upsilon_\nu$ and using~\eqref{e:JJ'Ups}, we deduce
\begin{equ}
|\langle (\mu - \fL_{\mu, \nu}^\eff)^{-1} \fA_+^\sharp \phi, \fA_+^\sharp \psi \rangle - \langle - \fL_0^\fe[\Upsilon_\nu]_\mu \, \phi, \psi\rangle| \lesssim (1 + \nu) \, \|(- \fL_0^\fe)^{1/2} \phi\| \, \|(- \fL_0^\fe)^{1/2} \psi\| \, ,
\end{equ}
which gives the desired bound~\eqref{e:ReplFixPt}.
\ss

Second, by applying~\eqref{e:SecondReplLemma} with $(f, g) = (\Upsilon_\nu, \tilde \Upsilon_\nu^\theta)$ and using~\eqref{e:JJ'Ups}, we deduce
\begin{equs}
\langle - \fL_0^\fe\big[\tilde \Upsilon_\nu^\theta\big]_\mu \, (\mu - \fL_{\mu, \nu}^\eff)^{-2} \, &\fA_+^\sharp \phi, \fA_+^\sharp \phi \rangle \\
&\leq \frac{1}{2 \theta} \langle - \fL_0^\fe\big[\tilde \Upsilon_\nu^\theta\big]_\mu \, \phi, \phi \rangle + \Big(C - \frac{(\det(D_0) + D_0[\ff, \ff] \nu)^\theta}{2 \theta}\Big) \langle - \fL_0^\fe \, \phi, \phi\rangle \, ,
\end{equs}
for some constant $C > 0$ (that only depends on $D_0$ and $\lambda$). Then, we deduce the desired bound~\eqref{e:ReplContraction} by assuming $\nu \geq \nu_1$ with $\nu_1$ sufficiently large to make the last term non-positive for every $\theta \in [0,1]$.
\end{proof}

\begin{proof}[of Lemma~\ref{lem:ReplLemmaGeneral}]

First, we prove the claim relative to the particular case $(f,g) = (\Upsilon_\nu, \tilde \Upsilon_\nu^\theta)$. From the formulas~\eqref{e:ODEexplicit} and~\eqref{e:Upstheta}, it is easily checked that the assumption~\eqref{e:AssReplLemma} holds with some constants $C_f$ and $C_g$ that depend only on $D_0$ and $\lambda$. Next, the identity $J(\Upsilon_\nu) = \Upsilon_\nu - \nu$ readily follows from the ODE~\eqref{e:ODE} solved by $\Upsilon_\nu$. Finally, by making the change of variables $y = \det(D_0) + D_0[\ff, \ff] \Upsilon_\nu(\ell)$ (note that $\dd y = \frac{C_\Upsilon \, D_0[\ff, \ff]}{\sqrt{\det(D_0) + D_0[\ff, \ff] \Upsilon_\nu(\ell)}} \, \dd \ell$ by the ODE~\eqref{e:ODE}), we obtain
\begin{equ}
J'(\Upsilon_\nu ; \tilde \Upsilon_\nu^\theta) \, (x) = \frac{1}{2} \int_{\det(D_0) + D_0[\ff, \ff] \, \nu}^{\det(D_0) + D_0[\ff, \ff] \, \Upsilon_\nu(x)} y^{\theta - 1} \, \dd y = \frac{1}{2 \theta} \, \big[\tilde \Upsilon_\nu^\theta(x) - (\det(D_0) + D_0[\ff, \ff] \nu)^\theta\big] \, ,
\end{equ}
which concludes the proof of~\eqref{e:JJ'Ups}.
\ss

Now, we notice that the two statements~\eqref{e:FirstReplLemma} and~\eqref{e:SecondReplLemma} directly follow from the identity~\eqref{e:SharpDiag} and the integral estimates~\eqref{e:FirstReplLemmaInt} and~\eqref{e:SecondReplLemmaInt} (with $\bullet = \sharp$), used together with the identity~\eqref{e:SharpDiag}.
\ss

Hence, it only remains to show~\eqref{e:FirstReplLemmaInt} and~\eqref{e:SecondReplLemmaInt}. To do so, we divide~\eqref{e:FirstReplLemmaInt} into the following two claims
\begin{equ}
\Big|\int_{\R^2} \dd p \, \frac{\chi^\bullet(...) \, \Pi_{p, k_i - p}}{\mu + \Xi + f(\Log(\mu + \Xi_0)) \Xi_0^\fe} - \int_{\R^2} \dd p \, \frac{\1_{\mu + |k_{1:n}|^2/2 < |p|^2 < 1}}{|\sqrt{D_0} \, p|^2 + f(\Log(|p|^2)) \, |\fe \cdot p|^2} \Big| \lesssim C_f \, , \label{e:FirstReplLemmaTechnical}
\end{equ}
\begin{equ}
\frac{\lambda^2 \rho(1 - \rho)}{\pi^2} \int_{\R^2} \dd p \, \frac{\1_{\mu + |k_{1:n}|^2/2 < |p|^2 < 1}}{|\sqrt{D_0} \, p|^2 + f(\Log(|p|^2)) |\fe \cdot p|^2} = J(f) \, (\Log(\mu + |k_{1:n}|^2/2)) \, , \label{e:FirstReplLemmaExact}
\end{equ}
and, similarly, we divide~\eqref{e:SecondReplLemmaInt} into the following two claims
\begin{equ}
\Big|\int_{\R^2} \dd p \, \frac{\chi^\bullet(...) \, \Pi_{p, k_i - p} \ g(\Log(\mu + \Xi_0)) \Xi_0^\fe}{(\mu + \Xi + f(\Log(\mu + \Xi_0)) \Xi_0^\fe)^2} - \int_{\R^2} \dd p \, \frac{\1_{\mu + |k_{1:n}|^2/2 < |p|^2 < 1} \ g(\Log(|p|^2)) |\fe \cdot p|^2}{(|\sqrt{D_0} \, p|^2 + f(\Log(|p|^2)) \, |\fe \cdot p|^2)^2} \Big| \lesssim C_f \, C_g \, , \label{e:SecondReplLemmaTechnical}
\end{equ}
\begin{equ}
\frac{\lambda^2 \rho(1 - \rho)}{\pi^2} \int_{\R^2} \dd p \, \frac{\1_{\mu + |k_{1:n}|^2/2 < |p|^2 < 1} \ g(\Log(|p|^2)) |\fe \cdot p|^2}{(|\sqrt{D_0} \, p|^2 + f(\Log(|p|^2)) \, |\fe \cdot p|^2)^2} = J'(f; g) \, (\Log(\mu + |k_{1:n}|^2/2)) \, . \label{e:SecondReplLemmaExact}
\end{equ}

We start with the proof of~\eqref{e:FirstReplLemmaExact}. We first change to polar coordinates (in particular, we denote $e_\theta := (\cos \theta, \sin \theta)^T$), then we use the identity
\footnote{More generally, for any positive-definite symmetric $2 \times 2$ matrix $S$ we have $\int_0^{2 \pi} \frac{\dd \theta}{e_\theta^T S e_\theta} = \frac{2 \pi}{\sqrt{\det(S)}}$, which can easily be seen by computing the Gaussian integral $\int_{\R^2} \exp(- p^T S p/2) \, \dd p = \frac{2 \pi}{\sqrt{\det(S)}}$ in polar coordinates.}
$\int_0^{2 \pi} \frac{\dd \theta}{e_\theta^T (D_0 + A \fe \fe^T) \, e_\theta} = \frac{2 \pi}{\sqrt{\det(D_0) + D_0[\ff, \ff] A}}$ for any $A \geq 0$, and finally we make the change of variables $\ell = \Log(r^2)$ (notice that formally $\dd \ell = \frac{2 \, \dd r}{r}$), which gives
\begin{equs}
\int_{\R^2} \dd p \, \frac{\1_{\mu + |k_{1:n}|^2/2 < |p|^2 < 1}}{|\sqrt{D_0} \, p|^2 + f(\Log(|p|^2)) |\fe \cdot p|^2} &= \int_{\sqrt{\mu + |k_{1:n}|^2/2}}^1 r \, \dd r \, \int_0^{2 \pi} \frac{\dd \theta}{r^2 \ e_\theta^T \, (D_0 + f(\Log(r^2)) \, \fe \fe^T) \, e_\theta} \\
&= \int_{\sqrt{\mu + |k_{1:n}|^2/2}}^1 \frac{\dd r}{r} \, \frac{2 \pi}{\sqrt{\det(D_0) + D_0[\ff, \ff] \, f(\Log(r^2))}} \\
&= \pi \int_0^{\Log(\mu + |k_{1:n}|^2/2)} \frac{\dd \ell}{\sqrt{\det(D_0) + D_0[\ff, \ff] \, f(\ell)}} \, \, ,
\end{equs}
which leads to~\eqref{e:FirstReplLemmaExact} (multiplying both sides by $\frac{\lambda^2 \rho(1 - \rho)}{\pi^2}$).

The formula~\eqref{e:SecondReplLemmaExact} can be proved similarly: we pass to polar coordinates, use the identity $\int_0^{2 \pi} \frac{\dd \theta \, e_\theta^T \fe \fe^T e_\theta}{(e_\theta^T (D_0 + A \fe \fe^T) \, e_\theta)^2} = \frac{D_0[\ff, \ff]}{2} \, \frac{2 \pi}{(\det(D_0) + D_0[\ff, \ff] A)^{3/2}}$ for any $A \geq 0$ (which can be derived by differentiating the previous identity with respect to $A$), and use the change of variables $\ell = \Log(r^2)$.
\ss

Thus we are left to prove the bounds~\eqref{e:FirstReplLemmaTechnical} and~\eqref{e:SecondReplLemmaTechnical}. In both cases, we want to show
\begin{equs} \label{e:ReplLemmaPhi}
\Big|\int_{\R^2} \dd p \, \chi^\bullet(...) \, &\Pi_{p, k_i - p} \, \Phi(\mu + \Xi, \Xi_0^\fe, \Log(\mu + \Xi_0)) \\
&- \int_{\R^2} \dd p \, \1_{\mu + |k_{1:n}|^2/2 < |p|^2 < 1} \, \Phi(|D_0^{1/2} p|^2, |\fe \cdot p|^2, \Log(|p|^2))\Big| \lesssim C \, ,
\end{equs}
where $\Phi(x_1,x_2,x_3)$ and $C$ equal $\frac{1}{x_1 + f(x_3) \, x_2}$ and $C_f$ for~\eqref{e:FirstReplLemmaTechnical}, and equal $\frac{g(x_3) \, x_2}{(x_1 + f(x_3) \, x_2)^2}$ and  $C_f \, C_g$ for~\eqref{e:SecondReplLemmaTechnical} (we fix these notations for the rest of the proof). First, it is easily checked from the assumptions~\eqref{e:AssReplLemma} that, in both cases, $\Phi$ is $C^1$ on $(\R_+^*)^3$ and satisfies, for every $x_1, x_2, x_3 > 0$ and $\alpha \in \{1, 2, 3\}$,
\begin{equ} \label{e:AssumptionPhi}
|\Phi(x_1, x_2, x_3)| \lesssim C_* \, \frac{1 \vee \frac{x_2}{x_1}}{x_1} \, , \quad |\partial_\alpha \Phi(x_1, x_2, x_3)| \lesssim C_* \, \frac{1 \vee \frac{x_2}{x_1}}{x_1 \, x_\alpha} \, ,
\end{equ}
We mention that, in practice, we are only interested in $x_1, x_2, x_3$ such that $x_2 \lesssim x_1$ (roughly speaking, because of the trivial bound $|\fe \cdot q| \leq |q|$), for which we can bound the above numerators as $1 \vee \frac{x_2}{x_1} \lesssim 1$. To prove~\eqref{e:ReplLemmaPhi}, we distinguish the good region 
\begin{equ} \label{e:GoodRegion}
    G := \{p \in \R^2 \colon |\fe \cdot p| > 10 \, \sqrt{\mu + |k_{1:n}|^2} \mbox{ and } |p| < 1/10\}
\end{equ}
and the bad region $B = \R^2 - G$. The idea is the following: on the one hand, the two integrands will be easier to compare on $G$ ; on the other hand, both integrands will give a negligible contribution on $B$. 

Indeed, we can bound the two contributions on the bad region as follows, by using in particular~\eqref{e:AssumptionPhi} and~\eqref{e:Triangular}, 
\begin{equs}
\int_B \dd p \, \chi^\bullet(...) \, \Pi_{p, k_i -p} \, |\Phi(\mu + \Xi, \Xi_0^\fe, &\Log(\mu + \Xi))| \vee \int_B \dd p \, \1_{\mu + |k_{1:n}|^2/2 < |p|^2 < 1} \, |\Phi(|D_0^{1/2} p|^2, |\fe \cdot p|^2, \Log(|p|^2))| \\
\lesssim C_* \int_B \dd p \, \frac{\1_{|p| \leq 1}}{\mu + |k_{1:n}|^2 + |p|^2} &\leq C_* \int_{\R^2} \frac{\1_{|\fe \cdot p| \leq 10 (\mu + |k_{1:n}|^2)^{1/2}} \, \dd p}{\mu + |k_{1:n}|^2 + |p|^2} + C_* \int_{\R^2} \frac{\1_{1/10 \leq |p| \leq 1} \, \dd p}{|p|^2} \\
&= C_* \int_{\R^2} \frac{\1_{|\fe \cdot q| \leq 10} \, \dd q}{1 + |q|^2} + C_* \int_{\R^2} \frac{\1_{1/10 \leq |p| \leq 1} \, \dd p}{|p|^2} \lesssim C_* \, , \label{e:BadRegionBound}
\end{equs}
where we used the change of variable $q = \frac{p}{\sqrt{\mu + |k_{1:n}|^2}}$ in the last line.

Now, we compare the two integrands on $G$. For the rest of this proof, let $p \in G$ so that
\begin{equ} \label{e:GoodRegionBis}
10 \, \sqrt{\mu + |k_{1:n}|^2} < |\fe \cdot p| \leq |p| < 1/10 \, .
\end{equ}
In particular, it is easily checked that
\begin{equ} \label{e:ComparisonGammaBrutal}
\mu + \Xi_0 \asymp \mu + \Xi \asymp |p|^2 \, , \quad \Xi_0^\fe \asymp |\fe \cdot p|^2 \, , \quad |p|^2  \vee (\mu + \Xi_0) \leq e^{-1} \, , \quad \Log(|p|^2) \wedge \Log(\mu + \Xi_0) \geq 1 \, .
\end{equ}
Now we want to finely compare $(\mu + \Xi, \Xi_0^\fe, \Log(\mu + \Xi_0))$ to $(|D_0^{1/2} p|^2, |\fe \cdot p|^2, \Log(|p|^2))$. First, by writing $\Xi_0$ as $\frac{1}{2}(|p|^2 + |p|^2 - 2 \langle k_i, p \rangle + |k_i|^2 + |k_{1:n \setminus i}|^2)$, we deduce
\begin{equ} \label{e:ComparisonGamma0Fine}
|\mu + \Xi_0 - |p|^2| \leq \mu + |k_{1:n}|^2 + 2 |k_i| |p| \lesssim \sqrt{\mu + |k_{1:n}|^2} \, |p| \, ,
\end{equ}
where the last bound comes from~\eqref{e:GoodRegionBis}. Similarly, we have
\begin{equs}
|\mu + \Xi - |D_0^{1/2} p|^2| &\leq \mu + |D_0^{1/2} k_{1:n}|^2 + 2 |D_0^{1/2} k_i| |D_0^{1/2} p| \lesssim \sqrt{\mu + |k_{1:n}|^2} \, |p| \, , \label{e:ComparisonGammaFine} \\
|\Xi_0^\fe - |\fe \cdot p|^2| &\leq \mu + |\fe \cdot k_{1:n}|^2 + 2 |\fe \cdot k_i| |\fe \cdot p| \lesssim \sqrt{\mu + |k_{1:n}|^2} \, |\fe \cdot p| \, . \label{e:ComparisonGammafeFine} 
\end{equs}
Finally, by using the mean value theorem, the fact that $|\Log'(x)| = 1/x$ for $x \in (0, 1)$, and the estimates~\eqref{e:ComparisonGammaBrutal} and~\eqref{e:ComparisonGamma0Fine}, we deduce
\begin{equ} \label{e:ComparisonLGamma0Fine} 
|\Log(\mu + \Xi_0) - \Log(|p|^2)| \lesssim \frac{|\mu + \Xi_0 - |p|^2|}{|p|^2} \lesssim \frac{\sqrt{\mu + |k_{1:n}|^2}}{|p|} \, .
\end{equ}
Now, we replace successively the arguments of $\Phi$ by using the mean value theorem and the estimates~\eqref{e:AssumptionPhi} and~\eqref{e:ComparisonGammaBrutal}, which gives
\begin{equs}
|\Phi(\mu + \Xi, \Xi_0^\fe, \Log(\mu + \Xi_0)) &- \Phi(|D_0^{1/2} p|^2, |\fe \cdot p|^2, \Log(|p|^2))| \\
&\lesssim C_* \Big(\frac{|\mu + \Xi - |D_0^{1/2} p|^2|}{|p|^4} + \frac{|\Xi_0^\fe - |p|^2|}{|p|^2 \, |\fe \cdot p|^2} + \frac{|\Log(\mu + \Xi_0) - \Log(|p|^2)|}{|p|^2}\Big) \\
&\lesssim C_* \Big(\frac{\sqrt{\mu + |k_{1:n}|^2}}{|p|^3} + \frac{\sqrt{\mu + |k_{1:n}|^2}}{|p|^2 \, |\fe \cdot p|} + \frac{\sqrt{\mu + |k_{1:n}|^2}}{|p|^3}\Big) \\
&\lesssim C_* \Big(\frac{\sqrt{\mu + |k_{1:n}|^2}}{|p|^2 \, |\fe \cdot p|}\Big) \, ,
\end{equs}
where in the second to last line we used~\eqref{e:ComparisonGammaFine},~\eqref{e:ComparisonGammafeFine} and~\eqref{e:ComparisonLGamma0Fine}. In addition, we notice that we can deduce from~\eqref{e:GoodRegionBis} the following basic estimates
\begin{equ}
|p| > 10 |k_{1:n}| \, , \quad |k_i - p| > 9 |k_{1:n}| \, , \quad |k_{1:n}| < \frac{1}{100} \, , \quad |k_i - p| < \frac{11}{100} \, ,
\end{equ}
which gives $\chi^\sharp(...) = 1$, $\Pi_{p,k_i-p} = 1$, and $\1_{\mu + |k_{1:n}|^2/2 < |p|^2 < 1} = 1$. Therefore, we deduce
\begin{equs}
\Big|\int_{G} \dd p \, \chi^\bullet(...) \, \Pi_{p, k_i - p} \, \Phi(\mu + \Xi, \Xi_0^\fe, &\Log(\mu + \Xi_0)) - \int_{G} \dd p \, \1_{\mu + |k_{1:n}|^2/2 < |p|^2 < 1} \, \Phi(|D_0^{1/2} p|^2, |\fe \cdot p|^2, \Log(|p|^2))\Big| \\
\lesssim C_* \int_{G} \dd p \, \frac{\sqrt{\mu + |k_{1:n}|^2}}{|p|^2 |\fe \cdot p|} &\leq C_* \int_{\R^2} \dd p \, \1_{|\fe \cdot p| \geq 10 \sqrt{\mu + |k_{1:n}|^2}} \, \frac{\sqrt{\mu + |k_{1:n}|^2}}{|p|^2 |\fe \cdot p|} \\
&= C_* \int_{\R^2} \dd q \, \1_{|\fe \cdot q| \geq 10} \, \frac{1}{|q|^2 |\fe \cdot q|} \lesssim C_* \, ,
\label{e:GoodRegionBound}
\end{equs}
where we used the change of variable $q = p / \sqrt{\mu + |k_{1:n}|^2}$ in the last line, and where the finiteness of the last integral follows from the computation
\begin{equ}
\int_{\R^2} \dd q \, \1_{|\fe \cdot q| \geq 10} \, \frac{1}{|q|^2 |\fe \cdot q|} = \int_{\R} \dd x \, \1_{|x| \geq 10} \int_{\R} \dd y \, \frac{1}{(x^2 + y^2) \, |x|} = \int_{\R} \1_{|x| \geq 10} \, \frac{\pi}{|x|^{3/2}} < +\infty \, .
\end{equ}

By gathering~\eqref{e:GoodRegionBound} and~\eqref{e:BadRegionBound}, we deduce~\eqref{e:ReplLemmaPhi} and conclude the proof.
\end{proof}

\subsection{Proof of Proposition~\ref{p:GenericAnsatz}}

\label{app:GenericAnsatz}

In this section, we carry out Steps $1$ and $3$ of the proof of Proposition~\ref{p:GenericAnsatz}.
\ss

{\noindent \it Step 1. The $L^2$ estimates}
\ss

We start with the estimates of $\mathfrak{Ans}_{\mu}(\bm \vphi)$ in $L^2$-norm, which include the proof of Lemma~\ref{lem:SummabilityAnsatz} and of~\eqref{e:AnsatzL2}. The key is the following lemma (the analogue of~\cite[Proposition 3.12]{CMT25}), whose statement can be summarised as follows: up to considering a sufficiently large $\nu$, we can always assume that the operator norm of $(\mu - \fL_{\mu, \nu}^\eff)^{-1} \fA_+^\sharp$ is arbitrarily small, say smaller than $1/10$.

\begin{lemma} \label{lem:GA+Contraction}
There exists $\nu_2 \geq 0$ (only depending on $D_0$ and $\lambda$), such that for every $\nu \geq \nu_2$, $\mu > 0$, and $\phi \in \fock{\rap}{}$, it holds
\begin{equ} \label{e:GA+Contraction}
\|(\mu - \fL_{\mu, \nu}^\eff)^{-1} \, \fA_+^\sharp \phi\| \leq \frac{1}{10} \, \|\phi\| \, .
\end{equ}
\end{lemma}

In particular, we deduce that $(\mu - \fL_{\mu, \nu}^\eff)^{-1} \fA_+^\sharp$ is an $L^2$ contraction, which is useful to handle its iterations that appear in the definition~\ref{e:TrueAnsatz} of the Ansatz. However, the reader should not overvalue Lemma~\ref{lem:GA+Contraction}: indeed, most of the time, we are actually interested in estimating the Ansatz in $H^1$-type norms, for which the $L^2$-contraction estimate~\eqref{e:GA+Contraction} would be ineffective (in contrast to the second replacement lemma~\eqref{e:ReplContraction} which particularly shines in this context, as will be clear from the step $2$ of the proof).
\ss

\begin{proof}[of Lemma~\ref{lem:GA+Contraction}]
By orthogonality of chaoses, it is sufficient to prove the bound for any $\phi \in \fock{\rap,n}{}$ and $n \geq 1$. Then, using~\eqref{e:SharpDiag} and the notations~\eqref{e:XiNotation}, we deduce
\begin{equs}
\|(\mu - \fL_{\mu, \nu}^\eff)^{-1} \, \fA_+^\sharp \phi\|^2 &= \langle (\mu - \fL_{\mu, \nu}^\eff)^{-2} \, \fA_+^\sharp \phi, \fA_+^\sharp \phi \rangle_{\mathrm{diag}} \\
&\lesssim n! \int_{\R^{2n}} \dd k_{1:n} \, |\fF \phi(k_{1:n})|^2 \sum_{i=1}^n |\fe \cdot k_i|^2 \int_{\R^2} \frac{\dd p}{(\mu + \Xi + \Upsilon_\nu(\Log(\mu + \Xi_0)) \, \Xi_0^\fe)^2} \, \\
&\lesssim n! \int_{\R^{2n}} \dd k_{1:n} \, |\fF \phi(k_{1:n})|^2 \sum_{i=1}^n |\fe \cdot k_i|^2 \int_{\R^2} \frac{\dd p}{(|k_{1:n}|^2 + |p|^2 + \nu |\fe \cdot p|^2)^2} \, ,
\end{equs}
where in the last line we used~\eqref{e:Triangular} and the trivial lower bound $\Upsilon_\nu \geq \nu$. A simple change of variables shows that the last integral equals $\frac{C}{(1+\nu)^{1/2} \, |k_{1:n}|^2}$ for some constant $C > 0$. Thus we deduce
\begin{equ}
\|(\mu - \fL_{\mu, \nu}^\eff)^{-1} \, \fA_+^\sharp \phi\|^2 \lesssim \frac{1}{(1 + \nu)^{1/2}} \, \|\phi\|^2 \, ,
\end{equ}
which concludes the proof by assuming $\nu \geq \nu_2$ and taking $\nu_2$ sufficiently large.
\end{proof}

Now, we can explain how to choose the parameter $\nu$ in order to define the Ansatz.

\begin{definition} \label{def:nu}
Let $\nu_0 := \nu_1 \vee \nu_2$ where the constants $\nu_1, \nu_2 \geq 0$ come from the second part of Lemma~\ref{lem:ReplLemma} and from Lemma~\ref{lem:GA+Contraction}, respectively; in particular $\nu_0$ depends only on $D_0$ and $\lambda$.
\end{definition}

We highlight that we could have also chosen $\nu_0$ as any fixed value larger than $\nu_1$ and $\nu_2$. We also mention that it is in fact possible to take $\nu_0 = \nu_1 = \nu_2 = 0$ if the strength of the asymmetry $\lambda$ is sufficiently small compared to $D_0$.
\ss 

We can now easily prove Lemma~\ref{lem:SummabilityAnsatz} and~\eqref{e:AnsatzL2}.

\begin{proof}[of Lemma~\ref{lem:SummabilityAnsatz}]
We show that, for any $\mu > 0$, $a, b \geq 0$, $\phi \in \fockdo{m}{}$ ($m \geq 0$) and $n \geq 1$, we have
\begin{equ} \label{e:SummabilityAnsatzPrecise}
\|\fN_*^a \, (1 - \fL_0)^b \, \big((\mu - \fL_{\mu}^{\eff})^{-1} \, \fA_+^\sharp\big)^n \phi\| \leq 10^{-n} \, (m+n)_*^a \, (1 + 3 \pi^2)^b \, \|\phi\| \, ,
\end{equ}
from which Lemma~\ref{lem:SummabilityAnsatz} easily follows.

First, we deduce by iterating~\eqref{e:GA+Contraction} that
\begin{equ}
\|\big((\mu - \fL_{\mu}^{\eff})^{-1} \, \fA_+^\sharp\big)^n \phi\| \lesssim 10^{-n} \|\phi\| \, .
\end{equ}
It only remains to handle the extra factors $\fN_*^a$ and $(1 - \fL_0)^b$ appearing on the left hand side of~\eqref{e:SummabilityAnsatzPrecise}. First, we observe that $\big((\mu - \fL_{\mu}^{\eff})^{-1} \, \fA_+^\sharp\big)^n \phi \in \fock{m+n}{}$, so that the extra factor $\fN_*^a$ in~\eqref{e:SummabilityAnsatzPrecise} contributes as a factor $(m+n)_*
^a$. Second, we observe that, if for some vectors $p, q, k_{1:N} \in \R^2$ it holds
\begin{equ}
\chi^\sharp(p,q ; k_{1:N}) \, \Pi_{p,q} \ne 0 \, ,
\end{equ}
then we must have $|k_{1:N}| \leq |p| \vee |q| \leq \sqrt{2} \, \pi$, which leads to $|p|^2 + |q|^2 + |k_{1:N}|^2 \leq 6 \pi^2$. This shows that, if $\psi \in \fA_+^\sharp(\fock{\rap}{})$, then its Fourier transform is supported on $\{k_{1:N'} \colon |k_{1:N'}|^2 \leq 6 \pi^2, \, N' \geq 0\}$; and the same holds if $\psi \in (\mu - \fL_{\mu}^{\eff})^{-1} \, \fA_+^\sharp(\fock{\rap}{})$ since $(\mu - \fL_{\mu}^{\eff})^{-1}$ is a diagonal operator. In particular, we deduce that the extra factor $(1 - \fL_0)^b$ in~\eqref{e:SummabilityAnsatzPrecise} can be bounded by $(1 + 3 \pi^2)^b$. This concludes the proof of~\eqref{e:SummabilityAnsatzPrecise}.
\end{proof}

\begin{proof}[of~\eqref{e:AnsatzL2}]
First, by iterating~\eqref{e:GA+Contraction}, we deduce that, for every $\phi \in \fock{\rap}{}$ and $n \geq 0$, we have
\begin{equs}
\|\big((\mu - \fL_{\mu}^{\eff})^{-1} \fA_+^\sharp\big)^n \, (\mu - \fL_{\mu}^{\eff})^{-1} \phi\| \leq 10^{-n} \, \|(\mu - \fL_{\mu}^{\eff})^{-1} \phi\| \, .
\end{equs}
Now, notice that the operator $\big((\mu - \fL_{\mu}^{\eff})^{-1} \fA_+^\sharp\big)^n \, (\mu - \fL_{\mu}^{\eff})^{-1}$ maps $\fock{\rap,n_0}{}$ into $\fock{\rap, n_0 + n}{}$ for any $n_0 \geq 0$. Then, by using orthogonality of chaoses and the bound $(n_0 + n)_*^\kappa \, \alpha^{n_0 + n} \leq n_*^\kappa \alpha^n \, (n_0)_*^\kappa \alpha^{n_0}$, we deduce that, for any $\bm \vphi \in \fock{\rap}{}$ and $n \geq 0$, it holds
\begin{equ}
\|\fN_{*}^\kappa \alpha^\fN \, \big((\mu - \fL_{\mu}^{\eff})^{-1} \fA_+^\sharp\big)^n \, (\mu - \fL_{\mu}^{\eff})^{-1}  \bm \vphi\| \leq \frac{n_*^\kappa \alpha^n}{10^n} \, \|\fN_{*}^\kappa \alpha^\fN \, (\mu - \fL_{\mu}^{\eff})^{-1} \bm \vphi\| \, .
\end{equ}
By using~\eqref{e:TrueAnsatz}, the triangle inequality, and summing the above bound over $n \geq 0$ (note that $\sum_{n \geq 0} \frac{n_*^\kappa \alpha^n}{10^n} < +\infty$ because $\alpha \leq \sqrt{2}$), we deduce~\eqref{e:AnsatzL2} and conclude the proof.
\end{proof}
\ms

{\noindent \it Step 3. The generator estimates}
\ss

Finally, we study the action of the generator on the Ansatz and show to this end~\eqref{e:AnsatzGen} and~\eqref{e:AnsatzH1precise}. First, we compute $(\mu - \bar \fL) \, \mathfrak{Ans}_{\mu}(\bm \vphi)$ as in~\cite[Lemma 4.2]{CMT25}.

\begin{lemma} \label{lem:GenAnsatz}
For every $\mu > 0$ and $\bm \vphi \in \fock{\rap}{}$, we have
\begin{equs} \label{e:GenAnsatz}
(\mu - \bar \fL) \, \mathfrak{Ans}_{\mu}(\bm \vphi) &=  \bm \vphi - \fA_-^\sharp (\mu - \fL_{\mu}^\eff)^{-1} \bm \vphi \\
&\ \ + (\fL_0^\fe[\Upsilon]_\mu - \fA_-^\sharp (\mu - \fL_{\mu}^\eff)^{-1} \fA_+^\sharp) \, \mathfrak{Ans}_{\mu}(\bm \vphi) - \bar \fA^\flat \, \mathfrak{Ans}_{\mu}(\bm \vphi) \, .
\end{equs}
\end{lemma}

\begin{proof}

To shorten notations, we denote $\fs := \mathfrak{Ans}_{\mu}(\bm \vphi)$ throughout this proof. We write
\begin{equ} \label{e:GenAnsatzComputation}
(\mu - \bar \fL) \, \fs = (\mu - \fL_{\mu}^\eff) \, \fs + \fL_0^\fe[\Upsilon]_{\mu} \, \fs - \fA_+^\sharp \, \fs - \fA_-^\sharp \fs - \bar \fA^\flat \fs \, .
\end{equ}
By using the identity~\eqref{e:AnsatzIdentity} twice, we have
\begin{equs}
(&\mu - \fL_{\mu}^\eff) \, \fs = \fA_+^\sharp \, \fs + \bm \vphi \, , \\
-&\fA_-^\sharp \, \fs = - \fA_-^\sharp (\mu - \fL_{\mu}^\eff)^{-1} \fA_+^\sharp \, \fs - \fA_-^\sharp (\mu - \fL_{\mu}^\eff)^{-1} \bm \vphi \, .
\end{equs}
Plugging these two equalities in~\eqref{e:GenAnsatzComputation}, we end up with
\begin{equ} \label{e:GenAnsatzBis}
(\mu - \bar \fL) \, \fs = \bm \vphi + \fL_0^\fe[\Upsilon]_{\mu} \, \fs - \fA_-^\sharp (\mu - \fL_{\mu}^\eff)^{-1} \fA_+^\sharp \, \fs - \fA_-^\sharp (\mu - \fL_{\mu}^\eff)^{-1} \bm \vphi - \bar \fA^\flat \, \fs \, ,
\end{equ}
which concludes the proof.
\end{proof}

Let us stress that, on the right hand side of~\eqref{e:GenAnsatz}, the first two summands correspond to those appearing in the estimate~\eqref{e:AnsatzGen}, while the last two summands should be negligible: indeed, the third term can be controlled by the replacement lemma~\eqref{e:ReplFixPt}, and the fourth term can be controlled by the following lemma (the analogue of~\cite[Proposition 3.4]{CMT25}).

\begin{lemma} \label{lem:Aflat}
For every $\mu > 0$ and $\phi \in \fock{\rap}{}$, we have
\begin{equ} \label{e:Aflat}
\|(- \fL_0)^{-1/2} \, \fA_+^\flat \phi\| \vee \|(- \fL_0)^{-1/2} \, \fA_-^\flat \phi\| \lesssim \|\fN \, (- \fL_0^\fe)^{1/2} \, \phi\| \, .
\end{equ}
\end{lemma}

\begin{proof}
First, by an easy duality argument (and by letting $\mu \searrow 0$), it is sufficient to show for any $\phi \in \fock{}{}$ and $\mu > 0$
\begin{equ} \label{e:AflatProof}
\langle (\mu - \fL_0)^{-1} \, \fA_+^\flat \phi, \fA_+^\flat \phi \rangle \lesssim \langle \fN^2 (\mu - \fL_0^\fe) \, \phi, \phi \rangle \, , \quad \langle (\mu - \fL_0^\fe)^{-1} \, \fA_+^\flat \phi, \fA_+^\flat \phi \rangle \lesssim \langle \fN^2 (\mu - \fL_0) \, \phi, \phi\rangle \, .
\end{equ}
Further, by orthogonality of chaoses, it is sufficient to show the bound for any $\phi \in \fock{\rap,n}{}$ and $n \geq 1$.
\ss

We start with the first bound. Using~\eqref{e:CSDiag} and the notations~\eqref{e:XiNotation}, we obtain
\begin{equs}
\langle (\mu - \fL_0)^{-1} \, \fA_+^\flat \phi, \fA_+^\flat \phi \rangle &\lesssim n^2 \, n! \, \int_{\R^{2n}} \dd {k_{1:n}} |\fF \phi(k_{1:n})|^2 \sum_{i=1}^n |\fe \cdot k_i|^2 \, \int_{\R^2} \frac{\chi^\flat(p, k_i - p ; k_{1:n}) \, \dd p}{\mu + \Xi} \\
&\lesssim n^2 \, n! \, \int_{\R^{2n}} \dd {k_{1:n}} |\fF \phi(k_{1:n})|^2 \sum_{i=1}^n |\fe \cdot k_i|^2 \, \int_{\R^2} \frac{1_{|p| \leq 4 |k_{1:n}|} \, \dd p}{|k_{1:n}|^2} \\
&\lesssim n^2 \, \langle (- \fL_0^\fe) \, \phi, \phi \rangle \, ,
\end{equs}
where in the second line we used~\eqref{e:Triangular} and observed that, if $\chi^\flat(p, k_i - p ; k_{1:n}) = 1$, then either $|p| \leq 3 |k_{1:n}|$ or $|k_i - p| \leq 3 |k_{1:n}|$, so that in both cases it holds $|p| \leq 4 |k_{1:n}|$.

We can prove similarly the second bound. Indeed, by using~\eqref{e:CSDiag} and the notations~\eqref{e:XiNotation}, we obtain
\begin{equs}
\langle (\mu - \fL_0^\fe)^{-1} \, \fA_+^\flat \phi, \fA_+^\flat \phi \rangle &\lesssim n^2 \, n! \, \int_{\R^{2n}} \dd {k_{1:n}} |\fF \phi(k_{1:n})|^2 \sum_{i=1}^n |\fe \cdot k_i|^2 \, \int_{\R^2} \frac{\chi^\flat(p, k_i - p ; k_{1:n}) \, \dd p}{\mu + \Xi_0^\fe} \\
&\lesssim n^2 \, n! \, \int_{\R^{2n}} \dd {k_{1:n}} |\fF \phi(k_{1:n})|^2 \sum_{i=1}^n |\fe \cdot k_i|^2 \, \int_{\R^2} \frac{1_{|p| \leq 4 |k_{1:n}|} \, \dd p}{|\fe \cdot k_{1:n}|^2} \\
&\lesssim n^2 \langle (- \fL_0) \, \phi, \phi\rangle \, .
\end{equs}
This concludes the proof.
\end{proof}

We can now estimate $(\mu - \bar \fL) \, \mathfrak{Ans}_{\mu}(\bm \vphi)$ in $H^{-1}$-norm as in~\eqref{e:AnsatzGen}.

\begin{proof}[of~\eqref{e:AnsatzGen}]
First, by~\eqref{e:GenAnsatz} and the triangle inequality, we have
\begin{equs}
\|\fN_{*}^\kappa \alpha^{\fN} \, (- \fL_0)^{-1/2} \, &\big((\mu - \bar \fL) \, \mathfrak{Ans}_{\mu}(\bm \vphi) - \bm \vphi + \fA_-^\sharp (\mu - \fL_{\mu}^{\eff})^{-1} \bm \vphi\big)\| \\
&\leq \|\fN_{*}^\kappa \alpha^{\fN} \, (- \fL_0)^{-1/2} \big(\fL_0^\fe[\Upsilon]_{\mu} - \fA_-^\sharp (\mu - \fL_{\mu}^{\eff})^{-1} \fA_+^\sharp\big) \, \mathfrak{Ans}_{\mu}(\bm \vphi)\| \\
&\quad + \|\fN_{*}^\kappa \alpha^{\fN} \, (- \fL_0)^{-1/2} \, \bar \fA^\flat \, \mathfrak{Ans}_{\mu}(\bm \vphi)\| \, , \label{e:AnsatzGenComputation}
\end{equs}
Now, observe that we immediately deduce, respectively from~\eqref{e:ReplFixPt} and~\eqref{e:Aflat}, that for any $\phi \in \fock{\rap}{}$
\begin{equs}
\|\fN_{*}^{\kappa} \alpha^{\fN} \, (- \fL_0)^{-1/2} \big(\fL_0^\fe[\Upsilon]_{\mu} - \fA_-^\sharp (\mu - \fL_{\mu}^{\eff})^{-1} \fA_+^\sharp\big) \, \phi\| &\lesssim_{\kappa, \alpha} \|\fN_{*}^\kappa \alpha^{\fN} \, (- \fL_0^\fe)^{1/2} \, \phi\| \, , \\
\|\fN_{*}^\kappa \alpha^{\fN} \, (- \fL_0)^{-1/2} \, \bar \fA^\flat \, \phi\| &\lesssim_{\kappa, \alpha} \|\fN_{*}^{\kappa + 1} \alpha^{\fN} \, (- \fL_0^\fe)^{1/2} \, \phi\| \, .
\end{equs}
By taking $\phi = \mathfrak{Ans}_{\mu}(\bm \vphi)$ and using~\eqref{e:AnsatzH1fe}, we deduce
\begin{equs}
\|\fN_{*}^\kappa \alpha^{\fN} \, (- \fL_0)^{-1/2} \, [(\mu - \bar \fL) \, \mathfrak{Ans}_{\mu}(\bm \vphi) &- \bm \vphi + \fA_-^\sharp (\mu - \fL_{\mu}^{\eff})^{-1} \bm \vphi]\| \\
&\lesssim_{\kappa, \alpha} \|\fN_{*}^{\kappa + 1} \alpha^{\fN} \, (- \fL_0^\fe\big[\tilde \Upsilon^\theta\big]_\mu)^{1/2} (\mu - \fL_{\mu}^\eff)^{-1} \bm \vphi\| \, ,
\end{equs}
which readily shows~\eqref{e:AnsatzGen}.
\end{proof}

Finally, we estimate the $H^1$ norm of the Ansatz $\|(\mu - \fS)^{1/2} \, \mathfrak{Ans}_{\mu}(\bm \vphi)\|$ as in~\eqref{e:AnsatzH1precise}. For this last estimate, we assume that $\bm \vphi$ is homogeneous (that is, belonging to $\fock{n_0}{}$ for some $n_0 \geq 0$) as it considerably simplifies the proof, and because we only need~\eqref{e:AnsatzH1precise} to estimate the diffusivity, for which the relevant $\bm \vphi$ belongs to the second chaos $\fock{2}{}$. 

\begin{proof}[of~\eqref{e:AnsatzH1precise}]
Let $\bar \fL^\sharp := \fS + \bar \fA^\sharp = \bar \fL - \bar \fA^\flat$. We observe that the symmetric part of $\bar \fL^\sharp$ is given by $\fS$ (on $\fock{\rap}{} \subset \fock{}{}$), so that we deduce
\begin{equs}
\|(\mu - \fS)^{1/2} \, \mathfrak{Ans}_{\mu}(\bm \vphi)\|^2 &= \langle (\mu - \bar \fL^\sharp) \, \mathfrak{Ans}_{\mu}(\bm \vphi), \mathfrak{Ans}_{\mu}(\bm{\vphi}) \rangle \, , \\
&= \langle \bm \vphi, \mathfrak{Ans}_{\mu}(\bm \vphi) \rangle \\
&\quad + \langle - \fA_-^\sharp (\mu - \fL_{\mu}^\eff)^{-1} \bm \vphi, \mathfrak{Ans}_{\mu}(\bm \vphi) \rangle \\
&\quad + \langle (\fL_0^\fe[\Upsilon]_{\mu} - \fA_-^\sharp (\mu - \fL_{\mu}^\eff)^{-1} \fA_+^\sharp) \, \mathfrak{Ans}_{\mu}(\bm \vphi), \mathfrak{Ans}_{\mu}(\bm \vphi) \rangle \, , \\
&=: \one + \two + \three
\end{equs}
where we used Lemma~\ref{lem:GenAnsatz} in the second line. We recall that we assume $\bm \vphi \in \fock{\rap,n_0}{}$ for some $n_0 \geq 0$; in particular, the summand of index $n$ in the definition~\eqref{e:TrueAnsatz} of the Ansatz belongs to $\fock{n_0 + n}{}$. On the one hand, we deduce that only the summand $n=0$ contributes in $\one$ so that
\begin{equ}
\one = \langle \bm \vphi, (\mu - \fL_{\mu}^{\eff})^{-1} \bm \vphi \rangle = \|(\mu - \fL_{\mu}^{\eff})^{-1/2} \bm \vphi\|^2 \, .
\end{equ}
On the other hand, we deduce that $\two = 0$ because $- \fA_-^\sharp (\mu - \fL_{\mu}^\eff)^{-1} \bm \vphi \in \fock{n_0 - 1}{}$. Finally, we can bound $\three$ by using the first replacement lemma~\eqref{e:ReplFixPt}, as follows
\begin{equ}
|\three| \lesssim \|(- \fL_0^\fe)^{1/2} \, \mathfrak{Ans}_{\mu}(\bm \vphi)\|^2 \lesssim \|(-\fL_0^\fe\big[\tilde \Upsilon^\theta\big]_\mu)^{1/2} \, (\mu - \fL_{\mu}^\eff)^{-1} \bm \vphi\|^2 \, ,
\end{equ}
where in the last bound we used~\eqref{e:AnsatzH1fe}. Gathering the previous results, we deduce~\eqref{e:AnsatzH1precise} and conclude the proof.
\end{proof}
\ms

\subsection{Proof of Proposition~\ref{p:AnsatzDiff}}

\label{app:AnsatzProof}

In this section, we prove Proposition~\ref{p:AnsatzDiff}. We first state useful estimates on $\fB_+ \cV^\delta$.

\begin{lemma} \label{lem:InputB+}
For every $\mu \geq 0$, $\delta \in (0,1]$, $\theta \in (0,1]$, $\gamma \in (0,1]$, it holds
\begin{equs}
\sqrt{\mu} \, \|(\mu - \fL_{\mu}^{\eff})^{-1} \fB_+ \cV^\delta\| &\lesssim 1 \, , \label{e:InputB+L2} \\
\|(- \fL_0^\fe\big[\tilde \Upsilon^\theta\big]_\mu)^{1/2} (\mu - \fL_{\mu}^{\eff})^{-1} \fB_+ \cV^\delta\| &\lesssim_{\theta} \log_{*}(1/\mu)^{\theta / 3} \, , \label{e:InputB+H1e}\\
\|(\mu - \fL_0)^{(1+\gamma)/2} (\mu - \fL_{\mu}^{\eff})^{-1} \fB_+ \cV^\delta\| &\lesssim_\gamma 1 \, , \label{e:InputB+H1gamma} \\
\limsup_{\delta \to 0} \Big|\|(\mu - \fL_{\mu}^\eff)^{-1/2} \, \fB_+ \cV^\delta\|^2 - \frac{d^{\lim}}{2} \log_{*}(&1/\mu)^{2/3}\Big| \lesssim 1 \, , \label{e:InputB+H1} \\
\|\fA_0 \cV^\delta\| \xrightarrow[\delta \to 0]{} 0 \, ,  \quad \|\fA_-^\sharp (\mu - \fL_{\mu}^\eff)^{-1} \fB_+ &\cV^\delta\| \xrightarrow[\delta \to 0]{} 0 \, , \label{e:InputB+A-A0}
\end{equs}
\end{lemma}

\begin{proof}[of Lemma~\ref{lem:InputB+}]
In this proof, we repeatedly use the following fact: for any positive diagonal operator $\fD$ with kernel $d(\cdot)$ it holds
\begin{equs}
\|\fD^{1/2} \, \fB_+ \cV^\delta\|^2 &= \langle \fD \, \fB_+ \cV^\delta, \fB_+ \cV^\delta\rangle \\
&= \frac{\lambda^2 \rho (1-\rho)}{2 \pi^2} \int_{\R^2} \dd k \, |\fF \cV^\delta(k)|^2 \, \int_{\R^2} \dd p \, \Pi_{p, k-p} \, d(p,k-p) \, \label{e:B+Diag} \\
&\lesssim \sup_{k \in B(0, \delta)} \, \int_{\R^2} \dd p \, \Pi_{p, k-p} \, d(p,k-p) \, , \label{e:B+DiagBound}
\end{equs}
where in the last line we used that $\|\cV^\delta\|_{L^2(\R^2)} = 1$ and $\mathrm{Supp}(\fF \cV^\delta) \subset B(0, \delta)$ (recall Definition~\ref{def:V}).
\ss

We start with~\eqref{e:InputB+L2}. By using the trivial bound $- \fL_{\mu}^\eff \geq - \|D_0^{-1}\|^{-1} \fL_0$ and~\eqref{e:B+DiagBound}, we deduce
\begin{equ}
\|(\mu - \fL_{\mu}^\eff)^{-1} \fB_+ \cV^\delta\|^2 \lesssim \|(\mu - \fL_0)^{-1} \, \fB_+ \cV^\delta\|^2 \lesssim \int_{\R^2} \frac{\dd p}{\mu + |p|^2} \lesssim \mu^{-1} \, .
\end{equ}

We turn to~\eqref{e:InputB+H1e}. By using~\eqref{e:B+DiagBound}, the notations~\eqref{e:XiNotation} (with $k_{1:n} = k$),~\eqref{e:SecondReplLemmaInt} with $(f,g) = (\Upsilon, \tilde \Upsilon^\theta)$ and~\eqref{e:JJ'Ups}, we deduce
\begin{equs}
\|(- \fL_0^\fe[\tilde \Upsilon^\theta]_\mu)^{1/2} \, (\mu - \fL_{\mu}^{\eff})^{-1} \, \fB_+ \cV^\delta\|^2 &\lesssim \sup_{k \in B(0, \delta)} \int_{\R^2} \dd p \, \frac{\Pi_{p,k-p} \, \tilde \Upsilon^\theta(\Log(\mu + \Xi_0))}{(\mu + \Xi + \Upsilon(\Log(\mu + \Xi_0)) \Xi_0^\fe)^2} \\
&\lesssim_\theta \sup_{k \in B(0, \delta)} \tilde \Upsilon^\theta(\Log(\mu + |k|^2/2)) + 1 \, , \\
&\lesssim \log_{*}(1/\mu)^{2 \theta / 3} \, .
\end{equs}

We continue with~\eqref{e:InputB+H1gamma}. By using the trivial bound $- \fL_{\mu}^\eff \geq - \|D_0^{-1}\|^{-1} \fL_0$,~\eqref{e:B+DiagBound}, and recalling that $\gamma \in (0,1]$, we deduce
\begin{equ}
\|(\mu - \fL_0)^{(1+\gamma)/2} \, (\mu - \fL_{\mu}^\eff)^{-1} \fB_+ \cV^\delta\|^2 \lesssim \|(\mu - \fL_0)^{(-1+\gamma)/2} \, \fB_+ \cV^\delta\|^2 \lesssim \int_{\R^2} \frac{\1_{|p|_\infty \leq \pi} \, \dd p}{ |p|^{2 - 4 \gamma}} \lesssim_\gamma 1 \, .
\end{equ}

We tackle~\eqref{e:InputB+H1}. By using~\eqref{e:B+Diag} and the notations~\eqref{e:XiNotation} (with $k_{1:n} = k$), we can write
\begin{equ} \label{e:InputB+H1detail}
\|(\mu - \fL_{\mu}^{\eff})^{-1/2} \fB_+ \cV^\delta\|^2 = \int_{\R^2} \dd k \, |\fF \cV^\delta(k)|^2 \, \frac{\lambda^2 \rho (1-\rho)}{2 \pi^2} \int_{\R^2} \frac{\dd p \, \Pi_{p, k - p}}{\mu + \Xi + \Upsilon(\Log(\mu + \Xi_0)) \, \Xi_0^\fe} \, .
\end{equ}
Then, we evaluate the last integral by using~\eqref{e:SecondReplLemmaInt} with $f = \Upsilon$ and~\eqref{e:JJ'Ups}, which leads to
\begin{equ}
\Big|\frac{\lambda^2 \rho(1-\rho)}{2 \pi^2} \int_{\R^2} \frac{\dd p \, \Pi_{p, k - p}}{\mu + \Xi + \Upsilon(\Log(\mu + \Xi_0)) \, \Xi_0^\fe} - \frac{1}{2} \, \Upsilon(\Log(\mu + |k|^2/2))\Big| \lesssim 1 \, .
\end{equ}
Then, by using~\eqref{e:AsymptoticUps} and the trivial fact $\Upsilon(\Log(\mu + |k|^2/2)) \to \Upsilon(L(\mu))$ as $k \to 0$, we deduce
\begin{equ}
\limsup_{\delta \to 0} \sup_{k \in B(0, \delta)} \Big|\frac{\lambda^2 \rho(1-\rho)}{2 \pi^2} \int_{\R^2} \frac{\dd p \, \Pi_{p, k - p}}{\mu + \Xi + \Upsilon(\Log(\mu + \Xi_0)) \, \Xi_0^\fe} - \frac{d^{\lim}}{2} \log_{*}(1/\mu)^{2/3} \Big| \lesssim 1 \, .
\end{equ}
By plugging this estimate in~\eqref{e:InputB+H1detail} and recalling that $\|\cV^\delta\|_{L^2(\R^2)} = 1$ and $\mathrm{Supp}(\cF \cV^\delta) \subset B(0, \delta)$, we deduce the desired estimate~\eqref{e:InputB+H1}.
\ss

Finally, we turn to~\eqref{e:InputB+A-A0}. First, we readily deduce from~\eqref{e:fA0} and Definition~\ref{def:V} (in particular,~\eqref{e:SqrtDirac}) that it holds
\begin{equ}
\|\fA_0 \cV^\delta\|^2 \lesssim \int_{\R^2} |\fe \cdot k|^2 \, |\fF \cV^\delta(k)|^2 \, \dd k \lesssim \delta^2 \, ,
\end{equ}
which shows the first statement of~\eqref{e:InputB+A-A0} by taking the limit as $\delta \to 0$. Second, by definition of $\fA_-^\sharp$ and by using the notations~\eqref{e:XiNotation} (with $k_{1:n} = k$), it holds
\begin{equs}
\|\fA_-^\sharp (\mu - \fL_{\mu}^{\eff})^{-1} &\fB_+ \cV^\delta\|^2 \\
&= \frac{\lambda^4 \rho^2 (1-\rho)^2}{\pi^4} \int_{\R^2} \dd k \, |\fe \cdot k|^2 \, |\fF \cV^\delta(k)|^2 \, \int_{\R^2} \frac{\dd p \, \Pi_{p,k-p} \, \chi^\sharp(p, k - p ; k)}{(\mu + \Xi + \Upsilon(\Log(\mu + \Xi_0)) \, \Xi_0^\fe)^2} \\
&\lesssim \mu^{-2} \int_{\R^2} \dd k \, |\fe \cdot k|^2 \, |\fF \cV^\delta(k)|^2 \\
&\lesssim \mu^{-2} \, \delta^2 \, ,
\end{equs}
which shows the second statement of~\eqref{e:InputB+A-A0} by taking the limit as $\delta \to 0$ (even though the estimate deteriorates significantly as $\mu \to 0$). This concludes the proof.
\end{proof}

We can now prove Proposition~\ref{p:AnsatzDiff}.

\begin{proof}[of Proposition~\ref{p:AnsatzDiff}]

First, we highlight that the estimates~\eqref{e:InputB+L2},~\eqref{e:InputB+H1e},~\eqref{e:InputB+H1gamma} in Lemma~\ref{lem:InputB+} trivially hold if we add a factor $\alpha^\fN$ on the left hand side, up to multiplying the hidden constant in $\lesssim$ by $\alpha^2$, simply because $\fB_+ \cV^\delta \in \fock{2}{}$. Now, we want to apply Proposition~\ref{p:GenericAnsatz} with $\bm \vphi = \fB_+ \cV^\delta$ and $\theta = 1/2 + 3 \eps$; we can bound the constants appearing in~\eqref{e:AnsatzConstantK} by using Lemma~\ref{lem:InputB+} (more precisely, by using~\eqref{e:InputB+L2},~\eqref{e:InputB+H1e},~\eqref{e:InputB+H1gamma} respectively), as follows
\begin{equs}
\, \sqrt{\mu} \, K_{\alpha, 0}^0(\fB_+ \cV^\delta) &\lesssim_\alpha 1 \, , \label{e:InputB+K0} \\
K_{\alpha, 0}^{1/2, \fe, 1/2 + 3 \eps}(\fB_+ \cV^\delta) &\lesssim_{\alpha, \eps} \log_{*}(1/\mu)^{1/6 + \eps} , \qquad \ \label{e:InputB+K1/2fe} \\
K_{\alpha, 0}^{(1+\gamma)/2}(\fB_+ \cV^\delta) &\lesssim_{\alpha, \gamma} 1 \, . \label{e:InputB+K1/2gamma}
\end{equs}
We then readily deduce the desired estimates~\eqref{e:AnsatzDiffL2},~\eqref{e:AnsatzDiffH1fe} and~\eqref{e:AnsatzDiffH1+gamma} from~\eqref{e:AnsatzL2},~\eqref{e:AnsatzH1fe} and~\eqref{e:AnsatzH1+gamma} respectively.

Now, we show~\eqref{e:AnsatzDiffGen}; first, by the triangle inequality, we have
\begin{equs}
\|\alpha^\fN \, &(\mu - \fL_0)^{-1/2} \, \big(\fL \, \mathfrak{Ans}_{\mu}(\fB_+ \cV^\delta) + \fB_+ \cV^\delta\big)\| \\
&\leq \|\alpha^\fN \, (- \fL_0)^{-1/2} \, \big((\mu - \bar \fL) \, \mathfrak{Ans}_{\mu}(\fB_+ \cV^\delta) - \fB_+ \cV^\delta + \fA_-^\sharp (\mu - \fL_{\mu}^{\eff})^{-1} \fB_+ \cV^\delta \big)\| \\
&\quad + \|\alpha^\fN \, \mu^{-1/2} \, \mu \, \mathfrak{Ans}_{\mu}(\fB_+ \cV^\delta)\| \\
&\quad + \|\alpha^\fN \, \mu^{-1/2} \, \fA_0 \, \mathfrak{Ans}_{\mu}(\fB_+ \cV^\delta)\| + \|\alpha \, \mu^{-1/2} \, \fA_-^\sharp (\mu - \fL_{\mu}^{\eff})^{-1} \fB_+ \cV^\delta\| \, . \label{e:AnsatzDiffGenProof} 
\end{equs}
Then, we bound the first summand by using~\eqref{e:AnsatzGen} and~\eqref{e:InputB+K1/2fe}, the second summand by using~\eqref{e:InputB+L2}, and we notice that the two lasts summands vanish as $\delta \to 0$ because of~\eqref{e:InputB+A-A0}. We give more details for the second-to-last summand: first we observe that $\fA_0$ commutes with $\mathfrak{Ans}_{\mu}$ and $\fB_+$ (cf. Remark~\ref{rmk:fA0}); we deduce, by using also~\eqref{e:AnsatzL2} and the trivial bound $\|\fB_+ \cdot\| \lesssim \|\cdot\|$,
\begin{equ}
\|\alpha^\fN \, \fA_0 \, \mathfrak{Ans}_{\mu}(\fB_+ \cV^\delta)\| = \|\alpha^\fN \, \mathfrak{Ans}_{\mu}(\fB_+ \fA_0 \cV^\delta)\| \lesssim_\alpha \|\alpha^\fN \, (\mu - \fL_{\mu}^{\eff})^{-1} \fB_+ \fA_0 \cV^\delta\| \lesssim \alpha^2 \mu^{-1} \|\fA_0 \cV^\delta\| \, ,
\end{equ}
which indeed vanishes as $\delta \to 0$ by~\eqref{e:InputB+A-A0}. Coming back to the bounds of the summands in~\eqref{e:AnsatzDiffGenProof}, we deduce
\begin{equ}
\limsup_{\delta \to 0} \|\alpha^\fN \, (\mu - \fL_0)^{-1/2} \, \big(\fL \, \mathfrak{Ans}_{\mu, \delta}^{\fB+} + \fB_+ \cV^\delta\big)\| \lesssim_{\alpha, \eps} \log_{*}(1/\mu)^{1/6 + \eps} + 1 + 0 + 0 \lesssim \log_{*}(1/\mu)^{1/6 + \eps} \, ,
\end{equ}
which shows~\eqref{e:AnsatzDiffGen}.
\ss

Finally, we tackle~\eqref{e:AnsatzDiffH1}; first, by the triangle inequality, we have
\begin{equs}
\big|\|(- \fS)^{1/2} \, &\mathfrak{Ans}_{\mu}(\fB_+ \cV^\delta)\|^2 - \tfrac{1}{2} \, d^{\lim} \log_{*}(1/\mu)^{2/3}\big| \\
&\leq \big|\|(\mu - \fS)^{1/2} \, \mathfrak{Ans}_{\mu}(\fB_+ \cV^\delta)\|^2 - \|(\mu - \fL_{\mu}^{\eff})^{-1/2} \fB_+ \cV^\delta\|^2\big| \\
&\quad + \mu \, \|\mathfrak{Ans}_{\mu}(\fB_+ \cV^\delta)\|^2 + \big|\|(\mu - \fL_{\mu}^{\eff})^{-1/2} \, \fB_+ \cV^\delta\|^2 - \tfrac{1}{2} \, d^{\lim} \log_{*}(1/\mu)^{2/3}\big| \, .
\end{equs}
Then, we can bound each summand by using, respectively,~\eqref{e:AnsatzH1precise} together with~\eqref{e:InputB+K1/2fe},~\eqref{e:AnsatzDiffL2},~\eqref{e:InputB+H1detail}; we end up with
\begin{equ}
\limsup_{\delta \to 0} \big|\|(- \fS)^{1/2} \, \mathfrak{Ans}_{\mu}(\fB_+ \cV^\delta)\|^2 - \tfrac{1}{2} \, d^{\lim} \log_{*}(1/\mu)^{2/3}\big| \lesssim_{\eps} \log_{*}(1/\mu)^{1/6 + \eps} + 1 + 1 \lesssim \log_{*}(1/\mu)^{1/6 + \eps} \, ,
\end{equ}
which proves~\eqref{e:AnsatzDiffH1} and concludes the proof.
\end{proof}

\section{Proof of some technical statements}

\subsection{Comparison between the hardcore and softcore $H^1$ norms (in the spirit of~\cite{LQSY04})}

\label{app:HCSC}

First, we prove Lemma~\ref{lem:ZiBoundHC}. The idea is to rely on Lemma~\ref{lem:ZiBound}, by suitably extending elements of the hard core Fock space $\fockdo{d}{}$ to the full Fock space $\fockd{d}{}$; in fact, this approach works well here because it turns out that we can reduce the problem to a small chaos degree $d$.

\begin{proof}[of Lemma~\ref{lem:ZiBoundHC}]

By orthogonality of chaoses, it is sufficient to prove the bound for any $g \in \fockdo{n}{}$ and $n \geq 1$. Let us define, for any $d \geq 1$, the function $\rho_d(g) \in \fockdo{d}{}$ given by
\begin{equ}
\rho_d(g) \, (x_{1:d}) \eqdef \sum_{x_{d+1:n} \in \Z^{2(n-d)}} g(x_{1:n})^2 \, .
\end{equ}
We first notice that
\begin{equs}
\|\1_{\cZ_{\geq 1}(l)} \, g\|_{\ell^2(\Z^{2n})}^2 &\leq n^2 \, \|\1_{|x_1 - x_2|_\infty \leq l} \, g\|_{\ell^2(\Z^{2n})}^2 = n^2 \, \|\1_{|x_1 - x_2|_\infty \leq l} \, \sqrt{\rho_2(g)}\|_{\ell^2(\Z^{4})}^2 \\
\|\1_{\cZ_{\geq 2}(l)} \, g\|_{\ell^2(\Z^{2n})}^2 &\leq n^3 \, \|\1_{|x_1 - x_3|_\infty \vee |x_2 - x_3|_\infty \leq l} \, g\|_{\ell^2(\Z^{2n})}^2 + n^4 \, \|\1_{|x_1 - x_3|_\infty \vee |x_2 - x_4|_\infty \leq l} \, g\|_{\ell^2(\Z^{2n})}^2 \\
&= n^3 \, \|\1_{|x_1 - x_3|_\infty \vee |x_2 - x_3|_\infty \leq \ell} \sqrt{\rho_3(g)}\|_{\ell^2(\Z^{6})}^2 + n^4 \, \|\1_{|x_1 - x_3|_\infty \vee |x_2 - x_4|_\infty \leq l} \sqrt{\rho_4(g)}\|_{\ell^2(\Z^{8})}^2 \, .
\end{equs}
Now we extend the functions $\sqrt{\rho_d(g)}$ to the diagonal $DI$ as follows: we define $G_d \in \fock{d}{}$ as
\begin{equ}
G_d(x_{1:d}) \eqdef \frac{1}{\# N_d(x_{1:d})} \sum_{y_{1:d} \in N_d(x_{1:d})} \sqrt{\rho_d(g)} \, (y_{1:d}) \, ,
\end{equ}
where $N_d(x_{1:d}) := \{x_{1:d}\}$ if $x_{1:d} \in HC$ and $N_d(x_{1:d}) := HC \cap B_\infty(x_{1:d}, d)$ else (we choose the $\ell^\infty$-ball of radius $d$ in order to guarantee that $N_d(x_{1:d})$ is never empty). Since the function $G_d$ extends $\sqrt{\rho_d(g)}$, we deduce in particular that
\begin{equs}
\|\1_{|x_1 - x_2|_\infty \leq l} \sqrt{\rho_2(g)}\|_{\ell^2(\Z^4)}^2 \leq \|&\1_{|x_1 - x_2|_\infty \leq \ell} \, G_2\|_{\ell^2(\Z^6)}^2 \lesssim l^2 \, \log_{*}(1/\mu) \, \|(\mu - \cS)^{1/2} \, G_2\|_{\ell^2(\Z^4)}^2 \, , \\
\|\1_{|x_1 - x_3|_\infty \vee |x_2 - x_3|_\infty \leq l} \sqrt{\rho_3(g)}\|_{\ell^2(\Z^6)}^2 &\leq \|\1_{|x_1 - x_3|_\infty \vee |x_2 - x_3|_\infty \leq l} \, G_3\|_{\ell^2(\Z^6)}^2 \lesssim l^4 \, \|(- \cS)^{1/2} \, G_3\|_{\ell^2(\Z^6)}^2 \, , \\
\|\1_{|x_1 - x_3|_\infty \vee |x_2 - x_4|_\infty \leq l} \sqrt{\rho_4(g)}\|_{\ell^2(\Z^8)}^2 &\leq \|\1_{|x_1 - x_3|_\infty \vee |x_2 - x_4|_\infty \leq l} \, G_4\|_{\ell^2(\Z^8)}^2 \lesssim l^4 \, \|(- \cS)^{1/2} \, G_4\|_{\ell^2(\Z^8)}^2 \, ,
\end{equs}
where we used~\eqref{e:ZiBoundReductionBis} in the last bound of each line. Therefore, it only remains to show that, for every $d \geq 1$ and $\mu \geq 0$, it holds
\begin{equ} \label{ZiBoundHCclaim}
\|(\mu - \cS)^{1/2} \, G_d\|_{\ell^2(\Z^{2d})}^2 \lesssim_d \|(\mu - S)^{1/2} \, g\|_{\ell^2(\Z^{2d})}^2 \, .
\end{equ}
First, we recall that we equipped $\Z^{2d}$ with a graph structure in the preamble. It is easily checked that, for any given $R \geq 1$, there exists a collection of paths $\Gamma(x_{1:d}, y_{1:d})$, for every $x_{1:d}, y_{1:d}$ in $\Z^{2d} \cap HC$ satisfying $|x_{1:d} - y_{1:d}|_\infty \leq R$, such that: $\Gamma(x_{1:d}, y_{1:d})$ connects $x_{1:d}$ to $y_{1:d}$ and only passes through vertices in $HC$; the length of $\Gamma(x_{1:d}, y_{1:d})$ is bounded by a constant $C = C(d, R)$. In particular, each edge of $\Z^{2d}$ is crossed by at most $C' = C'(d, R)$ distinct paths (because if $\Gamma(x_{1:d}, y_{1:d})$ passes through an edge $e$ then both $x_{1:d}$ and $y_{1:d}$ must be at graph distance smaller than $C$ from $e$). Then, we take $R := l_0 + 2 d$ and deduce by an easy application of the Cauchy-Schwarz inequality (in the spirit of the canonical path method used in Dirichlet form comparisons) that
\begin{equ}
\|(\mu - \cS)^{1/2} \, G_d\|_{\ell^2(\Z^{2d})}^2 \lesssim_d \|(\mu - S)^{1/2} \sqrt{\rho_d(g)}\|_{\ell^2(\Z^{2d})}^2 \, .
\end{equ}
Besides, another easy application of the Cauchy-Schwarz inequality (more precisely, the upper bound $\big(\big(\sum a_j^2\big)^{1/2} - \big(\sum b_j^2\big)^{1/2}\big)^2 \leq \sum (a_j - b_j)^2$) shows that
\begin{equ}
\|(\mu - S)^{1/2} \sqrt{\rho_d(g)}\|_{\ell^2(\Z^d)}^2 \leq \, \|(\mu - S)^{1/2} g\|_{\ell^2(\Z^{2d})}^2 \, ,
\end{equ}
which concludes the proof.
\end{proof}

Finally, we show the bound~\eqref{e:BoundQH1dot}, which states that the $H^1$ norm of $\fsQ g$ can be bounded by the one of $g$, up to a polynomial cost in the number operator.

\begin{proof}[of~\eqref{e:BoundQH1dot}]
Let $g \in \fockdo{n}{}$. We recall that we introduced a graph structure on $\Z^{2n}$ in the preamble. We start by writing 
\begin{equs}
\|(- \cS)^{1/2} \, \fsQ g\|_{\ell^2(\Z^{2n})}^2 &\asymp \sum_{x_{1:n} \sim \, y_{1:n} \in \Z^{2n}} \big(\fsQ g(y_{1:n}) - \fsQ g(x_{1:n})\big)^2 \\
&= \sum_{0 \leq i, j \leq 2} \sum_{x_{1:n} \sim \, y_{1:n} \in \Z^{2n}} \1_{\cE_i}(x_{1:n}) \, \1_{\cE_j}(y_{1:n}) \, \big(\fsQ g(y_{1:n}) - \fsQ g(x_{1:n})\big)^2 \, ,
\end{equs}
where we have denoted $\cE_0 := HC$, $\cE_1 := DI \cap \cZ_{1}$ and $\cE_2 := DI \cap \cZ_{\geq 2}$ (we recall that $\cZ_\cdot = \cZ_\cdot(2 l_0 + 1)$). We denote by $\Phi_{i,j}$ the summand relative to indices $i,j$ (note that $\Phi_{i,j} = \Phi_{j,i}$).
\ss

First, recalling that $Q g \equiv g$ on $\cE_0 = HC$, we immediately deduce
\begin{equ}
\Phi_{0,0} \lesssim \|(- \cS)^{1/2} g\|_{\ell^2(\Z^{2n})}^2 \, .
\end{equ}

Next, observe that $\Phi_{2,0} + \Phi_{2,1} + \Phi_{2,2}$ equals
\begin{equs}
\sum_{x_{1:n} \in \cE_2, y_{1:n} \in \Z^{2n} \colon x_{1:n} \sim y_{1:n}} \fsQ g(y_{1:n})^2 &\lesssim n \, \sum_{y_{1:n} \in \cZ_{\geq 2}(3 l_0 + 1)} \fsQ g(y_{1:n})^2 \lesssim n \, \|1_{\cZ_{\geq 2}(4 l_0 + 1)} \, g\|^2 \, ,
\end{equs}
where in the first bound we used the simple fact that any $y_{1:n} \in \Z^{2n}$ has at most $\# \mathrm{Supp}(s) \times n$ neighbours, while in the second bound we applied~\eqref{e:BoundL2QProof} together with the observation that $\fsQ g$ and $\fsQ(1_{\cZ_{\geq 2}(4 l_0 + 1)} \, g)$ coincide on $\cZ_{\geq 2}(3 l_0 + 1)$. Then, by using~\eqref{e:Z2BoundHC}, we deduce
\begin{equ}
\Phi_{2,0} + \Phi_{2,1} + \Phi_{2,2} \lesssim n^5 \, \|(- S)^{1/2} g\|_{\ell^2(\Z^{2n})}^2 \, .
\end{equ}

Now, we turn to $\Phi_{1,1}$. Observe that, if $x_{1:n} \sim y_{1:n}$ with $x_{1:n}, y_{1:n} \in \cE_1$, then there exist distinct indices $a, b, c$ such that $x_a = x_b$, $y_c - x_c \in \mathrm{Supp}(s)$ and $y_d = x_d$ for every $d \ne c$. Hence, we deduce
\begin{equ}
\Phi_{1,1} \leq n^3 \sum_{z \in \mathrm{Supp}(s)} \sum_{x_{1:n} \in \cE_1 \, : \, x_1 = x_2} \big(\fsQ g(x_1, x_1, x_3 + z, x_{4:n}) -  \fsQ g(x_1, x_1, x_3, x_{4:n})\big)^2 \, .
\end{equ}
Then, using the definition~\eqref{e:defQ} of $\fsQ f$ and the Cauchy-Schwarz inequality we deduce
\begin{equ}
\Phi_{1,1} \lesssim n^3 \sum_{z \in \mathrm{Supp}(s)} \sum_{x_{1:n} \in \cE_1 \, : \, x_1 = x_2} \big(g(x_1 + z_0, x_1, x_3 + z, x_{4:n}) -  g(x_1 + z_0, x_1, x_3, x_{4:n})\big)^2 \, .
\end{equ}
Now, because $x_{1:n} \in \cE_1 = DI \cap \cZ_1(2 l_0 + 1)$ with $x_1 = x_2$, we deduce that $(x_1 + z_0, x_1, x_3 + z, x_{4:n})$ and $(x_1 + z_0, x_1, x_3, x_{4:n})$ belong to $HC$ (and are neighbours). Thus, the last sum over $x_{1:n}$ is bounded, up to a multiplicative constant, by $\|(- S)^{1/2} g\|_{\ell^2(\Z^{2n})}^2$; we end up with
\begin{equ}
\Phi_{1,1} \lesssim n^3 \, \|(-S)^{1/2} g\|_{\ell^2(\Z^{2n})}^2 \, .
\end{equ}

Finally, we tackle $\Phi_{1,0}$. Observe that, if $x_{1:n} \sim y_{1:n}$ with $x_{1:n} \in \cE_1$ and $y_{1:n} \in \cE_0$, then there exist distinct indices $a \ne b$ such that $x_a = x_b$, $y_a - x_a \in \mathrm{Supp}(s)$ and $y_c = x_c$ for every $c \ne a$. Recalling the definition~\eqref{e:defQ} of $\fsQ$, we deduce
\begin{equs}
\Phi_{1,0} &\leq n^2 \sum_{z \in \mathrm{Supp}(s)} \sum_{x_{1:n} \in \cE_1 \colon x_1 = x_2} \big(g(x_1 + z_0, x_{2:n}) - g(x_1 + z, x_{2:n})\big)^2 \, .
\end{equs}
Now, we make the following observation\footnote{Here is a proof: we distinguish the cases $z \ne -z'$ and $z = -z'$. In the first case, we can take the path $z \to z + z' \to z'$. In the second case, let $z'' \in \mathrm{Supp}(s) \setminus \R z$ (it exists because $D_0$ is not degenerate), then we use the previous case to first connect $z$ to $z''$ and then $z''$ to $z'$.}: for any $z, z'$ in $\mathrm{Supp}(s)$, there exists a path $\Gamma(z,z')$ of length at most $4$ that connects $z$ to $z'$ in $(\Z^2 \setminus \{0\}) \cap B_\infty(0, 2 l_0)$. Then, by using the Cauchy-Schwarz inequality we deduce
\begin{equ}
\Phi_{1,0} \lesssim n^2 \sum_{z \in \mathrm{Supp}(s)} \sum_{(v, v') \in \Gamma(z,z_0)} \sum_{x_{1:n} \in \cE_1 \colon x_1 = x_2}\, \big(g(x_1 + v', x_{2:n}) - g(x_1 + v, x_{2:n})\big)^2 \, .
\end{equ}
Now, because $x_{1:n} \in \cE_1 = DI \cap \cZ_1(2 l_0 + 1)$ with $x_1 = x_2$ and because $\Gamma(z,z')$ takes values in $B_\infty(0, 2 l_0) \setminus \{0\}$, we deduce that $(x_1 + v', x_{2:n})$ and $(x_1 + v, x_{2:n})$ belong to $HC$ (and are neighbours). Thus the last sum over $x_{1:n}$ is bounded, up to a multiplicative constant, by $\|(-S)^{1/2} g\|_{\ell^2(\Z^{2n})}^2$ ; we end up with
\begin{equ}
\Phi_{1,0} \lesssim n^2 \, \|(-S)^{1/2} g\|_{\ell^2(\Z^{2n})}^2.
\end{equ}

Collecting all the previous bounds, we end up with
\begin{equ}
\|(-\cS)^{1/2} \fsQ g\|^2 \lesssim n^5 \, \|(- S)^{1/2} g\|^2 \, ,
\end{equ}
which concludes the proof of~\eqref{e:BoundQ}.
\end{proof}

\subsection{An elementary bilinear interpolation lemma}

We state an elementary bilinear interpolation lemma that will be useful for the proof of Proposition~\ref{p:HardcoreRemoval} when comparing the operators $A_0$ and $\cA_0$.

\begin{lemma} \label{lem:Interpolation}
Let $\mathscr S$ be a bounded positive definite self-adjoint operator acting on a Hilbert space $(H, \|\cdot\|)$, and let $B \colon H \times H \to V$ be a continuous bilinear map taking values in a normed vector space $(V, \|\cdot\|_V)$. We assume that $\mathscr S$ is bounded by $C > 0$, and that there exist $\alpha, \beta \geq 0$ such that for every $f, g \in H$
\begin{equ} \label{e:AssumptionInterpolation}
\|B(f,g)\|_V \lesssim \begin{cases}
\|\mathscr S^\alpha f\| \, \|\mathscr S^\beta g\| \\
\|\mathscr S^\beta f\| \, \|\mathscr S^\alpha g\|
\end{cases}
\end{equ}
Then, for any $a, b \geq 0$ such that $\max\{a,b\} < \max\{\alpha, \beta\}$ and $a + b < \alpha + \beta$, we have for every $f, g \in H$
\begin{equ} \label{e:Interpolation}
\|B(f,g)\|_V \lesssim_{a, b, \alpha, \beta} C^{\alpha + \beta - a - b} \, \|\mathscr S^a f\| \, \|\mathscr S^b g\| \, .
\end{equ}
\end{lemma}

\begin{proof}
For every $p \geq 0$, we define the dyadic partition of $f$ and $g$ as follows (thanks to the spectral theorem)
\begin{equ}
f_p := \1_{2^{-p-1} C < \mathscr S \leq 2^{-p} C} \, f \, , \quad g_p := \1_{2^{-p-1} C < \mathscr S \leq 2^{-p} C} \, g \, .
\end{equ}
By bilinearity and continuity of $B$, we see that $B(f,g)$ is the limit of $\sum_{0 \leq p,q \leq M} B(f_p, g_q)$ as $M \to \infty$, and thus we deduce
\begin{equ}
\|B(f,g)\|_V \leq \sum_{p,q = 0}^\infty \|B(f_p, g_q)\|_V \, .
\end{equ}
But using~\eqref{e:AssumptionInterpolation} and the definition of $f_p$ and $g_q$, we obtain
\begin{equ}
\|B(f_p, g_q)\|_V \lesssim_{\alpha, \beta}
\begin{cases}
 C^{\alpha+\beta} \, 2^{- \alpha p - \beta q} \, \|f_p\| \|g_q\| \\
 C^{\alpha+\beta} \, 2^{- \beta p - \alpha q} \, \|f_p\| \|g_q\|
\end{cases}
\end{equ}
from which it follows, for any $\theta \in [0,1]$ and $\kappa_1, \kappa_2 \geq 0$,
\begin{equs}
\|B(f_p, g_q)\|_V &\lesssim_{\alpha, \beta} C^{\alpha + \beta} 2^{- (\theta \alpha + (1 - \theta) \beta) p - ((1 - \theta) \alpha + \theta \beta) q} \|f_p\| \|g_q\| \, \\
&\lesssim_{\alpha, \beta, \theta, \kappa_1, \kappa_2} C^{\kappa_1 + \kappa_2} \, 2^{- \kappa_1 p - \kappa_2 q} \, \|\mathscr S^{\theta \alpha + (1 - \theta) \beta - \kappa_1} f_p\| \|\mathscr S^{(1 - \theta) \alpha + \theta \beta - \kappa_2} g_q\| \, .
\end{equs}
Then, we assume $\kappa_1, \kappa_2 > 0$ and use the Cauchy-Schwarz inequality to deduce
\begin{equs}
\|B(f, g)\|_V &\lesssim_{\alpha, \beta, \theta, \kappa_1, \kappa_2} C^{\kappa_1 + \kappa_2} \, \Big(\sum_{p \geq 0} 2^{- \kappa_1 p} \, \|\mathscr S^{\theta \alpha + (1 - \theta) \beta - \kappa_1} f_p\|\Big) \, \Big(\sum_{q \geq 0} 2^{- \kappa_2 q} \, \|\mathscr S^{(1 - \theta) \alpha + \theta \beta - \kappa_2} g_q\|\Big) \\
&\lesssim_{\kappa_1, \kappa_2} C^{\kappa_1 + \kappa_2} \, \|\mathscr S^{\theta \alpha + (1 - \theta) \beta - \kappa_1} f\| \, \|\mathscr S^{(1 - \theta) \alpha + \theta \beta - \kappa_2} g\| \,
\end{equs}
Finally, because of the assumptions on $a$ and $b$, it can be checked that there exist $\theta \in [0,1]$ and $\kappa_1, \kappa_2 > 0$ such that $a = \theta \alpha + (1 - \theta) \beta - \kappa_1$ and $b = (1-\theta) \alpha + \theta \beta - \kappa_2$, which concludes the proof of~\eqref{e:Interpolation}.
\end{proof}

\subsection{Hypercontractivity}

\label{sec:Hypercontractivity}

The purpose of this section is to recall and exploit the classical hypercontractivity property enjoyed by the semigroup $(e^{-t N})_{t \geq 0}$ (which, in the literature, is sometimes referred to as the ``noise operator''). Ultimately, we are interested in the following estimates (let us recall that the ``puissance $q$ du champ'' $\Gamma_q \psi$ was defined in~\eqref{e:PuissanceDuChamp}).

\begin{lemma} \label{lem:UltimateHypercontractivity}

There exists a family of constants $(\alpha(p))_{p > 2}$ (that depends on $\rho$) satisfying $\alpha(p) \to 1$ as $p \to 2$, such that the following statement holds. For any $\phi \in L^2(\nu_\rho)$ and local functional\footnote{In this section, we often restrict ourselves to local functionals in order to avoid any unnecessary technical difficulties. But, of course, it is possible to extend the upcoming estimates to much larger classes of functionals, simply by using standard approximation arguments.} $\psi \colon \Omega \to \R$, we have
\begin{equs}
\nu_\rho(|\phi|^p) &\leq \|\alpha(p)^N \phi\|^p \, , \label{e:UltHyper} \\
\nu_\rho(|\Gamma_2 \psi|^{p/2}) &\lesssim_p \|\alpha(p)^N (- S)^{1/2} \psi\|^p \, , \label{e:UltHyperGamma2} \\
\nu_\rho(\Gamma_{p} \psi) &\lesssim_p \|\alpha(p)^N \psi\|^{p-2}  \, \|\alpha(p)^N (-S)^{1/2} \psi\|^2 \, . \label{e:UltHyperGammap}
\end{equs}
\end{lemma}

Before proving Lemma~\ref{lem:UltimateHypercontractivity}, we give some details about hypercontractivity. First, observe that the semigroup $(e^{-t N})_{t \geq 0}$ acting on $\fockdo{}{} \simeq L^2(\nu_\rho)$ can be interpreted as the semigroup of the $\nu_\rho$-stationary Markov process on $\Omega = \{0,1\}^{\Z^2}$ which consists in resampling each bit according to a Poisson point process with intensity $\bigotimes_{x \in \Z^2} \mathrm{Leb}$. Then, it is classical (for more details, we refer to~\cite[Lemma 3.2, Theorem 3.5, Theorem A.1]{DSC96}) that this semigroup is hypercontractive, in the following sense: for every $1 < p \leq q < \infty$ and local functional $f \colon \Omega \to \R$, we have
\begin{equ} \label{e:HypercontractivityClassical}
\|e^{- t(p,q,\rho) N} f\|_{L^q(\nu_\rho)} \leq \|f\|_{L^p(\nu_\rho)} \, , \mbox{ where } t(p,q, \rho) := \frac{1}{4 \, C_{LS}(\rho)} \log\Big(\frac{q-1}{p-1}\Big) \, ,
\end{equ}
and where $C_{LS}(\rho) := \frac{1 - 2 \min\{\rho, 1 - \rho\}}{\log(1/\min\{\rho, 1 - \rho\} \ - \, 1)} \in (0, \infty)$ (with the convention $C_{LS}(1/2) := 1/2$) is the so-called log-Sobolev constant (of the $\cB(\rho)$-stationary Markov process on $\{0,1\}$ that resamples itself according to a Poisson point process with intensity $\mathrm{Leb}$). Obviously, we can trivially reformulate the hypercontractivity property~\eqref{e:HypercontractivityClassical} as follows: for every $1 < p \leq q < \infty$ and local functional $f \colon \Omega \to \R$, it holds
\begin{equ} \label{e:Hypercontractivity}
\|f\|_{L^q(\nu_\rho)} \leq \|\beta(p,q,\rho)^N f\|_{L^p(\nu_\rho)} \, , \mbox{ where } \beta(p,q, \rho) := \Big(\frac{q-1}{p-1}\Big)^{\frac{1}{4 \, C_{LS}(\rho)}} \, .
\end{equ}
We are especially interested in the case where either $p = 2$ or $q = 2$, so that we denote $\beta_2(r, \rho) := \beta(\min\{2,r\}, \max\{2,r\}, \rho)$ for every $r \in (1,\infty)$. Intuitively, the estimate~\ref{e:Hypercontractivity} tells us that we can reverse Jensen's inequality provided that we pay an exponential cost in the number operator $N$.
\ss

But let us stress that~\eqref{e:Hypercontractivity} only works for $p > 1$, which is stressed by the limit $\beta(p,q,\rho) \to \infty$ as $p \to 1$. However, for our purposes, we also need to be able to ``reverse Jensen's inequality'' when $p = 1$. In fact, the next lemma shows that this is still possible, provided that we assume that $f$ belongs to chaoses $\fockdo{0}{} \oplus \dots \oplus \fockdo{n}{}$ and replace the exponential cost in $N$ by an exponential cost in $n$.

\begin{lemma} \label{lem:HypercontractivityL1}
There exists a family of constants $(\beta_1(r))_{r \geq 1}$ (depending on $\rho$) satisfying $\beta_1(r) \to 1$ as $r \to 1$, such that the following statement holds. For any $r \geq 1$, $n \geq 0$, and local functional $f \colon \Omega \to \R$ belonging to $\bigoplus_{k=0}^n \fockdo{k}{}$, we have
\begin{equ} \label{e:HypercontractivityL1}
    \|f\|_{L^r(\nu_\rho)} \leq \beta_1(r)^n \, \|f\|_{L^1(\nu_\rho)} \, .
\end{equ}
\end{lemma}

It might seem counter-intuitive at first glance that we have $\beta_1(r) \to 1$ as $r \to 1$ while $\beta(p,q,\rho) \to \infty$ as $p \to 1$. Actually, the proof below shows that it works because the divergence of $\beta(p,q,\rho)$ as $p \to 1$ is sufficiently slow (it only grows as a power of $1/(p-1)$).

\begin{proof}[of Lemma~\ref{lem:HypercontractivityL1}]
First, we notice that for any $r \in (1,2]$ and $s \in [2,\infty)$, we have
\begin{equ} \label{e:HypercontractivityL2}
    \|f\|_{L^s(\nu_\rho)} \leq \beta_2(s)^n \, \|f\|_{L^2(\nu_\rho)} \, , \quad \|f\|_{L^2(\nu_\rho)} \leq \beta_2(r)^n \, \|f\|_{L^r(\nu_\rho)} \, .
\end{equ}
Indeed, this follows from~\eqref{e:Hypercontractivity} (applied respectively to $p = 2$, $q = s$, $f$ ; $p = r$, $q = 2$, $\beta_2(r)^{-N} f$), together with the trivial bounds $\|\beta^{N} f\|_{L^2(\nu_\rho)} \leq \beta^{n} \, \|f\|_{L^2(\nu_\rho)}$ and $\beta^{-n} \, \|f\|_{L^2(\nu_\rho)} \leq \|\beta^{-N} f\|_{L^2(\nu_\rho)}$ for any $\beta \geq 1$ (which follow from the orthogonality of the chaoses).
\ss

Now, we notice that if we prove~\eqref{e:HypercontractivityL1} for some $r_* \in (1, 2)$, say $r_* = 3/2$, then we deduce that~\eqref{e:HypercontractivityL1} also holds for any $r \geq 2$ by taking $\beta_1(r) := \beta_2(r) \, \beta_2(r_*) \, \beta_1(r_*)$ (to see this, bound the $L^r$-norm by the $L^{r_*}$-norm using~\eqref{e:HypercontractivityL2}). Hence, we can assume that $r \in (1, 2)$ without loss of generality. The key idea is the following: we first interpolate the $L^r$-norm between the $L^1$-norm and the $L^2$-norm (by Hölder's inequality) and then we apply~\eqref{e:HypercontractivityL2} to bound the $L^2$-norm by the $L^r$-norm, which can be reabsorbed. More precisely, we write
\begin{equ}
\|f\|_{L^r(\nu_\rho)}^r \leq \|f\|_{L^1(\nu_\rho)}^{2 - r} \, \|f\|_{L^2(\nu_\rho)}^{2 r - 2} \leq \beta_2(r)^{(2r - 2) n} \, \|f\|_{L^1(\nu_\rho)}^{2 - r} \, \|f\|_{L^r(\nu_\rho)}^{2r - 2} \, ,
\end{equ}
and then we divide both sides by $\|f\|_{L^r(\nu_\rho)}^{2 r - 2}$ and raise both sides to the power $1/(2-r)$, which gives
\begin{equ}
\|f\|_{L^r(\nu_\rho)} \leq \beta_2(r)^{\frac{2 r - 2}{2 - r} \, n} \, \|f\|_{L^1(\nu_\rho)} \, .
\end{equ}
Thus, we deduce that the estimate~\eqref{e:HypercontractivityL1} holds with $\beta_1(r) := \beta_2(r)^{\frac{2 r - 2}{2 - r}}$. Observe that, as $r \to 1$, we have $\beta_2(r) \to \infty$ and $\frac{2 r - 2}{2 - r} \to 0$, so we need to be careful to show that $\beta_1(r) \to 1$. But, recalling~\eqref{e:Hypercontractivity}, we have
\begin{equ}
\beta_1(r) = \exp\Big(\frac{1}{2 \, (2 - r) \, C_{LS}} \, (r - 1) \log\Big(\frac{1}{r-1}\Big)\Big) \, ,
\end{equ}
which indeed converges to $1$ as $r \to 1$. This concludes the proof.
\end{proof}

We are now ready to prove Lemma~\ref{lem:UltimateHypercontractivity}.

\begin{proof}[of Lemma~\ref{lem:UltimateHypercontractivity}]

We will show that the estimates~\eqref{e:UltHyper},~\eqref{e:UltHyperGamma2},~\eqref{e:UltHyperGammap} hold for any choice of $\alpha(p)$ such that
\begin{equ} \label{e:ChoiceAlpha(p)}
\alpha(p) \geq \beta_2(p) \mbox{ and } \alpha(p) > \beta_1(p/2) \, .
\end{equ}
Since $\beta_2(r) \to 1$ as $r \to 2$ (by the formula~\eqref{e:Hypercontractivity}) and $\beta_1(r) \to 1$ as $r \to 1$ (by Lemma~\ref{lem:HypercontractivityL1}), we can indeed find a family of constants $(\alpha(p))_{p > 2}$ satisfying~\eqref{e:ChoiceAlpha(p)} and such that $\alpha(p) \to 1$ as $p \to 2$.
\ss

First, by using~\eqref{e:Hypercontractivity} and that $\alpha(p) \geq \beta_2(p)$, we immediately deduce that the desired bound~\eqref{e:UltHyper} holds for any local functional $\phi$; but then an easy approximation argument allows us to extend the estimate to any $\phi \in L^2(\nu_\rho)$.
\ss

Now we tackle the estimates~\eqref{e:UltHyperGamma2} and~\eqref{e:UltHyperGammap} relative to the ``puissance $q$ du champ''. We first point out two basic properties of this object. First, we can bound the ``puissance $q$ du champ'' by its symmetric counterpart, as follows: by using the trivial bound $r(z) \leq 2 \, s(z)$, we deduce $\Gamma_q \psi \leq 2 \, \Gamma_q^s \psi$ where $\Gamma_q^s \psi$ is given by
\begin{equ} \label{e:PuissanceDuChampSymmetric}
\Gamma_q^{s} \psi(\eta) \eqdef \sum_{x, y \in \Z^2} s(y-x) \, \eta(x) (1 - \eta(y)) \, \big|\psi(\eta^{x,y}) - \psi(\eta)\big|^q = \frac{1}{2} \sum_{x, z \in \Z^2} s(z) \, \big|\psi(\eta^{x,x+z}) - \psi(\eta)\big|^q \, .
\end{equ}
Second, a standard computation shows that the expectation of the ``carré du champ'' is given by the Dirichlet form, as follows
\begin{equ} \label{e:MalliavinIPP}
\nu_\rho(\Gamma_2 \psi) = 2 \, \|(- S)^{1/2} \, \psi\|^2 \ \, \big(= \nu_\rho(\Gamma_2^s \psi)\big) \, .
\end{equ}

We turn to the proof of~\eqref{e:UltHyperGamma2}. We start from the bound $\Gamma_2 \psi \leq 2 \, \Gamma_2^s \psi$, then we write $\psi = \sum_{n \geq 0} \psi_n$ with $\psi_n \in \fockdo{n}{}$, and we use twice the generic bound $\big(\sum_{n \geq 0} u_n\big)^\theta \lesssim_{\theta, \eps} \sum_{n \geq 0} \, (1+n)^{\theta - 1 + \eps} \, u_n^\theta$ for $\theta \geq 1$ and $\eps > 0$ (use Hölder's inequality), which results in
\begin{equ}
|\Gamma_2 \psi|^{p/2} \lesssim_p |\Gamma_2^s \psi|^{p/2} \lesssim_{p, \eps} \sum_{n \geq 0} (1+n)^{\frac{p}{2} (1 + \eps) + \frac{p}{2} - 1 + \eps} \, |\Gamma_2^s \psi_n|^{p/2} = \sum_{n \geq 0} (1+n)^{\frac{3 p}{2}} \, |\Gamma_2^s \psi_n|^{p/2} \, ,
\end{equ}
where for the last equality we decided to take $\eps = 1$. Now, we notice that $\Gamma_2^s \psi_n \in \bigoplus_{0 \leq k \leq 2n} \fockdo{k}{}$ and apply Lemma~\ref{lem:HypercontractivityL1} (with $2 n$ and $p/2$), which leads to
\begin{equ}
\nu_\rho\big(|\Gamma_2^s \psi_n|^{p/2}\big) \leq \beta_1(p/2)^{p n} \, \nu_\rho\big(\Gamma_2^s \psi_n\big)^{p/2} \lesssim_p \beta_1(p/2)^{p n} \, \|(- S)^{1/2} \, \psi_n\|^p \, ,
\end{equ}
where the last equality comes from~\eqref{e:MalliavinIPP}. Thus, we deduce by the orthogonality of chaoses (and the general bound $\sum_{n \geq 0} u_n^p \leq \big(\sum_{n \geq 0} u_n^2\big)^{p/2}$ for $p \geq 2$) that
\begin{equs}
\nu_\rho\big(|\Gamma_2 \psi|^{p/2}\big) \lesssim_p \sum_{n \geq 0} (1 + n)^{\frac{3 p}{2}} \beta_1(p/2)^{p n} \|(- S)^{1/2} \psi_n\|^p
&\leq \|(1 + N)^{3/2} \beta_1(p/2)^{N} \, (- S)^{1/2} \psi\|^p \\
&\lesssim_{\alpha(p), p} \|\alpha(p)^N (- S)^{1/2} \psi\|^p \, , \label{e:BoundGamma2}
\end{equs}
where in the last bound we used that $\alpha(p) > \beta_1(p/2)$. This shows the desired bound~\eqref{e:UltHyperGamma2}.

We continue with the proof of~\eqref{e:UltHyperGammap}. We start from $\Gamma_p \psi \leq 2 \, \Gamma_p^s \psi$ and we apply~\eqref{e:Hypercontractivity}, which leads to (we use the notation $\nabla^{x,y} F(\eta) := F(\eta^{x,y}) - F(\eta)$ for conciseness; note that $\nabla^{x,y}$ commutes with $N$)
\begin{equ}
\nu_\rho(\Gamma_p \psi) \leq 2 \, \nu_\rho(\Gamma_p^s \psi) = \sum_{x, z \in \Z^2} s(z) \, \nu_\rho(|\nabla^{x,x+z} \psi|^p) \leq \sum_{x, z \in \Z^2} s(z) \, \nu_\rho(|\nabla^{x,x+z} \beta_2(p)^N \psi|^{2})^{p/2} \, .
\end{equ}
Now, we split the last $\nu_\rho(...)^{p/2}$ as $\nu_\rho(...)^{p/2 - 1} \times \nu_\rho(...)$, then we bound the first factor by using the trivial bound $\|\nabla^{x,y} \cdot\| \leq 2 \, \|\cdot\|$, and finally we notice that the remaining sum equals $2 \, \nu_\rho(\Gamma_2^s \beta_2(p/2)^N \psi)$ and use~\eqref{e:MalliavinIPP}. This results in
\begin{equs}
\nu_\rho(\Gamma_p \psi) &\lesssim_p \|\beta_2(p)^N \psi\|^{p - 2} \sum_{x, z \in \Z^2} s(z) \, \nu_\rho(|\nabla^{x,x+z} \beta_2(p)^{N} \psi|^{2}) \\
&\lesssim \|\beta_2(p)^N \psi\|^{p - 2} \, \|\beta_2(p)^{N} (- S)^{1/2} \psi\| \\
&\leq \|\alpha(p)^N \psi\|^{p - 2} \, \|\alpha(p)^{N} (- S)^{1/2} \psi\| \, , \label{e:BoundGammap}
\end{equs}
where in the third line we used that $\alpha(p) \geq \beta_2(p)$. This shows~\eqref{e:UltHyperGammap} and concludes the proof.
\end{proof}

\subsection{Convergence of the Bernoulli chaos expansion to the Gaussian chaos expansion}

\label{app:Cyl}

The main purpose of this section is to prove Lemma~\ref{lem:CVchaos}.
We start by proving the following useful subgaussian estimates.

\begin{lemma} \label{lem:Subgaussian}
For every $\tau \geq 1$, $t_0, t \geq 0$, $\vphi \in \cS(\R^2)$ and $\theta \in \R$ we have
\begin{equ} \label{e:Subgaussian}
\E[\exp(\theta \zeta(\vphi))] \vee \nu_\rho[\exp(\theta Z_t^\tau(\vphi))] \vee \nu_\rho[E^{\lim}_{Z_{t_0}^\tau(\eta)}[\exp(\theta Y_t^{\lim}(\vphi))]] \leq \exp(c \, \|\vphi\|_{L^2 L_{\mathrm{loc}}^\infty}^2 \, \theta^2/2) \, ,
\end{equ}
where the norm $\|\cdot\|_{L^2 L_{\mathrm{loc}}^\infty(\R^2)}$ on $\cS(\R^2)$ is defined as
\begin{equ} \label{e:L2Linfty}
\|\vphi\|_{L^2 L_{\mathrm{loc}}^\infty(\R^2)} \eqdef \Big(\int_{\R^2} \dd x \, \sup_{y \in B(x, 10)} |\vphi(y)|^2\Big)^{1/2} \, .
\end{equ}
\end{lemma}

\begin{proof}[of Lemma~\ref{lem:Subgaussian}]

First, the subgaussian estimate for $\zeta$ is trivial since $\zeta(\vphi) \in \cN(0,\|\vphi\|_{L^2(\R^2)}^2)$.

Next, we show the subgaussian estimates for $Z_t^\tau$ under the probability measure $\nu_\rho$. Observe that
\begin{equ}
Z_t^\tau(\vphi) = \det(\Sigma^\tau)^{-1/2} \sum_{x \in \Z^2} \vphi\big((\Sigma^\tau)^{-1} (x - \tau t v)\big) \, \sigma(x) \, ,
\end{equ}
and that the random variables $(\sigma(x))_{x \in \Z^2}$ are iid and $c$-subgaussian for some $c = c(\rho) \in (0,\infty)$. We deduce that $Y^\tau(\vphi)$ is subgaussian with parameter
\begin{equ}
c \, \det(\Sigma^\tau)^{-1} \sum_{x \in \Z^2} \vphi\big((\Sigma^\tau)^{-1} (x - \tau t v)\big)^2 \leq c \, \|\vphi\|_{L^2 L^\infty(\R^2)}^2 \, ,
\end{equ}
where the last bound comes from recognizing a Riemann sum (more precisely, we used that every cell of the affine lattice $(\Sigma^\tau)^{-1} (\Z^2 - \tau t v)$ has volume $\det(\Sigma^\tau)^{-1}$ and diameter less than $10$).
\ss

Finally, we show the subgaussian estimates for $Y_t^{\lim}$. By using~\eqref{e:YlimExplicit}, we can write, under the probability measure $\nu_\rho \, \mathbf E^{\lim}_{Z_{t_0}^\tau}$,
\begin{equ}
Y_t^{\lim}(\vphi) = Z_{t_0}^\tau(K^{\lim}_t \vphi) + \int_0^t \vec{\xi}(\dd s) \, [K_{t-s}^{\lim} \, \sqrt{D^{\lim}} \nabla \vphi] \, .
\end{equ}
Observe first that the two summands on the right hand side are independent. On the one hand, we just showed that the first summand is $c \, \|\vphi\|_{L^2 L^\infty(\R^2)}^2$-subgaussian (using in addition the general bound $\|k * \vphi\|_{L^2 L^{\infty}(\R^2)} \leq \|k\|_{L^1(\R^2)} \|\vphi\|_{L^2 L^{\infty}(\R^2)}$, with $k$ being the kernel of $K_t^{\lim}$). On the other hand, the second summand is a centred Gaussian random variable whose variance is bounded by $\int_0^\infty \dd r \, \|K_r^{\lim} \sqrt{D^{\lim}} \nabla \vphi\|_{L^2(\R^2)}^2 = \|\vphi\|_{L^2(\R^2)}^2$. Therefore we deduce
\begin{equ}
\nu_\rho[E^{\lim}_{Y^\tau}[\exp(\theta Y^{\lim}(\vphi))]] \leq \exp((c+1) \|\vphi\|_{L^2 L^\infty(\R^2)}^2/2) \, ,
\end{equ}
which concludes the proof.
\end{proof}

Now, we turn to the proof of Lemma~\ref{lem:CVchaos}.

\begin{proof}[of Lemma~\ref{lem:CVchaos}]

{\noindent \it Step 1. Reduction to $F(w) = H_n(w(h), \|h\|_{L^2(\R^2)})$}

We first show that it is sufficient to prove Lemma~\ref{lem:CVchaos} for polynomial cylinder functionals $F$, that is to say of the form~\eqref{e:CylinderFunctional} but with a polynomial function $f$. Let $F$ be a cylinder functional of the form~\eqref{e:CylinderFunctional}. Since $f \in \cC(\R^m, \R)$ grows at most polynomially fast, it is classical that there exists a sequence of polynomial functions $(p_n)_{n \geq 1}$ such that, as $n \to \infty$,
\begin{equ}
\eps_n := \sup_{y \in \R^m} |f(y) - p_n(y)| \exp(- (|y_1| + \dots + |y_m|)) \to 0 \, .
\end{equ}
We denote by $P_n$ the polynomial cylinder functional $P_n(w) := p_n(w(\vphi_1), \dots, w(\vphi_m))$. Then, by using Lemma~\ref{lem:Subgaussian} and the contractivity of $\fsi_t^\tau \colon L^2(\mathbb P) \to L^2(\nu_\rho)$, we immediately deduce
\begin{equ}
\|F(Z_t^\tau) - P_n(Z_t^\tau)\|_{L^2(\nu_\rho)} \lesssim \eps_n \, , \quad \|\fsi_t^\tau (F(\zeta) - P_n(\zeta))\|_{L^2(\nu_\rho)} \lesssim \eps_n \, .
\end{equ}
Thus, if Lemma~\ref{lem:CVchaos} holds for each $P_n$ then we immediately deduce that it also holds for $F$. Consequently, it is sufficient to show Lemma~\ref{lem:CVchaos} for polynomial cylinder functionals.
\ss

Furthermore, by using a linearity and polarisation argument, we can reduce the proof of Lemma~\ref{lem:CVchaos} to polynomial cylinder functionals $F$ of the form $F(w) = H_n(w(h), \|h\|_{L^2(\R^2)})$ for $n \geq 0$ and $h \in \cS(\R^2)$. Since $F(\zeta) = \fI(h^{\otimes n})$ by definition, we are left to show that
\begin{equ} \label{e:CVchaosGoal1}
\|H_n(Z_t^\tau(h), \|h\|_{L^2(\R^2)}) - \fsi_t^\tau h^{\otimes n}\|_{L^2(\nu_\rho)} \xrightarrow[\tau \to \infty]{} 0 \, .
\end{equ}

{\noindent \it Step 2. Replacement of $\fsi_t^\tau h^{\otimes n}$ by $\restriction{(h_t^\tau)^{\otimes n}}{\Z^{2n} \cap HC}$}

For conciseness, we denote $h_t^\tau := \fU_t^\tau h$ (note that $\fU_t^\tau h^{\otimes n} = (h_t^\tau)^{\otimes n}$). We claim that
\begin{equ} \label{e:RemovalPRn}
\|\fsi_t^\tau h^{\otimes n} - \restriction{(h_t^\tau)^{\otimes n}}{\Z^{2n} \cap HC}\| \xrightarrow[\tau \to \infty]{} 0 \, .
\end{equ}
By recalling that $\fsi_t^\tau = \fsR \fsP \fU_t^\tau$ and using that $\fsR \colon \fockd{}{} \to \fockdo{}{}$ is a contraction, it is sufficient to show that $\|\fsP (h_t^\tau)^{\otimes n} - \restriction{(h_t^\tau)^{\otimes n}}{\Z^{2n}}\|_{\ell^2(\Z^{2n})} \to 0$. 
Observe that, by using the Poisson summation formula~\eqref{e:PoissonSummation}, it holds for any $k_{1:n} \in (- \pi, \pi)^{2n} := \square_n$
\begin{equ}
\cF (\restriction{(h_t^\tau)^{\otimes n}}{\Z^{2n}} - \fsP (h_t^\tau)^{\otimes n}) \, (k_{1:n}) = \sum_{\ell_{1:n} \in 2 \pi \Z^{2n} \setminus \{0\}} (\fF h_t^\tau)^{\otimes n} (k_{1:n} + \ell_{1:n}) \, ,
\end{equ}
In particular, we deduce that, for any $m \geq n + 1$,
\begin{equs}
\|\restriction{(h_t^\tau)^{\otimes n}}{\Z^{2n}} - \fsP (h_t^\tau)^{\otimes n}\|_{\ell^2(\Z^{2n})} &\leq \sum_{\ell_{1:n} \in 2 \pi \Z^{2n} \setminus \{0\}} \|\fF h_t^\tau\|_{L^2(\square_1 + \ell_1)} \dots \|\fF h_t^\tau\|_{L^2(\square_1 + \ell_n)} \\
&\lesssim_{n,m} \sum_{\ell_{1:n} \in 2 \pi \Z^{2n} \setminus \{0\}} |\ell_{1:n}|_\infty^{-2 m} \, \|(- \fL_0)^m \, h_t^\tau\| \, \|h_t^\tau\|^{n-1} \\
&\lesssim_{n, m} \|(- \fL_0)^m \, h_t^\tau\| \, \|h_t^\tau\|^{n-1} \, ,
\end{equs}
where in the second line we used the bound $\|\fF h_t^\tau\|_{L^2(\square_1 + \ell_i)} \lesssim_{m} |\ell_i|^{-2m} \, \|(- \fL_0)^m \, h_t^\tau\|$ with an index $i$ such that $|\ell_i| = |\ell_{1:n}|_\infty$. Now, note that by definition of $\fU_t^\tau$ and the bound $|(\Sigma^\tau)^{-1} \, k|^2 \leq \tau^{-1} \, |k|^2$, we have $\|h_t^\tau\| \leq \|h\|$ and $\|(- \fL_0)^m h_t^\tau\| \leq \tau^{-m} \, \|(- \fL_0)^m h\|$, which gives
\begin{equs}
\|\restriction{(h_t^\tau)^{\otimes n}}{\Z^{2n}} - \fsP (h_t^\tau)^{\otimes n}\|_{\ell^2(\Z^{2n})} &\lesssim_{n, m, h} \tau^{-m} \, ,
\end{equs}
and concludes the proof of~\eqref{e:RemovalPRn}.
\ss

Consequently, it only remains to show that
\begin{equ} \label{e:CVchaosGoal2}
\|H_n(Z_t^\tau(h), \|h\|_{L^2(\R^2)}) - \restriction{(h_t^\tau)^{\otimes n}}{\Z^{2n} \cap HC}\|_{L^2(\nu_\rho)} \xrightarrow[\tau \to \infty]{} 0 \, .
\end{equ}

{\noindent \it Step 3. The combinatorial analysis}

We first introduce some combinatorial notation and facts, that can be found in~\cite[Section 2.5]{peccati2011wiener}. Let $\Pi_n$ be the set of partitions of $\{1, \dots, n\}$, equipped with the refinement order: that is $\pi \leq \pi'$ if any block of $\pi$ is included in a block of $\pi'$. If $\pi \in \Pi_n$ and $1 \leq k \leq n$, then we denote by $\pi_k$ the set of blocks of size $k$ of $\pi$. Let $\hat{0} \in \Pi_n$ be the singleton partition $\{\{1\}, \dots, \{n\}\}$. Let $\mathrm{Pair}_n \subset \Pi_n$ be the set of partial pairings, that is partitions whose blocks are of size $1$ or $2$. The Möbius function corresponding to the poset $(\Pi_n, \leq)$ is  given by
\begin{equ}
\mu(\pi) \eqdef \prod_{B \in \pi} (-1)^{|B|-1} \, (|B|-1)! \, .
\end{equ}
Note that, if $\pi \in \mathrm{Pair}_n$, then $\mu(\pi) = (-1)^{\# \pi_2}$. Since $\hat{0}$ is the minimum element of $\Pi_n$, the Möbius inversion formula gives that, for any $\pi \in \Pi_n$,
\begin{equ} \label{e:MobiusInv}
\1_{\pi = \hat{0}} = \sum_{\pi' \in \Pi_n} \mu(\pi') \, \1_{\pi' \leq \pi} \, .  
\end{equ}
Finally, if $x_{1:n} \in \R^{2n}$, then we denote by $p(x_{1:n}) = \hat{0}$ the corresponding ``equality'' partition, that is $i$ and $j$ belong to the same block if and only if $x_i = x_j$. Observe that $p(x_{1:n}) = \hat{0}$ if and only if $x_{1:n}$ satisfies the hard core condition; note also that for any given $\pi \in \Pi_n$, it holds $p(x_{1:n}) \geq \pi$ if and only if $(x_i)_{1 \leq i \leq n}$ is constant on each block of $\pi$.
\ms

Now, we come back to~\eqref{e:CVchaosGoal2} and rewrite the two appearing summands. On the one hand, by recalling $Z_t^\tau(h) = \sum_{x \in \Z^2} h_t^\tau(x) \, \sigma(x)$ and the formula~\eqref{e:Hermite} for the Hermite polynomial (recall also that $c_{k,n}$ is the number of partial pairings with $k$ blocks of size $2$ and $n-2k$ blocks of size $1$), we have
\begin{equs}
H_n(Z_t^\tau(h), \|h\|_{L^2(\R^2)}) &= \sum_{0 \leq k \leq n/2} c_{k,n} \, (-1)^{k} \, \|h\|_{L^2(\R^2)}^{2 k} \, \Big(\sum_{x \in \Z^2} h_t^\tau(x) \, \sigma(x)\Big)^{n - 2k} \\
&= \sum_{\pi \in \Pi_n} \1_{\pi \in \mathrm{Pair}_n} \, \mu(\pi) \, \|h\|_{L^2(\R^2)}^{2 \# \pi_2} \, \Big(\sum_{x \in \Z^2} h_t^\tau(x) \, \sigma(x)\Big)^{\# \pi_1} \, . \label{e:HermiteZ}
\end{equs}
On the other hand, by using that $x_{1:n} \in HC$ if and only if $p(x_{1:n}) = \hat{0}$, we have
\begin{equ}
I(\restriction{(h_t^\tau)^{\otimes n}}{\Z^{2n} \cap HC}) = \sum_{x_{1:n} \in \Z^{2n}} \1_{p(x_{1:n}) = \hat{0}} \, h_t^\tau(x_1) \dots h_t^\tau(x_n) \, \sigma(x_1) \dots \sigma(x_n) \, .
\end{equ}
By using the Möbius inversion formula~\eqref{e:MobiusInv} to rewrite $\1_{p(x_{1:n}) = 0}$ (recall also the couple of observations we made just after~\eqref{e:MobiusInv}), we end up with
\begin{equ} \label{e:ItauhMobius}
I(\restriction{(h_t^\tau)^{\otimes n}}{\Z^{2n} \cap HC}) = \sum_{\pi \in \Pi_n} \mu(\pi) \, \prod_{B \in \pi} \Big(\sum_{x \in \Z^2} h_t^\tau(x)^{\# B} \sigma(x)^{\# B}\Big) \, .
\end{equ}

For conciseness, we denote $S_{t,\tau}^{(k)} := \sum_{x \in \Z^2} h_t^\tau(x)^{k} \, \sigma(x)^{k}$ for any $k \in \{1, \dots, n\}$. By comparing the two formulas~\eqref{e:ItauhMobius} and~\eqref{e:HermiteZ}, we deduce that, in order to prove~\eqref{e:CVchaosGoal2}, it only remains to show that, for any $\pi \in \Pi_n$, it holds
\begin{equ} \label{e:CVchaosGoal3}
\Big\|\prod_{B \in \pi} S_{t, \tau}^{(\# B)} - \1_{\pi \in \mathrm{Pair}_n} \, \big(S_{t,\tau}^{(1)}\big)^{\# \pi_1} \|h\|_{L^2(\R^2)}^{2 \# \pi_2} \Big\|_{L^2(\nu_\rho)} \xrightarrow[\tau \to \infty]{} 0 \, .
\end{equ}

{\noindent \it Step 4. Conclusion}

First, we observe that, by using Hölder's inequality as follows
\begin{equ}
\|f_1 \dots f_m - g_1 \dots g_m\|_{L^2} \leq \sum_{i=1}^m \|f_1\|_{L^{2m}} \dots \|f_{i-1}\|_{L^{2m}} \|f_i - g_i\|_{L^{2m}} \|g_{i+1}\|_{L^{2m}} \dots \|g_{m}\|_{L^{2m}} \, ,
\end{equ}
we can reduce the proof of~\eqref{e:CVchaosGoal3} to the proof of the following estimate: for every $p \geq 1$, as $\tau \to \infty$, it holds
\begin{equ} \label{e:Sktau}
\|S_{t,\tau}^{(1)}\|_{L^p(\nu_\rho)} = O(1) \, , \quad \|S_{t,\tau}^{(2)} - \|h\|_{L^2(\R^2)}^2\|_{L^p(\nu_\rho)} \to 0 \, , \quad \|S_{t,\tau}^{(k)}\|_{L^p(\nu_\rho)} \xrightarrow[]{} 0 \mbox{ if } k \geq 3 \, .
\end{equ}

We start with the following key observation: by using the definition of $\fU_t^\tau$ and recognizing a Riemann sum, it is easily checked that, for any $k \geq 1$,
\begin{equ} \label{e:RiemannLk}
    \sum_{x \in \Z^2} |h_t^\tau(x)|^k \underset{\tau \to \infty}{\sim} \det(\Sigma^\tau)^{1 - k/2} \, \|h\|_{L^k(\R^2)}^k \, .
\end{equ}
First, by using~\eqref{e:RiemannLk} and the boundedness of $\sigma$, we deduce that, for any $k \geq 3$,
\begin{equ}
    \|S_{t, \tau}^{(k)}\|_{L^\infty(\nu_\rho)} \lesssim_k \sum_{x \in \Z^2} |h_t^\tau(x)|^k \xrightarrow[\tau \to \infty]{} 0 \, ,
\end{equ}
which shows the third statement of~\eqref{e:Sktau}. Second, by using classical subgaussian estimates for bounded centred iid random variables and~\eqref{e:RiemannLk}, we deduce 
\begin{equs}
    \Big\|\sum_{x \in \Z^2} h_t^\tau(x) \, \sigma(x)\Big\|_{L^p(\nu_\rho)} &\lesssim_p \Big(\sum_{x \in \Z^2} h_t^\tau(x)^2\Big)^{1/2} \lesssim_h 1 \, , \\
    \Big\|\sum_{x \in \Z^2} h_t^\tau(x)^2 \, (\sigma(x)^2 - 1)\Big\|_{L^p(\nu_\rho)} &\lesssim_p \Big(\sum_{x \in \Z^2} h_t^\tau(x)^4\Big)^{1/2} \xrightarrow[\tau \to \infty]{} 0 \, .
\end{equs}
By using also that $\sum_{x \in \Z^2} h_t^\tau(x)^2 \to \|h\|_{L^2(\R^2)}^2$ by~\eqref{e:RiemannLk}, we deduce the first two statements of~\eqref{e:Sktau}. This concludes the proof.
\end{proof}

\end{appendix}

\bibliography{refs}

\bibliographystyle{Martin}

\end{document}